\documentclass[a4paper]{amsart}
\usepackage{amssymb,amsthm}
\usepackage[abbrev,nobysame]{amsrefs}

\theoremstyle{plain}
\newtheorem{theorem}{Theorem}[section]
\newtheorem{lemma}[theorem]{Lemma}
\newtheorem{corollary}[theorem]{Corollary}

\theoremstyle{definition}

\newtheorem{assumption}[theorem]{Assumption}

\theoremstyle{remark}
\newtheorem{remark}[theorem]{Remark}

\numberwithin{equation}{section}

\begin{document}

\title[Navier--Stokes equations in a curved thin domain, Part I]{Navier--Stokes equations in a curved thin domain, Part I: uniform estimates for the Stokes operator}

\author[T.-H. Miura]{Tatsu-Hiko Miura}
\address{Department of Mathematics, Kyoto University, Kitashirakawa Oiwake-cho, Sakyo-ku, Kyoto 606-8502, Japan}
\email{t.miura@math.kyoto-u.ac.jp}



\begin{abstract}
  In the series of this paper and the forthcoming papers \cites{Miu_NSCTD_02,Miu_NSCTD_03} we study the Navier--Stokes equations in a three-dimensional curved thin domain around a given closed surface under Navier's slip boundary conditions.
  We focus on the study of the Stokes operator for the curved thin domain in this paper.
  The uniform norm equivalence for the Stokes operator and a uniform difference estimate for the Stokes and Laplace operators are established in which constants are independent of the thickness of the curved thin domain.
  To prove these results we show a uniform Korn inequality and a uniform a priori estimate for the vector Laplace operator on the curved thin domain based on a careful analysis of vector fields and surface quantities on the boundary.
  We also present examples of curved thin domains and vector fields for which the uniform Korn inequality is not valid but a standard Korn inequality holds with a constant that blows up as the thickness of a thin domain tends to zero.
\end{abstract}

\maketitle

\section{Introduction} \label{S:Intro}
\subsection{Problem and main results} \label{SS:Int_Pro}
Let $\Gamma$ be a closed surface in $\mathbb{R}^3$ with unit outward normal vector field $n$.
Also, let $g_0$ and $g_1$ be functions on $\Gamma$ satisfying
\begin{align*}
  g := g_1-g_0 \geq c \quad\text{on}\quad \Gamma
\end{align*}
with some constant $c>0$.
For a sufficiently small $\varepsilon>0$ we define a curved thin domain $\Omega_\varepsilon$ in $\mathbb{R}^3$ with small thickness of order $\varepsilon$ by
\begin{align} \label{E:CTD_Intro}
  \Omega_\varepsilon := \{y+rn(y) \mid y\in\Gamma,\,\varepsilon g_0(y)<r<\varepsilon g_1(y)\}
\end{align}
and write $\Gamma_\varepsilon:=\Gamma_\varepsilon^0\cup\Gamma_\varepsilon^1$ and $n_\varepsilon$ for the boundary of $\Omega_\varepsilon$ and its unit outward normal vector field, where $\Gamma_\varepsilon^0$ and $\Gamma_\varepsilon^1$ are the inner and outer boundaries given by
\begin{align*}
  \Gamma_\varepsilon^i := \{y+\varepsilon g_i(y)n(y) \mid y\in\Gamma\}, \quad i=0,1.
\end{align*}
In the series of this paper and the forthcoming papers \cites{Miu_NSCTD_02,Miu_NSCTD_03} we consider the Navier--Stokes equations with Navier's slip boundary conditions
\begin{align} \label{E:NS_CTD}
  \left\{
  \begin{alignedat}{3}
    \partial_tu^\varepsilon+(u^\varepsilon\cdot\nabla)u^\varepsilon-\nu\Delta u^\varepsilon+\nabla p^\varepsilon &= f^\varepsilon &\quad &\text{in} &\quad &\Omega_\varepsilon\times(0,\infty), \\
    \mathrm{div}\,u^\varepsilon &= 0 &\quad &\text{in} &\quad &\Omega_\varepsilon\times(0,\infty), \\
    u^\varepsilon \cdot n_\varepsilon &= 0 &\quad &\text{on} &\quad &\Gamma_\varepsilon\times(0,\infty), \\
    [\sigma(u^\varepsilon,p^\varepsilon)]_{\mathrm{tan}}+\gamma_\varepsilon u^\varepsilon &= 0 &\quad &\text{on} &\quad &\Gamma_\varepsilon\times(0,\infty), \\
    u^\varepsilon|_{t=0} &= u_0^\varepsilon &\quad &\text{in} &\quad &\Omega_\varepsilon.
  \end{alignedat}
  \right.
\end{align}
Here $\nu>0$ is the viscosity coefficient independent of $\varepsilon$ and $\gamma_\varepsilon\geq0$ is the friction coefficient on $\Gamma_\varepsilon$ given by
\begin{align} \label{E:Def_Fric}
  \gamma_\varepsilon := \gamma_\varepsilon^i \quad\text{on}\quad \Gamma_\varepsilon^i,\,i=0,1,
\end{align}
where $\gamma_\varepsilon^0$ and $\gamma_\varepsilon^1$ are nonnegative constants depending on $\varepsilon$.
Also, we denote by
\begin{align*}
  \sigma(u^\varepsilon,p^\varepsilon) := 2\nu D(u^\varepsilon)-p^\varepsilon I_3, \quad [\sigma(u^\varepsilon,p^\varepsilon)n_\varepsilon]_{\mathrm{tan}} := P_\varepsilon[\sigma(u^\varepsilon,p^\varepsilon)n_\varepsilon]
\end{align*}
the stress tensor and the tangential component of the stress vector on $\Gamma_\varepsilon$, where $I_3$ is the $3\times 3$ identity matrix, $n_\varepsilon\otimes n_\varepsilon$ is the tensor product of $n_\varepsilon$ with itself, and
\begin{align*}
  D(u^\varepsilon) := \frac{\nabla u^\varepsilon+(\nabla u^\varepsilon)^T}{2}, \quad P_\varepsilon := I_3-n_\varepsilon\otimes n_\varepsilon
\end{align*}
are the strain rate tensor and the orthogonal projection onto the tangent plane of $\Gamma_\varepsilon$.
Note that $[\sigma(u^\varepsilon,p^\varepsilon)n_\varepsilon]_{\mathrm{tan}}=2\nu P_\varepsilon D(u^\varepsilon)n_\varepsilon$ is independent of the pressure $p^\varepsilon$ and the slip boundary conditions can be expressed as
\begin{align} \label{E:Slip_Intro}
  u^\varepsilon\cdot n_\varepsilon = 0, \quad 2\nu P_\varepsilon D(u^\varepsilon)n_\varepsilon+\gamma_\varepsilon u^\varepsilon = 0 \quad\text{on}\quad \Gamma_\varepsilon.
\end{align}
Hereafter we mainly refer to \eqref{E:Slip_Intro} as the slip boundary conditions.

The aims of our study are to establish the global-in-time existence of a strong solution to \eqref{E:NS_CTD} for large data and to study the behavior of the strong solution as $\varepsilon\to0$.
In this paper, however, we focus on the study of the Stokes operator $A_\varepsilon$ associated with the Stokes problem in $\Omega_\varepsilon$ under the slip boundary conditions
\begin{align} \label{E:Stokes_CTD}
  \left\{
  \begin{alignedat}{3}
    -\nu\Delta u+\nabla p &= f &\quad &\text{in} &\quad &\Omega_\varepsilon, \\
    \mathrm{div}\,u &= 0 &\quad &\text{in} &\quad &\Omega_\varepsilon, \\
    u\cdot n_\varepsilon &= 0 &\quad &\text{on} &\quad &\Gamma_\varepsilon, \\
    2\nu P_\varepsilon D(u)n_\varepsilon+\gamma_\varepsilon u &= 0 &\quad &\text{on} &\quad &\Gamma_\varepsilon
  \end{alignedat}
  \right.
\end{align}
and provide fundamental results on $A_\varepsilon$ for the aims of our study.
The goal of this paper is to show the uniform norm equivalence for $A_\varepsilon$ and its square root
\begin{align} \label{E:NE_Intro}
  c^{-1}\|u\|_{H^k(\Omega_\varepsilon)} \leq \|A_\varepsilon^{k/2} u\|_{L^2(\Omega_\varepsilon)} \leq c\|u\|_{H^k(\Omega_\varepsilon)}, \quad u\in D(A_\varepsilon^{k/2}),\, k=1,2
\end{align}
and the uniform difference estimate for $A_\varepsilon$ and $-\nu\Delta$ of the form
\begin{align} \label{E:Diff_Intro}
  \|A_\varepsilon u+\nu\Delta u\|_{L^2(\Omega_\varepsilon)} \leq c\|u\|_{H^1(\Omega_\varepsilon)}, \quad u\in D(A_\varepsilon)
\end{align}
with a constant $c>0$ independent of $\varepsilon$ (see Section \ref{S:Main} for the precise statements).

The estimates \eqref{E:NE_Intro} and \eqref{E:Diff_Intro} play a fundamental role in the second part \cite{Miu_NSCTD_02} of our study.
In \cite{Miu_NSCTD_02} we prove the global existence of a strong solution $u^\varepsilon$ to \eqref{E:NS_CTD} for large data $u_0^\varepsilon$ and $f^\varepsilon$ such that
\begin{align*}
  \|u_0^\varepsilon\|_{H^1(\Omega_\varepsilon)},\, \|f^\varepsilon\|_{L^\infty(0,\infty;L^2(\Omega_\varepsilon))} = O(\varepsilon^{-1/2})
\end{align*}
when $\varepsilon$ is sufficiently small.
We also derive estimates for $u^\varepsilon$ with constants explicitly depending on $\varepsilon$ which are essential for the last paper \cite{Miu_NSCTD_03}.
To get the global existence we show that the $L^2(\Omega_\varepsilon)$-norm of $A_\varepsilon^{1/2}u^\varepsilon$ is bounded uniformly in time by a standard energy method.
A key tool for the proof is a good estimate for the trilinear term
\begin{align*}
  \bigl((u\cdot\nabla)u,A_\varepsilon u\bigr)_{L^2(\Omega_\varepsilon)}, \quad u\in D(A_\varepsilon)
\end{align*}
which implies a differential inequality in time for the $L^2(\Omega_\varepsilon)$-norm of $A_\varepsilon^{1/2}u^\varepsilon$ similar to the one for the two-dimensional Navier--Stokes equations.
To derive that estimate we require \eqref{E:NE_Intro} and \eqref{E:Diff_Intro}.
Note that we have the $H^1(\Omega_\varepsilon)$-norm of $u$, not its $H^2(\Omega_\varepsilon)$-norm, in the right-hand side of \eqref{E:Diff_Intro}.
This fact is important in order to get a good estimate for the trilinear term.

Let us also mention the last part \cite{Miu_NSCTD_03} of our study.
We consider the thin-film limit for \eqref{E:NS_CTD} and study the behavior of the strong solution $u^\varepsilon$ as $\varepsilon\to0$ in \cite{Miu_NSCTD_03}.
Using the results of this paper and \cite{Miu_NSCTD_02} we show that the average in the thin direction of $u^\varepsilon$ converges on $\Gamma$ as $\varepsilon\to0$.
Moreover, we derive limit equations on $\Gamma$ for \eqref{E:NS_CTD} by characterizing the limit of the average of $u^\varepsilon$ as a solution to the limit equations.
When the thickness of $\Omega_\varepsilon$ is $\varepsilon$ (i.e. $g\equiv1$) and there is no friction between the fluid and the boundary $\Gamma_\varepsilon$ (i.e. $\gamma_\varepsilon=0$), the limit equations derived in \cite{Miu_NSCTD_03} agree with the Navier--Stokes equations on a Riemannian manifold
\begin{align} \label{E:SNS_Intro}
  \partial_tv+\overline{\nabla}_vv-\nu\{\Delta_Bv+\mathrm{Ric}(v)\}+\nabla_\Gamma q = f, \quad \mathrm{div}_\Gamma v = 0 \quad\text{on}\quad \Gamma\times(0,\infty)
\end{align}
introduced in \cites{EbMa70,MitYa02,Ta92} and studied in many works (see e.g. \cites{ChaCzu13,DinMit04,KheMis12,KohWen18,MitTa01,Nag99,Pi17,Prie94}).
In the above, $\overline{\nabla}_vv$ is the covariant derivative of $v$ along itself, $\mathrm{Ric}$ is the Ricci curvature of $\Gamma$, and $\Delta_B$, $\nabla_\Gamma$, and $\mathrm{div}_\Gamma$ are the Bochner Laplacian, the tangential gradient, and the surface divergence on $\Gamma$ (see \cite{Miu_NSCTD_03} for details).
We emphasize that the last paper \cite{Miu_NSCTD_03} provides the first result on a rigorous derivation of the surface Navier--Stokes equations on a general closed surface in $\mathbb{R}^3$ by the thin-film limit and that for \cite{Miu_NSCTD_03} the results of this paper and \cite{Miu_NSCTD_02} are essential.

\subsection{Ideas of the proofs} \label{SS:Int_Ide}
Let us explain the ideas of the proofs of \eqref{E:NE_Intro} and \eqref{E:Diff_Intro} (see Section \ref{S:Pf_Main} for details).
Since the bilinear form for \eqref{E:Stokes_CTD} is of the form
\begin{align*}
  a_\varepsilon(u_1,u_2) = 2\nu\bigl(D(u_1),D(u_2)\bigr)_{L^2(\Omega_\varepsilon)}+\sum_{i=0,1}\gamma_\varepsilon^i(u_1,u_2)_{L^2(\Gamma_\varepsilon^i)}
\end{align*}
due to the slip boundary conditions (see Lemma \ref{L:IbP_St}), we show that $a_\varepsilon$ is bounded and coercive uniformly in $\varepsilon$ on an appropriate function space on $\Omega_\varepsilon$ in order to get \eqref{E:NE_Intro} with $k=1$ (see Theorem \ref{T:Uni_aeps}).
To this end, we use the trace inequality
\begin{align*}
  \|\varphi\|_{L^2(\Gamma_\varepsilon^i)} \leq c\varepsilon^{-1/2}\|\varphi\|_{H^1(\Omega_\varepsilon)}, \quad \varphi\in H^1(\Omega_\varepsilon),\, i=0,1
\end{align*}
with a constant $c>0$ independent of $\varepsilon$, which follows from a more precise inequality given in Lemma \ref{L:Poincare}, and the uniform Korn inequality
\begin{align} \label{E:Korn_Intro}
  \|u\|_{H^1(\Omega_\varepsilon)} \leq c\|D(u)\|_{L^2(\Omega_\varepsilon)}
\end{align}
for $u\in H^1(\Omega_\varepsilon)^3$ satisfying the impermeable boundary condition
\begin{align} \label{E:Imp_Intro}
  u\cdot n_\varepsilon = 0 \quad\text{on}\quad \Gamma_\varepsilon
\end{align}
and
\begin{align} \label{E:KOK_Intro}
  |(u,\bar{v})_{L^2(\Omega_\varepsilon)}| \leq \beta\|u\|_{L^2(\Omega_\varepsilon)}\|\bar{v}\|_{L^2(\Omega_\varepsilon)}
\end{align}
with a constant $\beta\in[0,1)$ independent of $\varepsilon$ for every Killing vector field $v$ on $\Gamma$ (see Section \ref{S:Main}) that satisfies
\begin{align} \label{E:KilG_Intro}
  v\cdot\nabla_\Gamma g = 0 \quad\text{on}\quad \Gamma,
\end{align}
where $\bar{v}$ is the constant extension of $v$ in the normal direction of $\Gamma$ and $\nabla_\Gamma$ is the tangential gradient on $\Gamma$ (see Section \ref{SS:Pre_Surf}).
We prove \eqref{E:Korn_Intro} under the conditions \eqref{E:Imp_Intro} and \eqref{E:KOK_Intro} in Lemma \ref{L:Korn_H1}.
Moreover, we observe in Lemma \ref{L:Korn_Rg} that, if every Killing vector field on $\Gamma$ satisfying \eqref{E:KilG_Intro} is the restriction on $\Gamma$ of an infinitesimal rigid displacement of $\mathbb{R}^3$, i.e. a vector filed on $\mathbb{R}^3$ of the form
\begin{align} \label{E:ID_Intro}
  w(x) = a\times x+b, \quad x\in\mathbb{R}^3
\end{align}
with $a,b\in\mathbb{R}^3$, then \eqref{E:Korn_Intro} holds under the conditions \eqref{E:Imp_Intro} and, instead of \eqref{E:KOK_Intro},
\begin{align} \label{E:KOI_Intro}
  |(u,w)_{L^2(\Omega_\varepsilon)}| \leq \beta\|u\|_{L^2(\Omega_\varepsilon)}\|w\|_{L^2(\Omega_\varepsilon)}
\end{align}
for every vector field $w$ of the form \eqref{E:ID_Intro} satisfying
\begin{align} \label{E:IDG_Intro}
  w|_\Gamma\cdot n = w|_\Gamma\cdot\nabla_\Gamma g = 0 \quad\text{on}\quad \Gamma,
\end{align}
where $\beta\in[0,1)$ is again a constant independent of $\varepsilon$.
The proof of \eqref{E:Korn_Intro} consists of two steps.
First we estimate the $L^2(\Omega_\varepsilon)$-norm of $\nabla u$ to show
\begin{align} \label{E:KGrad_Intro}
  \|\nabla u\|_{L^2(\Omega_\varepsilon)}^2 \leq 4\|D(u)\|_{L^2(\Omega_\varepsilon)}^2+c\|u\|_{L^2(\Omega_\varepsilon)}^2
\end{align}
in Lemma \ref{L:Korn_Grad}.
To this end, we apply integration by parts twice to get
\begin{align*}
  \|\nabla u\|_{L^2(\Omega_\varepsilon)}^2 \leq 2\|D(u)\|_{L^2(\Omega_\varepsilon)}^2-\int_{\Gamma_\varepsilon}(u\cdot\nabla)u\cdot n_\varepsilon\,d\mathcal{H}^2
\end{align*}
and estimate the last term by reducing the order of the derivatives of $u$ on $\Gamma_\varepsilon$ with the aid of \eqref{E:Imp_Intro} and interpolating integrals over the inner and outer boundaries $\Gamma_\varepsilon^0$ and $\Gamma_\varepsilon^1$.
Next for a given $\alpha>0$ we prove the uniform $L^2(\Omega_\varepsilon)$-estimate
\begin{align} \label{E:KUL2_Intro}
  \|u\|_{L^2(\Omega_\varepsilon)}^2 \leq \alpha\|\nabla u\|_{L^2(\Omega_\varepsilon)}+c\|D(u)\|_{L^2(\Omega_\varepsilon)}^2
\end{align}
in Lemma \ref{L:Korn_U} by contradiction as in the case of a flat thin domain studied in \cite{HoSe10}.
We transform $\Omega_\varepsilon$ into a domain with fixed thickness and show that a sequence of vector fields failing to satisfy \eqref{E:KUL2_Intro} converges to the constant extension of a Killing vector field on $\Gamma$ satisfying \eqref{E:KilG_Intro} as $\varepsilon\to0$.
Then we take that Killing vector field in \eqref{E:KOK_Intro} or \eqref{E:KOI_Intro}, send $\varepsilon\to0$, and use $\beta<1$ to get a contradiction.
Note that both steps are based on a careful analysis of surface quantities of $\Gamma_\varepsilon$.

To establish \eqref{E:Diff_Intro} we follow the idea of the works \cites{Ho08,Ho10} on a flat thin domain.
Using the slip boundary conditions \eqref{E:Slip_Intro} we derive the integration by parts formula
\begin{multline*}
  \int_{\Omega_\varepsilon}\mathrm{curl}\,\mathrm{curl}\,u\cdot\Phi\,dx = -\int_{\Omega_\varepsilon}\mathrm{curl}\,G(u)\cdot\Phi\,dx+\int_{\Omega_\varepsilon}\{\mathrm{curl}\,u+G(u)\}\cdot\mathrm{curl}\,\Phi\,dx
\end{multline*}
for $\Phi\in L^2(\Omega_\varepsilon)^3$ with $\mathrm{curl}\,\Phi\in L^2(\Omega_\varepsilon)^3$, where $G(u)$ is a vector field on $\Omega_\varepsilon$ whose $H^1(\Omega_\varepsilon)$-norm is uniformly bounded by that of $u$ (see Lemmas \ref{L:G_Bound} and \ref{L:IbP_Curl}).
Then we combine this formula and the Helmholtz--Leray decomposition for $-\nu\Delta u$ on $\Omega_\varepsilon$ to get \eqref{E:Diff_Intro}.
Here the uniform estimate for $G(u)$ plays an important role, but its proof involves a complicated calculations of surface quantities of $\Gamma_\varepsilon^0$ and $\Gamma_\varepsilon^1$ since we construct $G(u)$ by interpolating surface quantities of $\Gamma_\varepsilon^0$ and those of $\Gamma_\varepsilon^1$.

To prove \eqref{E:NE_Intro} with $k=2$ we employ \eqref{E:Diff_Intro} and the uniform a priori estimate for the vector Laplace operator
\begin{align} \label{E:UA_Intro}
  \|u\|_{H^2(\Omega_\varepsilon)} \leq c\left(\|\Delta u\|_{L^2(\Omega_\varepsilon)}+\|u\|_{H^1(\Omega_\varepsilon)}\right)
\end{align}
for $u\in H^2(\Omega_\varepsilon)^3$ satisfying \eqref{E:Slip_Intro} (see Lemma \ref{L:Lap_Apri}).
The proof of \eqref{E:UA_Intro} proceeds as in that of \eqref{E:KGrad_Intro}, but calculations are more involved.
We first show that the above $u$ is approximated by $H^3$ vector fields on $\Omega_\varepsilon$ satisfying \eqref{E:Slip_Intro} to assume $u\in H^3(\Omega_\varepsilon)^3$ (see Lemma \ref{L:NSl_Approx}).
Then we carry out integration by parts twice to get (see Appendix \ref{S:Ap_Vec} for notations)
\begin{align*}
  \|\nabla^2u\|_{L^2(\Omega_\varepsilon)}^2 = \|\Delta u\|_{L^2(\Omega_\varepsilon)}^2+\int_{\Gamma_\varepsilon}\nabla u:\{(n_\varepsilon\cdot\nabla)\nabla u-n_\varepsilon\otimes\Delta u\}\,d\mathcal{H}^2.
\end{align*}
Thus we intend to show the uniform estimate for the last term
\begin{multline} \label{E:BD_Intro}
  \left|\int_{\Gamma_\varepsilon}\nabla u:\{(n_\varepsilon\cdot\nabla)\nabla u-n_\varepsilon\otimes\Delta u\}\,d\mathcal{H}^2\right| \\
  \leq c\left(\|u\|_{H^1(\Omega_\varepsilon)}^2+\|u\|_{H^1(\Omega_\varepsilon)}\|\nabla^2u\|_{L^2(\Omega_\varepsilon)}\right).
\end{multline}
To this end, we first reduce the second order derivatives of $u$ on $\Gamma_\varepsilon$ to the first order ones by using \eqref{E:Slip_Intro}.
In this step we make use of formulas for the covariant derivatives of tangential vector fields on $\Gamma_\varepsilon$ given in Appendix \ref{S:Ap_Cov} to perform calculations on $\Gamma_\varepsilon$ without a change of variables.
Then we interpolate integrals of $u$ and its first order derivatives over the inner and outer boundaries $\Gamma_\varepsilon^0$ and $\Gamma_\varepsilon^1$ to get \eqref{E:BD_Intro}.
For this purpose, we apply estimates for the differences between surface quantities of $\Gamma_\varepsilon^0$ and those of $\Gamma_\varepsilon^1$ given in Section \ref{SS:Pre_Dom}.
However, the proofs of those estimates involve complicated calculations of differential geometry of surfaces (see Appendix \ref{S:Ap_Proof}).

\subsection{Literature overview} \label{SS:Int_Lit}
The study of the Navier--Stokes equations in thin domains has a long history.
A main subject is to prove the global existence of a strong solution for large data depending on the smallness of the thickness of a thin domain, since a thin domain in $\mathbb{R}^3$ with very small thickness is almost two-dimensional.
It is also important to study the behavior of a solution as the thickness of a thin domain tends to zero in order to understand the dependence of a solution on the thin and other directions.
Raugel and Sell \cite{RaSe93} first studied the Navier--Stokes equations in a thin product domain $Q\times(0,\varepsilon)$ in $\mathbb{R}^3$ with a rectangle $Q$ and a sufficiently small $\varepsilon>0$ under the purely periodic or mixed Dirichlet-periodic boundary conditions and obtained the global existence of a strong solution.
Temam and Ziane \cite{TeZi96} generalized the result of \cite{RaSe93} to a thin product domain $\omega\times(0,\varepsilon)$ in $\mathbb{R}^3$ around a bounded domain $\omega$ in $\mathbb{R}^2$ under combinations of the Dirichlet, periodic, and Hodge boundary conditions.
They also proved that the average in the thin direction of a solution to the original equations under suitable boundary conditions converges towards a solution to the two-dimensional Navier--Stokes equations in $\omega$ as $\varepsilon\to0$.
For further results on the Navier--Stokes equations in three-dimensional thin product domains we refer to \cites{Hu07,If99,IfRa01,MoTeZi97,Mo99} and the references cited therein.

Thin product domains appearing in the above cited papers are flat in the sense that they shrink to domains in $\mathbb{R}^2$ as $\varepsilon\to0$ and their top and bottom boundaries are flat, but in physical problems we frequently encounter nonflat thin domains (see \cite{Ra95} for examples of them).
Temam and Ziane \cite{TeZi97} first dealt with a nonflat thin domain in the study of the Navier--Stokes equations.
Under the Hodge boundary conditions they proved the global existence of a strong solution to the Navier--Stokes equations in a thin spherical shell
\begin{align*}
  \{x\in\mathbb{R}^3 \mid a < |x| < a+\varepsilon a\}, \quad a>0
\end{align*}
and the convergence of its average towards a solution of limit equations on a sphere as $\varepsilon\to0$.
Iftimie, Raugel, and Sell \cite{IfRaSe07} considered a flat thin domain with a nonflat top boundary
\begin{align*}
  \{(x',x_3)\in\mathbb{R}^3 \mid x'\in(0,1)^2,\,0<x_3<\varepsilon g(x')\}, \quad g\colon(0,1)^2\to\mathbb{R}
\end{align*}
under the horizontally periodic and vertically slip boundary conditions and obtained the global existence of a strong solution.
They also compared the strong solution with a solution to limit equations in $(0,1)^2$.
Hoang \cites{Ho10,Ho13} and Hoang and Sell \cite{HoSe10} generalized the existence result of \cite{IfRaSe07} to a flat thin domain with nonflat top and bottom boundaries (in \cite{Ho13} two-phase flows were studied).

Let us also mention the slip boundary conditions \eqref{E:Slip_Intro} and the Stokes problem \eqref{E:Stokes_CTD}.
The slip boundary conditions introduced by Navier \cite{Na1823} state that the fluid slips on the boundary with velocity proportional to the tangential component of the stress vector.
These conditions are considered as an appropriate model for flows with free boundaries and for flows past chemically reacting walls in which the usual no-slip boundary condition is not valid (see \cite{Ve87}).
They also arise in the study of the atmosphere and ocean dynamics \cites{LiTeWa92a,LiTeWa92b,LiTeWa95} and in the homogenization of the no-slip boundary condition on a rough boundary \cites{Hi16,JaMi01}.
The Stokes problem \eqref{E:Stokes_CTD} under the slip boundary conditions for a general bounded domain in $\mathbb{R}^3$ was first studied by Solonnikov and \v{S}\v{c}adilov \cite{SoSc73} in the $L^2$-setting.
Beir\~{a}o da Veiga \cite{Be04} considered the generalized system for \eqref{E:Stokes_CTD} and proved the $H^2$-regularity estimate for a solution.
The $L^p$-theory for \eqref{E:Stokes_CTD} in a bounded domain in $\mathbb{R}^3$ were established by Amrouche and Rejaiba \cite{AmRe14}.
Note that the main results \eqref{E:NE_Intro} and \eqref{E:Diff_Intro} of this paper are not covered by \cites{AmRe14,Be04,SoSc73} since we show that the constant $c$ in these estimates does not depend on the thickness of the curved thin domain.

In this paper and the forthcoming papers \cites{Miu_NSCTD_02,Miu_NSCTD_03} we deal with the curved thin domain $\Omega_\varepsilon$ of the form \eqref{E:CTD_Intro} which degenerates into the closed surface $\Gamma$ as $\varepsilon\to0$.
Curved thin domains around hypersurfaces and lower dimensional manifolds were considered in the study of eigenvalues of the Laplace operator \cites{JiKu16,Kr14,Sch96,Yac18} and of reaction-diffusion equations \cites{PrRiRy02,PrRy03,Yan90}.
The series of our works gives the first study of the Navier--Stokes equations in a curved thin domain in $\mathbb{R}^3$ whose limit set is a general closed surface.
Our aim is not just to generalize the shape of a thin domain, but to provide the first result on a rigorous derivation of the surface Navier--Stokes equations \eqref{E:SNS_Intro} by the thin-film limit.

Although the main purpose of this paper is to present preliminary results for the study of \eqref{E:NS_CTD}, we show new results on the uniform Korn inequality \eqref{E:Korn_Intro}.
Korn's inequality is a basic tool in the theory of linear elasticity and fluid mechanics and has been studied in various contexts (see \cite{Hor95} and the references cited therein).
The uniform Korn inequality \eqref{E:Korn_Intro} in a curved thin domain in $\mathbb{R}^k$ with $k\geq2$ around a closed hypersurface was first given by Lewicka and M\"{u}ller \cite{LeMu11}.
In \cite{LeMu11}*{Theorem 2.2} they proved \eqref{E:Korn_Intro} under the conditions \eqref{E:Imp_Intro} and \eqref{E:KOK_Intro} (see also \cite{LeMu11}*{Theorem 2.1} for other conditions).
Their proof was based on a uniform Korn inequality in a thin cylinder and Korn's inequality on a hypersurface for which Killing vector fields on the hypersurface play a fundamental role.
In this paper we present another proof of \eqref{E:Korn_Intro} under the same conditions by following the idea of the work \cite{HoSe10} on a flat thin domain.
Moreover, we prove \eqref{E:Korn_Intro} by imposing \eqref{E:Imp_Intro} and the new condition \eqref{E:KOI_Intro} under the assumption that every Killing vector field on $\Gamma$ satisfying \eqref{E:KilG_Intro} is the restriction on $\Gamma$ of an infinitesimal rigid displacement of $\mathbb{R}^3$.
This assumption is valid for many kinds of closed surfaces in $\mathbb{R}^3$ (see Remark \ref{R:Killing}).
In particular, we can use \eqref{E:KOI_Intro} instead of \eqref{E:KOK_Intro} for curved thin domains around the unit sphere in $\mathbb{R}^3$.
We also note that we take a vector field $w$ defined on $\mathbb{R}^3$ itself in \eqref{E:KOI_Intro}, not the constant extension of a vector field on $\Gamma$ as in \eqref{E:KOK_Intro}.
This fact is crucial in order to relate the Stokes operator $A_\varepsilon$ properly to the Stokes problem \eqref{E:Stokes_CTD} (see Remark \ref{R:Cond_A3}).
In Section \ref{SS:Diff_Korn} we further show that the conditions \eqref{E:KOK_Intro} and \eqref{E:KOI_Intro} are more strict than the condition for a standard Korn inequality related to the axial symmetry of a domain by giving examples of both axially symmetric and not axially symmetric curved thin domains.

Besides the new results on the uniform Korn inequality, we give new techniques for the analysis of vector fields on embedded surfaces in $\mathbb{R}^3$ (or on higher dimensional hypersurfaces in Euclidean spaces) such as the boundary of a domain.
In the proof of \eqref{E:BD_Intro} we need to compute the second order derivatives of $u\in H^3(\Omega_\varepsilon)^3$ on $\Gamma_\varepsilon$ to reduce the order of the derivatives.
To carry out such calculations we usually take a local coordinate system of $\Gamma_\varepsilon$ or transform a part of $\Gamma_\varepsilon$ into the boundary of a half-space, but here these choices will result in too complicated calculations which we can hardly complete.
Instead we use a local orthonormal frame for the tangent bundle of $\Gamma_\varepsilon$ and formulas for the covariant derivatives of tangential vector fields on $\Gamma_\varepsilon$ given in Appendix \ref{S:Ap_Cov} to work without a change of variables.
The most important tool is the Gauss formula
\begin{align*}
  (X\cdot\nabla)Y = \overline{\nabla}_X^\varepsilon Y+(W_\varepsilon X\cdot Y)n_\varepsilon \quad\text{on}\quad \Gamma_\varepsilon
\end{align*}
for tangential vector fields $X$ and $Y$ on $\Gamma_\varepsilon$, which expresses the directional derivative $(X\cdot\nabla)Y$ in $\mathbb{R}^3$ in terms of the covariant derivative $\overline{\nabla}_X^\varepsilon Y$ on $\Gamma_\varepsilon$ and the second fundamental form $(W_\varepsilon X\cdot Y)n_\varepsilon$ of $\Gamma_\varepsilon$ (see Lemma \ref{L:Gauss}).
It enables us to apply formulas of differential geometry to quantities on $\Gamma_\varepsilon$ expressed in a fixed coordinate system of $\mathbb{R}^3$ and to write resulting expressions in the same coordinate system.
Our method is useful to deduce properties of functions on a domain from their behavior on the boundary since it avoids a change of variables.
It also provides an easy and understandable way to compute vector fields on surfaces without introducing local coordinate systems and differential forms.
We expect that our method will be applicable to other problems involving complicated calculations of vector fields on surfaces, especially to partial differential equations for vector fields on stationary or moving surfaces such as the surface Navier--Stokes and Stokes equations (see e.g. \cites{JaOlRe18,KoLiGi17,OlQuReYu18,Reus20}).

\subsection{Organization of this paper} \label{SS:Int_Org}
The rest of this paper is organized as follows.
In Section \ref{S:Main} we provide the main results of this paper.
Notations and basic results on a closed surface and a curved thin domain are presented in Section \ref{S:Pre}.
Section \ref{S:Fund} gives fundamental inequalities and formulas for functions on the curved thin domain and its boundary.
In Section \ref{S:UniKorn} we establish the uniform Korn inequality \eqref{E:Korn_Intro} and compare it with a standard Korn inequality.
We also derive the uniform a priori estimate for the vector Laplace operator \eqref{E:UA_Intro} in Section \ref{S:UniLap}.
Using the results of Sections \ref{S:Fund}--\ref{S:UniLap} we prove our main results in Section \ref{S:Pf_Main}.
Appendix \ref{S:Ap_Vec} fixes notations on vectors and matrices.
Some auxiliary results related to the closed surface are shown in Appendix \ref{S:Ap_RCS}.
In Appendix \ref{S:Ap_Proof} we provide the proofs of lemmas in Section \ref{S:Pre} and Lemmas \ref{L:CoV_Fixed}, \ref{L:KAux_Du}, and \ref{L:G_Bound} involving elementary but long calculations of differential geometry of surfaces.
Appendix \ref{S:Ap_Cov} presents formulas for the covariant derivatives of tangential vector fields on the closed surface used in Section \ref{S:UniLap}.
In Appendix \ref{S:Ap_Rig} we show some properties of infinitesimal rigid displacements of $\mathbb{R}^3$ related to the axial symmetry of the closed surface and the curved thin domain.

Most results of this paper were obtained in the doctoral thesis of the author \cite{Miu_DT}.
In this paper, however, we newly prove the uniform Korn inequality \eqref{E:Korn_Intro} under the condition \eqref{E:KOI_Intro} and give Appendix \ref{S:Ap_Rig} to study properties of infinitesimal rigid displacements of $\mathbb{R}^3$ related to the axial symmetry of a closed surface and a curved thin domain.
By these new results we can add the condition (A3) in Assumption \ref{Assump_2} to consider some curved thin domains excluded in \cite{Miu_DT}.
The most important example of a curved thin domain newly included in this paper is the thin spherical shell
\begin{align*}
  \Omega_\varepsilon = \{x\in\mathbb{R}^3 \mid 1<|x|<1+\varepsilon\}
\end{align*}
under the perfect slip boundary conditions \eqref{E:Slip_Intro} with $\gamma_\varepsilon=0$.
This kind of curved thin domain was studied by Temam and Ziane \cite{TeZi97} under different boundary conditions (see Remark \ref{R:Ex_Assump}).
We also add Section \ref{SS:Diff_Korn} in which we discuss the difference between the uniform Korn inequality and a standard Korn inequality.

\section{Main results} \label{S:Main}
In this section we present the main results of this paper.
The proofs of theorems in this section will be given in Section \ref{S:Pf_Main}.

To state the main results we first fix some notations (see also Section \ref{S:Pre}).
Let $\Gamma$ be a two-dimensional closed (i.e. compact and without boundary), connected, and oriented surface in $\mathbb{R}^3$ with unit outward normal vector field $n$ and $g_0,g_1\in C^4(\Gamma)$.
We assume that $\Gamma$ is of class $C^5$ and there exists a constant $c>0$ such that
\begin{align} \label{E:G_Inf}
  g := g_1-g_0 \geq c \quad\text{on}\quad \Gamma.
\end{align}
Note that we do not assume $g_0\leq0$ or $g_1\geq0$ on $\Gamma$.
For a sufficiently small $\varepsilon\in(0,1]$ let $\Omega_\varepsilon$ be the curved thin domain in $\mathbb{R}^3$ of the form \eqref{E:CTD_Intro} and
\begin{align*}
  L_\sigma^2(\Omega_\varepsilon) := \{u\in L^2(\Omega_\varepsilon)^3 \mid \text{$\mathrm{div}\,u=0$ in $\Omega_\varepsilon$, $u\cdot n_\varepsilon=0$ on $\Gamma_\varepsilon$}\}
\end{align*}
the standard $L^2$-solenoidal space on $\Omega_\varepsilon$.
By integration by parts we observe that the bilinear form for the Stokes probelm \eqref{E:Stokes_CTD} is given by
\begin{align} \label{E:Def_aeps}
  a_\varepsilon(u_1,u_2) := 2\nu\bigl(D(u_1),D(u_2)\bigr)_{L^2(\Omega_\varepsilon)}+\sum_{i=0,1}\gamma_\varepsilon^i(u_1,u_2)_{L^2(\Gamma_\varepsilon^i)}
\end{align}
for $u_1,u_2\in H^1(\Omega_\varepsilon)^3$ (see Lemma \ref{L:IbP_St}).
Here
\begin{align*}
  D(u) := \frac{\nabla u+(\nabla u)^T}{2} \quad\text{on}\quad \Omega_\varepsilon
\end{align*}
is the strain rate tensor for a vector field $u$ on $\Omega_\varepsilon$ and $\gamma_\varepsilon^0$ and $\gamma_\varepsilon^1$ are the friction coefficients appearing in \eqref{E:Def_Fric}.
Clearly, $a_\varepsilon$ is symmetric.
To make it uniformly in $\varepsilon$ bounded and coercive on an appropriate function space, we define function spaces and impose assumptions on $\gamma_\varepsilon^0$, $\gamma_\varepsilon^1$, and $\Gamma$.
Let
\begin{align} \label{E:Def_R}
  \mathcal{R} := \{w(x)=a\times x+b,\,x\in\mathbb{R}^3 \mid a,b\in\mathbb{R}^3,\,\text{$w|_\Gamma\cdot n=0$ on $\Gamma$}\}
\end{align}
be the space of all infinitesimal rigid displacements of $\mathbb{R}^3$ whose restrictions on $\Gamma$ are tangential.
Note that $\mathcal{R}$ is of finite dimension and that $\mathcal{R}\neq\{0\}$ if and only if $\Gamma$ is axially symmetric, i.e. invariant under a rotation by any angle around some line (see Lemma \ref{L:IR_Surf}).
Let $\nabla_\Gamma$ the tangential gradient operator on $\Gamma$ (see Section \ref{SS:Pre_Surf} for its definition).
We define subspaces of $\mathcal{R}$ by
\begin{align} \label{E:Def_Rg}
  \begin{aligned}
    \mathcal{R}_i &:= \{w\in\mathcal{R} \mid \text{$w|_\Gamma\cdot\nabla_\Gamma g_i=0$ on $\Gamma$}\}, \quad i=0,1, \\
    \mathcal{R}_g &:= \{w\in\mathcal{R} \mid \text{$w|_\Gamma\cdot\nabla_\Gamma g=0$ on $\Gamma$}\} \quad (g=g_1-g_0).
  \end{aligned}
\end{align}
Note that $\mathcal{R}_0\cap\mathcal{R}_1\subset\mathcal{R}_g$.
It turns out (see Lemmas \ref{L:CTD_AS} and \ref{L:CTD_Rg}) that $\Omega_\varepsilon$ is axially symmetric around the same line for all $\varepsilon\in(0,1]$ if $\mathcal{R}_0\cap\mathcal{R}_1\neq\{0\}$, while $\Omega_\varepsilon$ is not axially symmetric around any line for all $\varepsilon>0$ sufficiently small if $\mathcal{R}_g=\{0\}$.
Next we define the surface strain rate tensor $D_\Gamma(v)$ by
\begin{align} \label{E:Def_SSR}
  D_\Gamma(v) := P(\nabla_\Gamma v)_SP \quad\text{on}\quad \Gamma
\end{align}
for a (not necessarily tangential) vector field $v$ on $\Gamma$, where
\begin{align*}
  P := I_3-n\otimes n, \quad (\nabla_\Gamma v)_S := \frac{\nabla_\Gamma v+(\nabla_\Gamma v)^T}{2}
\end{align*}
are the orthogonal projection onto the tangent plane of $\Gamma$ and the symmetric part of the tangential gradient matrix of $v$ (see Section \ref{SS:Pre_Surf} for details).
Then we set
\begin{align} \label{E:Def_Kil}
  \begin{aligned}
    \mathcal{K}(\Gamma) &:= \{v \in H^1(\Gamma)^3 \mid \text{$v\cdot n=0$, $D_\Gamma(v)=0$ on $\Gamma$}\}, \\
    \mathcal{K}_g(\Gamma) &:= \{v\in\mathcal{K}(\Gamma) \mid \text{$v\cdot\nabla_\Gamma g=0$ on $\Gamma$}\}.
  \end{aligned}
\end{align}
If $\Gamma$ is of class $C^4$, then $v\in\mathcal{K}(\Gamma)$ is in fact of class $C^1$ (see Lemma \ref{L:Kil_Reg}) and
\begin{align*}
  \overline{\nabla}_Xv\cdot Y+X\cdot\overline{\nabla}_Yv = 0 \quad\text{on}\quad \Gamma
\end{align*}
for all tangential vector fields $X$ and $Y$ on $\Gamma$, where $\overline{\nabla}_Xv:=P(X\cdot\nabla_\Gamma)v$ denotes the covariant derivative of $v$ along $X$.
Such a vector field generates a one-parameter group of isometries of $\Gamma$ and is called a Killing vector field on $\Gamma$.
It is known that $\mathcal{K}(\Gamma)$ is a Lie algebra of dimension at most three.
For details of Killing vector fields we refer to \cites{Jo11,Pe06}.

\begin{remark} \label{R:Killing}
  For $w(x)=a\times x+b$, $x\in\mathbb{R}^3$ with $a,b\in\mathbb{R}^3$, direct calculations show that $D_\Gamma(w)=0$ on $\Gamma$.
  Hence $w$ is Killing on $\Gamma$ if it tangential on $\Gamma$, i.e.
  \begin{align*}
    \mathcal{R}|_\Gamma := \{w|_\Gamma\mid w\in\mathcal{R}\} \subset \mathcal{K}(\Gamma).
  \end{align*}
  The set $\mathcal{R}|_\Gamma$ represents the extrinsic infinitesimal symmetry of the embedded surface $\Gamma$, while $\mathcal{K}(\Gamma)$ describes the intrinsic one of the abstract Riemannian manifold $\Gamma$.
  It is known that $\mathcal{R}|_\Gamma=\mathcal{K}(\Gamma)$ if $\Gamma$ is a surface of revolution (see also Lemma \ref{L:IR_Kil}).
  The same relation holds if $\Gamma$ is closed and convex since any isometry between two closed and convex surfaces in $\mathbb{R}^3$ is a motion in $\mathbb{R}^3$ (a rotation and a translation) or a motion and a reflection by the Cohn-Vossen theorem (see \cite{Sp79}).
  However, it is not known whether $\mathcal{R}|_\Gamma$ agrees with $\mathcal{K}(\Gamma)$ for a general (nonconvex and not axially symmetric) closed surface.
  In particular, the existence of a closed surface in $\mathbb{R}^3$ that is not axially symmetric but admits a nontrivial Killing vector field, i.e. $\mathcal{R}=\{0\}$ but $\mathcal{K}(\Gamma)\neq\{0\}$, is an open problem.
\end{remark}

We make the following assumptions on the friction coefficients $\gamma_\varepsilon^0$ and $\gamma_\varepsilon^1$, the closed surface $\Gamma$, and the functions $g_0$ and $g_1$ (see also Remarks \ref{R:Ex_Assump} and \ref{R:Cond_A3}).

\begin{assumption} \label{Assump_1}
  There exists a constant $c>0$ such that
  \begin{align} \label{E:Fric_Upper}
    \gamma_\varepsilon^0 \leq c\varepsilon, \quad \gamma_\varepsilon^1 \leq c\varepsilon
  \end{align}
  for all $\varepsilon\in(0,1]$.
\end{assumption}

\begin{assumption} \label{Assump_2}
  Either of the following conditions is satisfied:
  \begin{itemize}
    \item[(A1)] There exists a constant $c>0$ such that
    \begin{align*}
      \gamma_\varepsilon^0 \geq c\varepsilon \quad\text{for all}\quad \varepsilon\in(0,1] \quad\text{or}\quad \gamma_\varepsilon^1 \geq c\varepsilon \quad\text{for all}\quad \varepsilon\in(0,1].
    \end{align*}
    \item[(A2)] The space $\mathcal{K}_g(\Gamma)$ contains only a trivial vector field, i.e. $\mathcal{K}_g(\Gamma)=\{0\}$.
    \item[(A3)] The relations
    \begin{align*}
      \mathcal{R}_g=\mathcal{R}_0\cap\mathcal{R}_1, \quad \mathcal{R}_g|_\Gamma:=\{w|_\Gamma\mid w\in\mathcal{R}_g\}=\mathcal{K}_g(\Gamma)
    \end{align*}
    hold and $\gamma_\varepsilon^0=\gamma_\varepsilon^1=0$ for all $\varepsilon\in(0,1]$.
  \end{itemize}
\end{assumption}

These assumptions are imposed only in this section and Section \ref{S:Pf_Main}.
Under Assumptions \ref{Assump_1} and \ref{Assump_2} we define subspaces of $L^2(\Omega_\varepsilon)^3$ and $H^1(\Omega_\varepsilon)^3$ by
\begin{align} \label{E:Def_Heps}
  \begin{aligned}
    \mathcal{H}_\varepsilon &:=
    \begin{cases}
      L_\sigma^2(\Omega_\varepsilon) &\text{if the condition (A1) or (A2) is satisfied}, \\
      L_\sigma^2(\Omega_\varepsilon)\cap\mathcal{R}_g^\perp &\text{if the condition (A3) is satisfied},
  \end{cases} \\
  \mathcal{V}_\varepsilon &:= \mathcal{H}_\varepsilon\cap H^1(\Omega_\varepsilon)^3,
    \end{aligned}
\end{align}
where $\mathcal{R}_g^\perp$ is the orthogonal complement of $\mathcal{R}_g$ in $L^2(\Omega_\varepsilon)^3$.
Here we consider vector fields in $\mathcal{R}_g$ defined on the whole space $\mathbb{R}^3$ as elements of $L^2(\Omega_\varepsilon)^3$ just by restricting them on $\overline{\Omega}_\varepsilon$.
Note that $\mathcal{R}_0\cap\mathcal{R}_1\subset L_\sigma^2(\Omega_\varepsilon)$ by Lemma \ref{L:IR_Sole} and thus $\mathcal{R}_g\subset L_\sigma^2(\Omega_\varepsilon)$ under the condition (A3).
Also, $\mathcal{H}_\varepsilon$ and $\mathcal{V}_\varepsilon$ are closed in $L^2(\Omega_\varepsilon)^3$ and $H^1(\Omega_\varepsilon)^3$.
By $\mathbb{P}_\varepsilon$ we denote the orthogonal projection from $L^2(\Omega_\varepsilon)^3$ onto $\mathcal{H}_\varepsilon$.
Note that $\mathbb{P}_\varepsilon$ may be slightly different from the standard Helmholtz--Leray projection from $L^2(\Omega_\varepsilon)^3$ onto $L_\sigma^2(\Omega_\varepsilon)$ under the condition (A3).

Now let us present the main results of this paper.
The first result is the uniform boundedness and coerciveness of the bilinear form $a_\varepsilon$ given by \eqref{E:Def_aeps} on $\mathcal{V}_\varepsilon$.

\begin{theorem} \label{T:Uni_aeps}
  Under Assumptions \ref{Assump_1} and \ref{Assump_2}, there exist constants $\varepsilon_0\in(0,1]$ and $c>0$ such that
  \begin{align} \label{E:Uni_aeps}
    c^{-1}\|u\|_{H^1(\Omega_\varepsilon)}^2 \leq a_\varepsilon(u,u) \leq c\|u\|_{H^1(\Omega_\varepsilon)}^2
  \end{align}
  for all $\varepsilon\in(0,\varepsilon_0]$ and $u\in\mathcal{V}_\varepsilon$.
\end{theorem}

Throughout this section we fix the constant $\varepsilon_0$ given in Theorem \ref{T:Uni_aeps} and take $\varepsilon\in(0,\varepsilon_0]$.
By Theorem \ref{T:Uni_aeps} the bilinear form $a_\varepsilon$ is bounded, coercive, and symmetric on the Hilbert space $\mathcal{V}_\varepsilon$.
Hence by the Lax--Milgram theorem there exists a bounded linear operator $A_\varepsilon$ from $\mathcal{V}_\varepsilon$ into its dual space $\mathcal{V}_\varepsilon'$ such that
\begin{align*}
  {}_{\mathcal{V}_\varepsilon'}\langle A_\varepsilon u_1,u_2\rangle_{\mathcal{V}_\varepsilon} = a_\varepsilon(u_1,u_2), \quad u_1,u_2\in \mathcal{V}_\varepsilon,
\end{align*}
where ${}_{\mathcal{V}_\varepsilon'}\langle\cdot,\cdot\rangle_{\mathcal{V}_\varepsilon}$ is the duality product between $\mathcal{V}_\varepsilon'$ and $\mathcal{V}_\varepsilon$.
We consider $A_\varepsilon$ as an unbounded operator on $\mathcal{H}_\varepsilon$ with domain
\begin{align*}
  D(A_\varepsilon) = \{u\in\mathcal{V}_\varepsilon \mid A_\varepsilon u\in\mathcal{H}_\varepsilon\}.
\end{align*}
Then the Lax--Milgram theory implies that $A_\varepsilon$ is a positive self-adjoint operator on $\mathcal{H}_\varepsilon$ and thus its square root $A_\varepsilon^{1/2}$ is well-defined on $D(A_\varepsilon^{1/2})=\mathcal{V}_\varepsilon$.
Moreover,
\begin{align} \label{E:L2in_Ahalf}
  (A_\varepsilon u_1,u_2)_{L^2(\Omega_\varepsilon)} = (A_\varepsilon^{1/2}u_1,A_\varepsilon^{1/2} u_2)_{L^2(\Omega_\varepsilon)}
\end{align}
for all $u_1\in D(A_\varepsilon)$ and $u_2\in \mathcal{V}_\varepsilon$, and
\begin{align} \label{E:L2nor_Ahalf}
  \|A_\varepsilon^{1/2}u\|_{L^2(\Omega_\varepsilon)}^2 = a_\varepsilon(u,u) = 2\nu\|D(u)\|_{L^2(\Omega_\varepsilon)}^2+\gamma_\varepsilon^0\|u\|_{L^2(\Gamma_\varepsilon^0)}^2+\gamma_\varepsilon^1\|u\|_{L^2(\Gamma_\varepsilon^1)}^2
\end{align}
for all $u\in \mathcal{V}_\varepsilon$ (see e.g. \cites{BoFa13,So01} for details).
From a regularity result for a solution to the Stokes problem \eqref{E:Stokes_CTD} (see \cites{AmRe14,Be04,SoSc73}) it also follows that
\begin{align} \label{E:Dom_St}
  D(A_\varepsilon) = \{u\in \mathcal{V}_\varepsilon\cap H^2(\Omega_\varepsilon)^3 \mid \text{$2\nu P_\varepsilon D(u)n_\varepsilon+\gamma_\varepsilon u=0$ on $\Gamma_\varepsilon$}\}
\end{align}
and $A_\varepsilon u=-\nu\mathbb{P}_\varepsilon\Delta u$ for $u\in D(A_\varepsilon)$.
We call $A_\varepsilon$ the Stokes operator associated with \eqref{E:Stokes_CTD} or the Stokes operator for $\Omega_\varepsilon$ under the slip boundary conditions.

Let us give basic inequalities for $A_\varepsilon^{1/2}$ with constants independent of $\varepsilon$.

\begin{lemma} \label{L:Stokes_H1}
  Under Assumptions \ref{Assump_1} and \ref{Assump_2}, let $\varepsilon_0$ be the constant given in Theorem \ref{T:Uni_aeps}.
  There exists a constant $c>0$ such that
  \begin{align} \label{E:Stokes_H1}
    c^{-1}\|u\|_{H^1(\Omega_\varepsilon)} \leq \|A_\varepsilon^{1/2}u\|_{L^2(\Omega_\varepsilon)} \leq c\|u\|_{H^1(\Omega_\varepsilon)}
  \end{align}
  for all $\varepsilon\in(0,\varepsilon_0]$ and $u\in \mathcal{V}_\varepsilon$.
  Moreover, if $u\in D(A_\varepsilon)$, then we have
  \begin{align} \label{E:Stokes_Po}
    \|A_\varepsilon^{1/2}u\|_{L^2(\Omega_\varepsilon)} \leq c\|A_\varepsilon u\|_{L^2(\Omega_\varepsilon)}.
  \end{align}
\end{lemma}

\begin{proof}
  The inequality \eqref{E:Stokes_H1} is an immediate consequence of \eqref{E:Uni_aeps} and \eqref{E:L2nor_Ahalf}.
  To prove \eqref{E:Stokes_Po} for $u\in D(A_\varepsilon)$ we see by \eqref{E:L2in_Ahalf} and H\"{o}lder's inequality that
  \begin{align*}
    \|A_\varepsilon^{1/2}u\|_{L^2(\Omega_\varepsilon)}^2 = (u,A_\varepsilon u)_{L^2(\Omega_\varepsilon)} \leq \|u\|_{L^2(\Omega_\varepsilon)}\|A_\varepsilon u\|_{L^2(\Omega_\varepsilon)}.
  \end{align*}
  By this inequality, $\|u\|_{L^2(\Omega_\varepsilon)}\leq \|u\|_{H^1(\Omega_\varepsilon)}$, and \eqref{E:Stokes_H1} we get \eqref{E:Stokes_Po}.
\end{proof}

Since $A_\varepsilon=-\nu\mathbb{P}_\varepsilon\Delta$ on $\mathcal{H}_\varepsilon$ and $\mathbb{P}_\varepsilon$ is the orthogonal projection from $L^2(\Omega_\varepsilon)^3$ onto $\mathcal{H}_\varepsilon$, we easily observe that
\begin{align*}
  \|A_\varepsilon u+\nu\Delta u\|_{L^2(\Omega_\varepsilon)} = \nu\|\Delta u-\mathbb{P}_\varepsilon\Delta u\|_{L^2(\Omega_\varepsilon)} \leq \nu\|\Delta u\|_{L^2(\Omega_\varepsilon)} \leq c\|u\|_{H^2(\Omega_\varepsilon)}
\end{align*}
for all $u\in D(A_\varepsilon)$ with a constant $c>0$ independent of $\varepsilon$.
The next theorem shows that the right-hand side of the above inequality can be replaced by the $H^1(\Omega_\varepsilon)$-norm of $u$ under the slip boundary conditions \eqref{E:Slip_Intro}.

\begin{theorem} \label{T:Comp_Sto_Lap}
  Under Assumptions \ref{Assump_1} and \ref{Assump_2}, let $\varepsilon_0$ be the constant given in Theorem \ref{T:Uni_aeps}.
  There exists a constant $c>0$ such that
  \begin{align} \label{E:Comp_Sto_Lap}
    \|A_\varepsilon u+\nu\Delta u\|_{L^2(\Omega_\varepsilon)} \leq c\|u\|_{H^1(\Omega_\varepsilon)}
  \end{align}
  for all $\varepsilon\in(0,\varepsilon_0]$ and $u\in D(A_\varepsilon)$.
\end{theorem}

The inequality \eqref{E:Comp_Sto_Lap} is useful to derive a good estimate for the trilinear term
\begin{align*}
  \bigl((u\cdot\nabla)u,A_\varepsilon u\bigr)_{L^2(\Omega_\varepsilon)}, \quad u\in D(A_\varepsilon),
\end{align*}
which is essential for the proof of the global existence of a strong solution to the Navier--Stokes equations \eqref{E:NS_CTD}.
For details, we refer to \cite{Miu_NSCTD_02}.

Finally, we present the uniform norm equivalence for the Stokes operator $A_\varepsilon$.

\begin{theorem} \label{T:Stokes_H2}
  Under Assumptions \ref{Assump_1} and \ref{Assump_2}, let $\varepsilon_0$ be the constant given in Theorem \ref{T:Uni_aeps}.
  There exists a constant $c>0$ such that
  \begin{align} \label{E:Stokes_H2}
    c^{-1}\|u\|_{H^2(\Omega_\varepsilon)} \leq \|A_\varepsilon u\|_{L^2(\Omega_\varepsilon)} \leq c\|u\|_{H^2(\Omega_\varepsilon)}
  \end{align}
  for all $\varepsilon\in(0,\varepsilon_0]$ and $u\in D(A_\varepsilon)$.
\end{theorem}

As a consequence of Lemma \ref{L:Stokes_H1} and Theorem \ref{T:Stokes_H2} we obtain an interpolation inequality for a vector field in $D(A_\varepsilon)$.

\begin{corollary} \label{C:St_Inter}
  Under Assumptions \ref{Assump_1} and \ref{Assump_2}, let $\varepsilon_0$ be the constant given in Theorem \ref{T:Uni_aeps}.
  Then there exists a constant $c>0$ such that
  \begin{align} \label{E:St_Inter}
    \|u\|_{H^1(\Omega_\varepsilon)} \leq c\|u\|_{L^2(\Omega_\varepsilon)}^{1/2}\|u\|_{H^2(\Omega_\varepsilon)}^{1/2}
  \end{align}
  for all $\varepsilon\in(0,\varepsilon_0]$ and $u\in D(A_\varepsilon)$.
\end{corollary}

\begin{proof}
  Let $u\in D(A_\varepsilon)$.
  From \eqref{E:L2in_Ahalf} and \eqref{E:Stokes_H1} it follows that
  \begin{align*}
    \|u\|_{H^1(\Omega_\varepsilon)}^2 \leq c\|A_\varepsilon^{1/2}u\|_{L^2(\Omega_\varepsilon)}^2 = c(A_\varepsilon u,u)_{L^2(\Omega_\varepsilon)} \leq c\|A_\varepsilon u\|_{L^2(\Omega_\varepsilon)}\|u\|_{L^2(\Omega_\varepsilon)}.
  \end{align*}
  Applying \eqref{E:Stokes_H2} to the right-hand side of this inequality we get
  \begin{align*}
    \|u\|_{H^1(\Omega_\varepsilon)}^2 \leq c\|u\|_{L^2(\Omega_\varepsilon)}\|u\|_{H^2(\Omega_\varepsilon)}.
  \end{align*}
  Hence \eqref{E:St_Inter} is valid.
\end{proof}

We conclude this section with two remarks on Assumption \ref{Assump_2}.

\begin{remark} \label{R:Ex_Assump}
  The conditions of Assumption \ref{Assump_2} are valid in the following cases:
  \begin{itemize}
    \item[(A1)] When at least one of $\gamma_\varepsilon^0$ and $\gamma_\varepsilon^1$ is bounded from below by $\varepsilon$, we may consider any closed surface $\Gamma$.
    In this case, however, the perfect slip (i.e. $\gamma_\varepsilon=0$) of the fluid on the boundary $\Gamma_\varepsilon$ is not allowed.
    \item[(A2)] It is known (see e.g. \cite{Sh_18pre}*{Proposition 2.2}) that there exists no nontrivial Killing vector field on $\Gamma$ (i.e. $\mathcal{K}(\Gamma)=\{0\}$) if the genus of $\Gamma$ is greater than one.
    In this case $\mathcal{K}_g(\Gamma)=\{0\}$ for any $g=g_1-g_0$ and we may take arbitrary nonnegative $\gamma_\varepsilon^0$ and $\gamma_\varepsilon^1$ (bounded above by $\varepsilon$).
    Note that, if $\mathcal{K}_g(\Gamma)=\{0\}$, then $\mathcal{R}_g=\{0\}$ and the curved thin domain $\Omega_\varepsilon$ is not axially symmetric around any line for all $\varepsilon>0$ sufficiently small (see Lemma \ref{L:CTD_Rg}).
    \item[(A3)] As mentioned in Remark \ref{R:Killing}, if $\Gamma$ is a surface of revolution or it is closed and convex then $\mathcal{R}|_\Gamma=\mathcal{K}(\Gamma)$ and thus $\mathcal{R}_g|_\Gamma=\mathcal{K}_g(\Gamma)$ for any $g=g_1-g_0$.
    Also, the relation $\mathcal{R}_0\cap\mathcal{R}_1=\mathcal{R}_g$ holds if, for example, $g_0$ or $g_1$ is constant.
    In this case we only consider the perfect slip boundary conditions
    \begin{align} \label{E:Per_Slip}
      u\cdot n_\varepsilon = 0, \quad 2\nu P_\varepsilon D(u)n_\varepsilon = 0 \quad\text{on}\quad \Gamma_\varepsilon.
    \end{align}
    A typical but important example of this case is the thin spherical shell
    \begin{align*}
      \Omega_\varepsilon = \{x\in\mathbb{R}^3 \mid 1<|x|<1+\varepsilon\} \quad (\Gamma = S^2,\,g_0 \equiv 0,\,g_1 \equiv 1)
    \end{align*}
    around the unit sphere $S^2$ in $\mathbb{R}^3$ considered by Temam and Ziane \cite{TeZi97} under the Hodge boundary conditions
    \begin{align} \label{E:Hodge_BC}
      u\cdot n_\varepsilon = 0, \quad \mathrm{curl}\,u\times n_\varepsilon = 0 \quad\text{on}\quad \Gamma_\varepsilon.
    \end{align}
    Note that, if $u\cdot n_\varepsilon=0$ on $\Gamma_\varepsilon$, then (see \cite{MitMon09}*{Section 2} and Lemma \ref{L:Diff_SH})
    \begin{align*}
      2P_\varepsilon D(u)n_\varepsilon-\mathrm{curl}\,u\times n_\varepsilon = 2W_\varepsilon u \quad\text{on}\quad \Gamma_\varepsilon,
    \end{align*}
    where $W_\varepsilon$ is the Weingarten map (or the shape operator) of $\Gamma_\varepsilon$ representing the curvatures of $\Gamma_\varepsilon$ (see Section \ref{SS:Pre_Dom} for its definition).
    Hence the perfect slip boundary conditions \eqref{E:Per_Slip} are different from the Hodge boundary conditions \eqref{E:Hodge_BC} by the curvatures of the boundary.
  \end{itemize}
  We also note that, if $\Gamma=\mathbb{T}^2$ is the flat torus, then
  \begin{align*}
    \mathcal{R}_i &= \{(a_1,a_2,0)^T\in\mathbb{R}^2\times\{0\} \mid \text{$a_1\partial_1g_i+a_2\partial_2g_i=0$ on $\mathbb{T}^2$}\}, \quad i=0,1, \\
    \mathcal{R}_g &= \mathcal{K}_g(\Gamma) = \{(a_1,a_2,0)^T\in\mathbb{R}^2\times\{0\} \mid \text{$a_1\partial_1g+a_2\partial_2g=0$ on $\mathbb{T}^2$}\}
  \end{align*}
  and the conditions (A2) and (A3) were imposed in \cite{Ho10} and \cites{HoSe10,IfRaSe07}, respectively, which studied the Naiver--Stokes equations in a flat thin domain around $\Gamma=\mathbb{T}^2$.
\end{remark}

\begin{remark} \label{R:Cond_A3}
  For a function $\eta$ on $\Gamma$ let $\bar{\eta}$ be its constant extension in the normal direction of $\Gamma$ (see Section \ref{SS:Pre_Surf} for the precise definition) and
  \begin{align*}
    \overline{\mathcal{K}}_g(\Gamma) := \{\bar{v} \mid v\in\mathcal{K}_g(\Gamma)\}, \quad \mathbb{H}_\varepsilon := L_\sigma^2(\Omega_\varepsilon)\cap\overline{\mathcal{K}}_g(\Gamma)^\perp, \quad \mathbb{V}_\varepsilon := \mathbb{H}_\varepsilon\cap H^1(\Omega_\varepsilon)^3.
  \end{align*}
  Then by the uniform Korn inequality given in Lemma \ref{L:Korn_H1} we see that the bilinear form $a_\varepsilon$ is uniformly coercive on $\mathbb{V}_\varepsilon$ even if Assumption \ref{Assump_2} is not imposed.
  Since we can also show that $a_\varepsilon$ is uniformly bounded on $\mathbb{V}_\varepsilon$ under Assumption \ref{Assump_1} as in Theorem \ref{T:Uni_aeps}, we obtain a bounded linear operator $\mathbb{A}_\varepsilon$ from $\mathbb{V}_\varepsilon$ into its dual space induced by $a_\varepsilon$.
  This $\mathbb{A}_\varepsilon$, however, is not properly related to the Stokes problem \eqref{E:Stokes_CTD}.
  To see this, let $u\in\mathbb{V}_\varepsilon$ such that $f:=\mathbb{A}_\varepsilon u\in\mathbb{H}_\varepsilon$.
  Then
  \begin{align} \label{E:Wrong_Weak}
    a_\varepsilon(u,\varphi) = (f,\varphi)_{L^2(\Omega_\varepsilon)} \quad\text{for all}\quad \varphi \in \mathbb{V}_\varepsilon.
  \end{align}
  If \eqref{E:Wrong_Weak} was valid for all $\varphi\in L_\sigma^2(\Omega_\varepsilon)\cap H^1(\Omega_\varepsilon)^3$ then we could recover the Stokes problem \eqref{E:Stokes_CTD} from \eqref{E:Wrong_Weak} by a standard argument (see \cites{CoFo88,BoFa13,So01,Te79}), but we cannot verify it because of the condition $\varphi\in\overline{\mathcal{K}}_g(\Gamma)^\perp$ for the test function $\varphi$.
  Indeed, let $\varphi\in L_\sigma^2(\Omega_\varepsilon)\cap H^1(\Omega_\varepsilon)^3$ and assume that it can be decomposed into $\varphi=\Phi+\bar{v}$ with some $\Phi\in\mathbb{V}_\varepsilon$ and $\bar{v}\in\overline{\mathcal{K}}_g(\Gamma)$ (this is possible if $\overline{\mathcal{K}}_g(\Gamma)\subset L_\sigma^2(\Omega_\varepsilon)$, but such a relation is not always valid since $\bar{v}\in\overline{\mathcal{K}}_g(\Gamma)$ does not satisfy $\bar{v}\cdot n_\varepsilon=0$ on $\Gamma_\varepsilon$ in general).
  Then since \eqref{E:Wrong_Weak} is valid for $\Phi\in\mathbb{V}_\varepsilon$ and $(f,\bar{v})_{L^2(\Omega_\varepsilon)}=0$ by $f\in\mathbb{H}_\varepsilon$, to verify \eqref{E:Wrong_Weak} for $\varphi=\Phi+\bar{v}$ we need to show that
  \begin{align*}
    a_\varepsilon(u,\bar{v}) = 2\nu\bigl(D(u),D(\bar{v})\bigr)_{L^2(\Omega_\varepsilon)}+\gamma_\varepsilon^0(u,\bar{v})_{L^2(\Gamma_\varepsilon^0)}+\gamma_\varepsilon^1(u,\bar{v})_{L^2(\Gamma_\varepsilon^1)}
  \end{align*}
  vanishes.
  However, the second and third terms on the right-hand side do not vanish unless $\gamma_\varepsilon^0=\gamma_\varepsilon^1=0$.
  The first term also does not vanish in general, since for the constant extension $\bar{v}$ of a vector field $v$ on $\Gamma$ we observe by \eqref{E:ConDer_Dom} that
  \begin{align*}
    D(\bar{v})(x) = \frac{1}{2}\Bigl[\{I_3-rW(y)\}^{-1}\nabla_\Gamma v(y)+\{\nabla_\Gamma v(y)\}^T\{I_3-rW(y)\}^{-1}\Bigr]
  \end{align*}
  for $x=y+rn(y)\in\Omega_\varepsilon$ with $y\in\Gamma$ and $r\in(\varepsilon g_0(y),\varepsilon g_1(y))$, where $W$ is the Weingarten map of $\Gamma$ (see Section \ref{SS:Pre_Surf}), and $D(\bar{v})$ does not vanish on $\Omega_\varepsilon$ just by $D_\Gamma(v)=0$ on $\Gamma$ (even if $\Gamma=S^2$ and $\bar{v}(x)=e_3\times(x/|x|)$ is the constant extension of $v(y)=e_3\times y\in\mathcal{K}(S^2)$ with $e_3=(0,0,1)^T$).
  Thus we fail to show \eqref{E:Wrong_Weak} for $\varphi\in L_\sigma^2(\Omega_\varepsilon)\cap H^1(\Omega_\varepsilon)^3$ and it is not clear whether $u$ is a solution to the Stokes problem \eqref{E:Stokes_CTD} with $f=\mathbb{A}_\varepsilon u$.
  This observation implies that the operator $\mathbb{A}_\varepsilon$ is not appropriate for the study of the Navier--Stokes equations \eqref{E:NS_CTD}.

  The above problem does not occur if we impose Assumption \ref{Assump_2} and consider the bilinear form $a_\varepsilon$ on the function space $\mathcal{V}_\varepsilon$ given by \eqref{E:Def_Heps}.
  In this case, for $u\in D(A_\varepsilon)$ and $f:=A_\varepsilon u\in\mathcal{H}_\varepsilon$ we a priori have
  \begin{align} \label{E:Corr_Weak}
    a_\varepsilon(u,\varphi) = (f,\varphi)_{L^2(\Omega_\varepsilon)} \quad\text{for all}\quad \varphi \in \mathcal{V}_\varepsilon.
  \end{align}
  Under the condition (A1) or (A2) we have $\mathcal{V}_\varepsilon=L_\sigma^2(\Omega_\varepsilon)\cap H^1(\Omega_\varepsilon)^3$ and thus \eqref{E:Corr_Weak} implies that $u$ is indeed a solution to the Stokes problem \eqref{E:Stokes_CTD} with $f=A_\varepsilon u$.
  When we impose the condition (A3), $\mathcal{V}_\varepsilon$ may be smaller than $L_\sigma^2(\Omega_\varepsilon)\cap H^1(\Omega_\varepsilon)^3$.
  In this case, however, since $\mathcal{R}_g=\mathcal{R}_0\cap\mathcal{R}_1$ is of finite dimension and contained in $L_\sigma^2(\Omega_\varepsilon)$ by Lemma \ref{L:IR_Sole}, each $\varphi\in L_\sigma^2(\Omega_\varepsilon)\cap H^1(\Omega_\varepsilon)^3$ can be decomposed into $\varphi=\Phi+w$ with some $\Phi\in\mathcal{V}_\varepsilon$ and $w\in\mathcal{R}_g$.
  Then \eqref{E:Corr_Weak} holds for $\Phi\in\mathcal{V}_\varepsilon$ and, since $w\in\mathcal{R}_g$ is of the form $w(x)=a\times x+b$, $x\in\mathbb{R}^3$ with $a,b\in\mathbb{R}^3$, we easily get $D(w)=0$ on $\mathbb{R}^3$.
  From this fact and $\gamma_\varepsilon^0=\gamma_\varepsilon^1=0$ by the condition (A3) it follows that
  \begin{align*}
    a_\varepsilon(u,w) = 2\nu\bigl(D(u),D(w)\bigr)_{L^2(\Omega_\varepsilon)}+\gamma_\varepsilon^0(u,w)_{L^2(\Gamma_\varepsilon^0)}+\gamma_\varepsilon^1(u,w)_{L^2(\Gamma_\varepsilon^1)} = 0.
  \end{align*}
  Thus \eqref{E:Corr_Weak} is also valid for all $\varphi\in L_\sigma^2(\Omega_\varepsilon)\cap H^1(\Omega_\varepsilon)^3$ under the condition (A3).
\end{remark}

\section{Preliminaries} \label{S:Pre}
We fix notations on a closed surface and a curved thin domain and give their basic properties.
Notations on vectors and matrices are given in Appendix \ref{S:Ap_Vec}.

Some lemmas in this section are proved just by calculations involving differential geometry of surfaces.
We provide the proofs of them in Appendix \ref{S:Ap_Proof} to avoid making this section too long.
Also, some results in this section are not used in the following sections but essential for the second and third parts \cites{Miu_NSCTD_02,Miu_NSCTD_03} of our study.
We include them here since they easily follow from other results used in this paper or we can prove them just by a few discussions along with the proofs of the other results.

Throughout this paper we denote by $c$ a general positive constant independent of the parameter $\varepsilon$.
Also, we fix a coordinate system of $\mathbb{R}^3$ and write $x_i$, $i=1,2,3$ for the $i$-th component of a point $x\in\mathbb{R}^3$ under this coordinate system.

\subsection{Closed surface} \label{SS:Pre_Surf}
Let $\Gamma$ be a two-dimensional closed, connected, and oriented surface in $\mathbb{R}^3$.
We assume that $\Gamma$ is of class $C^\ell$ with $\ell\geq 2$.
By $n$ and $d$ we denote the unit outward normal vector field of $\Gamma$ and the signed distance function from $\Gamma$ increasing in the direction of $n$.
Also, let $\kappa_1$ and $\kappa_2$ be the principal curvatures of $\Gamma$.
From the $C^\ell$-regularity of $\Gamma$ it follows that $n\in C^{\ell-1}(\Gamma)^3$ and $\kappa_1,\kappa_2\in C^{\ell-2}(\Gamma)$.
In particular, $\kappa_1$ and $\kappa_2$ are bounded on the compact set $\Gamma$.
Hence we can take a tubular neighborhood
\begin{align*}
  N := \{x\in\mathbb{R}^3\mid \mathrm{dist}(x,\Gamma)<\delta\}, \quad \delta > 0
\end{align*}
of $\Gamma$ such that for each $x\in N$ there exists a unique point $\pi(x)\in\Gamma$ satisfying
\begin{align} \label{E:Nor_Coord}
  x = \pi(x)+d(x)n(\pi(x)), \quad \nabla d(x) = n(\pi(x)).
\end{align}
Moreover, $d$ and $\pi$ are of class $C^\ell$ and $C^{\ell-1}$ on $\overline{N}$ (see \cite{GiTr01}*{Section 14.6} for details).
By the boundedness of $\kappa_1$ and $\kappa_2$ we also have
\begin{align} \label{E:Curv_Bound}
  c^{-1} \leq 1-r\kappa_i(y) \leq c \quad\text{for all}\quad y\in\Gamma,\,r\in(-\delta,\delta),\,i=1,2
\end{align}
if we take $\delta>0$ sufficiently small.

Let us define differential operators on $\Gamma$.
For $y\in\Gamma$ we set
\begin{align*}
  P(y) := I_3-n(y)\otimes n(y), \quad Q(y) := n(y)\otimes n(y).
\end{align*}
The matrices $P$ and $Q$ are the orthogonal projections onto the tangent plane and the normal direction of $\Gamma$ and satisfy $|P|=2$, $|Q|=1$, and
\begin{gather*}
  I_3 = P+Q, \quad PQ = QP = 0, \quad P^T = P^2 = P, \quad Q^T = Q^2 = Q, \\
  |a|^2 = |Pa|^2+|Qa|^2, \quad |Pa| \leq |a|, \quad Pa\cdot n = 0, \quad a\in\mathbb{R}^3
\end{gather*}
on $\Gamma$.
Also, $P,Q\in C^{\ell-1}(\Gamma)^{3\times3}$ by the $C^\ell$-regularity of $\Gamma$.
For $\eta\in C^1(\Gamma)$ we define the tangential gradient and the tangential derivatives of $\eta$ as
\begin{align} \label{E:Def_TGr}
  \nabla_\Gamma\eta(y) := P(y)\nabla\tilde{\eta}(y), \quad \underline{D}_i\eta(y) := \sum_{j=1}^3P_{ij}(y)\partial_j\tilde{\eta}(y), \quad y\in\Gamma,\,i=1,2,3
\end{align}
so that $\nabla_\Gamma\eta=(\underline{D}_1\eta,\underline{D}_2\eta,\underline{D}_3\eta)^T$.
Here $\tilde{\eta}$ is a $C^1$-extension of $\eta$ to $N$ with $\tilde{\eta}|_\Gamma=\eta$.
Since $P^2=P$ and $n\cdot Pa=0$ on $\Gamma$ for $a\in\mathbb{R}^3$ we have
\begin{align} \label{E:P_TGr}
  P\nabla_\Gamma\eta = \nabla_\Gamma\eta, \quad n\cdot\nabla_\Gamma\eta = 0 \quad\text{on}\quad \Gamma.
\end{align}
Note that $\nabla_\Gamma\eta$ given by \eqref{E:Def_TGr} agrees with the gradient of a function on a Riemannian manifold expressed under a local coordinate system (see Lemma \ref{L:TGr_DG}).
Hence the values of $\nabla_\Gamma\eta$ and $\underline{D}_i\eta$ are independent of the choice of an extension $\tilde{\eta}$.
In particular, the constant extension $\bar{\eta}:=\eta\circ\pi$ of $\eta$ in the normal direction of $\Gamma$ satisfies
\begin{align} \label{E:ConDer_Surf}
  \nabla\bar{\eta}(y) = \nabla_\Gamma\eta(y), \quad \partial_i\bar{\eta}(y) = \underline{D}_i\eta(y), \quad y\in\Gamma,\,i=1,2,3
\end{align}
since $\nabla\pi(y)=P(y)$ for $y\in\Gamma$ by \eqref{E:Nor_Coord} and $d(y)=0$.
In what follows, the notation $\bar{\eta}$ with an overline always stands for the constant extension of a function $\eta$ on $\Gamma$ in the normal direction of $\Gamma$.
The tangential Hessian matrix of $\eta\in C^2(\Gamma)$ and the Laplace--Beltrami operator are given by
\begin{align*}
  \nabla_\Gamma^2\eta := (\underline{D}_i\underline{D}_j\eta)_{i,j}, \quad \Delta_\Gamma\eta := \mathrm{tr}[\nabla_\Gamma^2\eta] = \sum_{i=1}^3\underline{D}_i^2\eta \quad\text{on}\quad \Gamma.
\end{align*}
Note that $\nabla_\Gamma^2\eta$ is not symmetric in general (see Lemma \ref{L:TD_Exc}).
For a (not necessarily tangential) vector field $v=(v_1,v_2,v_3)^T\in C^1(\Gamma)^3$ we define the tangential gradient matrix and the surface divergence of $v$ by
\begin{align} \label{E:Def_DivG}
  \nabla_\Gamma v :=
  \begin{pmatrix}
    \underline{D}_1v_1 & \underline{D}_1v_2 & \underline{D}_1v_3 \\
    \underline{D}_2v_1 & \underline{D}_2v_2 & \underline{D}_2v_3 \\
    \underline{D}_3v_1 & \underline{D}_3v_2 & \underline{D}_3v_3
  \end{pmatrix}, \quad
  \mathrm{div}_\Gamma v := \mathrm{tr}[\nabla_\Gamma v] = \sum_{i=1}^3\underline{D}_iv_i \quad\text{on}\quad \Gamma
\end{align}
and the surface strain rate tensor for $v$ by
\begin{align*}
  D_\Gamma(v) := P(\nabla_\Gamma v)_SP \quad\text{on}\quad \Gamma, \quad (\nabla_\Gamma v)_S = \frac{\nabla_\Gamma v+(\nabla_\Gamma v)^T}{2}.
\end{align*}
Also, for $v\in C^1(\Gamma)^3$ and $\eta\in C(\Gamma)^3$ we set
\begin{align*}
  (\eta\cdot\nabla_\Gamma)v :=
  \begin{pmatrix}
    \eta\cdot\nabla_\Gamma v_1 \\
    \eta\cdot\nabla_\Gamma v_2 \\
    \eta\cdot\nabla_\Gamma v_3
  \end{pmatrix}
   = (\nabla_\Gamma v)^T\eta \quad\text{on}\quad \Gamma.
\end{align*}
Note that for any $C^1$-extension $\tilde{v}$ of $v$ to $N$ with $\tilde{v}|_\Gamma=v$ we have
\begin{align} \label{E:Tgrad_Surf}
  \nabla_\Gamma v = P\nabla\tilde{v}, \quad (\eta\cdot\nabla_\Gamma)v = [(P\eta)\cdot\nabla]\tilde{v} \quad\text{on}\quad \Gamma.
\end{align}
Next we define the Weingarten map $W$ and (twice) the mean curvature $H$ of $\Gamma$ by
\begin{align*}
  W := -\nabla_\Gamma n, \quad H := \mathrm{tr}[W] = -\mathrm{div}_\Gamma n \quad\text{on}\quad \Gamma.
\end{align*}
Note that $W$ and $H$ are of class $C^{\ell-2}$ and thus bounded on $\Gamma$.

\begin{lemma} \label{L:Form_W}
  The Weingarten map $W$ is symmetric and
  \begin{align} \label{E:Form_W}
    Wn = 0, \quad PW = WP = W \quad\text{on}\quad \Gamma.
  \end{align}
  Also, if $v\in C^1(\Gamma)^3$ is tangential, i.e. $v\cdot n=0$ on $\Gamma$, then
  \begin{align} \label{E:Grad_W}
    (\nabla_\Gamma v)n = Wv, \quad \nabla_\Gamma v = P(\nabla_\Gamma v)P+(Wv)\otimes n \quad\text{on}\quad \Gamma.
  \end{align}
\end{lemma}

\begin{proof}
  By \eqref{E:Nor_Coord}, \eqref{E:ConDer_Surf}, and $|n|^2=1$ on $\Gamma$ we have
  \begin{align*}
    W = -\nabla\bar{n} = -\nabla^2d, \quad Wn = -(\nabla_\Gamma n)n = -\frac{1}{2}\nabla_\Gamma(|n|^2) = 0 \quad\text{on}\quad \Gamma.
  \end{align*}
  Hence $W$ is symmetric and the equalities \eqref{E:Form_W} are valid.

  Let $v\in C^1(\Gamma)^3$ satisfy $v\cdot n=0$ on $\Gamma$.
  Then
  \begin{align*}
    (\nabla_\Gamma v)n = \nabla_\Gamma(v\cdot n)-(\nabla_\Gamma n)v = Wv \quad\text{on}\quad \Gamma.
  \end{align*}
  Thus the first equality of \eqref{E:Grad_W} holds.
  Also, by $I_3=P+Q$ on $\Gamma$ and \eqref{E:P_TGr} we have
  \begin{align*}
    \nabla_\Gamma v = (\nabla_\Gamma v)P+(\nabla_\Gamma v)Q = P(\nabla_\Gamma v)P+\{(\nabla_\Gamma v)n\}\otimes n \quad\text{on}\quad \Gamma.
  \end{align*}
  Hence the second equality of \eqref{E:Grad_W} follows from the first one.
\end{proof}

By \eqref{E:Form_W} we see that $W$ has the eigenvalue zero associated with the eigenvector $n$.
Its other eigenvalues are the principal curvatures $\kappa_1$ and $\kappa_2$ and thus $H=\kappa_1+\kappa_2$ on $\Gamma$ (see e.g. \cites{GiTr01,Lee18}).

The Weingarten map $W$ appears when we exchange the tangential derivatives and compute the gradient of the constant extension of a function on $\Gamma$.

\begin{lemma} \label{L:TD_Exc}
  For $\eta\in C^2(\Gamma)$ we have
  \begin{align} \label{E:TD_Exc}
    \underline{D}_i\underline{D}_j\eta-\underline{D}_j\underline{D}_i\eta = [W\nabla_\Gamma\eta]_in_j-[W\nabla_\Gamma\eta]_jn_i \quad\text{on}\quad \Gamma,\,i,j=1,2,3.
  \end{align}
  Here $[W\nabla_\Gamma\eta]_i$ is the $i$-th component of the vector field $W\nabla_\Gamma\eta$ for $i=1,2,3$.
\end{lemma}

\begin{lemma} \label{L:Wein}
  The matrix
  \begin{align*}
    I_3-d(x)\overline{W}(x) = I_3-rW(y)
  \end{align*}
  is invertible for all $x=y+rn(y)\in N$ with $y\in\Gamma$ and $r\in(-\delta,\delta)$.
  Moreover,
  \begin{align} \label{E:WReso_P}
    \{I_3-rW(y)\}^{-1}P(y) = P(y)\{I_3-rW(y)\}^{-1}
  \end{align}
  for all $y\in\Gamma$ and $r\in(-\delta,\delta)$ and there exists a constant $c>0$ such that
  \begin{gather}
    c^{-1}|a| \leq \bigl|\{I_3-rW(y)\}^ka\bigr| \leq c|a|, \quad k=\pm1, \label{E:Wein_Bound} \\
    \bigl|I_3-\{I_3-rW(y)\}^{-1}\bigr| \leq c|r| \label{E:Wein_Diff}
  \end{gather}
  for all $y\in\Gamma$, $r\in(-\delta,\delta)$, and $a\in\mathbb{R}^3$.
\end{lemma}

\begin{lemma} \label{L:Pi_Der}
  For all $x\in N$ we have
  \begin{align} \label{E:Pi_Der}
    \nabla\pi(x) &= \left\{I_3-d(x)\overline{W}(x)\right\}^{-1}\overline{P}(x).
  \end{align}
  Let $\eta\in C^1(\Gamma)$.
  Then its constant extension $\bar{\eta}=\eta\circ\pi$ satisfies
  \begin{align} \label{E:ConDer_Dom}
    \nabla\bar{\eta}(x) = \left\{I_3-d(x)\overline{W}(x)\right\}^{-1}\overline{\nabla_\Gamma\eta}(x),\quad x\in N
  \end{align}
  and there exists a constant $c>0$ independent of $\eta$ such that
  \begin{gather}
    c^{-1}\left|\overline{\nabla_\Gamma\eta}(x)\right| \leq |\nabla\bar{\eta}(x)| \leq c\left|\overline{\nabla_\Gamma\eta}(x)\right|, \label{E:ConDer_Bound} \\
    \left|\nabla\bar{\eta}(x)-\overline{\nabla_\Gamma\eta}(x)\right| \leq c\left|d(x)\overline{\nabla_\Gamma\eta}(x)\right| \label{E:ConDer_Diff}
  \end{gather}
  for all $x\in N$.
  If $\Gamma$ is of class $C^3$ and $\eta\in C^2(\Gamma)$, then we have
  \begin{align} \label{E:Con_Hess}
    |\nabla^2\bar{\eta}(x)| \leq c\left(\left|\overline{\nabla_\Gamma\eta}(x)\right|+\left|\overline{\nabla_\Gamma^2\eta}(x)\right|\right), \quad x\in N.
  \end{align}
\end{lemma}

Lemmas \ref{L:TD_Exc}--\ref{L:Pi_Der} are proved in Appendix \ref{S:Ap_Proof}.
Note that
\begin{align} \label{E:Nor_Grad}
  \nabla\bar{n}(x) = -\left\{I_3-d(x)\overline{W}(x)\right\}^{-1}\overline{W}(x), \quad x\in N
\end{align}
by \eqref{E:ConDer_Dom} and $W=-\nabla_\Gamma n$ on $\Gamma$.

Let us define the Sobolev spaces on $\Gamma$.
To this end, we give an integration by parts formula for the tangential derivatives of functions on $\Gamma$ (see also \cite{DzEl13}*{Theorem 2.10} and \cite{GiTr01}*{Lemma 16.1}).

\begin{lemma} \label{L:IbP_TD}
  For $v\in C^1(\Gamma)^3$ we have
  \begin{align} \label{E:IbP_DivG}
    \int_\Gamma\mathrm{div}_\Gamma v\,d\mathcal{H}^2 = -\int_\Gamma(v\cdot n)H\,d\mathcal{H}^2,
  \end{align}
  where $\mathcal{H}^2$ is the two-dimensional Hausdorff measure.
  Moreover,
  \begin{align} \label{E:IbP_TD}
    \int_\Gamma(\eta\underline{D}_i\xi+\xi\underline{D}_i\eta)\,d\mathcal{H}^2 = -\int_\Gamma \eta\xi Hn_i\,d\mathcal{H}^2
  \end{align}
  for $\eta,\xi\in C^1(\Gamma)$ and $i=1,2,3$.
\end{lemma}

\begin{proof}
  First note that, when $X\in C^1(\Gamma)^3$ is tangential on $\Gamma$, its surface divergence defined by \eqref{E:Def_DivG} agrees with the divergence on a Riemannian manifold expressed under a local coordinate system (see Lemma \ref{L:DivG_DG}).
  Hence, noting that $\Gamma$ is closed, by a standard localization argument we have the surface divergence theorem
  \begin{align} \label{Pf_ITD:SD_Thm}
    \int_\Gamma\mathrm{div}_\Gamma X\,d\mathcal{H}^2 = 0.
  \end{align}
  Let $v\in C^1(\Gamma)^3$.
  By the decomposition $v=Pv+(v\cdot n)n$ we have
  \begin{align*}
    \int_\Gamma\mathrm{div}_\Gamma v\,d\mathcal{H}^2 = \int_\Gamma\mathrm{div}_\Gamma(Pv)\,d\mathcal{H}^2+\int_\Gamma\mathrm{div}_\Gamma[(v\cdot n)n]\,d\mathcal{H}^2.
  \end{align*}
  Since $Pv$ is tangential on $\Gamma$, the first term on the right-hand side vanishes by \eqref{Pf_ITD:SD_Thm} with $X=Pv$.
  Moreover, from \eqref{E:P_TGr} and the definition of $H$ it follows that
  \begin{align*}
    \mathrm{div}_\Gamma[(v\cdot n)n] = [\nabla_\Gamma(v\cdot n)]\cdot n+(v\cdot n)\mathrm{div}_\Gamma n = -(v\cdot n)H \quad\text{on}\quad \Gamma.
  \end{align*}
  Hence \eqref{E:IbP_DivG} follows.
  We also get \eqref{E:IbP_TD} by setting $v=\eta\xi e_i$ for $\eta,\xi\in C^1(\Gamma)$ and $i=1,2,3$ in \eqref{E:IbP_DivG}, where $\{e_1,e_2,e_3\}$ is the standard basis of $\mathbb{R}^3$.
\end{proof}

Based on \eqref{E:IbP_TD}, for $p\in[1,\infty]$ and $i=1,2,3$ we say that $\eta\in L^p(\Gamma)$ has the $i$-th weak tangential derivative if there exists $\eta_i\in L^p(\Gamma)$ such that
\begin{align} \label{E:Def_WTD}
  \int_\Gamma \eta_i\xi\,d\mathcal{H}^2 = -\int_\Gamma \eta(\underline{D}_i\xi+\xi Hn_i)\,d\mathcal{H}^2
\end{align}
for all $\xi\in C^1(\Gamma)$.
In this case we write $\underline{D}_i\eta=\eta_i$ and define the Sobolev space
\begin{align*}
  W^{1,p}(\Gamma) &:= \{\eta \in L^p(\Gamma) \mid \text{$\underline{D}_i\eta\in L^p(\Gamma)$ for all $i=1,2,3$}\}, \\
  \|\eta\|_{W^{1,p}(\Gamma)} &:=
  \begin{cases}
    \left(\|\eta\|_{L^p(\Gamma)}^p+\|\nabla_\Gamma\eta\|_{L^p(\Gamma)}^p\right)^{1/p} &\text{if}\quad p\in[1,\infty), \\
    \|\eta\|_{L^\infty(\Gamma)}+\|\nabla_\Gamma\eta\|_{L^\infty(\Gamma)} &\text{if}\quad p=\infty.
  \end{cases}
\end{align*}
Here $\nabla_\Gamma\eta:=(\underline{D}_1\eta,\underline{D}_2\eta,\underline{D}_3\eta)^T$ is the weak tangential gradient of $\eta\in W^{1,p}(\Gamma)$.
This notation is consistent with \eqref{E:Def_TGr} for a $C^1$ function on $\Gamma$.
Also, by \eqref{E:Def_WTD},
\begin{align} \label{E:IbP_WTGr}
  \int_\Gamma\nabla_\Gamma\eta\cdot v\,d\mathcal{H}^2 = -\int_\Gamma\eta\{\mathrm{div}_\Gamma v+(v\cdot n)H\}\,d\mathcal{H}^2
\end{align}
for $\eta\in W^{1,p}(\Gamma)$ and $v\in C^1(\Gamma)^3$.
We also define the second order Sobolev space
\begin{align*}
  W^{2,p}(\Gamma) &:= \{\eta \in W^{1,p}(\Gamma) \mid \text{$\underline{D}_i\underline{D}_j\eta\in L^p(\Gamma)$ for all $i,j=1,2,3$}\}, \\
  \|\eta\|_{W^{2,p}(\Gamma)} &:=
  \begin{cases}
    \left(\|\eta\|_{W^{1,p}(\Gamma)}^p+\|\nabla_\Gamma^2\eta\|_{L^p(\Gamma)}^p\right)^{1/p} &\text{if}\quad p\in[1,\infty), \\
    \|\eta\|_{W^{1,\infty}(\Gamma)}+\|\nabla_\Gamma^2\eta\|_{L^\infty(\Gamma)} &\text{if}\quad p=\infty
  \end{cases}
\end{align*}
and the higher order Sobolev space $W^{m,p}(\Gamma)$ with $m\geq 2$ similarly, and write
\begin{align*}
  W^{0,p}(\Gamma):=L^p(\Gamma), \quad H^m(\Gamma):=W^{m,2}(\Gamma), \quad p\in[1,\infty], \, m \geq 0.
\end{align*}
Here $\nabla_\Gamma^2\eta:=(\underline{D}_i\underline{D}_j\eta)_{i,j}$ for $\eta\in W^{2,p}(\Gamma)$.
Note that $W^{m,p}(\Gamma)$ is a Banach space.
Since $\Gamma$ is of class $C^\ell$, a function in $W^{m,p}(\Gamma)$ is approximated by $C^\ell$ functions on $\Gamma$ when $m\leq \ell$ and $p\neq\infty$.

\begin{lemma} \label{L:Wmp_Appr}
  Let $m=0,1,\dots,\ell$ and $p\in[1,\infty)$.
  Then $C^\ell(\Gamma)$ is dense in $W^{m,p}(\Gamma)$.
\end{lemma}

By Lemma \ref{L:Wmp_Appr} we can apply the equalities and inequalities for functions in $C^1(\Gamma)$ or $C^2(\Gamma)$ given in this subsection to those in $W^{1,p}(\Gamma)$ or $W^{2,p}(\Gamma)$ with $p\in[1,\infty)$.
Lemma \ref{L:Wmp_Appr} is shown by standard localization and mollification arguments.
We give its proof in Appendix \ref{S:Ap_Proof} for the completeness.

Let $\mathcal{X}(\Gamma)$ be a function space on $\Gamma$ such as $C^m(\Gamma)$ and $W^{m,p}(\Gamma)$.
We define the space of all tangential vector fields on $\Gamma$ whose components belong to $\mathcal{X}(\Gamma)$ by
\begin{align*}
  \mathcal{X}(\Gamma,T\Gamma) := \{v\in\mathcal{X}(\Gamma)^3 \mid \text{$v\cdot n=0$ on $\Gamma$}\}.
\end{align*}
Then $W^{m,p}(\Gamma,T\Gamma)$ is a closed subspace of $W^{m,p}(\Gamma)^3$ for $m\geq0$ and $p\in[1,\infty]$.
Also, for $v\in W^{1,p}(\Gamma,T\Gamma)$ with $p\in[1,\infty]$ we have
\begin{align} \label{E:IbP_WDivG_T}
  \int_\Gamma\mathrm{div}_\Gamma v\,d\mathcal{H}^2 = -\int_\Gamma(v\cdot n)H\,d\mathcal{H}^2 = 0
\end{align}
by \eqref{E:Def_WTD} with $\xi\equiv1$ (note that $\nabla_\Gamma\xi=0$ on $\Gamma$ if $\xi$ is constant).
When $m\leq\ell-1$ and $p\neq\infty$, an element of $W^{m,p}(\Gamma,T\Gamma)$ is approximated by $C^{\ell-1}$ tangential vector fields on $\Gamma$.

\begin{lemma} \label{L:Wmp_Tan_Appr}
  Let $m=0,1,\dots,\ell-1$ and $p\in[1,\infty)$.
  Then $C^{\ell-1}(\Gamma,T\Gamma)$ is dense in $W^{m,p}(\Gamma,T\Gamma)$ with respect to the norm $\|\cdot\|_{W^{m,p}(\Gamma)}$.
\end{lemma}

\begin{proof}
  Let $v\in W^{m,p}(\Gamma,T\Gamma)\subset W^{m,p}(\Gamma)^3$.
  By Lemma \ref{L:Wmp_Appr} we can take a sequence $\{\tilde{v}_k\}_{k=1}^\infty$ in $C^\ell(\Gamma)^3$ that converges to $v$ strongly in $W^{m,p}(\Gamma)^3$.
  For each $k\in\mathbb{N}$ let $v_k:=P\tilde{v}_k$ on $\Gamma$.
  Then $v_k\in C^{\ell-1}(\Gamma,T\Gamma)$ since $P$ is of class $C^{\ell-1}$ on $\Gamma$.
  Moreover, since $v$ is tangential on $\Gamma$, we have $v-v_k=P(v-\tilde{v}_k)$ on $\Gamma$ and thus
  \begin{align*}
    \|v-v_k\|_{W^{m,p}(\Gamma)} \leq c\|v-\tilde{v}_k\|_{W^{m,p}(\Gamma)} \to 0 \quad\text{as}\quad k\to\infty
  \end{align*}
  by the $C^{\ell-1}$-regularity of $P$ on $\Gamma$ and the strong convergence of $\{\tilde{v}_k\}_{k=1}^\infty$ to $v$ in $W^{m,p}(\Gamma)^3$.
  Hence the claim is valid.
\end{proof}

\subsection{Curved thin domain} \label{SS:Pre_Dom}
From now on, we assume that the closed surface $\Gamma$ is of class $C^5$.
Let $g_0,g_1\in C^4(\Gamma)$ such that $g:=g_1-g_0$ satisfies \eqref{E:G_Inf}.
For $\varepsilon\in(0,1]$ we define a curved thin domain $\Omega_\varepsilon$ in $\mathbb{R}^3$ by \eqref{E:CTD_Intro}, i.e.
\begin{align*}
  \Omega_\varepsilon := \{y+rn(y) \mid y\in\Gamma,\,\varepsilon g_0(y) < r < \varepsilon g_1(y)\}.
\end{align*}
Since $g_0$ and $g_1$ are bounded on $\Gamma$, there exists $\tilde{\varepsilon}\in(0,1]$ such that $\tilde{\varepsilon}|g_i|<\delta$ on $\Gamma$ for $i=0,1$, where $\delta>0$ is the radius of the tubular neighborhood $N$ of $\Gamma$ given in Section \ref{SS:Pre_Surf}.
Hence $\overline{\Omega}_\varepsilon\subset N$ and the lemmas in Section \ref{SS:Pre_Surf} are applicable in $\overline{\Omega}_\varepsilon$ for all $\varepsilon\in(0,\tilde{\varepsilon}]$.
In what follows, we assume $\tilde{\varepsilon}=1$ by replacing $g_i$ with $\tilde{\varepsilon} g_i$ for $i=0,1$.

Let $\Gamma_\varepsilon^0$ and $\Gamma_\varepsilon^1$ be the inner and outer boundaries of $\Omega_\varepsilon$ defined as
\begin{align*}
  \Gamma_\varepsilon^i:=\{y+\varepsilon g_i(y)n(y) \mid y\in\Gamma\}, \quad i=0,1.
\end{align*}
Then the whole boundary of $\Omega_\varepsilon$ is given by $\Gamma_\varepsilon:=\Gamma_\varepsilon^0\cup\Gamma_\varepsilon^1$.
Note that $\Gamma_\varepsilon$ is of class $C^4$ by the $C^5$-regularity of $\Gamma$ and $g_0,g_1\in C^4(\Gamma)$.
We use this fact in the proof of a uniform a priori estimate for the vector Laplace operator (see Section \ref{S:UniLap}).

Let us give surface quantities on $\Gamma_\varepsilon$.
We define vector fields $\tau_\varepsilon^i$ and $n_\varepsilon^i$ on $\Gamma$ as
\begin{align}
  \tau_\varepsilon^i(y) &:= \{I_3-\varepsilon g_i(y)W(y)\}^{-1}\nabla_\Gamma g_i(y), \label{E:Def_NB_Aux}\\
  n_\varepsilon^i(y) &:= (-1)^{i+1}\frac{n(y)-\varepsilon\tau_\varepsilon^i(y)}{\sqrt{1+\varepsilon^2|\tau_\varepsilon^i(y)|^2}} \label{E:Def_NB}
\end{align}
for $y\in\Gamma$ and $i=0,1$.
Note that $\tau_\varepsilon^i$ is tangential on $\Gamma$ by \eqref{E:P_TGr}, \eqref{E:WReso_P}, and $Pa\cdot n=0$ on $\Gamma$ for $a\in\mathbb{R}^3$.
Also, $\tau_\varepsilon^i$ and $n_\varepsilon^i$ are bounded on $\Gamma$ uniformly in $\varepsilon$ along with their first and second order tangential derivatives.

\begin{lemma} \label{L:NB_Aux}
  There exists a constant $c>0$ independent of $\varepsilon$ such that
  \begin{gather}
    |\tau_\varepsilon^i(y)| \leq c, \quad |\underline{D}_k\tau_\varepsilon^i(y)| \leq c, \quad |\underline{D}_l\underline{D}_k\tau_\varepsilon^i(y)| \leq c, \label{E:Tau_Bound} \\
    |\tau_\varepsilon^i(y)-\nabla_\Gamma g_i(y)| \leq c\varepsilon, \quad |\nabla_\Gamma\tau_\varepsilon^i(y)-\nabla_\Gamma^2g_i(y)| \leq c\varepsilon \label{E:Tau_Diff}
  \end{gather}
  for all $y\in\Gamma$, $i=0,1$, and $k,l=1,2,3$.
  We also have
  \begin{gather}
    |n_\varepsilon^i(y)| = 1, \quad |\underline{D}_kn_\varepsilon^i(y)| \leq c, \quad |\underline{D}_l\underline{D}_kn_\varepsilon^i(y)| \leq c, \label{E:N_Bound} \\
    |n_\varepsilon^0(y)+n_\varepsilon^1(y)| \leq c\varepsilon, \quad |\nabla_\Gamma n_\varepsilon^0(y)+\nabla_\Gamma n_\varepsilon^1(y)| \leq c\varepsilon \label{E:N_Diff}
  \end{gather}
  for all $y\in\Gamma$, $i=0,1$, and $k,l=1,2,3$.
\end{lemma}

We present the proof of Lemmas \ref{L:NB_Aux} in Appendix \ref{S:Ap_Proof}.

Let $n_\varepsilon$ be the unit outward normal vector field of $\Gamma_\varepsilon$.
For $i=0,1$ the direction of $n_\varepsilon$ on $\Gamma_\varepsilon^i$ is the same as that of $(-1)^{i+1}\bar{n}$ since the signed distance function $d$ from $\Gamma$ increases in the direction of $n$.

\begin{lemma} \label{L:Nor_Bo}
  The unit outward normal vector field $n_\varepsilon$ of $\Gamma_\varepsilon$ is given by
  \begin{align} \label{E:Nor_Bo}
    n_\varepsilon(x) = \bar{n}_\varepsilon^i(x), \quad x\in\Gamma_\varepsilon^i,\,i=0,1.
  \end{align}
  Here $\bar{n}_\varepsilon^i=n_\varepsilon^i\circ\pi$ is the constant extension of the vector field $n_\varepsilon^i$ given by \eqref{E:Def_NB}.
\end{lemma}

\begin{proof}
  We observe in Lemma \ref{L:Para_Nor} that, if we define
  \begin{align} \label{Pf_NB:Vec}
    \tau_h(y) := \{I_3-h(y)W(y)\}^{-1}\nabla_\Gamma h(y), \quad n_h(y) := \frac{n(y)-\tau_h(y)}{\sqrt{1+|\tau_h(y)|^2}}
  \end{align}
  for $y\in\Gamma$ and $h\in C^1(\Gamma)$ satisfying $|h|<\delta$ on $\Gamma$, then the constant extension of $n_h$ is a unit normal vector field of the parametrized surface
  \begin{align} \label{Pf_NB:Para}
    \Gamma_h := \{y+h(y)n(y) \mid y\in\Gamma\}.
  \end{align}
  Setting $h=\varepsilon g_i$ in Lemma \ref{L:Para_Nor} and noting that the direction of $n_\varepsilon$ on $\Gamma_\varepsilon^i$ is the same as that of $(-1)^{i+1}\bar{n}$ for $i=0,1$ we obtain \eqref{E:Nor_Bo}.
\end{proof}

As in Section \ref{SS:Pre_Surf} we set $P_\varepsilon:=I_3-n_\varepsilon\otimes n_\varepsilon$ and $Q_\varepsilon:=n_\varepsilon\otimes n_\varepsilon$ on $\Gamma_\varepsilon$ and define the tangential gradient and the tangential derivatives of $\varphi\in C^1(\Gamma_\varepsilon)$ by
\begin{align*}
  \nabla_{\Gamma_\varepsilon}\varphi := P_\varepsilon\nabla\tilde{\varphi}, \quad \underline{D}_i^\varepsilon\varphi := \sum_{j=1}^3[P_\varepsilon]_{ij}\partial_j\tilde{\varphi} \quad\text{on}\quad \Gamma_\varepsilon,\, i=1,2,3,
\end{align*}
where $\tilde{\varphi}$ is any $C^1$-extension of $\varphi$ to an open neighborhood of $\Gamma_\varepsilon$ with $\tilde{\varphi}|_{\Gamma_\varepsilon}=\varphi$.
For $u=(u_1,u_2,u_3)^T\in C^1(\Gamma_\varepsilon)^3$ we define the tangential gradient matrix and the surface divergence of $u$ by
\begin{align*}
  \nabla_{\Gamma_\varepsilon}u :=
  \begin{pmatrix}
    \underline{D}_1^\varepsilon u_1 & \underline{D}_1^\varepsilon u_2 & \underline{D}_1^\varepsilon u_3 \\
    \underline{D}_2^\varepsilon u_1 & \underline{D}_2^\varepsilon u_2 & \underline{D}_2^\varepsilon u_3 \\
    \underline{D}_3^\varepsilon u_1 & \underline{D}_3^\varepsilon u_2 & \underline{D}_3^\varepsilon u_3
  \end{pmatrix}, \quad
  \mathrm{div}_{\Gamma_\varepsilon}u := \mathrm{tr}[\nabla_{\Gamma_\varepsilon}u] = \sum_{i=1}^3\underline{D}_i^\varepsilon u_i \quad\text{on}\quad \Gamma_\varepsilon.
\end{align*}
Also, for $u\in C^1(\Gamma_\varepsilon)^3$ and $\varphi\in C(\Gamma_\varepsilon)^3$ we write
\begin{align*}
  (\varphi\cdot\nabla_{\Gamma_\varepsilon})u :=
  \begin{pmatrix}
    \varphi\cdot\nabla_{\Gamma_\varepsilon}u_1 \\
    \varphi\cdot\nabla_{\Gamma_\varepsilon}u_2 \\
    \varphi\cdot\nabla_{\Gamma_\varepsilon}u_3
  \end{pmatrix}
  = (\nabla_{\Gamma_\varepsilon}u)^T\varphi \quad\text{on}\quad \Gamma_\varepsilon.
\end{align*}
Note that, as in the case of $\Gamma$, we have
\begin{align} \label{E:Tgrad_Bo}
  \nabla_{\Gamma_\varepsilon}u = P_\varepsilon\nabla\tilde{u}, \quad (\varphi\cdot\nabla_{\Gamma_\varepsilon})u = [(P_\varepsilon\varphi)\cdot\nabla]\tilde{u} \quad\text{on}\quad \Gamma_\varepsilon
\end{align}
for any $C^1$-extension $\tilde{u}$ of $u$ to an open neighborhood of $\Gamma_\varepsilon$ with $\tilde{u}|_{\Gamma_\varepsilon}=u$.
We also define the Weingarten map $W_\varepsilon$ and (twice) the mean curvature $H_\varepsilon$ of $\Gamma_\varepsilon$ as
\begin{align*}
  W_\varepsilon := -\nabla_{\Gamma_\varepsilon}n_\varepsilon, \quad H_\varepsilon := \mathrm{tr}[W_\varepsilon] = -\mathrm{div}_{\Gamma_\varepsilon}n_\varepsilon \quad\text{on}\quad \Gamma_\varepsilon.
\end{align*}
Then by Lemma \ref{L:Form_W} we have
\begin{align} \label{E:PW_Bo}
  W_\varepsilon^T = P_\varepsilon W_\varepsilon = W_\varepsilon P_\varepsilon = W_\varepsilon \quad\text{on}\quad \Gamma_\varepsilon.
\end{align}
The weak tangential derivatives of functions on $\Gamma_\varepsilon$ and the Sobolev spaces $W^{m,p}(\Gamma_\varepsilon)$ for $m\in\mathbb{N}$ and $p\in[1,\infty]$ are also defined as in Section \ref{SS:Pre_Surf}.

By the expression \eqref{E:Def_NB} of the unit outward normal $n_\varepsilon$ to $\Gamma_\varepsilon$, we can compare the surface quantities of $\Gamma_\varepsilon$ with those of $\Gamma$.

\begin{lemma} \label{L:Comp_Nor}
  There exists a constant $c>0$ independent of $\varepsilon$ such that
  \begin{gather}
    \left|n_\varepsilon(x)-(-1)^{i+1}\left\{\bar{n}(x)-\varepsilon\overline{\nabla_\Gamma g_i}(x)\right\}\right| \leq c\varepsilon^2, \label{E:Comp_N} \\
    \left|P_\varepsilon(x)-\overline{P}(x)\right| \leq c\varepsilon, \quad \left|Q_\varepsilon(x)-\overline{Q}(x)\right| \leq c\varepsilon, \label{E:Comp_P} \\
    \left|W_\varepsilon(x)-(-1)^{i+1}\overline{W}(x)\right| \leq c\varepsilon, \quad \left|H_\varepsilon(x)-(-1)^{i+1}\overline{H}(x)\right| \leq c\varepsilon, \label{E:Comp_W} \\
    \left|\underline{D}_j^\varepsilon W_\varepsilon(x)-(-1)^{i+1}\overline{\underline{D}_jW}(x)\right| \leq c\varepsilon \label{E:Comp_DW}
  \end{gather}
  for all $x\in\Gamma_\varepsilon^i$, $i=0,1$, and $j=1,2,3$.
\end{lemma}

From Lemma \ref{L:Comp_Nor} it follows that $W_\varepsilon$, $H_\varepsilon$, and $\underline{D}_j^\varepsilon W_\varepsilon$ with $j=1,2,3$ are uniformly bounded in $\varepsilon$ on $\Gamma_\varepsilon$ (note that $|P_\varepsilon|=2$ and $|Q_\varepsilon|=1$ on $\Gamma_\varepsilon$).
Moreover, we can compare the surface quantities of the inner and outer boundaries.

\begin{lemma} \label{L:Diff_SQ_IO}
  There exists a constant $c>0$ independent of $\varepsilon$ such that
  \begin{align}
    |F_\varepsilon(y+\varepsilon g_1(y)n(y))-F_\varepsilon(y+\varepsilon g_0(y)n(y))| &\leq c\varepsilon, \label{E:Diff_PQ_IO} \\
    |G_\varepsilon(y+\varepsilon g_1(y)n(y))+G_\varepsilon(y+\varepsilon g_0(y)n(y))| &\leq c\varepsilon \label{E:Diff_WH_IO}
  \end{align}
  for all $y\in\Gamma$, where $F_\varepsilon=P_\varepsilon,Q_\varepsilon$ and $G_\varepsilon=W_\varepsilon,H_\varepsilon,\underline{D}_j^\varepsilon W_\varepsilon$ with $j=1,2,3$.
\end{lemma}

The proofs of Lemmas \ref{L:Comp_Nor} and \ref{L:Diff_SQ_IO} are given in Appendix \ref{S:Ap_Proof}.

Next we give a change of variables formula for an integral over $\Omega_\varepsilon$.
For functions $\varphi$ on $\Omega_\varepsilon$ and $\eta$ on $\Gamma_\varepsilon^i$, $i=0,1$ we use the notations
\begin{alignat}{2}
  \varphi^\sharp(y,r) &:= \varphi(y+rn(y)), &\quad &y\in\Gamma,\,r\in(\varepsilon g_0(y),\varepsilon g_1(y)), \label{E:Pull_Dom} \\
  \eta_i^\sharp(y) &:= \eta(y+\varepsilon g_i(y)n(y)), &\quad &y\in\Gamma. \label{E:Pull_Bo}
\end{alignat}
Let $J=J(y,r)$ be a function given by
\begin{align} \label{E:Def_Jac}
  J(y,r) := \mathrm{det}[I_3-rW(y)] = \{1-r\kappa_1(y)\}\{1-r\kappa_2(y)\}
\end{align}
for $y\in\Gamma$ and $r\in(-\delta,\delta)$.
By \eqref{E:Curv_Bound} and $\kappa_1,\kappa_2\in C^3(\Gamma)$ we have
\begin{gather}
  c^{-1} \leq J(y,r) \leq c, \quad \left|\frac{\partial J}{\partial r}(y,r)\right| \leq c, \label{E:Jac_Bound_01} \\
  |J(y,r)-1| \leq c|r| \label{E:Jac_Diff_01}
\end{gather}
for all $y\in\Gamma$ and $r\in(-\delta,\delta)$.
Note that $J$ is the Jacobian appearing in the change of variables formula
\begin{align} \label{E:CoV_Dom}
  \int_{\Omega_\varepsilon}\varphi(x)\,dx = \int_\Gamma\int_{\varepsilon g_0(y)}^{\varepsilon g_1(y)}\varphi(y+rn(y))J(y,r)\,dr\,d\mathcal{H}^2(y)
\end{align}
for a function $\varphi$ on $\Omega_\varepsilon$ (see e.g. \cite{GiTr01}*{Section 14.6}).
The formula \eqref{E:CoV_Dom} can be seen as a co-area formula.
From \eqref{E:Jac_Bound_01} and \eqref{E:CoV_Dom} it immediately follows that
\begin{align} \label{E:CoV_Equiv}
  c^{-1}\|\varphi\|_{L^p(\Omega_\varepsilon)}^p \leq \int_\Gamma\int_{\varepsilon g_0(y)}^{\varepsilon g_1(y)}|\varphi^\sharp(y,r)|^p\,dr\,d\mathcal{H}^2(y) \leq c\|\varphi\|_{L^p(\Omega_\varepsilon)}^p
\end{align}
for $\varphi\in L^p(\Omega_\varepsilon)$, $p\in[1,\infty)$.
In the sequel we frequently use this inequality and the following estimates for the constant extension $\bar{\eta}=\eta\circ\pi$ of a function $\eta$ on $\Gamma$.

\begin{lemma} \label{L:Con_Lp_W1p}
  For $p\in[1,\infty)$ we have $\eta\in L^p(\Gamma)$ if and only if $\bar{\eta}\in L^p(\Omega_\varepsilon)$, and there exists a constant $c>0$ independent of $\varepsilon$ and $\eta$ such that
  \begin{align} \label{E:Con_Lp}
    c^{-1}\varepsilon^{1/p}\|\eta\|_{L^p(\Gamma)} \leq \|\bar{\eta}\|_{L^p(\Omega_\varepsilon)} \leq c\varepsilon^{1/p}\|\eta\|_{L^p(\Gamma)}.
  \end{align}
  Moreover, $\eta\in W^{1,p}(\Gamma)$ if and only if $\bar{\eta}\in W^{1,p}(\Omega_\varepsilon)$ and we have
  \begin{align} \label{E:Con_W1p}
    c^{-1}\varepsilon^{1/p}\|\eta\|_{W^{1,p}(\Gamma)} \leq \|\bar{\eta}\|_{W^{1,p}(\Omega_\varepsilon)} \leq c\varepsilon^{1/p}\|\eta\|_{W^{1,p}(\Gamma)}.
  \end{align}
\end{lemma}

\begin{proof}
  Since $\bar{\eta}^\sharp(y,r)=\eta(y)$ for $y\in\Gamma$ and $r\in(\varepsilon g_0(y),\varepsilon g_1(y))$,
  \begin{align*}
    \int_\Gamma\int_{\varepsilon g_0(y)}^{\varepsilon g_1(y)}|\bar{\eta}^\sharp(y,r)|^p\,dr\,d\mathcal{H}^2(y) = \varepsilon\int_\Gamma g(y)|\eta(y)|^p\,d\mathcal{H}^2(y).
  \end{align*}
  By this equality, \eqref{E:G_Inf}, and \eqref{E:CoV_Equiv} we get \eqref{E:Con_Lp}.
  Similarly, we have
  \begin{align*}
    c^{-1}\varepsilon^{1/p}\|\nabla_\Gamma\eta\|_{L^p(\Gamma)} \leq \|\nabla\bar{\eta}\|_{L^p(\Omega_\varepsilon)} \leq c\varepsilon^{1/p}\|\nabla_\Gamma\eta\|_{L^p(\Gamma)}
  \end{align*}
  by \eqref{E:G_Inf}, \eqref{E:ConDer_Bound}, and \eqref{E:CoV_Equiv}.
  Combining this with \eqref{E:Con_Lp} we obtain \eqref{E:Con_W1p}.
\end{proof}

We also give a change of variables formula for an integral over $\Gamma_\varepsilon^i$, $i=0,1$.

\begin{lemma} \label{L:CoV_Surf}
  For $\varphi\in L^1(\Gamma_\varepsilon^i)$, $i=0,1$ let $\varphi_i^\sharp$ be given by \eqref{E:Pull_Bo}.
  Then
  \begin{align} \label{E:CoV_Surf}
    \int_{\Gamma_\varepsilon^i}\varphi(x)\,d\mathcal{H}^2(x) = \int_\Gamma \varphi_i^\sharp(y)J(y,\varepsilon g_i(y))\sqrt{1+\varepsilon^2|\tau_\varepsilon^i(y)|^2}\,d\mathcal{H}^2(y)
  \end{align}
  with $\tau_\varepsilon^i$ given by \eqref{E:Def_NB_Aux}.
  Moreover, if $\varphi\in L^p(\Gamma_\varepsilon^i)$, $p\in[1,\infty)$, then $\varphi_i^\sharp\in L^p(\Gamma)$ and
  \begin{align} \label{E:Lp_CoV_Surf}
    c^{-1}\|\varphi\|_{L^p(\Gamma_\varepsilon^i)} \leq \|\varphi_i^\sharp\|_{L^p(\Gamma)} \leq c\|\varphi\|_{L^p(\Gamma_\varepsilon^i)},
  \end{align}
  where $c>0$ is a constant independent of $\varepsilon$ and $\varphi$.
\end{lemma}

\begin{proof}
  In Lemma \ref{L:CoV_Para} we show the change of variables formula
  \begin{align*}
    \int_{\Gamma_h}\varphi(x)\,d\mathcal{H}^2(x) = \int_\Gamma\varphi_h^\sharp(y)J(y,h(y))\sqrt{1+|\tau_h(y)|^2}\,d\mathcal{H}^2(y)
  \end{align*}
  for $\varphi\in L^1(\Gamma_h)$, where $\tau_h$ and $\Gamma_h$ are the vector field on $\Gamma$ and the parametrized surface given by \eqref{Pf_NB:Vec} and \eqref{Pf_NB:Para} with $h\in C^1(\Gamma)$ satisfying $|h|<\delta$ on $\Gamma$ and
  \begin{align*}
    \varphi_h^\sharp(y) := \varphi(y+h(y)n(y)) \quad y\in\Gamma.
  \end{align*}
  Setting $h=\varepsilon g_i$, $i=0,1$ in the above formula we obtain \eqref{E:CoV_Surf}.
  Also, \eqref{E:Lp_CoV_Surf} follows from the formula \eqref{E:CoV_Surf} and the inequalities \eqref{E:Tau_Bound} and \eqref{E:Jac_Bound_01}.
\end{proof}

\section{Fundamental inequalities and formulas} \label{S:Fund}
Let us give fundamental inequalities and formulas for functions on the curved thin domain $\Omega_\varepsilon$ and its boundary $\Gamma_\varepsilon$.
For a function $\varphi$ on $\Omega_\varepsilon$ and $x\in\Omega_\varepsilon$ let
\begin{align} \label{E:Def_NorDer}
  \partial_n\varphi(x) := (\bar{n}(x)\cdot\nabla)\varphi(x) = \frac{d}{dr}\bigl(\varphi(y+rn(y))\bigr)\Big|_{r=d(x)} \quad (y=\pi(x)\in\Gamma)
\end{align}
be the derivative of $\varphi$ in the normal direction of $\Gamma$.
Note that
\begin{align} \label{E:NorDer_Con}
  \partial_n\bar{\eta}(x) = (\bar{n}(x)\cdot\nabla)\bar{\eta}(x) = 0, \quad x\in \Omega_\varepsilon
\end{align}
for the constant extension $\bar{\eta}=\eta\circ\pi$ of $\eta\in C^1(\Gamma)$.

First we show Poincar\'{e} and trace type inequalities on $\Omega_\varepsilon$.

\begin{lemma} \label{L:Poincare}
  There exists a constant $c>0$ such that
  \begin{align}
    \|\varphi\|_{L^p(\Omega_\varepsilon)} &\leq c\left(\varepsilon^{1/p}\|\varphi\|_{L^p(\Gamma_\varepsilon^i)}+\varepsilon\|\partial_n\varphi\|_{L^p(\Omega_\varepsilon)}\right), \label{E:Poin_Dom} \\
    \|\varphi\|_{L^p(\Gamma_\varepsilon^i)} &\leq c\left(\varepsilon^{-1/p}\|\varphi\|_{L^p(\Omega_\varepsilon)}+\varepsilon^{1-1/p}\|\partial_n\varphi\|_{L^p(\Omega_\varepsilon)}\right) \label{E:Poin_Bo}
  \end{align}
  for all $\varepsilon\in(0,1]$, $\varphi\in W^{1,p}(\Omega_\varepsilon)$ with $p\in[1,\infty)$, and $i=0,1$.
\end{lemma}

\begin{proof}
  We prove \eqref{E:Poin_Dom} and \eqref{E:Poin_Bo} for $i=0$.
  The proofs for $i=1$ are the same.
  We use the notations \eqref{E:Pull_Dom} and \eqref{E:Pull_Bo}.
  For $y\in\Gamma$ and $r\in(\varepsilon g_0(y),\varepsilon g_1(y))$, since
  \begin{align*}
    \varphi^\sharp(y,\varepsilon g_0(y)) = \varphi_0^\sharp(y), \quad \frac{\partial\varphi^\sharp}{\partial r}(y,r) = (\partial_n\varphi)^\sharp(y,r),
  \end{align*}
  where the second equality follows from \eqref{E:Def_NorDer}, we have
  \begin{align} \label{Pf_P:FTC}
    \varphi^\sharp(y,r) = \varphi_0^\sharp(y)+\int_{\varepsilon g_0(y)}^r(\partial_n\varphi)^\sharp(y,\tilde{r})\,d\tilde{r}.
  \end{align}
  From \eqref{Pf_P:FTC} and H\"{o}lder's inequality it follows that
  \begin{align*}
    |\varphi^\sharp(y,r)| &\leq |\varphi_0^\sharp(y)|+\int_{\varepsilon g_0(y)}^{\varepsilon g_1(y)}|(\partial_n\varphi)^\sharp(y,\tilde{r})|\,d\tilde{r} \\
    &\leq |\varphi_0^\sharp(y)|+c\varepsilon^{1-1/p}\left(\int_{\varepsilon g_0(y)}^{\varepsilon g_1(y)}|(\partial_n\varphi)^\sharp(y,\tilde{r})|^p\,d\tilde{r}\right)^{1/p}.
  \end{align*}
  Noting that the right-hand side is independent of $r$, we integrate the $p$-th power of both sides of this inequality with respect to $r$ to obtain
  \begin{align} \label{Pf_P:Int_r}
    \int_{\varepsilon g_0(y)}^{\varepsilon g_1(y)}|\varphi^\sharp(y,r)|^p\,dr \leq c\left(\varepsilon|\varphi_0^\sharp(y)|^p+\varepsilon^p\int_{\varepsilon g_0(y)}^{\varepsilon g_1(y)}|(\partial_n\varphi)^\sharp(y,\tilde{r})|^p\,d\tilde{r}\right)
  \end{align}
  for all $y\in\Gamma$.
  We integrate both sides with respect to $y$ and then use \eqref{E:CoV_Equiv} to get
  \begin{align*}
    \|\varphi\|_{L^p(\Omega_\varepsilon)}^p \leq c\left(\varepsilon\|\varphi_0^\sharp\|_{L^p(\Gamma)}^p+\varepsilon^p\|\partial_n\varphi\|_{L^p(\Omega_\varepsilon)}^p\right).
  \end{align*}
  Applying \eqref{E:Lp_CoV_Surf} to the first term on the right-hand side we obtain \eqref{E:Poin_Dom}.

  Next let us prove \eqref{E:Poin_Bo}.
  As in the proof of \eqref{Pf_P:Int_r}, we use \eqref{Pf_P:FTC} to get
  \begin{align*}
    g(y)|\varphi_0^\sharp(y)|^p \leq c\left(\varepsilon^{-1}\int_{\varepsilon g_0(y)}^{\varepsilon g_1(y)}|\varphi^\sharp(y,r)|^p\,dr+\varepsilon^{p-1}\int_{\varepsilon g_0(y)}^{\varepsilon g_1(y)}|(\partial_n\varphi)^\sharp(y,\tilde{r})|^p\,d\tilde{r}\right)
  \end{align*}
  for all $y\in\Gamma$.
  Here the function $g(y)$ on the left-hand side comes from the integration with respect to $r$.
  Integrating both sides of the above inequality with respect to $y$ and using \eqref{E:G_Inf} and \eqref{E:CoV_Equiv} we obtain
  \begin{align*}
    \|\varphi_0^\sharp\|_{L^p(\Gamma)}^p \leq c\left(\varepsilon^{-1}\|\varphi\|_{L^p(\Omega_\varepsilon)}^p+\varepsilon^{p-1}\|\partial_n\varphi\|_{L^p(\Omega_\varepsilon)}^p\right).
  \end{align*}
  We apply \eqref{E:Lp_CoV_Surf} to the left-hand side of this inequality to get \eqref{E:Poin_Bo}.
\end{proof}

Next we present two results related to the impermeable boundary condition
\begin{align} \label{E:Bo_Imp}
  u\cdot n_\varepsilon = 0 \quad\text{on}\quad \Gamma_\varepsilon
\end{align}
for a vector field $u\colon\Omega_\varepsilon\to\mathbb{R}^3$.

\begin{lemma} \label{L:Exp_Bo}
  For $i=0,1$ let $u\in C(\Gamma_\varepsilon^i)^3$ satisfy \eqref{E:Bo_Imp} on $\Gamma_\varepsilon^i$.
  Then
  \begin{align} \label{E:Exp_Bo}
    u\cdot\bar{n} = \varepsilon u\cdot\bar{\tau}_\varepsilon^i, \quad |u\cdot\bar{n}| \leq c\varepsilon|u| \quad\text{on}\quad \Gamma_\varepsilon^i,
  \end{align}
  where $\tau_\varepsilon^i$ is given by \eqref{E:Def_NB_Aux} and $c>0$ is a constant independent of $\varepsilon$ and $u$.
\end{lemma}

\begin{proof}
  The first equality of \eqref{E:Exp_Bo} follows from \eqref{E:Def_NB}, \eqref{E:Nor_Bo}, and \eqref{E:Bo_Imp} on $\Gamma_\varepsilon^i$.
  Also, we get the second inequality of \eqref{E:Exp_Bo} by the first one and \eqref{E:Tau_Bound}.
\end{proof}

\begin{lemma} \label{L:Int_UgUn}
  There exists a constant $c>0$ such that
  \begin{align} \label{E:Int_UgUn}
    \left|\int_{\Gamma_\varepsilon}(u\cdot\nabla)u\cdot n_\varepsilon\,d\mathcal{H}^2\right| \leq c\left(\|u\|_{L^2(\Omega_\varepsilon)}^2+\|u\|_{L^2(\Omega_\varepsilon)}\|\nabla u\|_{L^2(\Omega_\varepsilon)}\right)
  \end{align}
  for all $\varepsilon\in(0,1]$ and $u\in C^1(\overline{\Omega}_\varepsilon)^3\cup H^2(\Omega_\varepsilon)^3$ satisfying \eqref{E:Bo_Imp}.
\end{lemma}

Note that the second order derivatives of $u$ do not appear in the right-hand side of \eqref{E:Int_UgUn}.
The condition $u\in H^2(\Omega_\varepsilon)^3$ is just required to ensure $\nabla u|_{\Gamma_\varepsilon}\in L^2(\Gamma_\varepsilon)^{3\times 3}$.

\begin{proof}
  Noting that $u$ is tangential on $\Gamma_\varepsilon$ by \eqref{E:Bo_Imp}, we use \eqref{E:Tgrad_Bo} to get
  \begin{align*}
    (u\cdot\nabla)u\cdot n_\varepsilon = (u\cdot\nabla_{\Gamma_\varepsilon})u\cdot n_\varepsilon = u\cdot\nabla_{\Gamma_\varepsilon}(u\cdot n_\varepsilon)-u\cdot(u\cdot\nabla_{\Gamma_\varepsilon})n_\varepsilon \quad\text{on}\quad \Gamma_\varepsilon.
  \end{align*}
  The first term on the right-hand side vanishes by \eqref{E:Bo_Imp} (note that the tangential gradient on $\Gamma_\varepsilon$ depends only on the values of a function on $\Gamma_\varepsilon$).
  Also,
  \begin{align*}
    (u\cdot\nabla_{\Gamma_\varepsilon})n_\varepsilon = (\nabla_{\Gamma_\varepsilon}n_\varepsilon)^Tu = -W_\varepsilon^Tu = -W_\varepsilon u \quad\text{on}\quad \Gamma_\varepsilon
  \end{align*}
  by $-\nabla_{\Gamma_\varepsilon}n_\varepsilon=W_\varepsilon=W_\varepsilon^T$.
  Hence $(u\cdot\nabla)u\cdot n_\varepsilon=u\cdot W_\varepsilon u$ on $\Gamma_\varepsilon$ and
  \begin{align} \label{Pf_IU:Bd_Int}
    \int_{\Gamma_\varepsilon}(u\cdot\nabla)u\cdot n_\varepsilon\,d\mathcal{H}^2 = \int_{\Gamma_\varepsilon}u\cdot W_\varepsilon u\,d\mathcal{H}^2 = \sum_{i=0,1}\int_{\Gamma_\varepsilon^i}u\cdot W_\varepsilon u\,d\mathcal{H}^2.
  \end{align}
  To estimate the right-hand side we interpolate the integrals over the inner and outer boundaries $\Gamma_\varepsilon^0$ and $\Gamma_\varepsilon^1$ to produce an integral over $\Omega_\varepsilon$.
  Let
  \begin{align} \label{Pf_IU:Def_Aux}
    \begin{aligned}
      F_i(y) &:= \sqrt{1+\varepsilon^2|\tau_\varepsilon^i(y)|^2}\,W_{\varepsilon,i}^\sharp(y), \quad i=0,1, \\
      F(y,r) &:= \frac{1}{\varepsilon g(y)}\bigl\{\bigl(r-\varepsilon g_0(y)\bigr)F_1(y)-\bigl(\varepsilon g_1(y)-r\bigr)F_0(y)\bigr\}, \\
      \varphi(y,r) &:= u^\sharp(y,r)\cdot F(y,r)u^\sharp(y,r)J(y,r)
    \end{aligned}
  \end{align}
  for $y\in\Gamma$ and $r\in[\varepsilon g_0(y),\varepsilon g_1(y)]$ with $\tau_\varepsilon^i$, $i=0,1$ and $J$ given by \eqref{E:Def_NB_Aux} and \eqref{E:Def_Jac}.
  Here and in what follows we use the notations \eqref{E:Pull_Dom} and \eqref{E:Pull_Bo} and sometimes suppress the arguments $y$ and $r$.
  By \eqref{Pf_IU:Def_Aux} we see that
  \begin{align*}
    [u\cdot W_\varepsilon u]_i^\sharp(y)\sqrt{1+\varepsilon^2|\tau_\varepsilon^i(y)|^2}J(y,\varepsilon g_i(y)) = (-1)^{i+1}\varphi(y,\varepsilon g_i(y)), \quad y\in\Gamma,\,i=0,1.
  \end{align*}
  From this relation and \eqref{E:CoV_Surf} we deduce that
  \begin{align} \label{Pf_IU:Bd_FTC}
    \begin{aligned}
      \sum_{i=0,1}\int_{\Gamma_\varepsilon^i}[u\cdot W_\varepsilon u](x)\,d\mathcal{H}^2(x) &= \int_\Gamma\{\varphi(y,\varepsilon g_1(y))-\varphi(y,\varepsilon g_0(y))\}\,d\mathcal{H}^2(y) \\
      &= \int_\Gamma\int_{\varepsilon g_0(y)}^{\varepsilon g_1(y)}\frac{\partial\varphi}{\partial r}(y,r)\,dr\,d\mathcal{H}^2(y).
    \end{aligned}
  \end{align}
  To estimate the integrand on the last line we use \eqref{E:Jac_Bound_01} to get
  \begin{align} \label{PF_IU:Dphi_1}
    \left|\frac{\partial\varphi}{\partial r}\right| \leq c\left\{\left(|F|+\left|\frac{\partial F}{\partial r}\right|\right)|u^\sharp|^2+|F||u^\sharp||(\nabla u)^\sharp|\right\}.
  \end{align}
  By \eqref{E:Tau_Bound} and the uniform boundedness in $\varepsilon$ of $W_\varepsilon$ on $\Gamma_\varepsilon$ we observe that $F_0$ and $F_1$ are bounded on $\Gamma$ uniformly in $\varepsilon$.
  Thus we have
  \begin{align} \label{Pf_IU:Est_F}
    |F(y,r)| \leq \frac{c}{\varepsilon g(y)}\bigl\{\bigl(r-\varepsilon g_0(y)\bigr)+\bigl(\varepsilon g_1(y)-r\bigr)\} = c
  \end{align}
  for $y\in\Gamma$ and $r\in[\varepsilon g_0(y),\varepsilon g_1(y)]$.
  Also, by $\partial F/\partial r=(\varepsilon g)^{-1}(F_1+F_0)$ and \eqref{Pf_IU:Def_Aux},
  \begin{align} \label{Pf_IU:DF_1}
    \left|\frac{\partial F}{\partial r}\right| \leq c\varepsilon^{-1}\left(|W_{\varepsilon,1}^\sharp+W_{\varepsilon,0}^\sharp|+\sum_{i=0,1}\left(\sqrt{1+\varepsilon^2|\tau_\varepsilon^i|^2}-1\right)|W_{\varepsilon,i}^\sharp|\right).
  \end{align}
  By the mean value theorem for the function $\sqrt{1+s}$, $s\geq0$ and \eqref{E:Tau_Bound} we have
  \begin{align} \label{Pf_IU:Sqrt}
    0 \leq \sqrt{1+\varepsilon^2|\tau_\varepsilon^i(y)|^2}-1 \leq \frac{\varepsilon^2}{2}|\tau_\varepsilon^i(y)|^2 \leq c\varepsilon^2, \quad y\in\Gamma.
  \end{align}
  We apply this inequality, \eqref{E:Diff_WH_IO} with $G_\varepsilon=W_\varepsilon$, and the uniform boundedness in $\varepsilon$ of $W_\varepsilon$ on $\Gamma_\varepsilon$ to the right-hand side of \eqref{Pf_IU:DF_1} to obtain
  \begin{align} \label{Pf_IU:Est_DF}
    \left|\frac{\partial F}{\partial r}(y,r)\right| \leq c \quad\text{for all}\quad y\in\Gamma,\,r\in[\varepsilon g_0(y),\varepsilon g_1(y)].
  \end{align}
  From \eqref{PF_IU:Dphi_1}, \eqref{Pf_IU:Est_F}, and \eqref{Pf_IU:Est_DF} we deduce that
  \begin{align} \label{Pf_IU:Est_Dphi}
    \left|\frac{\partial\varphi}{\partial r}(y,r)\right| \leq c\Bigl(|u^\sharp(y,r)|^2+\Bigl[|u^\sharp||(\nabla u)^\sharp|\Bigr](y,r)\Bigr)
  \end{align}
  for all $y\in\Gamma$ and $r\in[\varepsilon g_0(y),\varepsilon g_1(y)]$.
  By \eqref{Pf_IU:Bd_Int}, \eqref{Pf_IU:Bd_FTC}, and \eqref{Pf_IU:Est_Dphi} we have
  \begin{align*}
    \begin{aligned}
      \left|\int_{\Gamma_\varepsilon}(u\cdot\nabla)u\cdot n_\varepsilon\,d\mathcal{H}^2\right| &\leq c\int_\Gamma\int_{\varepsilon g_0}^{\varepsilon g_1}\Bigl(|u^\sharp|^2+|u^\sharp||(\nabla u)^\sharp|\Bigr)\,dr\,d\mathcal{H}^2
    \end{aligned}
  \end{align*}
  and applying \eqref{E:CoV_Equiv} and Hold\"{o}r's inequality to the right-hand side we get \eqref{E:Int_UgUn}.
\end{proof}

Finally, we derive two formulas on $\Gamma_\varepsilon$ from the slip boundary conditions
\begin{align} \label{E:Bo_Slip}
  u\cdot n_\varepsilon = 0, \quad 2\nu P_\varepsilon D(u)n_\varepsilon+\gamma_\varepsilon u = 0 \quad\text{on}\quad \Gamma_\varepsilon
\end{align}
which are crucial for the proofs of a uniform a priori estimate for the vector Laplace operator on $\Omega_\varepsilon$ (see Lemma \ref{L:Lap_Apri}) and the uniform difference estimate for the Stokes and Laplace operators \eqref{E:Comp_Sto_Lap}.

\begin{lemma} \label{L:NSl}
  For $i=0,1$ let $u\in C^2(\overline{\Omega}_\varepsilon)^3$ satisfy \eqref{E:Bo_Slip} on $\Gamma_\varepsilon^i$.
  Then
  \begin{gather}
    P_\varepsilon(n_\varepsilon\cdot\nabla)u = -W_\varepsilon u-\frac{\gamma_\varepsilon}{\nu}u \quad\text{on}\quad \Gamma_\varepsilon^i, \label{E:NSl_ND} \\
    n_\varepsilon\times\mathrm{curl}\,u = -n_\varepsilon\times\left\{n_\varepsilon\times\left(2W_\varepsilon u+\frac{\gamma_\varepsilon}{\nu}u\right)\right\} \quad\text{on}\quad \Gamma_\varepsilon^i. \label{E:NSl_Curl}
  \end{gather}
\end{lemma}

\begin{proof}
  Taking the tangential gradient of $u\cdot n_\varepsilon=0$ on $\Gamma_\varepsilon^i$ we have
  \begin{align} \label{Pf_PSl:Gu_N}
    (\nabla_{\Gamma_\varepsilon}u)n_\varepsilon = -(\nabla_{\Gamma_\varepsilon}n_\varepsilon)u = W_\varepsilon u \quad\text{on}\quad \Gamma_\varepsilon^i.
  \end{align}
  From this equality and \eqref{E:Tgrad_Bo} it follows that
  \begin{align*}
    2P_\varepsilon D(u)n_\varepsilon &= P_\varepsilon\{(\nabla u)n_\varepsilon+(\nabla u)^Tn_\varepsilon\} = (\nabla_{\Gamma_\varepsilon}u)n_\varepsilon+P_\varepsilon(n_\varepsilon\cdot\nabla)u \\
    &= W_\varepsilon u+P_\varepsilon(n_\varepsilon\cdot\nabla)u
  \end{align*}
  on $\Gamma_\varepsilon^i$.
  By this equality and the second equality of \eqref{E:Bo_Slip} we observe that
  \begin{align*}
    P_\varepsilon(n_\varepsilon\cdot\nabla)u = 2P_\varepsilon D(u)n_\varepsilon-W_\varepsilon u = -\frac{\gamma_\varepsilon}{\nu}u-W_\varepsilon u \quad\text{on}\quad \Gamma_\varepsilon^i.
  \end{align*}
  Thus \eqref{E:NSl_ND} is valid.
  To prove \eqref{E:NSl_Curl} we observe that the vector field $n_\varepsilon\times\mathrm{curl}\,u$ is tangential on $\Gamma_\varepsilon^i$.
  By this fact, \eqref{E:Tgrad_Bo}, \eqref{E:NSl_ND}, and \eqref{Pf_PSl:Gu_N} we have
  \begin{align*}
    n_\varepsilon\times\mathrm{curl}\,u &= P_\varepsilon(n_\varepsilon\times\mathrm{curl}\,u) = P_\varepsilon\{(\nabla u)n_\varepsilon-(\nabla u)^Tn_\varepsilon\} \\
    &= (\nabla_{\Gamma_\varepsilon}u)n_\varepsilon-P_\varepsilon(n_\varepsilon\cdot\nabla)u = 2W_\varepsilon u+\frac{\gamma_\varepsilon}{\nu}u
  \end{align*}
  on $\Gamma_\varepsilon^i$.
  Noting that $n_\varepsilon\cdot u=n_\varepsilon\cdot W_\varepsilon u=0$ and $|n_\varepsilon|^2=1$ on $\Gamma_\varepsilon^i$, we conclude by
  \begin{align*}
    a\times(a\times b) = (a\cdot b)a-|a|^2b, \quad a,b\in\mathbb{R}^3
  \end{align*}
  with $a=n_\varepsilon$ and $b=2W_\varepsilon u+\nu^{-1}\gamma_\varepsilon u$ and the above equality that \eqref{E:NSl_Curl} is valid.
\end{proof}

\section{Korn's inequality on a curved thin domain} \label{S:UniKorn}
In this section we establish a uniform Korn inequality on $\Omega_\varepsilon$ that is essential for the uniform coerciveness of the bilinear form $a_\varepsilon$ given by \eqref{E:Def_aeps}.
We also compare it with a standard Korn inequality for simple examples of $\Omega_\varepsilon$.

\subsection{Uniform Korn inequality on a curved thin domain} \label{SS:Pr_UniKorn}
Let
\begin{align*}
  D(u) = (\nabla u)_S = \frac{\nabla u+(\nabla u)^T}{2}
\end{align*}
be the strain rate tensor of a vector field $u$ on $\Omega_\varepsilon$.
The goal of this subsection is to show the uniform Korn inequality
\begin{align*}
  \|u\|_{H^1(\Omega_\varepsilon)} \leq c\|D(u)\|_{L^2(\Omega_\varepsilon)}
\end{align*}
with a constant $c>0$ independent of $\varepsilon$ under suitable assumptions on $u$.
First we give an $L^2$-estimate for the gradient matrix of a vector field on $\Omega_\varepsilon$.

\begin{lemma} \label{L:Korn_Grad}
  There exists a constant $c_{K,1}>0$ independent of $\varepsilon$ such that
  \begin{align} \label{E:Korn_Grad}
    \|\nabla u\|_{L^2(\Omega_\varepsilon)}^2 \leq 4\|D(u)\|_{L^2(\Omega_\varepsilon)}^2+c_{K,1}\|u\|_{L^2(\Omega_\varepsilon)}^2
  \end{align}
  for all $\varepsilon\in(0,1]$ and $u\in H^1(\Omega_\varepsilon)^3$ satisfying \eqref{E:Bo_Imp}.
\end{lemma}

Let us prove an auxiliary density result.

\begin{lemma} \label{L:H1N0_Dense}
  Let $u\in H^1(\Omega_\varepsilon)^3$ satisfy \eqref{E:Bo_Imp}.
  Then there exists a sequence $\{u_k\}_{k=1}^\infty$ in $C^2(\overline{\Omega}_\varepsilon)^3$ such that $u_k$ satisfies \eqref{E:Bo_Imp} for each $k\in\mathbb{N}$ and
  \begin{align*}
    \lim_{k\to\infty}\|u-u_k\|_{H^1(\Omega_\varepsilon)} = 0.
  \end{align*}
\end{lemma}

\begin{proof}
  We follow the idea of the proof of \cite{BoFa13}*{Theorem IV.4.7}, but here it is not necessary to localize a vector field on $\Omega_\varepsilon$.
  For $x\in N$ we define
  \begin{align*}
    \tilde{n}(x) := \frac{1}{\varepsilon \bar{g}(x)}\bigl\{\bigl(d(x)-\varepsilon\bar{g}_0(x)\bigr)\bar{n}_\varepsilon^1(x)+\bigl(\varepsilon\bar{g}_1(x)-d(x)\bigr)\bar{n}_\varepsilon^0(x)\bigr\},
  \end{align*}
  where $n_\varepsilon^0$ and $n_\varepsilon^1$ are given by \eqref{E:Def_NB} and $\bar{\eta}=\eta\circ\pi$ denotes the constant extension of a function $\eta$ on $\Gamma$.
  Then $\tilde{n}\in C^2(N)^3$ by the regularity of $\Gamma$, $g_0$, and $g_1$.
  Moreover, $\tilde{n}=n_\varepsilon$ on $\Gamma_\varepsilon$ by Lemma \ref{L:Nor_Bo}.
  Hence if $u\in H^1(\Omega_\varepsilon)^3$ satisfies \eqref{E:Bo_Imp}, then we have $u\cdot\tilde{n}\in H_0^1(\Omega_\varepsilon)$ and $w:=u-(u\cdot\tilde{n})\tilde{n}\in H^1(\Omega_\varepsilon)^3$.
  Since $\Gamma_\varepsilon$ is of class $C^4$, there exist sequences $\{\varphi_k\}_{k=1}^\infty$ in $C_c^\infty(\Omega_\varepsilon)$ and $\{w_k\}_{k=1}^\infty$ in $C^\infty(\overline{\Omega}_\varepsilon)^3$ such that
  \begin{align*}
    \lim_{k\to\infty}\|u\cdot\tilde{n}-\varphi_k\|_{H^1(\Omega_\varepsilon)} = \lim_{k\to\infty}\|w-w_k\|_{H^1(\Omega_\varepsilon)} = 0.
  \end{align*}
  Here $C_c^\infty(\Omega_\varepsilon)$ is the space of all smooth and compactly supported functions on $\Omega_\varepsilon$.
  Therefore, setting $u_k:=\varphi_k\tilde{n}+w_k-(w_k\cdot\tilde{n})\tilde{n}\in C^2(\overline{\Omega}_\varepsilon)^3$ we see that
  \begin{align*}
    u_k\cdot n_\varepsilon = u_k\cdot\tilde{n} = \varphi_k = 0 \quad\text{on}\quad \Gamma_\varepsilon
  \end{align*}
  for each $k\in\mathbb{N}$ and (note that $u=(u\cdot\tilde{n})\tilde{n}+w$ and $w\cdot\tilde{n}=0$ in $\Omega_\varepsilon$)
  \begin{align*}
    \|u-u_k\|_{H^1(\Omega_\varepsilon)} &= \|(u\cdot\tilde{n}-\varphi_k)\tilde{n}+(w-w_k)-\{(w-w_k)\cdot\tilde{n}\}\tilde{n}\|_{H^1(\Omega_\varepsilon)} \\
    &\leq c_\varepsilon\left(\|u\cdot\tilde{n}-\varphi_k\|_{H^1(\Omega_\varepsilon)}+\|w-w_k\|_{H^1(\Omega_\varepsilon)}\right) \to 0
  \end{align*}
  as $k\to\infty$ (here $c_\varepsilon>0$ may depend on $\varepsilon$ but is independent of $k$).
\end{proof}

\begin{proof}[Proof of Lemma \ref{L:Korn_Grad}]
  By Lemma \ref{L:H1N0_Dense} and a density argument it is sufficient to show \eqref{E:Korn_Grad} for all $u\in C^2(\overline{\Omega}_\varepsilon)^3$ satisfying \eqref{E:Bo_Imp}.
  Then by the $C^2$-regularity of $u$ on $\overline{\Omega}_\varepsilon$ we can perform integration by parts twice to get
  \begin{align*}
    \int_{\Omega_\varepsilon}\nabla u:(\nabla u)^T\,dx = \int_{\Omega_\varepsilon}(\mathrm{div}\,u)^2\,dx+\int_{\Gamma_\varepsilon}\{(u\cdot\nabla)u\cdot n_\varepsilon-(u\cdot n_\varepsilon)\mathrm{div}\,u\}\,d\mathcal{H}^2.
  \end{align*}
  Since $(\mathrm{div}\,u)^2\geq0$ in $\Omega_\varepsilon$ and $u\cdot n_\varepsilon=0$ on $\Gamma_\varepsilon$, the above equality implies
  \begin{align*}
    \int_{\Omega_\varepsilon}\nabla u:(\nabla u)^T\,dx \geq \int_{\Gamma_\varepsilon}(u\cdot\nabla)u\cdot n_\varepsilon\,d\mathcal{H}^2.
  \end{align*}
  From this inequality and $|\nabla u|^2=2|D(u)|^2-\nabla u:(\nabla u)^T$ in $\Omega_\varepsilon$ we deduce that
  \begin{align*}
    \|\nabla u\|_{L^2(\Omega_\varepsilon)}^2 \leq 2\|D(u)\|_{L^2(\Omega_\varepsilon)}^2-\int_{\Gamma_\varepsilon}(u\cdot\nabla)u\cdot n_\varepsilon\,d\mathcal{H}^2.
  \end{align*}
  Noting that $u\in C^2(\overline{\Omega}_\varepsilon)^3$ satisfies \eqref{E:Bo_Imp}, we apply \eqref{E:Int_UgUn} to the last term to obtain
  \begin{align*}
    \|\nabla u\|_{L^2(\Omega_\varepsilon)}^2 &\leq 2\|D(u)\|_{L^2(\Omega_\varepsilon)}^2+c\left(\|u\|_{L^2(\Omega_\varepsilon)}^2+\|u\|_{L^2(\Omega_\varepsilon)}\|\nabla u\|_{L^2(\Omega_\varepsilon)}\right) \\
    &\leq 2\|D(u)\|_{L^2(\Omega_\varepsilon)}^2+c\|u\|_{L^2(\Omega_\varepsilon)}^2+\frac{1}{2}\|\nabla u\|_{L^2(\Omega_\varepsilon)}^2.
  \end{align*}
  Hence \eqref{E:Korn_Grad} follows.
\end{proof}

Next we show a uniform $L^2$-estimate for a vector field on $\Omega_\varepsilon$ by the $L^2$-norms of the gradient matrix and the strain rate tensor on $\Omega_\varepsilon$.
Recall that for a function $\eta$ on $\Gamma$ we denote by $\bar{\eta}=\eta\circ\pi$ its constant extension in the normal direction of $\Gamma$.

\begin{lemma} \label{L:Korn_U}
  For given $\alpha>0$ and $\beta\in[0,1)$ there exist constants
  \begin{align*}
    \varepsilon_K = \varepsilon_K(\alpha,\beta)\in(0,1], \quad c_{K,2} = c_{K,2}(\alpha,\beta) > 0
  \end{align*}
  independent of $\varepsilon$ such that
  \begin{align} \label{E:Korn_U}
    \|u\|_{L^2(\Omega_\varepsilon)}^2 \leq \alpha\|\nabla u\|_{L^2(\Omega_\varepsilon)}^2+c_{K,2}\|D(u)\|_{L^2(\Omega_\varepsilon)}^2
  \end{align}
  for all $\varepsilon\in(0,\varepsilon_K]$ and $u\in H^1(\Omega_\varepsilon)^3$ satisfying \eqref{E:Bo_Imp} and
  \begin{align} \label{E:Korn_Orth}
    \left|(u,\bar{v})_{L^2(\Omega_\varepsilon)}\right| \leq \beta\|u\|_{L^2(\Omega_\varepsilon)}\|\bar{v}\|_{L^2(\Omega_\varepsilon)} \quad\text{for all}\quad v\in\mathcal{K}_g(\Gamma).
  \end{align}
  Here $\mathcal{K}_g(\Gamma)$ is the function space on $\Gamma$ given by \eqref{E:Def_Kil}.
\end{lemma}

To prove Lemma \ref{L:Korn_U} we transform integrals over $\Omega_\varepsilon$ into those over the domain $\Omega_1$ with fixed thickness by using the following lemmas (note that we assume $\overline{\Omega}_1\subset N$ by scaling $g_0$ and $g_1$).

\begin{lemma} \label{L:CoV_Fixed}
  For $\varepsilon\in(0,1]$ let
  \begin{align} \label{E:Def_Bij}
    \Phi_\varepsilon(X) := \pi(X)+\varepsilon d(X)\bar{n}(X), \quad X \in \Omega_1.
  \end{align}
  Then $\Phi_\varepsilon$ is a bijection from $\Omega_1$ onto $\Omega_\varepsilon$ and for a function $\varphi$ on $\Omega_\varepsilon$ we have
  \begin{align} \label{E:CoV_Fixed}
    \int_{\Omega_\varepsilon}\varphi(x)\,dx = \varepsilon\int_{\Omega_1}\xi(X)J(\pi(X),d(X))^{-1}J(\pi(X),\varepsilon d(X))\,dX,
  \end{align}
  where $\xi:=\varphi\circ\Phi_\varepsilon$ on $\Omega_1$ and $J$ is given by \eqref{E:Def_Jac}.
  Moreover, if $\varphi\in L^2(\Omega_\varepsilon)$, then $\xi\in L^2(\Omega_1)$ and there exist constants $c_1,c_2>0$ independent of $\varepsilon$ and $\varphi$ such that
  \begin{align} \label{E:L2_Fixed}
    c_1\varepsilon^{-1}\|\varphi\|_{L^2(\Omega_\varepsilon)}^2 \leq \|\xi\|_{L^2(\Omega_1)}^2 \leq c_2\varepsilon^{-1}\|\varphi\|_{L^2(\Omega_\varepsilon)}^2.
  \end{align}
  If in addition $\varphi\in H^1(\Omega_\varepsilon)$, then $\xi\in H^1(\Omega_1)$ and
  \begin{align} \label{E:H1_Fixed}
    \varepsilon^{-1}\|\nabla\varphi\|_{L^2(\Omega_\varepsilon)}^2 \geq c\left(\left\|\overline{P}\nabla\xi\right\|_{L^2(\Omega_1)}^2+\varepsilon^{-2}\|\partial_n\xi\|_{L^2(\Omega_1)}^2\right),
  \end{align}
  where $\partial_n\xi=(\bar{n}\cdot\nabla)\xi$ on $\Omega_1$ and $c>0$ is a constant independent of $\varepsilon$ and $\varphi$.
\end{lemma}

\begin{lemma} \label{L:KAux_Du}
  For $\varepsilon\in(0,1]$ let $\Phi_\varepsilon\colon\Omega_1\to\Omega_\varepsilon$ be the bijection given by \eqref{E:Def_Bij}.
  Also, let $u\in H^1(\Omega_\varepsilon)^3$.
  Then $U:=u\circ\Phi_\varepsilon\in H^1(\Omega_1)^3$ and the inequality \eqref{E:H1_Fixed} holds with $\varphi$ and $\xi$ replaced by $u$ and $U$, respectively.
  Moreover,
  \begin{align} \label{E:KAux_Du}
    \varepsilon^{-1}\|D(u)\|_{L^2(\Omega_\varepsilon)}^2 \geq c\left(\left\|\overline{P}F_\varepsilon(U)_S\overline{P}\right\|_{L^2(\Omega_1)}^2+\varepsilon^{-2}\|\partial_n(U\cdot\bar{n})\|_{L^2(\Omega_1)}^2\right),
  \end{align}
  where $F_\varepsilon(U)_S=\{F_\varepsilon(U)+F_\varepsilon(U)^T\}/2$ is the symmetric part of
  \begin{align} \label{E:KAux_Matrix}
    F_\varepsilon(U) := \Bigl(I_3-\varepsilon d\overline{W}\Bigr)^{-1}\Bigl(I_3-d\overline{W}\Bigr)\nabla U \quad\text{on}\quad \Omega_1
  \end{align}
  and $c>0$ is a constant independent of $\varepsilon$ and $u$.
\end{lemma}

We give the proofs of Lemmas \ref{L:CoV_Fixed} and \ref{L:KAux_Du} in Appendix \ref{S:Ap_Proof}.

\begin{proof}[Proof of Lemma \ref{L:Korn_U}]
  Following the idea of the proof of \cite{HoSe10}*{Lemma 4.14} we prove \eqref{E:Korn_U} by contradiction.
  Assume to the contrary that there exist a sequence $\{\varepsilon_k\}_{k=1}^\infty$ of positive numbers with $\lim_{k\to\infty}\varepsilon_k=0$ and vector fields $u_k\in H^1(\Omega_{\varepsilon_k})^3$ satisfying \eqref{E:Bo_Imp} on $\Gamma_{\varepsilon_k}$, \eqref{E:Korn_Orth}, and
  \begin{align} \label{PF_KU:Ineq_Contra}
    \|u_k\|_{L^2(\Omega_{\varepsilon_k})}^2 > \alpha\|\nabla u_k\|_{L^2(\Omega_{\varepsilon_k})}^2+k\|D(u_k)\|_{L^2(\Omega_{\varepsilon_k})}^2, \quad k\in\mathbb{N}.
  \end{align}
  For each $k\in\mathbb{N}$ let $\Phi_{\varepsilon_k}$ be the bijection from $\Omega_1$ onto $\Omega_{\varepsilon_k}$ given by \eqref{E:Def_Bij} and
  \begin{align*}
    U_k := u_k\circ\Phi_{\varepsilon_k} \in H^1(\Omega_1)^3.
  \end{align*}
  Also, let $F_{\varepsilon_k}(U_k)$ be the matrix given by \eqref{E:KAux_Matrix}.
  Dividing both sides of \eqref{PF_KU:Ineq_Contra} by $\varepsilon_k$ and using \eqref{E:L2_Fixed}--\eqref{E:KAux_Du} we get
  \begin{multline*}
    \|U_k\|_{L^2(\Omega_1)}^2 > c\alpha\left(\left\|\overline{P}\nabla U_k\right\|_{L^2(\Omega_1)}^2+\varepsilon_k^{-2}\|\partial_nU_k\|_{L^2(\Omega_1)}^2\right) \\
    +ck\left(\left\|\overline{P}F_{\varepsilon_k}(U_k)_S\overline{P}\right\|_{L^2(\Omega_1)}^2+\varepsilon_k^{-2}\|\partial_n(U_k\cdot\bar{n})\|_{L^2(\Omega_1)}^2\right).
  \end{multline*}
  Since $\|U_k\|_{L^2(\Omega_1)}>0$, we may assume
  \begin{align} \label{Pf_KU:L2_Uk}
    \|U_k\|_{L^2(\Omega_1)} = 1, \quad k\in\mathbb{N}
  \end{align}
  by replacing $U_k$ with $U_k/\|U_k\|_{L^2(\Omega_1)}$.
  Then
  \begin{align}
    \left\|\overline{P}\nabla U_k\right\|_{L^2(\Omega_1)}^2+\varepsilon_k^{-2}\|\partial_nU_k\|_{L^2(\Omega_1)}^2 &< c\alpha^{-1}, \label{Pf_KU:L2_Grad_Uk} \\
    \left\|\overline{P}F_{\varepsilon_k}(U_k)_S\overline{P}\right\|_{L^2(\Omega_1)}^2+\varepsilon_k^{-2}\|\partial_n(U_k\cdot\bar{n})\|_{L^2(\Omega_1)}^2 &< ck^{-1} \label{Pf_KU:L2_DUk}
  \end{align}
  and $\{U_k\}_{k=1}^\infty$ is bounded in $H^1(\Omega_1)^3$ by \eqref{Pf_KU:L2_Uk}, \eqref{Pf_KU:L2_Grad_Uk}, and
  \begin{align*}
    |\nabla U_k|^2 = \left|\overline{P}\nabla U_k\right|^2+\left|\overline{Q}\nabla U_k\right|^2, \quad \left|\overline{Q}\nabla U_k\right| = |\bar{n}\otimes\partial_nU_k| = |\partial_nU_k| \quad\text{in}\quad \Omega_1.
  \end{align*}
  Hence there exists $U\in H^1(\Omega_1)^3$ such that (up to a subsequence) $\{U_k\}_{k=1}^\infty$ converges to $U$ weakly in $H^1(\Omega_1)^3$.
  By the compact embedding $H^1(\Omega_1)\hookrightarrow L^2(\Omega_1)$ and \eqref{Pf_KU:L2_Uk} it also converges to $U$ strongly in $L^2(\Omega_1)^3$ and
  \begin{align} \label{Pf_KU:L2_Limit}
    \|U\|_{L^2(\Omega_1)} = \lim_{k\to\infty}\|U_k\|_{L^2(\Omega_1)} = 1.
  \end{align}
  Our goal is to show $U=0$ in $\Omega_1$, which contradicts with \eqref{Pf_KU:L2_Limit}.
  Since
  \begin{align} \label{Pf_KU:Lim_NDer}
    \lim_{k\to\infty}\|\partial_nU_k\|_{L^2(\Omega_1)} = 0
  \end{align}
  by \eqref{Pf_KU:L2_Grad_Uk} and $\{U_k\}_{k=1}^\infty$ converges to $U$ weakly in $H^1(\Omega_1)^3$, it follows that $\partial_nU=0$ in $\Omega_1$, i.e. $U$ is independent of the normal direction of $\Gamma$.
  Hence setting
  \begin{align*}
    v(y) := U(y+g_0(y)n(y)), \quad y\in\Gamma
  \end{align*}
  we can consider $U$ as the constant extension of $v$, i.e. $U=\bar{v}$ in $\Omega_1$.
  Moreover, from \eqref{E:Poin_Bo} with $\varepsilon=1$ and $\partial_n\bar{v}=0$ in $\Omega_1$ we deduce that
  \begin{align*}
    \|U_k-\bar{v}\|_{L^2(\Gamma_1)} \leq c\left(\|U_k-\bar{v}\|_{L^2(\Omega_1)}+\|\partial_nU_k\|_{L^2(\Omega_1)}\right), \quad k\in\mathbb{N}.
  \end{align*}
  Thus, by the strong convergence of $\{U_k\}_{k=1}^\infty$ to $\bar{v}=U$ in $L^2(\Omega_1)^3$ and \eqref{Pf_KU:Lim_NDer},
  \begin{align} \label{Pf_KU:Lim_Bo}
    \lim_{k\to\infty}\|U_k-\bar{v}\|_{L^2(\Gamma_1)} = 0.
  \end{align}
  Let us show $v\in\mathcal{K}_g(\Gamma)$.
  Since $\bar{v}=U\in H^1(\Omega_1)^3$, we have $v\in H^1(\Gamma)^3$ by Lemma \ref{L:Con_Lp_W1p}.
  For each $k\in\mathbb{N}$, since $u_k$ satisfies \eqref{E:Bo_Imp} on $\Gamma_{\varepsilon_k}$ we can use \eqref{E:Exp_Bo} to get
  \begin{align*}
    |u_k\cdot\bar{n}|\leq c\varepsilon_k|u_k| \quad\text{on}\quad \Gamma_{\varepsilon_k}, \quad\text{i.e.}\quad |U_k\cdot\bar{n}|\leq c\varepsilon_k|U_k| \quad\text{on}\quad \Gamma_1.
  \end{align*}
  By this inequality, \eqref{E:Poin_Bo} with $\varepsilon=1$, and the boundedness of $\{U_k\}_{k=1}^\infty$ in $H^1(\Omega_1)^3$,
  \begin{align} \label{Pf_KU:Lim_NC}
    \|U_k\cdot\bar{n}\|_{L^2(\Gamma_1)} \leq c\varepsilon_k\|U_k\|_{L^2(\Gamma_1)} \leq c\varepsilon_k\|U_k\|_{H^1(\Omega_1)} \to 0 \quad\text{as}\quad k\to\infty.
  \end{align}
  Combining this with \eqref{Pf_KU:Lim_Bo} we get $\bar{v}\cdot\bar{n}=0$ on $\Gamma_1$ and thus $v\cdot n=0$ on $\Gamma$.
  Next we show that $D_\Gamma(v)=P(\nabla_\Gamma v)_SP$ vanishes on $\Gamma$.
  Since
  \begin{align*}
    \left|\left\{I_3-\varepsilon_kd(X)\overline{W}(X)\right\}^{-1}-I_3\right| \leq c\varepsilon_k|d(X)| \leq c\varepsilon_k \to 0 \quad\text{as}\quad k\to\infty
  \end{align*}
  uniformly in $X\in\Omega_1$ by \eqref{E:Wein_Diff} and $\{U_k\}_{k=1}^\infty$ converges to $U=\bar{v}$ weakly in $H^1(\Omega_1)^3$, for the matrix $F_{\varepsilon_k}(U_k)$ of the form \eqref{E:KAux_Matrix} we have
  \begin{align*}
    \lim_{k\to\infty}F_{\varepsilon_k}(U_k) = \Bigl(I_3-d\overline{W}\Bigr)\nabla\bar{v} = \overline{\nabla_\Gamma v} \quad\text{weakly in}\quad L^2(\Omega_1)^{3\times3}.
  \end{align*}
  Here the last equality follows from \eqref{E:ConDer_Dom}.
  Thus we get
  \begin{align*}
    \lim_{k\to\infty}\overline{P}F_{\varepsilon_k}(U_k)_S\overline{P} = \overline{P}\Bigl(\overline{\nabla_\Gamma v}\Bigr)_S\overline{P} = \overline{D_\Gamma(v)} \quad\text{weakly in}\quad L^2(\Omega_1)^{3\times3}.
  \end{align*}
  Moreover, the inequality \eqref{Pf_KU:L2_DUk} yields
  \begin{align*}
    \lim_{k\to\infty}\left\|\overline{P}F_{\varepsilon_k}(U_k)_S\overline{P}\right\|_{L^2(\Omega_1)} = 0.
  \end{align*}
  Therefore,
  \begin{align*}
    \overline{D_\Gamma(v)} = 0 \quad\text{in}\quad \Omega_1, \quad\text{i.e.}\quad D_\Gamma(v) = 0 \quad\text{on}\quad \Gamma.
  \end{align*}
  To prove $v\in\mathcal{K}_g(\Gamma)$ it remains to verify $v\cdot\nabla_\Gamma g=0$ on $\Gamma$.
  In what follows, we use the notations \eqref{E:Pull_Dom} and \eqref{E:Pull_Bo} (with $\varepsilon=1$).
  For each $k\in\mathbb{N}$, since $u_k$ satisfies \eqref{E:Bo_Imp} on $\Gamma_{\varepsilon_k}$, it follows from \eqref{E:Exp_Bo} that
  \begin{align*}
    u_k\cdot\bar{\tau}_{\varepsilon_k}^i = \varepsilon_k^{-1}u_k\cdot\bar{n} \quad\text{on}\quad \Gamma_{\varepsilon_k}^i,\,i=0,1.
  \end{align*}
  This equality yields $U_k\cdot\bar{\tau}_{\varepsilon_k}^i=\varepsilon_k^{-1}U_k\cdot\bar{n}$ on $\Gamma_1^i$, $i=0,1$, or equivalently,
  \begin{align*}
    U_{k,i}^\sharp(y)\cdot\tau_{\varepsilon_k}^i(y) = \varepsilon_k^{-1}U_{k,i}^\sharp(y)\cdot n(y), \quad y\in\Gamma,\,i=0,1.
  \end{align*}
  Hence
  \begin{align} \label{Pf_KU:L2_Diff}
    \|U_{k,1}^\sharp\cdot\tau_{\varepsilon_k}^1-U_{k,0}^\sharp\cdot\tau_{\varepsilon_k}^0\|_{L^2(\Gamma)} \leq \varepsilon_k^{-1}\|U_{k,1}^\sharp\cdot n-U_{k,0}^\sharp\cdot n\|_{L^2(\Gamma)}.
  \end{align}
  Moreover, since $\bar{n}^\sharp(y,r)=n(y)$ for $y\in\Gamma$ and $r\in(g_0(y),g_1(y))$,
  \begin{align*}
    (U_{k,1}^\sharp\cdot n)(y)-(U_{k,0}^\sharp\cdot n)(y) &= \int_{g_0(y)}^{g_1(y)}\frac{\partial}{\partial r}\Bigl((U_k\cdot\bar{n})^\sharp(y,r)\Bigr)\,dr \\
    &= \int_{g_0(y)}^{g_1(y)}[\partial_n(U_k\cdot\bar{n})]^\sharp(y,r)\,dr.
  \end{align*}
  By this equality, H\"{o}lder's inequality, \eqref{E:CoV_Equiv}, and \eqref{Pf_KU:L2_DUk},
  \begin{align*}
    \|U_{k,1}^\sharp\cdot n-U_{k,0}^\sharp\cdot n\|_{L^2(\Gamma)}^2 &= \int_\Gamma\left(\int_{g_0(y)}^{g_1(y)}[\partial_n(U_k\cdot\bar{n})]^\sharp(y,r)\,dr\right)^2d\mathcal{H}^2(y) \\
    &\leq c\|\partial_n(U_k\cdot\bar{n})\|_{L^2(\Omega_1)}^2 \leq c\varepsilon_k^2k^{-1}.
  \end{align*}
  Applying this inequality to the right-hand side of \eqref{Pf_KU:L2_Diff} we get
  \begin{align} \label{Pf_KU:Lim_Diff}
    \|U_{k,1}^\sharp\cdot\tau_{\varepsilon_k}^1-U_{k,0}^\sharp\cdot\tau_{\varepsilon_k}^0\|_{L^2(\Gamma)} \leq ck^{-1/2} \to 0 \quad\text{as}\quad k\to\infty.
  \end{align}
  Also, by \eqref{E:Tau_Bound}, \eqref{E:Tau_Diff}, and \eqref{E:Lp_CoV_Surf},
  \begin{align*}
    \|U_{k,i}^\sharp\cdot\tau_{\varepsilon_k}^i-v\cdot\nabla_\Gamma g_i\|_{L^2(\Gamma)} &\leq \|(U_{k,i}^\sharp-v)\cdot\tau_{\varepsilon_k}^i\|_{L^2(\Gamma)}+\|v\cdot(\tau_{\varepsilon_k}^i-\nabla_\Gamma g_i)\|_{L^2(\Gamma)} \\
    &\leq c\left(\|U_{k,i}^\sharp-v\|_{L^2(\Gamma)}+\varepsilon_k\|v\|_{L^2(\Gamma)}\right) \\
    &\leq c\left(\|U_k-\bar{v}\|_{L^2(\Gamma_1^i)}+\varepsilon_k\|v\|_{L^2(\Gamma)}\right).
  \end{align*}
  Since the right-hand side tends to zero as $k\to\infty$ by \eqref{Pf_KU:Lim_Bo},
  \begin{align*}
    \lim_{k\to\infty}\|U_{k,i}^\sharp\cdot\tau_{\varepsilon_k}^i-v\cdot\nabla_\Gamma g_i\|_{L^2(\Gamma)} = 0, \quad i=0,1.
  \end{align*}
  This equality and \eqref{Pf_KU:Lim_Diff} imply
  \begin{align*}
     \|v\cdot\nabla_\Gamma g\|_{L^2(\Gamma)} = \|v\cdot\nabla_\Gamma g_1-v\cdot\nabla_\Gamma g_0\|_{L^2(\Gamma)} = 0.
   \end{align*}
  Hence $v\cdot\nabla_\Gamma g=0$ on $\Gamma$ and we obtain $v\in\mathcal{K}_g(\Gamma)$.

  Now we recall that $u_k\in H^1(\Omega_{\varepsilon_k})^3$ satisfies \eqref{E:Korn_Orth}.
  Then since $v\in\mathcal{K}_g(\Gamma)$,
  \begin{align} \label{Pf_KU:Fin_Ineq}
    |(u_k,\bar{v})_{L^2(\Omega_{\varepsilon_k})}| \leq \beta\|u_k\|_{L^2(\Omega_{\varepsilon_k})}\|\bar{v}\|_{L^2(\Omega_{\varepsilon_k})}, \quad k\in\mathbb{N}
  \end{align}
  with $\beta\in[0,1)$.
  We express this inequality in terms of $U_k$ and send $k\to\infty$.
  Let
  \begin{align} \label{Pf_Ku:Phik}
    \varphi_k(X):=J(\pi(X),d(X))^{-1}J(\pi(X),\varepsilon_kd(X)), \quad X\in \Omega_1.
  \end{align}
  Then by \eqref{E:CoV_Fixed} and $U_k=u_k\circ\Phi_{\varepsilon_k}$ on $\Omega_1$ we have
  \begin{gather*}
    (u_k,\bar{v})_{L^2(\Omega_{\varepsilon_k})} = \varepsilon_k\int_{\Omega_1}U_k\cdot(\bar{v}\circ\Phi_{\varepsilon_k})\varphi_k\,dX.
  \end{gather*}
  Here $\bar{v}\circ\Phi_{\varepsilon_k}=\bar{v}$ in $\Omega_1$ since $\pi\circ\Phi_{\varepsilon_k}=\pi$ in $\Omega_1$ by \eqref{E:Def_Bij}.
  Moreover, since
  \begin{align*}
    \varphi_k(X)-J(\pi(X),d(X))^{-1} = J(\pi(X),d(X))^{-1}\{J(\pi(X),\varepsilon_kd(X))-1\}
  \end{align*}
  for $X\in\Omega_1$, we observe by \eqref{E:Jac_Bound_01}, \eqref{E:Jac_Diff_01}, and $|d|\leq c$ in $\Omega_1$ that
  \begin{align} \label{Pf_Ku:Conv_Phik}
    |\varphi_k(X)-J(\pi(X),d(X))^{-1}| \leq c\varepsilon_k|d(X)| \leq c\varepsilon_k \to 0 \quad\text{as}\quad k\to\infty
  \end{align}
  uniformly in $X\in\Omega_1$.
  From these facts and the strong convergence of $\{U_k\}_{k=1}^\infty$ to $U=\bar{v}$ in $L^2(\Omega_1)^3$ we deduce that
  \begin{align*}
    \lim_{k\to\infty}\varepsilon_k^{-1}(u_k,\bar{v})_{L^2(\Omega_{\varepsilon_k})} = \lim_{k\to\infty}\int_{\Omega_1}(U_k\cdot\bar{v})\varphi_k\,dX = \int_{\Omega_1}|\bar{v}|^2J(\pi(\cdot),d(\cdot))^{-1}\,dX.
  \end{align*}
  By \eqref{E:CoV_Dom} with $\varepsilon=1$ the last term is of the form
  \begin{align*}
    \int_\Gamma\int_{g_0(y)}^{g_1(y)}|v(y)|^2J(y,r)^{-1}J(y,r)\,dr\,d\mathcal{H}^2(y) = \int_\Gamma g(y)|v(y)|^2\,d\mathcal{H}^2(y).
  \end{align*}
  Therefore,
  \begin{align} \label{PF_KU:Fin_Lim_1}
    \lim_{k\to\infty}\varepsilon_k^{-1}(u_k,\bar{v})_{L^2(\Omega_{\varepsilon_k})} = \|g^{1/2}v\|_{L^2(\Gamma)}^2.
  \end{align}
  By the same arguments we have
  \begin{align} \label{Pf_KU:Fin_Lim_2}
    \lim_{k\to\infty}\varepsilon_k^{-1}\|u_k\|_{L^2(\Omega_{\varepsilon_k})}^2 = \lim_{k\to\infty}\varepsilon_k^{-1}\|\bar{v}\|_{L^2(\Omega_{\varepsilon_k})}^2 = \|g^{1/2}v\|_{L^2(\Gamma)}^2.
  \end{align}
  We divide both sides of \eqref{Pf_KU:Fin_Ineq} by $\varepsilon_k$, send $k\to\infty$, and use \eqref{PF_KU:Fin_Lim_1}--\eqref{Pf_KU:Fin_Lim_2} to get
  \begin{align*}
    \|g^{1/2}v\|_{L^2(\Gamma)}^2 \leq \beta\|g^{1/2}v\|_{L^2(\Gamma)}^2.
  \end{align*}
  By this inequality, $\beta<1$, and \eqref{E:G_Inf} we obtain $v=0$ on $\Gamma$ and thus $U=\bar{v}=0$ in $\Omega_1$, which contradicts with \eqref{Pf_KU:L2_Limit}. Hence \eqref{E:Korn_U} is valid.
\end{proof}

Combining Lemmas \ref{L:Korn_Grad} and \ref{L:Korn_U} we obtain the uniform Korn inequality on $\Omega_\varepsilon$.

\begin{lemma} \label{L:Korn_H1}
  For $\beta\in[0,1)$ there exist $\varepsilon_{K,\beta}\in(0,1]$ and $c_{K,\beta}>0$ such that
  \begin{align} \label{E:Korn_H1}
    \|u\|_{H^1(\Omega_\varepsilon)}^2 \leq c_{K,\beta}\|D(u)\|_{L^2(\Omega_\varepsilon)}^2
  \end{align}
  for all $\varepsilon\in(0,\varepsilon_{K,\beta}]$ and $u\in H^1(\Omega_\varepsilon)^3$ satisfying \eqref{E:Bo_Imp} and \eqref{E:Korn_Orth}.
\end{lemma}

\begin{proof}
  Let $c_{K,1}>0$ be the constant given in Lemma \ref{L:Korn_Grad}.
  Also, let $\varepsilon_K\in(0,1]$ and $c_{K,2}>0$ be the constants given in Lemma \ref{L:Korn_U} with $\alpha:=1/2c_{K,1}$.
  For $\varepsilon\in(0,\varepsilon_K]$ let $u\in H^1(\Omega_\varepsilon)^3$ satisfy \eqref{E:Bo_Imp} and \eqref{E:Korn_Orth}.
  By \eqref{E:Korn_Grad} and \eqref{E:Korn_U} we have
  \begin{align*}
    \|\nabla u\|_{L^2(\Omega_\varepsilon)}^2 \leq (4+c_{K,1}c_{K,2})\|D(u)\|_{L^2(\Omega_\varepsilon)}^2+c_{K,1}\alpha\|\nabla u\|_{L^2(\Omega_\varepsilon)}^2.
  \end{align*}
  Since $\alpha=1/2c_{K,1}$, the above inequality implies that
  \begin{align} \label{Pf_KH:K_Grad}
    \|\nabla u\|_{L^2(\Omega_\varepsilon)}^2 \leq c_{\beta,1}\|D(u)\|_{L^2(\Omega_\varepsilon)}^2, \quad c_{\beta,1} := 2(4+c_{K,1}c_{K,2}).
  \end{align}
  From this inequality and \eqref{E:Korn_U} we further deduce that
  \begin{align} \label{Pf_KH:K_U}
    \|u\|_{L^2(\Omega_\varepsilon)}^2 \leq c_{\beta,2}\|D(u)\|_{L^2(\Omega_\varepsilon)}^2, \quad c_{\beta,2} := 2(2c_{K,1}^{-1}+c_{K,2}).
  \end{align}
  By \eqref{Pf_KH:K_Grad} and \eqref{Pf_KH:K_U} we get \eqref{E:Korn_H1} with $\varepsilon_{K,\beta}:=\varepsilon_K$ and $c_{K,\beta}:=c_{\beta,1}+c_{\beta,2}$.
\end{proof}

Let $\mathcal{R}_g$ be the space of infinitesimal rigid displacements of $\mathbb{R}^3$ given by \eqref{E:Def_Rg}:
\begin{align*}
  \mathcal{R}_g = \{w(x)=a\times x+b,\,x\in\mathbb{R}^3 \mid \text{$a,b\in\mathbb{R}^3$, $w|_\Gamma\cdot n=w|_\Gamma\cdot \nabla_\Gamma g=0$ on $\Gamma$}\}.
\end{align*}
We show that \eqref{E:Korn_H1} holds with $\mathcal{K}_g(\Gamma)$ in the condition \eqref{E:Korn_Orth} replaced by $\mathcal{R}_g$ if $\mathcal{K}_g(\Gamma)$ agrees with $\mathcal{R}_g|_\Gamma:=\{w|_\Gamma \mid w\in\mathcal{R}_g\}$ (see also Remark \ref{R:Killing}).

\begin{lemma} \label{L:Korn_Rg}
  Suppose that $\mathcal{R}_g|_\Gamma=\mathcal{K}_g(\Gamma)$.
  Then for $\beta\in[0,1)$ there exist constants $\varepsilon_{K,\beta}\in(0,1]$ and $c_{K,\beta}>0$ such that the inequality \eqref{E:Korn_H1} holds for all $\varepsilon\in(0,\varepsilon_{K,\beta}]$ and $u\in H^1(\Omega_\varepsilon)^3$ satisfying \eqref{E:Bo_Imp} and
  \begin{align} \label{E:KoRg_Orth}
    \left|(u,w)_{L^2(\Omega_\varepsilon)}\right| \leq \beta\|u\|_{L^2(\Omega_\varepsilon)}\|w\|_{L^2(\Omega_\varepsilon)} \quad\text{for all}\quad w\in\mathcal{R}_g.
  \end{align}
\end{lemma}

Note that the vector field $w\in\mathcal{R}_g$ in \eqref{E:KoRg_Orth} has an explicit form $w(x)=a\times x+b$ for $x\in\mathbb{R}^3$, which is essential for the proof of Lemma \ref{L:Korn_Rg}.

\begin{proof}
  The proof is the same as that of Lemma \ref{L:Korn_H1} if we show that the statement of Lemma \ref{L:Korn_U} is still valid under the condition \eqref{E:KoRg_Orth} instead of \eqref{E:Korn_Orth}.
  Assume to the contrary that there exist a sequence $\{\varepsilon_k\}_{k=1}^\infty$ of positive numbers convergent to zero and vector fields $u_k\in H^1(\Omega_{\varepsilon_k})^3$, $k\in\mathbb{N}$ satisfying \eqref{E:Bo_Imp} on $\Gamma_{\varepsilon_k}$, \eqref{PF_KU:Ineq_Contra}, and \eqref{E:KoRg_Orth}.
  Let $\Phi_{\varepsilon_k}$ be the bijection from $\Omega_1$ onto $\Omega_{\varepsilon_k}$ given by \eqref{E:Def_Bij} and
  \begin{align*}
    U_k := u_k\circ\Phi_{\varepsilon_k} \in H^1(\Omega_1)^3.
  \end{align*}
  Then, after replacing $U_k$ with $U_k/\|U_k\|_{L^2(\Omega_1)}$, we can show as in the proof of Lemma \ref{L:Korn_U} that $\{U_k\}_{k=1}^\infty$ converges (up to a subsequence) strongly in $L^2(\Omega_1)^3$ to the constant extension $\bar{v}$ of some $v\in\mathcal{K}_g(\Gamma)$ and
  \begin{align} \label{Pf_KRg:Lim_L2}
    \|\bar{v}\|_{L^2(\Omega_1)} = \lim_{k\to\infty}\|U_k\|_{L^2(\Omega_1)} =1.
  \end{align}
  Now we can take $w\in\mathcal{R}_g$ such that $w|_\Gamma=v$ on $\Gamma$ by the assumption $\mathcal{R}_g|_\Gamma=\mathcal{K}_g(\Gamma)$.
  Then since $u_k$ satisfies \eqref{E:KoRg_Orth} and $w\in\mathcal{R}_g$,
  \begin{align} \label{Pf_KRg:Orth}
    \left|(u_k,w)_{L^2(\Omega_{\varepsilon_k})}\right| \leq \beta\|u_k\|_{L^2(\Omega_{\varepsilon_k})}\|w\|_{L^2(\Omega_{\varepsilon_k})}, \quad k\in\mathbb{N}.
  \end{align}
  Let $\varphi_k$ be the function on $\Omega_1$ given by \eqref{Pf_Ku:Phik}.
  Then by \eqref{E:CoV_Fixed} we get
  \begin{align*}
    (u_k,w)_{L^2(\Omega_{\varepsilon_k})} = \varepsilon_k\int_{\Omega_1}U_k\cdot(w\circ\Phi_{\varepsilon_k})\varphi_k\,dX.
  \end{align*}
  Since $w\in\mathcal{R}_g$ is of the form $w(x)=a\times x+b$ with $a,b\in\mathbb{R}^3$,
  \begin{align*}
    w(\Phi_{\varepsilon_k}(X)) = a\times\{\pi(X)+\varepsilon_kd(X)\bar{n}(X)\}+b = w(\pi(X))+\varepsilon_kd(X)\{a\times\bar{n}(X)\}
  \end{align*}
  for $X\in\Omega_1$.
  Hence by $|d|\leq c$ and $|\bar{n}|=1$ in $\Omega_1$ we have
  \begin{align*}
    |w\circ\Phi_{\varepsilon_k}-w\circ\pi| = \varepsilon_k|d(a\times\bar{n})| \leq c\varepsilon_k \to 0 \quad\text{as}\quad k\to\infty
  \end{align*}
  uniformly on $\Omega_1$.
  By this fact, \eqref{Pf_Ku:Conv_Phik}, and the strong convergence of $\{U_k\}_{k=1}^\infty$ to $\bar{v}$ in $L^2(\Omega_1)^3$ we observe that
  \begin{align*}
    \lim_{k\to\infty}\varepsilon_k^{-1}(u_k,w)_{L^2(\Omega_{\varepsilon_k})} = \int_{\Omega_1}\bar{v}\cdot(w\circ\pi)J(\pi(\cdot),d(\cdot))^{-1}\,dX.
  \end{align*}
  To the right-hand side we further apply \eqref{E:CoV_Dom} with $\varepsilon=1$,
  \begin{align*}
    \pi(y+rn(y)) = y, \quad y\in\Gamma, \, r\in(g_0(y),g_1(y)),
  \end{align*}
  and $w|_\Gamma=v$ on $\Gamma$ to obtain
  \begin{align*}
    \lim_{k\to\infty}\varepsilon_k^{-1}(u_k,w)_{L^2(\Omega_{\varepsilon_k})} = \int_\Gamma g(y)v(y)\cdot w(y)\,d\mathcal{H}^2(y) = \|g^{1/2}v\|_{L^2(\Gamma)}^2.
  \end{align*}
  In the same way we can show that
  \begin{align*}
    \lim_{k\to\infty}\varepsilon_k^{-1}\|u_k\|_{L^2(\Omega_{\varepsilon_k})}^2 = \lim_{k\to\infty}\varepsilon_k^{-1}\|w\|_{L^2(\Omega_{\varepsilon_k})}^2 = \|g^{1/2}v\|_{L^2(\Gamma)}^2.
  \end{align*}
  Thus, as in the last part of the proof of Lemma \ref{L:Korn_U}, we can derive $v=0$ on $\Gamma$ by dividing both sides of \eqref{Pf_KRg:Orth} by $\varepsilon_k$, sending $k\to\infty$, and using the above equalities, $\beta<1$, and \eqref{E:G_Inf}.
  This implies $\bar{v}=0$ on $\Omega_1$, which contradicts with \eqref{Pf_KRg:Lim_L2}.
  Hence the statement of Lemma \ref{L:Korn_U} holds under the condition \eqref{E:KoRg_Orth} instead of \eqref{E:Korn_Orth}.
\end{proof}

\begin{remark} \label{R:Korn_H1}
  The uniform Korn inequality \eqref{E:Korn_H1} was first established by Lewicka and M\"{u}ller \cite{LeMu11}*{Theorem 2.2} under the condition \eqref{E:Korn_Orth}.
  They combined a uniform Korn inequality on a thin cylinder and Korn's inequality on a surface to prove \eqref{E:Korn_H1}.
  In Lemma \ref{L:Korn_H1} we gave a more direct proof of \eqref{E:Korn_H1} under the same condition.

  The condition \eqref{E:KoRg_Orth} under the assumption $\mathcal{R}_g|_\Gamma=\mathcal{K}_g(\Gamma)$ is a new condition for the uniform Korn inequality \eqref{E:Korn_H1}.
  Note that we take a vector field $w\in\mathcal{R}_g$ defined on $\mathbb{R}^3$ itself in \eqref{E:KoRg_Orth}, not its restriction on $\Gamma$ as in \cite{LeMu11}.
  Due to this fact, Lemma \ref{L:Korn_Rg} under the assumption $\mathcal{R}_g=\mathcal{K}_g(\Gamma)$ gives an improvement of \cite{LeMu11}*{Theorem 2.3} which shows Korn's inequality with a constant of order $\varepsilon^{-1}$.

  As we mentioned in Remark \ref{R:Killing}, we have $\mathcal{R}|_\Gamma=\mathcal{K}(\Gamma)$ and thus $\mathcal{R}_g|_\Gamma=\mathcal{K}_g(\Gamma)$ for any $g$ if $\Gamma$ is axially symmetric or it is closed and convex.
  In particular, Lemma \ref{L:Korn_Rg} is applicable for curved thin domains around the unit sphere $S^2$ in $\mathbb{R}^3$.
\end{remark}

\subsection{Difference between the uniform and standard Korn inequalities} \label{SS:Diff_Korn}
In this subsection we discuss the difference between the uniform Korn inequality \eqref{E:Korn_H1} and a standard Korn inequality related to the axial symmetry of a domain.

For a fixed $\varepsilon\in(0,1]$ let
\begin{align*}
  \mathcal{R}_\varepsilon := \{w(x)=a\times x+b,\,x\in\mathbb{R}^3 \mid \text{$a,b\in\mathbb{R}^3$, $w|_{\Gamma_\varepsilon}\cdot n_\varepsilon=0$ on $\Gamma_\varepsilon$}\}.
\end{align*}
The set $\mathcal{R}_\varepsilon$ stands for the axial symmetry of $\Omega_\varepsilon$, i.e. $\mathcal{R}_\varepsilon\neq\{0\}$ if and only if $\Omega_\varepsilon$ is axially symmetric around some line (see Lemma \ref{L:IR_Surf}).
It appears in the following standard Korn inequality with a constant depending on a domain (see also \cites{AmRe14,Be04,SoSc73}).

\begin{lemma} \label{L:Korn_Reps}
  For fixed $\varepsilon\in(0,1]$ and $\beta\in[0,1)$ there exists a constant $c_\varepsilon>0$ depending on $\varepsilon$ (and $\beta$) such that
  \begin{align} \label{E:Korn_Reps}
    \|u\|_{H^1(\Omega_\varepsilon)} \leq c_\varepsilon\|D(u)\|_{L^2(\Omega_\varepsilon)}
  \end{align}
  for all $u\in H^1(\Omega_\varepsilon)^3$ satisfying \eqref{E:Bo_Imp} and
  \begin{align} \label{E:Reps_Orth}
    \left|(u,w)_{L^2(\Omega_\varepsilon)}\right| \leq \beta\|u\|_{L^2(\Omega_\varepsilon)}\|w\|_{L^2(\Omega_\varepsilon)} \quad\text{for all}\quad w\in\mathcal{R}_\varepsilon.
  \end{align}
\end{lemma}

\begin{proof}
  The proof is much easier than those of Lemmas \ref{L:Korn_H1} and \ref{L:Korn_Rg} since we fix $\varepsilon$ and do not use a change of variables.
  By \eqref{E:Korn_Grad} it is sufficient to show that
  \begin{align} \label{Pf_KRe:Goal}
    \|u\|_{L^2(\Omega_\varepsilon)} \leq c_\varepsilon\|D(u)\|_{L^2(\Omega_\varepsilon)}
  \end{align}
  for all $u\in H^1(\Omega_\varepsilon)^3$ satisfying \eqref{E:Bo_Imp} and \eqref{E:Reps_Orth}.
  Assume to the contrary that there exist vector fields $u_k\in H^1(\Omega_\varepsilon)^3$, $k\in\mathbb{N}$ satisfying \eqref{E:Bo_Imp}, \eqref{E:Reps_Orth}, and
  \begin{align} \label{Pf_KRe:Cont}
    \|u_k\|_{L^2(\Omega_\varepsilon)} > k\|D(u_k)\|_{L^2(\Omega_\varepsilon)}.
  \end{align}
  Replacing $u_k$ with $u_k/\|u_k\|_{L^2(\Omega_\varepsilon)}$ we may assume that
  \begin{align} \label{Pf_KRe:Uk_L2}
    \|u_k\|_{L^2(\Omega_\varepsilon)} = 1, \quad \|D(u_k)\|_{L^2(\Omega_\varepsilon)} < k^{-1}, \quad k\in\mathbb{N}.
  \end{align}
  By \eqref{E:Korn_Grad}, \eqref{Pf_KRe:Uk_L2}, and the compact embedding $H^1(\Omega_\varepsilon)\hookrightarrow L^2(\Omega_\varepsilon)$ we observe that (up to a subsequence) $\{u_k\}_{k=1}^\infty$ converges to some $u\in H^1(\Omega_\varepsilon)^3$ strongly in $L^2(\Omega_\varepsilon)^3$ and weakly in $H^1(\Omega_\varepsilon)^3$.
  Moreover, $\{u_k|_{\Gamma_\varepsilon}\}_{k=1}^\infty$ converges to $u|_{\Gamma_\varepsilon}$ strongly in $L^2(\Gamma_\varepsilon)^3$ by the trace inequality \eqref{E:Trace_L2} given below (note that we fix $\varepsilon$).
  By these facts, \eqref{Pf_KRe:Uk_L2}, and the fact that $u_k$ satisfies \eqref{E:Bo_Imp} for each $k\in\mathbb{N}$ we have
  \begin{align} \label{Pf_KRe:Limit}
    \|u\|_{L^2(\Omega_\varepsilon)} = 1, \quad D(u) = 0 \quad\text{in}\quad \Omega_\varepsilon, \quad u\cdot n_\varepsilon = 0 \quad\text{on}\quad \Gamma_\varepsilon.
  \end{align}
  The second equality of \eqref{Pf_KRe:Limit} implies that $u$ is of the form $u(x)=a\times x+b$ with some $a,b\in\mathbb{R}^3$.
  Hence by the third equality of \eqref{Pf_KRe:Limit} we get $u\in\mathcal{R}_\varepsilon$ and thus
  \begin{align*}
    \left|(u_k,u)_{L^2(\Omega_\varepsilon)}\right| \leq \beta\|u_k\|_{L^2(\Omega_\varepsilon)}\|u\|_{L^2(\Omega_\varepsilon)}
  \end{align*}
  for each $k\in\mathbb{N}$ since $u_k$ satisfies \eqref{E:Reps_Orth}.
  We send $k\to\infty$ in this inequality and use the strong convergence of $\{u_k\}_{k=1}^\infty$ to $u$ in $L^2(\Omega_\varepsilon)^3$ and \eqref{Pf_KRe:Limit} to obtain
  \begin{align*}
    1 = \|u\|_{L^2(\Omega_\varepsilon)}^2 \leq \beta\|u\|_{L^2(\Omega_\varepsilon)}^2 = \beta,
  \end{align*}
  which contradicts with $\beta\in[0,1)$.
  Therefore, the inequality \eqref{Pf_KRe:Goal} is valid.
\end{proof}

Let us give a trace inequality for $\Omega_\varepsilon$ used in the above proof.

\begin{lemma} \label{L:Trace_L2}
  There exists a constant $c>0$ such that
  \begin{align} \label{E:Trace_L2}
    \|\varphi\|_{L^2(\Gamma_\varepsilon^i)} \leq c\left(\varepsilon^{-1/2}\|\varphi\|_{L^2(\Omega_\varepsilon)}+\|\varphi\|_{L^2(\Omega_\varepsilon)}^{1/2}\|\partial_n\varphi\|_{L^2(\Omega_\varepsilon)}^{1/2}\right)
  \end{align}
  for all $\varepsilon\in(0,1]$ and $\varphi\in H^1(\Omega_\varepsilon)$, where $\partial_n\varphi=(\bar{n}\cdot\nabla)\varphi$ is the derivative of $\varphi$ in the normal direction of $\Gamma$.
\end{lemma}

\begin{proof}
  We use the notations \eqref{E:Pull_Dom}--\eqref{E:Pull_Bo}.
  The proof is almost the same as that of \eqref{E:Poin_Bo}.
  For $y\in\Gamma$ and $r\in(\varepsilon g_0(y),\varepsilon g_1(y))$ we employ the formula
  \begin{align*}
    |\varphi_i^\sharp(y)|^2 &= |\varphi^\sharp(y,r)|^2+\int_r^{\varepsilon g_i(y)}\frac{\partial}{\partial\tilde{r}}\Bigl(|\varphi^\sharp(y,\tilde{r})|^2\Bigr)d\tilde{r} \\
    &= |\varphi^\sharp(y,r)|^2+2\int_r^{\varepsilon g_i(y)}\varphi^\sharp(y,\tilde{r})(\partial_n\varphi)^\sharp(y,\tilde{r})d\tilde{r}
  \end{align*}
  instead of \eqref{Pf_P:FTC} to get
  \begin{align*}
    |\varphi_i^\sharp(y)|^2 \leq |\varphi^\sharp(y,r)|^2+2\int_{\varepsilon g_0(y)}^{\varepsilon g_1(y)}|\varphi^\sharp(y,\tilde{r})||(\partial_n\varphi)^\sharp(y,\tilde{r})|\,d\tilde{r}.
  \end{align*}
  We divide both sides by $\varepsilon$, integrate the resulting inequality with respect to $y$ and $r$, and then apply \eqref{E:G_Inf}, \eqref{E:CoV_Equiv}, \eqref{E:Lp_CoV_Surf}, and H\"{o}lder's inequality to obtain
  \begin{align*}
    \|\varphi\|_{L^2(\Gamma_\varepsilon^i)}^2 &\leq c\left(\varepsilon^{-1}\|\varphi\|_{L^2(\Omega_\varepsilon)}^2+\int_{\Omega_\varepsilon}|\varphi||\partial_n\varphi|\,dx\right) \\
    &\leq c\left(\varepsilon^{-1}\|\varphi\|_{L^2(\Omega_\varepsilon)}^2+\|\varphi\|_{L^2(\Omega_\varepsilon)}\|\partial_n\varphi\|_{L^2(\Omega_\varepsilon)}\right).
  \end{align*}
  Hence \eqref{E:Trace_L2} follows.
\end{proof}

The constant $c_\varepsilon$ in \eqref{E:Korn_Reps} may blow up as $\varepsilon\to0$ (see \cite{HoSe10}*{Corollary 4.11} for the case of a flat thin domain).
To see this, we use the following vector field which was introduced in \cite{LeMu11}*{Section 4} as a counterexample to \eqref{E:Korn_H1}.

\begin{lemma} \label{L:ExLM}
  Suppose that $\mathcal{K}_g(\Gamma)\neq\{0\}$.
  Let $v\in\mathcal{K}_g(\Gamma)$, $v\not\equiv0$ and
  \begin{align} \label{E:Def_ExLM}
    v^\varepsilon(x) := \left\{I_3-d(x)\overline{W}(x)\right\}\bar{v}(x)+\varepsilon\left\{\bar{v}(x)\cdot\overline{\nabla_\Gamma g_0}(x)\right\}\bar{n}(x), \quad x\in N.
  \end{align}
  Then $v^\varepsilon$ satisfies \eqref{E:Bo_Imp} for all $\varepsilon\in(0,1]$.
  Moreover, there exist constants $c_v^1,c_v^2>0$ and $\varepsilon_v\in(0,1]$ depending on $v$ such that
  \begin{align} \label{E:Est_ExLM}
    \|v^\varepsilon\|_{H^1(\Omega_\varepsilon)} \geq c_v^1\varepsilon^{1/2}, \quad \|D(v^\varepsilon)\|_{L^2(\Omega_\varepsilon)} \leq c_v^2\varepsilon^{3/2}
  \end{align}
  for all $\varepsilon\in(0,\varepsilon_v]$.
\end{lemma}

The results of Lemma \ref{L:ExLM} were originally shown in \cite{LeMu11}*{Section 4}.
Here we give different proofs of them in our notations.

\begin{proof}
  Let $\tau_\varepsilon^0$ and $\tau_\varepsilon^1$ be the vector fields on $\Gamma$ given by \eqref{E:Def_NB_Aux}.
  Since $W$ is symmetric and $v$, $Wv$, $\tau_\varepsilon^0$, and $\tau_\varepsilon^1$ are tangential on $\Gamma$,
  \begin{align*}
    v^\varepsilon(x)\cdot\bar{\tau}_\varepsilon^i(x) &= \{I_3-\varepsilon g_i(y)W(y)\}v(y)\cdot\{I_3-\varepsilon g_i(y)W(y)\}^{-1}\nabla_\Gamma g_i(y) \\
    &= v(y)\cdot\nabla_\Gamma g_i(y)
  \end{align*}
  for all $x=y+\varepsilon g_i(y)n(y)\in\Gamma_\varepsilon^i$ with $y\in\Gamma$ and $i=0,1$.
  On the other hand,
  \begin{align*}
    v(y)\cdot\nabla_\Gamma g_0(y) = v(y)\cdot\nabla_\Gamma g_1(y),\quad y\in\Gamma
  \end{align*}
  by $v\in\mathcal{K}_g(\Gamma)$ and $g=g_1-g_0$ on $\Gamma$ and thus
  \begin{align*}
    v^\varepsilon(x)\cdot\bar{n}(x) = \varepsilon\{v(y)\cdot\nabla_\Gamma g_0(y)\} = \varepsilon\{v(y)\cdot\nabla_\Gamma g_i(y)\} = \varepsilon\{v^\varepsilon(x)\cdot\bar{\tau}_\varepsilon^i(x)\}
  \end{align*}
  for all $x=y+\varepsilon g_i(y)n(y)\in\Gamma_\varepsilon^i$ with $y\in\Gamma$ and $i=0,1$.
  From this equality, \eqref{E:Def_NB}, and Lemma \ref{L:Nor_Bo} it follows that $v^\varepsilon$ satisfies \eqref{E:Bo_Imp}.

  Next let us prove \eqref{E:Est_ExLM}.
  We differentiate both sides of \eqref{E:Def_ExLM}.
  Then
  \begin{align} \label{Pf_ELM:Grad_Ve}
    \begin{aligned}
      \nabla v^\varepsilon &= \nabla\bar{v}-\nabla d\otimes\Bigl(\overline{W}\bar{v}\Bigr)-d\nabla\Bigl(\overline{W}\bar{v}\Bigr)+\varepsilon\nabla\Bigl[\Bigl(\bar{v}\cdot\overline{\nabla_\Gamma g_0}\Bigr)\bar{n}\Bigr] \\
      &= \overline{\nabla_\Gamma v}-\bar{n}\otimes\Bigl(\overline{W}\bar{v}\Bigr)+R_\varepsilon
    \end{aligned}
  \end{align}
  in $N$ by $\nabla d=\bar{n}$ in $N$, where
  \begin{align*}
    R_\varepsilon := \Bigl(\nabla{v}-\overline{\nabla_\Gamma v}\Bigr)-d\nabla\Bigl(\overline{W}\bar{v}\Bigr)+\varepsilon\nabla\Bigl[\Bigl(\bar{v}\cdot\overline{\nabla_\Gamma g_0}\Bigr)\bar{n}\Bigr] \quad\text{in}\quad N.
  \end{align*}
  Noting that $n$, $W$, and $\nabla_\Gamma g_0$ are bounded on $\Gamma$ along with their first order derivatives, we apply \eqref{E:ConDer_Bound}, \eqref{E:ConDer_Diff}, and $|d|\leq c\varepsilon$ in $\Omega_\varepsilon$ to $R_\varepsilon$ to get
  \begin{align} \label{Pf_ELM:Est_Re}
    \left|\nabla v^\varepsilon-\Bigl\{\overline{\nabla_\Gamma v}-\bar{n}\otimes\Bigl(\overline{W}\bar{v}\Bigr)\Bigr\}\right| = |R_\varepsilon| \leq c\varepsilon\left(|\bar{v}|+\left|\overline{\nabla_\Gamma v}\right|\right) \quad\text{in}\quad \Omega_\varepsilon
  \end{align}
  with some constant $c>0$ independent of $\varepsilon$.
  Hence
  \begin{align*}
    \|\nabla v^\varepsilon\|_{L^2(\Omega_\varepsilon)} &\geq \left\|\overline{\nabla_\Gamma v}-\bar{n}\otimes\Bigl(\overline{W}\bar{v}\Bigr)\right\|_{L^2(\Omega_\varepsilon)}-\|R_\varepsilon\|_{L^2(\Omega_\varepsilon)} \\
    &\geq \left\|\overline{\nabla_\Gamma v}-\bar{n}\otimes\Bigl(\overline{W}\bar{v}\Bigr)\right\|_{L^2(\Omega_\varepsilon)}-c\varepsilon\left(\|\bar{v}\|_{L^2(\Omega_\varepsilon)}+\left\|\overline{\nabla_\Gamma v}\right\|_{L^2(\Omega_\varepsilon)}\right)
  \end{align*}
  and, applying \eqref{E:Con_Lp} to the last line we further get
  \begin{align} \label{Pf_ELM:L2_Grad}
    \|\nabla v^\varepsilon\|_{L^2(\Omega_\varepsilon)} \geq c\varepsilon^{1/2}\left(\|\nabla_\Gamma v-n\otimes(Wv)\|_{L^2(\Gamma)}-\varepsilon\|v\|_{H^1(\Gamma)}\right).
  \end{align}
  On the other hand, by \eqref{E:Grad_W} and \eqref{Pf_ELM:Grad_Ve} we have
  \begin{align*}
    \nabla v^\varepsilon = \overline{P}\Bigl(\overline{\nabla_\Gamma v}\Bigr)\overline{P}+\Bigl(\overline{W}\bar{v}\Bigr)\otimes\bar{n}-\bar{n}\otimes\Bigl(\overline{W}\bar{v}\Bigr)+R_\varepsilon \quad\text{in}\quad N.
  \end{align*}
  From this equality and $D_\Gamma(v)=0$ on $\Gamma$ by $v\in\mathcal{K}_g(\Gamma)$ it follows that
  \begin{align*}
    D(v^\varepsilon) = \overline{D_\Gamma(v)}+(R_\varepsilon)_S = (R_\varepsilon)_S \quad\text{in}\quad N,
  \end{align*}
  where $(R_\varepsilon)_S=(R_\varepsilon+R_\varepsilon^T)/2$ is the symmetric part of $R_\varepsilon$, and thus
  \begin{align*}
    |D(v^\varepsilon)| \leq |R_\varepsilon| \leq c\varepsilon\left(|\bar{v}|+\left|\overline{\nabla_\Gamma v}\right|\right) \quad\text{in}\quad \Omega_\varepsilon
  \end{align*}
  by \eqref{Pf_ELM:Est_Re}.
  Hence we observe by \eqref{E:Con_Lp} that
  \begin{align} \label{Pf_ELM:L2_Str}
    \|D(v^\varepsilon)\|_{L^2(\Omega_\varepsilon)} \leq c\varepsilon\left(\|\bar{v}\|_{L^2(\Omega_\varepsilon)}+\left\|\overline{\nabla_\Gamma v}\right\|_{L^2(\Omega_\varepsilon)}\right) \leq c\varepsilon^{3/2}\|v\|_{H^1(\Gamma)}.
  \end{align}
  Now we claim $\|\nabla_\Gamma v-n\otimes(Wv)\|_{L^2(\Gamma)}>0$ if $v\not\equiv0$.
  Indeed, if
  \begin{align*}
    \nabla_\Gamma v-n\otimes(Wv) = 0, \quad\text{i.e.}\quad \underline{D}_iv_j-n_i[Wv]_j = 0 \quad\text{on}\quad \Gamma,\,i,j=1,2,3,
  \end{align*}
  then $\nabla_\Gamma v_j=[Wv]_jn$ on $\Gamma$, but this implies $\nabla_\Gamma v_j=0$ on $\Gamma$ since $\nabla_\Gamma v_j$ is tangential on $\Gamma$.
  Thus $v=(v_1,v_2,v_3)^T$ is constant on $\Gamma$.
  However, since $v$ is tangential on the closed surface $\Gamma$, it must identically vanish on $\Gamma$.
  Hence the claim is valid and we can take a constant $\varepsilon_v\in(0,1]$ so that
  \begin{align} \label{Pf_ELM:LB}
    \|\nabla_\Gamma v-n\otimes(Wv)\|_{L^2(\Gamma)}-\varepsilon\|v\|_{H^1(\Gamma)} \geq \frac{1}{2}\|\nabla_\Gamma v-n\otimes(Wv)\|_{L^2(\Gamma)} > 0
  \end{align}
  for all $\varepsilon\in(0,\varepsilon_v]$ by $v\not\equiv0$.
  Therefore, we deduce from \eqref{Pf_ELM:L2_Grad}--\eqref{Pf_ELM:LB} that
  \begin{align*}
    \|v^\varepsilon\|_{H^1(\Omega_\varepsilon)} \geq \|\nabla v^\varepsilon\|_{L^2(\Omega_\varepsilon)} \geq c_v^1\varepsilon^{1/2}, \quad \|D(v^\varepsilon)\|_{L^2(\Omega_\varepsilon)} \leq c_v^2\varepsilon^{3/2}
  \end{align*}
  for all $\varepsilon\in(0,\varepsilon_v]$, where $c_v^1,c_v^2>0$ are constants depending on $v$ but independent of $\varepsilon$, and the inequalities \eqref{E:Est_ExLM} are valid.
\end{proof}

From Lemma \ref{L:ExLM} it immediately follows that $c_\varepsilon$ blows up if $\mathcal{K}_g(\Gamma)\neq\{0\}$ and $\Omega_\varepsilon$ is not axially symmetric for all $\varepsilon\in(0,1]$ sufficiently small.
Let us give an example.

\begin{lemma} \label{L:ExBl_NAS}
  Let $\Gamma=S^2$ be the unit sphere in $\mathbb{R}^3$ and
  \begin{align*}
    g_0(y) = y_3, \quad g_1(y) = y_2+2, \quad y = (y_1,y_2,y_3)\in S^2.
  \end{align*}
  Then for all $\varepsilon\in(0,1]$ the curved thin domain
  \begin{align*}
    \Omega_\varepsilon = \{ry \mid y\in S^2,\, 1+\varepsilon y_3 < r < 1+\varepsilon(y_2+2)\}
  \end{align*}
  is not axially symmetric around any line.
  Also, there exist constants $c_b>0$ and $\varepsilon_b\in(0,1]$ such that the constant $c_\varepsilon$ given in Lemma \ref{L:Korn_Reps} with any $\beta\in[0,1)$ satisfies
  \begin{align} \label{E:ExBl}
    c_\varepsilon \geq c_b\varepsilon^{-1}
  \end{align}
  for all $\varepsilon\in(0,\varepsilon_b]$ and thus $c_\varepsilon\to\infty$ as $\varepsilon\to0$.
\end{lemma}

\begin{proof}
  First note that
  \begin{align*}
    g(y) = g_1(y)-g_0(y) = y_2-y_3+2 \geq 2-\sqrt{2}, \quad y\in S^2.
  \end{align*}
  Using the spherical coordinate system
  \begin{align*}
    S^2 = \{(\sin\vartheta_1\cos\vartheta_2,\sin\vartheta_1\sin\vartheta_2,\cos\vartheta_1) \mid \vartheta_1\in[0,\pi],\,\vartheta_2\in[0,2\pi]\}
  \end{align*}
  we can express the inner boundary of $\Omega_\varepsilon$ as
  \begin{gather*}
    \Gamma_\varepsilon^0 = \{(\varphi(\vartheta_1)\cos\vartheta_2,\varphi(\vartheta_1)\sin\vartheta_2,\psi(\vartheta_1)) \mid \vartheta_1\in[0,\pi],\,\vartheta_2\in[0,2\pi]\}, \\
    \varphi(\vartheta_1) := (1+\varepsilon\cos\vartheta_1)\sin\vartheta_1, \quad \psi(\vartheta_1) := (1+\varepsilon\cos\vartheta_1)\cos\vartheta_1.
  \end{gather*}
  Thus $\Gamma_\varepsilon^0$ is axially symmetric around the $x_3$-axis.
  Since $\Gamma_\varepsilon^0$ is not a sphere, it is not axially symmetric around other lines (see Remark \ref{R:IR_Kil}).
  Similarly, we see that the outer boundary $\Gamma_\varepsilon^1$ is axially symmetric only around the $x_2$-axis.
  Hence $\Omega_\varepsilon$ is not axially symmetric around any line, i.e. $\mathcal{R}_\varepsilon=\{0\}$ for all $\varepsilon\in(0,1]$ (see Lemma \ref{L:IR_Surf}).

  Next let us prove \eqref{E:ExBl}.
  Since $\Gamma=S^2$, we have
  \begin{align} \label{Pf_EBN:RK}
    \mathcal{R} = \{w(x)=a\times x,\, x\in\mathbb{R}^3 \mid a\in\mathbb{R}^3\}, \quad \mathcal{K}(S^2) = \mathcal{R}|_{S^2} = \{w|_{S^2} \mid w\in\mathcal{R}\}.
  \end{align}
  Let $v\in\mathcal{K}(S^2)$ be of the form $v(y)=a\times y$, $y\in S^2$ with $a\in\mathbb{R}^3$.
  Then since
  \begin{align*}
    \nabla_\Gamma g(y) = P(y)(e_2-e_3), \quad y\in S^2,
  \end{align*}
  where $\{e_1,e_2,e_3\}$ is the standard basis of $\mathbb{R}^3$, and $v$ is tangential on $S^2$,
  \begin{align*}
    v(y)\cdot\nabla_\Gamma g(y) = v(y)\cdot(e_2-e_3) = \{(e_2-e_3)\times a\}\cdot y, \quad y\in S^2.
  \end{align*}
  Hence $v\cdot\nabla_\Gamma g=0$ on $S^2$ if and only if $a=\alpha(e_2-e_3)$ with $\alpha\in\mathbb{R}$, i.e.
  \begin{align*}
    \mathcal{K}_g(S^2) = \{\alpha v_0 \mid \alpha\in\mathbb{R}\} \neq \{0\}, \quad v_0(y) := (e_2-e_3)\times y, \quad y\in S^2.
  \end{align*}
  Let $v^\varepsilon$ be the vector field defined by \eqref{E:Def_ExLM} with $v=v_0$.
  Also, let $c_v^1$, $c_v^2$, and $\varepsilon_v$ be the constants given in Lemma \ref{L:ExLM} and $\varepsilon_b:=\varepsilon_v$.
  Then $v^\varepsilon$ satisfies \eqref{E:Bo_Imp} and \eqref{E:Est_ExLM} for all $\varepsilon\in(0,\varepsilon_b]$ by Lemma \ref{L:ExLM}.
  Moreover, since $\mathcal{R}_\varepsilon=\{0\}$, the condition \eqref{E:Reps_Orth} with any $\beta\in[0,1)$ is automatically satisfied.
  Hence we can apply \eqref{E:Korn_Reps} to $v^\varepsilon$ and use \eqref{E:Est_ExLM} to obtain
  \begin{align*}
    c_v^1\varepsilon^{1/2} \leq \|v^\varepsilon\|_{H^1(\Omega_\varepsilon)} \leq c_\varepsilon\|D(v^\varepsilon)\|_{L^2(\Omega_\varepsilon)} \leq c_\varepsilon c_v^2\varepsilon^{3/2}
  \end{align*}
  for all $\varepsilon\in(0,\varepsilon_b]$.
  This inequality yields \eqref{E:ExBl} with $c_b:=c_v^1/c_v^2$.
\end{proof}

By Lemmas \ref{L:ExLM} and \ref{L:ExBl_NAS} we observe that the axial asymmetry of a curved thin domain is not sufficient for the uniform Korn inequality \eqref{E:Korn_H1}.
Next we give an example of an axially symmetric curved thin domain for which $c_\varepsilon$ blows up.

\begin{lemma} \label{L:ExBl_AS}
  Let $\Gamma=S^2$ be the unit sphere in $\mathbb{R}^3$ and
  \begin{align*}
    g_0(y) = y_3^2, \quad g_1(y) = y_3^2+1, \quad y=(y_1,y_2,y_3)\in S^2.
  \end{align*}
  Then for all $\varepsilon\in(0,1]$ the curved thin domain
  \begin{align} \label{E:EBA_CTD}
    \Omega_\varepsilon = \{ry \mid y\in S^2,\, 1+\varepsilon y_3^2 < r < 1+\varepsilon(y_3^2+1)\}
  \end{align}
  is axially symmetric only around the $x_3$-axis.
  Moreover, there exist constants $c_b>0$ and $\varepsilon_b\in(0,1]$ such that the constant $c_\varepsilon$ given in Lemma \ref{L:Korn_Reps} with any $\beta\in[0,1)$ satisfies \eqref{E:ExBl} for all $\varepsilon\in(0,\varepsilon_b]$ and thus $c_\varepsilon\to\infty$ as $\varepsilon\to0$.
\end{lemma}

\begin{proof}
  Since $g=g_1-g_0=1$ on $S^2$, we have $\nabla_\Gamma g=0$ on $S^2$ and
  \begin{gather} \label{Pf_EBA:RgKg}
    \mathcal{R}_g = \mathcal{R} = \{w(x) = a\times x,\, x\in\mathbb{R}^3 \mid a\in\mathbb{R}^3\}, \quad \mathcal{K}_g(S^2) = \mathcal{K}(S^2) = \mathcal{R}|_{S^2}.
  \end{gather}
  Also, as in the first part of the proof of Lemma \ref{L:ExBl_NAS} we observe that the inner and outer boundaries of $\Omega_\varepsilon$ are axially symmetric only around the $x_3$-axis.
  Hence $\Omega_\varepsilon$ is axially symmetric only around the $x_3$-axis and
  \begin{align} \label{Pf_EBA:Reps}
    \mathcal{R}_\varepsilon = \{\alpha w_3 \mid \alpha\in\mathbb{R}\}, \quad w_3(x) = e_3\times x, \quad x\in\mathbb{R}^3
  \end{align}
  for all $\varepsilon\in(0,1]$ by Lemma \ref{L:IR_Surf}, where $\{e_1,e_2,e_3\}$ is the standard basis of $\mathbb{R}^3$.

  Let us prove \eqref{E:ExBl}.
  We use the vector field $v^\varepsilon$ of the form \eqref{E:Def_ExLM}.
  Let
  \begin{align*}
    w_1(x) := e_1\times x, \quad x\in\mathbb{R}^3, \quad v_1 := w_1|_{S^2} \, (\neq0),
  \end{align*}
  and $v^\varepsilon$ be the vector field defined by \eqref{E:Def_ExLM} with $v=v_1$.
  Then $v^\varepsilon$ satisfies \eqref{E:Bo_Imp} and \eqref{E:Est_ExLM} by Lemma \ref{L:ExLM} since $v_1\in\mathcal{K}_g(S^2)$ by \eqref{Pf_EBA:RgKg}.
  Let us show that $v^\varepsilon$ satisfies the condition \eqref{E:Reps_Orth} for \eqref{E:Korn_Reps}.
  Since $v_1$ is tangential on $S^2$ and $g_0(y)=y_3^2$ for $y\in S^2$,
  \begin{gather*}
    n(y) = y, \quad W(y) = -\nabla_\Gamma n(y) = -P(y), \quad W(y)v_1(y) = -v_1(y), \\
    \nabla_\Gamma g_0(y) = 2y_3P(y)e_3, \quad v_1(y)\cdot\nabla_\Gamma g_0(y) = 2y_3v_1(y)\cdot e_3 = 2y_2y_3
  \end{gather*}
  for $y\in S^2$.
  Also, the signed distance function from $S^2$ and the constant extension of a function $\eta$ on $S^2$ are given by
  \begin{align*}
    d(x) = |x|-1, \quad \bar{\eta}(x) = \eta\left(\frac{x}{|x|}\right), \quad x\in N.
  \end{align*}
  By these formulas the vector field $v^\varepsilon$ is of the form
  \begin{align} \label{Pf_EBA:Ve}
    v^\varepsilon(x) = |x|\left(e_1\times\frac{x}{|x|}\right)+2\varepsilon\frac{x_2x_3}{|x|^2}\frac{x}{|x|} = w_1(x)+2\varepsilon\frac{x_2x_3}{|x|^2}\frac{x}{|x|}, \quad x\in N.
  \end{align}
  From this equality and $w_3(x)=e_3\times x$ it follows that
  \begin{align*}
    v^\varepsilon(x)\cdot w_3(x) = w_1(x)\cdot w_3(x) = -x_1x_3, \quad x\in N.
  \end{align*}
  Moreover, using the spherical coordinate system
  \begin{gather*}
    x_1 = r\sin\vartheta_1\cos\vartheta_2, \quad x_2 = r\sin\vartheta_1\sin\vartheta_2, \quad x_3 = r\cos\vartheta_1, \\
    \vartheta_1 \in [0,\pi], \quad \vartheta_2 \in [0,2\pi], \quad r \in \bigr(1+\varepsilon\cos^2\vartheta_1,1+\varepsilon(\cos^2\vartheta_1+1)\bigr)
  \end{gather*}
  for $\Omega_\varepsilon$ we have
  \begin{multline*}
    \int_{\Omega_\varepsilon}x_1x_3\,dx \\
    = \left(\int_0^{2\pi}\cos\vartheta_2\,d\vartheta_2\right)\left\{\int_0^\pi\left(\int_{1+\varepsilon\cos^2\vartheta_1}^{1+\varepsilon(\cos^2\vartheta_1+1)}r^4dr\right)\sin^2\vartheta_1\cos\vartheta_1d\vartheta_1\right\} = 0.
  \end{multline*}
  Thus $(v^\varepsilon,w_3)_{L^2(\Omega_\varepsilon)}=0$ and $v^\varepsilon$ satisfies \eqref{E:Reps_Orth} with any $\beta\in[0,1)$ by \eqref{Pf_EBA:Reps}.
  Since $v^\varepsilon$ also satisfies \eqref{E:Bo_Imp}, we can apply \eqref{E:Korn_Reps} to $v^\varepsilon$.
  Hence we get \eqref{E:ExBl} by \eqref{E:Korn_Reps} and \eqref{E:Est_ExLM} as in the proof of Lemma \ref{L:ExBl_NAS}.
\end{proof}

As we observed in the above proof, the vector field $v^\varepsilon$ of the form \eqref{Pf_EBA:Ve} satisfies the condition \eqref{E:Reps_Orth} for the standard Korn inequality \eqref{E:Korn_Reps}, but the uniform Korn inequality \eqref{E:Korn_H1} is not valid for $v^\varepsilon$.
Let us directly show that $v^\varepsilon$ does not satisfy the condition \eqref{E:Korn_Orth} or \eqref{E:KoRg_Orth} with any $\beta\in[0,1)$ for \eqref{E:Korn_H1}.

\begin{lemma} \label{L:ExBl_Con}
  Let $\Omega_\varepsilon$ and $v^\varepsilon$ be the curved thin domain and the vector field of the form \eqref{E:EBA_CTD} and \eqref{Pf_EBA:Ve}, respectively.
  Then for each $\beta\in[0,1)$ there exists a constant $\varepsilon_\beta\in(0,1]$ such that $v^\varepsilon$ does not satisfy \eqref{E:Korn_Orth} or \eqref{E:KoRg_Orth} for all $\varepsilon\in(0,\varepsilon_\beta]$.
\end{lemma}

\begin{proof}
  Let $v^\varepsilon$ be the vector field of the form \eqref{Pf_EBA:Ve}, i.e.
  \begin{align*}
    v^\varepsilon(x) = w_1(x)+\varepsilon u(x), \quad w_1(x) = e_1\times x, \quad u(x) = \frac{2x_2x_3}{|x|^2}\frac{x}{|x|}, \quad x\in N.
  \end{align*}
  By \eqref{Pf_EBA:RgKg} we have $w_1\in\mathcal{R}_g$ and $v_1:=w_1|_{S^2}\in\mathcal{K}_g(S^2)$.
  Since $w_1\cdot u=0$ in $N$,
  \begin{align} \label{Pf_EBC:VeW}
    (v^\varepsilon,w_1)_{L^2(\Omega_\varepsilon)} = \|w_1\|_{L^2(\Omega_\varepsilon)}^2, \quad \|v^\varepsilon\|_{L^2(\Omega_\varepsilon)}^2 = \|w_1\|_{L^2(\Omega_\varepsilon)}^2+\varepsilon^2\|u\|_{L^2(\Omega_\varepsilon)}^2.
  \end{align}
  Also, by $\bar{v}_1(x)=v_1(x/|x|)=|x|^{-1}w_1(x)$ for $x\in N$,
  \begin{align} \label{Pf_EBC:VeV_1}
    (v^\varepsilon,\bar{v}_1)_{L^2(\Omega_\varepsilon)} = \int_{\Omega_\varepsilon}\frac{|w_1(x)|^2}{|x|}\,dx, \quad \|\bar{v}_1\|_{L^2(\Omega_\varepsilon)}^2 = \int_{\Omega_\varepsilon}\frac{|w_1(x)|^2}{|x|^2}\,dx.
  \end{align}
  Since $w_1\not\equiv0$, the first equalities of \eqref{Pf_EBC:VeW} and \eqref{Pf_EBC:VeV_1} show that $v^\varepsilon$ does not satisfy \eqref{E:Korn_Orth} or \eqref{E:KoRg_Orth} with $\beta=0$ for all $\varepsilon\in(0,1]$.
  Now let $\beta\in(0,1)$.
  Since
  \begin{align} \label{Pf_EBC:Width}
    1 \leq |x| = r \leq 1+2\varepsilon, \quad x = ry \in \Omega_\varepsilon,\, y\in S^2,\, r\in\bigl(1+\varepsilon y_3^2,1+\varepsilon(y_3^2+1)\bigr),
  \end{align}
  we deduce from \eqref{Pf_EBC:VeV_1} and \eqref{Pf_EBC:Width} that
  \begin{align} \label{Pf_EBC:VeV_2}
    (v^\varepsilon,\bar{v}_1)_{L^2(\Omega_\varepsilon)} \geq \frac{1}{1+2\varepsilon}\|w_1\|_{L^2(\Omega_\varepsilon)}^2, \quad \|\bar{v}_1\|_{L^2(\Omega_\varepsilon)}^2 \leq \|w_1\|_{L^2(\Omega_\varepsilon)}^2.
  \end{align}
  Also, by the change of variables and \eqref{Pf_EBC:Width} we see that
  \begin{align*}
    \|w_1\|_{L^2(\Omega_\varepsilon)}^2 &= \int_{S^2}\int_{1+\varepsilon y_3^2}^{1+\varepsilon(y_3^2+1)}(y_2^2+y_3^2)r^4\,dr\,d\mathcal{H}^2(y) \geq c_1\varepsilon, \\
    \|u\|_{L^2(\Omega_\varepsilon)}^2 &= \int_{S^2}\int_{1+\varepsilon y_3^2}^{1+\varepsilon(y_3^2+1)}4y_2^2y_3^2r^2\,dr\,d\mathcal{H}^2(y) \leq c_2\varepsilon
  \end{align*}
  with constants $c_1,c_2>0$ independent of $\varepsilon$ and thus setting $c_3:=c_2/c_1$ we get
  \begin{align*}
    \|u\|_{L^2(\Omega_\varepsilon)}^2 \leq c_3\|w_1\|_{L^2(\Omega_\varepsilon)}^2.
  \end{align*}
 This inequality and the second equality of \eqref{Pf_EBC:VeW} imply that
 \begin{align} \label{Pf_EBC:L2_Ve}
   \|v^\varepsilon\|_{L^2(\Omega_\varepsilon)} = \left(\|w_1\|_{L^2(\Omega_\varepsilon)}^2+\varepsilon^2\|u\|_{L^2(\Omega_\varepsilon)}^2\right)^{1/2} \leq (1+c_3\varepsilon^2)^{1/2}\|w_1\|_{L^2(\Omega_\varepsilon)}.
 \end{align}
 Now since $\beta\in(0,1)$ there exists a constant $\varepsilon_\beta\in(0,1]$ such that
 \begin{align} \label{Pf_EBC:Con}
   \beta(1+c_3\varepsilon^2)^{1/2}(1+2\varepsilon) < 1, \quad\text{i.e.}\quad \beta(1+c_3\varepsilon^2)^{1/2} < \frac{1}{1+2\varepsilon} < 1
 \end{align}
 for all $\varepsilon\in(0,\varepsilon_\beta]$.
 Then by \eqref{Pf_EBC:VeW}, \eqref{Pf_EBC:L2_Ve}, and \eqref{Pf_EBC:Con} we get
 \begin{align*}
   \beta\|v^\varepsilon\|_{L^2(\Omega_\varepsilon)}\|w_1\|_{L^2(\Omega_\varepsilon)} \leq \beta(1+c_3\varepsilon^2)^{1/2}\|w_1\|_{L^2(\Omega_\varepsilon)}^2 < \|w_1\|_{L^2(\Omega_\varepsilon)}^2 = (v^\varepsilon, w_1)_{L^2(\Omega_\varepsilon)}
 \end{align*}
 and, by \eqref{Pf_EBC:VeV_2}, \eqref{Pf_EBC:L2_Ve}, and \eqref{Pf_EBC:Con},
 \begin{align*}
   \beta\|v^\varepsilon\|_{L^2(\Omega_\varepsilon)}\|\bar{v}_1\|_{L^2(\Omega_\varepsilon)} &\leq \beta(1+c_3\varepsilon^2)^{1/2}\|w_1\|_{L^2(\Omega_\varepsilon)}^2 \\
   &< \frac{1}{1+2\varepsilon}\|w_1\|_{L^2(\Omega_\varepsilon)}^2 \leq (v^\varepsilon,\bar{v}_1)_{L^2(\Omega_\varepsilon)}
 \end{align*}
 for all $\varepsilon\in(0,\varepsilon_\beta]$.
 Hence $v^\varepsilon$ fails to satisfy \eqref{E:Korn_Orth} with $v=v_1\in\mathcal{K}_g(S^2)$ and \eqref{E:KoRg_Orth} with $w=w_1\in\mathcal{R}_g$ for all $\varepsilon\in(0,\varepsilon_\beta]$.
\end{proof}

Lemmas \ref{L:ExBl_AS} and \ref{L:ExBl_Con} show that the conditions \eqref{E:Korn_Orth} and \eqref{E:KoRg_Orth} for the uniform Korn inequality \eqref{E:Korn_H1} can be more strict than the condition \eqref{E:Reps_Orth} for the standard Korn inequality \eqref{E:Korn_Reps} even if a curved thin domain is axially symmetric.
Note that it may also happen that the condition \eqref{E:KoRg_Orth} is the same as \eqref{E:Reps_Orth}.
For example, if
\begin{align*}
  \Omega_\varepsilon = \{x \in \mathbb{R}^3 \mid 1<|x|<1+\varepsilon\} \quad (\Gamma=S^2,\, g_0 \equiv 0,\, g_1 \equiv 1)
\end{align*}
is a thin spherical shell, then
\begin{align*}
  \mathcal{R}_g = \mathcal{R}_\varepsilon = \mathcal{R} = \{w(x) = a\times x,\, x\in\mathbb{R}^3 \mid a\in\mathbb{R}^3\}
\end{align*}
for all $\varepsilon\in(0,1]$ and thus the condition \eqref{E:KoRg_Orth} is the same as \eqref{E:Reps_Orth}.

\begin{remark} \label{R:Com_LM}
  The authors of \cite{LeMu11} constructed the vector field of the form \eqref{E:Def_ExLM} and proved \eqref{E:Est_ExLM} under the assumption $\mathcal{K}_g(\Gamma)\neq\{0\}$ to show that the uniform Korn inequality \eqref{E:Korn_H1} fails to hold without the condition \eqref{E:Korn_Orth}.
  Based on that result they mentioned at the end of \cite{LeMu11}*{Section 4} that the constant $c_\varepsilon$ in \eqref{E:Korn_Reps} blows up as $\varepsilon\to0$ even if the limit surface $\Gamma$ is not axially symmetric.
  Indeed, if $\Gamma$ is not axially symmetric, then $\Omega_\varepsilon$ is also not axially symmetric for all $\varepsilon\in(0,1]$ sufficiently small (see Lemma \ref{L:CTD_Rg}) and we get \eqref{E:ExBl} as in the proof of Lemma \ref{L:ExBl_NAS}.
  However, as we mentioned in Remark \ref{R:Killing}, it is not known whether there exists a closed surface in $\mathbb{R}^3$ that is not axially symmetric but admits a nontrivial Killing vector field, i.e. $\mathcal{R}=\{0\}$ but $\mathcal{K}(\Gamma)\neq\{0\}$.
  To avoid this problem, we presented the concrete examples of curved thin domains around the unit sphere in $\mathbb{R}^3$ for which the relations \eqref{Pf_EBN:RK} hold and $c_\varepsilon$ blows up as $\varepsilon\to0$ in Lemmas \ref{L:ExBl_NAS} and \ref{L:ExBl_AS}.
  Note that the comment at the end of \cite{LeMu11}*{Section 4} is valid for curved thin domains in $\mathbb{R}^2$, since every smooth closed curve in $\mathbb{R}^2$ has a nontrivial tangential vector field of constant length as a nontrivial Killing vector field.
\end{remark}

\section{Uniform a priori estimate for the vector Laplace operator} \label{S:UniLap}
The purpose of this section is to prove the following uniform a priori estimate for the vector Laplace operator on $\Omega_\varepsilon$ under the slip boundary conditions \eqref{E:Bo_Slip}.

\begin{lemma} \label{L:Lap_Apri}
  Suppose that the inequalities \eqref{E:Fric_Upper} are valid for all $\varepsilon\in(0,1]$.
  Then there exists a constant $c>0$ independent of $\varepsilon$ such that
  \begin{align} \label{E:Lap_Apri}
    \|u\|_{H^2(\Omega_\varepsilon)} \leq c\left(\|\Delta u\|_{L^2(\Omega_\varepsilon)}+\|u\|_{H^1(\Omega_\varepsilon)}\right)
  \end{align}
  for all $\varepsilon\in(0,1]$ and $u\in H^2(\Omega_\varepsilon)^3$ satisfying \eqref{E:Bo_Slip}.
\end{lemma}

First we give an approximation result for a vector field in $H^2(\Omega_\varepsilon)^3$ satisfying the slip boundary conditions \eqref{E:Bo_Slip}.
To this end, we consider the problem
\begin{align} \label{E:EE_LE}
  \begin{cases}
    -\nu\{\Delta u+\nabla(\mathrm{div}\,u)\}+u = f &\text{in}\quad \Omega_\varepsilon, \\
    u\cdot n_\varepsilon = 0, \quad 2\nu P_\varepsilon D(u)n_\varepsilon+\gamma_\varepsilon u = 0 &\text{on}\quad \Gamma_\varepsilon
  \end{cases}
\end{align}
for a given data $f\colon\Omega_\varepsilon\to\mathbb{R}^3$.
The bilinear form for \eqref{E:EE_LE} is given by
\begin{align*}
  \tilde{a}_\varepsilon(u_1,u_2) := 2\nu\bigl(D(u_1),D(u_2)\bigr)_{L^2(\Omega_\varepsilon)}+(u_1,u_2)_{L^2(\Omega_\varepsilon)}+\sum_{i=0,1}\gamma_\varepsilon^i(u_1,u_2)_{L^2(\Gamma_\varepsilon^i)}
\end{align*}
for $u_1,u_2\in H^1(\Omega_\varepsilon)^3$ (see Lemma \ref{L:IbP_St}).
We denote by $\langle\cdot,\cdot\rangle_{\Omega_\varepsilon}$ the duality product between $H^{-1}(\Omega_\varepsilon)$ and $H^1(\Omega_\varepsilon)$ and define
\begin{align*}
  H_{n,0}^1(\Omega_\varepsilon) := \{u \in H^1(\Omega_\varepsilon)^3 \mid \text{$u\cdot n_\varepsilon=0$ on $\Gamma_\varepsilon$}\}.
\end{align*}
Note that $H_{n,0}^1(\Omega_\varepsilon)$ is closed in $H^1(\Omega_\varepsilon)^3$ and thus a Hilbert space.

\begin{lemma} \label{L:EE_Reg}
  For $\varepsilon\in(0,1]$ let $f\in H^{-1}(\Omega_\varepsilon)^3$.
  Suppose that the inequalities \eqref{E:Fric_Upper} are valid.
  Then there exists a unique weak solution $u\in H_{n,0}^1(\Omega_\varepsilon)$ to \eqref{E:EE_LE} in the sense that
  \begin{align*}
    \tilde{a}_\varepsilon(u,\Phi) = \langle f,\Phi\rangle_{\Omega_\varepsilon} \quad\text{for all}\quad \Phi\in H_{n,0}^1(\Omega_\varepsilon).
  \end{align*}
  Moreover, if $f\in L^2(\Omega_\varepsilon)^3$, then $u\in H^2(\Omega_\varepsilon)^3$ and it satisfies \eqref{E:EE_LE} a.e. in $\Omega_\varepsilon$ and on $\Gamma_\varepsilon$, and there exists a constant $c_\varepsilon>0$ depending on $\varepsilon$ such that
  \begin{align} \label{E:EE_Reg}
    \|u\|_{H^2(\Omega_\varepsilon)} \leq c_\varepsilon\|f\|_{L^2(\Omega_\varepsilon)}.
  \end{align}
  If in addition $f\in H^1(\Omega_\varepsilon)^3$ then $u\in H^3(\Omega_\varepsilon)^3$.
\end{lemma}

Note that it does not matter how the constant $c_\varepsilon$ in \eqref{E:EE_Reg} depends on $\varepsilon$ since we apply Lemma \ref{L:EE_Reg} just for approximation of a vector field on $\Omega_\varepsilon$ (see Lemma \ref{L:NSl_Approx}).

\begin{proof}
  Let $u\in H_{n,0}^1(\Omega_\varepsilon)$.
  By the definition of $\tilde{a}_\varepsilon$ it is obvious that
  \begin{align*}
    \|u\|_{L^2(\Omega_\varepsilon)}^2 \leq \tilde{a}_\varepsilon(u,u).
  \end{align*}
  Also, the inequality \eqref{E:Korn_Grad} implies that
  \begin{align*}
    \|\nabla u\|_{L^2(\Omega_\varepsilon)}^2 \leq c\left(\|D(u)\|_{L^2(\Omega_\varepsilon)}^2+\|u\|_{L^2(\Omega_\varepsilon)}^2\right) \leq c\tilde{a}_\varepsilon(u,u)
  \end{align*}
  with a constant $c>0$ independent of $\varepsilon$.
  Hence
  \begin{align} \label{Pf_ER:Ta_Coer}
    \|u\|_{H^1(\Omega_\varepsilon)}^2 \leq c\tilde{a}_\varepsilon(u,u).
  \end{align}
  On the other hand, the inequalities \eqref{E:Fric_Upper} and \eqref{E:Poin_Bo} yield
  \begin{align*}
    \gamma_\varepsilon^i\|u\|_{L^2(\Gamma_\varepsilon^i)}^2 \leq c\varepsilon\left(\varepsilon^{-1}\|u\|_{L^2(\Omega_\varepsilon)}^2+\varepsilon\|\partial_nu\|_{L^2(\Omega_\varepsilon)}^2\right) \leq c\|u\|_{H^1(\Omega_\varepsilon)}
  \end{align*}
  for $i=0,1$.
  From this inequality and $|D(u)|\leq|\nabla u|$ in $\Omega_\varepsilon$ it follows that
  \begin{align} \label{Pf_ER:Ta_Bound}
    \tilde{a}_\varepsilon(u,u) \leq c\|u\|_{H^1(\Omega_\varepsilon)}^2.
  \end{align}
  By \eqref{Pf_ER:Ta_Coer} and \eqref{Pf_ER:Ta_Bound} the bilinear form $\tilde{a}_\varepsilon$ is bounded and coercive on $H_{n,0}^1(\Omega_\varepsilon)$.
  Thus the Lax--Milgram theorem implies the existence and uniqueness of a weak solution $u\in H_{n,0}^1(\Omega_\varepsilon)$ to \eqref{E:EE_LE}.

  The proof of the $H^2$-regularity of $u$ with \eqref{E:EE_Reg} for $f\in L^2(\Omega_\varepsilon)^3$ is carried out by a standard localization argument and a method of the difference quotient.
  Here we omit it since it is the same as the proofs of \cite{Be04}*{Theorem 1.2} and \cite{SoSc73}*{Theorem 2} which established the $H^2$-regularity of a weak solution to the Stokes problem in a general bounded domain under the slip boundary conditions.

  The $H^3$-regularity of $u$ for $f\in H^1(\Omega_\varepsilon)^3$ is proved by induction and a localization argument as in the case of a general second order elliptic equation.
  For details, we refer to \cite{Ev10}*{Section 6.3, Theorem 5}.
  (Note that the $C^4$-regularity of the boundary $\Gamma_\varepsilon$ is required for the $H^3$-regularity of $u$, see Section \ref{SS:Pre_Dom} and \cites{Be04,SoSc73}.)
\end{proof}

Based on Lemma \ref{L:EE_Reg} we show that a vector field in $H^2(\Omega_\varepsilon)^3$ is approximated by those in $H^3(\Omega_\varepsilon)^3$ under the slip boundary conditions \eqref{E:Bo_Slip}.

\begin{lemma} \label{L:NSl_Approx}
  For $\varepsilon\in(0,1]$ let $u\in H^2(\Omega_\varepsilon)^3$ satisfy \eqref{E:Bo_Slip} and suppose that the inequalities \eqref{E:Fric_Upper} are valid.
  Then there exists a sequence $\{u_k\}_{k=1}^\infty$ in $H^3(\Omega_\varepsilon)^3$ such that $u_k$ satisfies \eqref{E:Bo_Slip} for each $k\in\mathbb{N}$ and
  \begin{align*}
    \lim_{k\to\infty}\|u-u_k\|_{H^2(\Omega_\varepsilon)}=0.
  \end{align*}
\end{lemma}

\begin{proof}
  Let $u\in H^2(\Omega_\varepsilon)^3$ satisfy \eqref{E:Bo_Slip} and
  \begin{align*}
    f := -\nu\{\Delta u+\nabla(\mathrm{div}\,u)\}+u \in L^2(\Omega_\varepsilon)^3.
  \end{align*}
  Then we can take a sequence $\{f_k\}_{k=1}^\infty$ in $C_c^\infty(\Omega_\varepsilon)^3$ that converges to $f$ strongly in $L^2(\Omega_\varepsilon)^3$.
  For each $k\in\mathbb{N}$ let $u_k$ be a unique weak solution to \eqref{E:EE_LE} with data $f_k\in C_c^\infty(\Omega_\varepsilon)^3$.
  Then $u_k\in H^3(\Omega_\varepsilon)^3$ and it satisfies \eqref{E:Bo_Slip} by Lemma \ref{L:EE_Reg}.
  Moreover, since $u-u_k$ is a unique weak solution to \eqref{E:EE_LE} with data $f-f_k$,
  \begin{align*}
    \|u-u_k\|_{H^2(\Omega_\varepsilon)} \leq c_\varepsilon\|f-f_k\|_{L^2(\Omega_\varepsilon)} \to 0 \quad\text{as}\quad k\to\infty
  \end{align*}
  by \eqref{E:EE_Reg} and the strong convergence of $\{f_k\}_{k=1}^\infty$ to $f$ in $L^2(\Omega_\varepsilon)^3$ (note that the constant $c_\varepsilon$ does not depend on $k$).
\end{proof}

Now let us prove Lemma \ref{L:Lap_Apri}.
As in Section \ref{SS:Pre_Surf}, for a function space $\mathcal{X}(\Gamma_\varepsilon)$ on $\Gamma_\varepsilon$ we denote the space of all tangential vector fields on $\Gamma_\varepsilon$ of class $\mathcal{X}$ by
\begin{align*}
  \mathcal{X}(\Gamma_\varepsilon,T\Gamma_\varepsilon) := \{u\in \mathcal{X}(\Gamma_\varepsilon)^3 \mid \text{$u\cdot n_\varepsilon=0$ on $\Gamma_\varepsilon$}\}.
\end{align*}
For $u\in H^1(\Gamma_\varepsilon,T\Gamma_\varepsilon)$ and $v\in L^2(\Gamma_\varepsilon,T\Gamma_\varepsilon)$ we define the covariant derivative
\begin{align*}
  \overline{\nabla}_v^\varepsilon u := P_\varepsilon(v\cdot\nabla)\tilde{u} = P_\varepsilon(v\cdot\nabla_{\Gamma_\varepsilon})u \quad\text{on}\quad \Gamma_\varepsilon,
\end{align*}
where $\tilde{u}$ is any $H^1$-extension of $u$ to an open neighborhood of $\Gamma_\varepsilon$ with $\tilde{u}|_{\Gamma_\varepsilon}=u$.
We use the formulas for the covariant derivatives given in Appendix \ref{S:Ap_Cov}.

\begin{proof}[Proof of Lemma \ref{L:Lap_Apri}]
  Let $u\in H^2(\Omega_\varepsilon)^3$ satisfy \eqref{E:Bo_Slip}.
  Since
  \begin{align*}
    \|u\|_{H^2(\Omega_\varepsilon)}^2 = \|u\|_{H^1(\Omega_\varepsilon)}^2+\|\nabla^2u\|_{L^2(\Omega_\varepsilon)}^2,
  \end{align*}
  it is sufficient for \eqref{E:Lap_Apri} to show that
  \begin{align} \label{Pf_LA:Goal}
    \|\nabla^2u\|_{L^2(\Omega_\varepsilon)}^2 \leq c\left(\|\Delta u\|_{L^2(\Omega_\varepsilon)}^2+\|u\|_{H^1(\Omega_\varepsilon)}^2\right).
  \end{align}
  Moreover, by Lemma \ref{L:NSl_Approx} and a density argument we may assume that $u$ belongs to $H^3(\Omega_\varepsilon)^3$ and satisfies \eqref{E:Bo_Slip}, and thus $u\in H^2(\Gamma_\varepsilon,T\Gamma_\varepsilon)$.
  By $u\in H^3(\Omega_\varepsilon)^3$ we can carry out integration by parts twice to get
  \begin{align} \label{Pf_LA:IbP}
    \|\nabla^2u\|_{L^2(\Omega_\varepsilon)}^2 = \|\Delta u\|_{L^2(\Omega_\varepsilon)}^2+\int_{\Gamma_\varepsilon}\nabla u:\{(n_\varepsilon\cdot\nabla)\nabla u-n_\varepsilon\otimes\Delta u\}\,d\mathcal{H}^2.
  \end{align}
  Here $(n_\varepsilon\cdot\nabla)\nabla u$ denotes a $3\times3$ matrix whose $(i,j)$-entry is given by
  \begin{align*}
    [(n_\varepsilon\cdot\nabla)\nabla u]_{ij} := n_\varepsilon\cdot\nabla(\partial_iu_j), \quad i,j=1,2,3.
  \end{align*}
  Let us estimate the boundary integral in \eqref{Pf_LA:IbP}.
  Our goal is to show that
  \begin{multline} \label{Pf_LA:BD}
    \left|\int_{\Gamma_\varepsilon}\nabla u:\{(n_\varepsilon\cdot\nabla)\nabla u-n_\varepsilon\otimes\Delta u\}\,d\mathcal{H}^2\right| \\
    \leq c\left(\|u\|_{H^1(\Omega_\varepsilon)}^2+\|u\|_{H^1(\Omega_\varepsilon)}\|\nabla^2u\|_{L^2(\Omega_\varepsilon)}\right)
  \end{multline}
  with a constant $c>0$ independent of $\varepsilon$.
  Since $u$ satisfies \eqref{E:Bo_Slip}, we have
  \begin{align} \label{Pf_LA:ND_Bo}
    (\nabla u)^Tn_\varepsilon = (n_\varepsilon\cdot\nabla)u = -W_\varepsilon u-\tilde{\gamma}_\varepsilon u+\xi_\varepsilon n_\varepsilon \quad\text{on}\quad \Gamma_\varepsilon
  \end{align}
  by \eqref{E:NSl_ND} (note that $u$ and $W_\varepsilon u$ are tangential on $\Gamma_\varepsilon$), where
  \begin{align*}
    \tilde{\gamma}_\varepsilon := \frac{\gamma_\varepsilon}{\nu}, \quad \xi_\varepsilon := (n_\varepsilon\cdot\nabla)u\cdot n_\varepsilon = \nabla u: Q_\varepsilon.
  \end{align*}
  The first step for \eqref{Pf_LA:BD} is to prove
  \begin{align} \label{Pf_LA:Int_Sum}
    \int_{\Gamma_\varepsilon}\nabla u:\{(n_\varepsilon\cdot\nabla)\nabla u-n_\varepsilon\otimes\Delta u\}\,d\mathcal{H}^2 = \sum_{k=1}^4\int_{\Gamma_\varepsilon}\varphi_k\,d\mathcal{H}^2,
  \end{align}
  where
  \begin{align} \label{Pf_LA:Def_Phi}
    \begin{aligned}
      \varphi_1 &:= -2\{\nabla_{\Gamma_\varepsilon}W_\varepsilon\cdot u+(\nabla u)W_\varepsilon+\tilde{\gamma_\varepsilon}\nabla u\}:P_\varepsilon(\nabla u)P_\varepsilon,\\
      \varphi_2 &:= W_\varepsilon\nabla u:(\nabla u)P_\varepsilon \\
      &\qquad\qquad -2(u\cdot\mathrm{div}_{\Gamma_\varepsilon}W_\varepsilon+2\nabla u:W_\varepsilon)(\nabla u:Q_\varepsilon)+H_\varepsilon(\nabla u:Q_\varepsilon)^2, \\
      \varphi_3 &:= -(W_\varepsilon^3u-H_\varepsilon W_\varepsilon^2u)\cdot u, \\
      \varphi_4 &:= -\tilde{\gamma}_\varepsilon(2W_\varepsilon^2u-2H_\varepsilon W_\varepsilon u-\tilde{\gamma}_\varepsilon H_\varepsilon u)\cdot u.
    \end{aligned}
  \end{align}
  In \eqref{Pf_LA:Def_Phi} we used the notation $\nabla_{\Gamma_\varepsilon}W_\varepsilon\cdot u$ for the $3\times 3$ matrix with $(i,j)$-entry
  \begin{align} \label{Pf_LA:DW_u_Bo}
    [\nabla_{\Gamma_\varepsilon}W_\varepsilon\cdot u]_{ij} := \sum_{k=1}^3(\underline{D}_i^\varepsilon[W_\varepsilon]_{jk})u_k, \quad i,j=1,2,3
  \end{align}
  and the notation $\mathrm{div}_{\Gamma_\varepsilon}W_\varepsilon$ for the vector field with $j$-th component
  \begin{align} \label{Pf_LA:Div_W}
    [\mathrm{div}_{\Gamma_\varepsilon}W_\varepsilon]_j := \sum_{i=1}^3\underline{D}_i^\varepsilon[W_\varepsilon]_{ij}, \quad j=1,2,3.
  \end{align}
  Using a partition of unity on $\Gamma_\varepsilon$ we may assume that $u|_{\Gamma_\varepsilon}$ is compactly supported in a relatively open subset $O$ of $\Gamma_\varepsilon$ on which we can take a local orthonormal frame $\{\tau_1,\tau_2\}$ (see Appendix \ref{S:Ap_Cov}).
  Since $\{\tau_1,\tau_2,n_\varepsilon\}$ is an orthonormal basis of $\mathbb{R}^3$,
  \begin{align} \label{Pf_LA:Decom_Intg}
    \begin{aligned}
      \nabla u: \{(n_\varepsilon\cdot\nabla)\nabla u-n_\varepsilon\otimes\Delta u\} &= (\nabla u)^T: [\{(n_\varepsilon\cdot\nabla)\nabla u\}^T-\Delta u\otimes n_\varepsilon] \\
      &= \eta_1+\eta_2+\eta_3
    \end{aligned}
  \end{align}
  on $O$, where
  \begin{align}
    \eta_i &:= (\nabla u)^T\tau_i\cdot[\{(n_\varepsilon\cdot\nabla)\nabla u\}^T\tau_i-(\Delta u\otimes n_\varepsilon)\tau_i], \quad i=1,2, \label{Pf_LA:Def_etai}\\
    \eta_3 &:= (\nabla u)^Tn_\varepsilon\cdot[\{(n_\varepsilon\cdot\nabla)\nabla u\}^Tn_\varepsilon-(\Delta u\otimes n_\varepsilon)n_\varepsilon]. \label{Pf_LA:Def_eta3}
  \end{align}
  In what follows, we carry out calculations on $O$.
  By \eqref{E:Gauss} and $\tau_i\cdot n_\varepsilon=0$ we have
  \begin{align} \label{Pf_LA:etai_1}
    \begin{aligned}
      (\nabla u)^T\tau_i &= (\tau_i\cdot\nabla)u = \overline{\nabla}_i^\varepsilon u+(W_\varepsilon u\cdot\tau_i)n_\varepsilon, \\
      (\Delta u\otimes n_\varepsilon)\tau_i &= (\tau_i\cdot n_\varepsilon)\Delta u = 0,
    \end{aligned}
  \end{align}
  where $\overline{\nabla}_i^\varepsilon:=\overline{\nabla}_{\tau_i}^\varepsilon$, $i=1,2$.
  For $j=1,2,3$ let $\tau_i^j$ and $n_\varepsilon^j$ be the $j$-th components of $\tau_i$ and $n_\varepsilon$.
  Then the $j$-th component of $\{(n_\varepsilon\cdot\nabla)\nabla u\}^T\tau_i$ is of the form
  \begin{align*}
    \sum_{k,l=1}^3n_\varepsilon^k(\partial_k\partial_lu_j)\tau_i^l &= \sum_{k=1}^3n_\varepsilon^k(\tau_i\cdot\nabla)(\partial_ku_j) = \sum_{k=1}^3n_\varepsilon^k(\tau_i\cdot\nabla_{\Gamma_\varepsilon})(\partial_ku_j) \\
    &= \sum_{k=1}^3\{(\tau_i\cdot\nabla_{\Gamma_\varepsilon})(n_\varepsilon^k\partial_ku_j)-(\tau_i\cdot\nabla_{\Gamma_\varepsilon}n_\varepsilon^k)\partial_ku_j\} \\
    &= (\tau_i\cdot\nabla_{\Gamma_\varepsilon})\{(n_\varepsilon\cdot\nabla)u_j\}-\{(\tau_i\cdot\nabla_{\Gamma_\varepsilon})n_\varepsilon\cdot\nabla\}u_j
  \end{align*}
  by \eqref{E:Tgrad_Bo} and $P_\varepsilon\tau_i=\tau_i$ (also note that the tangential derivatives depend only on the values of functions on $\Gamma_\varepsilon$).
  Hence
  \begin{align*}
    \{(n_\varepsilon\cdot\nabla)\nabla u\}^T\tau_i = (\tau_i\cdot\nabla_{\Gamma_\varepsilon})\{(n_\varepsilon\cdot\nabla)u\}-\{(\tau_i\cdot\nabla_{\Gamma_\varepsilon})n_\varepsilon\cdot\nabla\}u.
  \end{align*}
  By \eqref{Pf_LA:ND_Bo}, \eqref{E:Gauss}, $-\nabla_{\Gamma_\varepsilon}n_\varepsilon=W_\varepsilon=W_\varepsilon^T$, and
  \begin{align} \label{Pf_LA:Direc_TauN}
    (\tau_i\cdot\nabla_{\Gamma_\varepsilon})n_\varepsilon = (\nabla_{\Gamma_\varepsilon}n_\varepsilon)^T\tau_i = -W_\varepsilon\tau_i,
  \end{align}
  we further observe that
  \begin{multline} \label{Pf_LA:etai_2}
    \{(n_\varepsilon\cdot\nabla)\nabla u\}^T\tau_i =-\overline{\nabla}_i^\varepsilon(W_\varepsilon u)-\tilde{\gamma}_\varepsilon\overline{\nabla}_i^\varepsilon u+\overline{\nabla}_{W_\varepsilon\tau_i}^\varepsilon u-\xi_\varepsilon W_\varepsilon\tau_i \\
    +\{(-\tilde{\gamma}_\varepsilon W_\varepsilon u+\nabla_{\Gamma_\varepsilon}\xi_\varepsilon)\cdot\tau_i\}n_\varepsilon.
  \end{multline}
  Note that the first four terms on the right-hand side of \eqref{Pf_LA:etai_2} are tangential on $\Gamma_\varepsilon$.
  From \eqref{Pf_LA:Def_etai}, \eqref{Pf_LA:etai_1}, and \eqref{Pf_LA:etai_2} we deduce that
  \begin{multline*}
    \eta_i = -\left\{\overline{\nabla}_i^\varepsilon(W_\varepsilon u)+\tilde{\gamma}_\varepsilon\overline{\nabla}_i^\varepsilon u-\overline{\nabla}_{W_\varepsilon\tau_i}^\varepsilon u+\xi_\varepsilon W_\varepsilon\tau_i\right\}\cdot\overline{\nabla}_i^\varepsilon u \\
    +(W_\varepsilon u\cdot\tau_i)\{(-\tilde{\gamma}_\varepsilon W_\varepsilon u+\nabla_{\Gamma_\varepsilon}\xi_\varepsilon)\cdot\tau_i\}, \quad i=1,2.
  \end{multline*}
  Since $W_\varepsilon u$ and $\nabla_{\Gamma_\varepsilon}\xi_\varepsilon$ are tangential on $\Gamma_\varepsilon$ and $\{\tau_1,\tau_2\}$ is an orthonormal basis of the tangent plane of $\Gamma_\varepsilon$, by the above equality and \eqref{E:Wtr_Cov}--\eqref{E:Winn_Cov} we obtain
  \begin{multline} \label{Pf_LA:Sum_etai}
    \eta_1+\eta_2 = -\{\nabla_{\Gamma_\varepsilon}(W_\varepsilon u)+\tilde{\gamma}_\varepsilon\nabla_{\Gamma_\varepsilon}u-W_\varepsilon\nabla_{\Gamma_\varepsilon}u\}:(\nabla_{\Gamma_\varepsilon}u)P_\varepsilon \\
    -\xi_\varepsilon(\nabla_{\Gamma_\varepsilon}u:W_\varepsilon)+W_\varepsilon u\cdot(-\tilde{\gamma}_\varepsilon W_\varepsilon u+\nabla_{\Gamma_\varepsilon}\xi_\varepsilon).
  \end{multline}
  To calculate $\eta_3$ we see that the $j$-th component of $\{(n_\varepsilon\cdot\nabla)\nabla u\}^Tn_\varepsilon$ is of the form
  \begin{align*}
    \sum_{k,l=1}^3n_\varepsilon^k(\partial_k\partial_lu_j)n_\varepsilon^l &= \mathrm{tr}[Q_\varepsilon\nabla^2u_j] = \mathrm{tr}[\nabla^2u_j]-\mathrm{tr}[P_\varepsilon\nabla^2u_j] \\
    &= \Delta u_j-\sum_{i=1,2}P_\varepsilon(\nabla^2u_j)\tau_i\cdot\tau_i-P_\varepsilon(\nabla^2u)n_\varepsilon\cdot n_\varepsilon \\
    &= \Delta u_j-\sum_{i=1,2}\{(\tau_i\cdot\nabla)\nabla u_j\}\cdot\tau_i
  \end{align*}
  for $j=1,2,3$ by $P_\varepsilon^T=P_\varepsilon$, $P_\varepsilon\tau_i=\tau_i$, and $P_\varepsilon n_\varepsilon=0$.
  From this equality and
  \begin{align*}
    \{(\tau_i\cdot\nabla)\nabla u_j\}\cdot\tau_i &= \{(\tau_i\cdot\nabla_{\Gamma_\varepsilon})\nabla u_j\}\cdot\tau_i \\
    &= (\tau_i\cdot\nabla_{\Gamma_\varepsilon})(\nabla u_j\cdot\tau_i)-\nabla u_j\cdot(\tau_i\cdot\nabla_{\Gamma_\varepsilon})\tau_i \\
    &= (\tau_i\cdot\nabla_{\Gamma_\varepsilon})\{(\tau_i\cdot\nabla)u_j\}-\{(\tau_i\cdot\nabla_{\Gamma_\varepsilon})\tau_i\cdot\nabla\}u_j
  \end{align*}
  by \eqref{E:Tgrad_Bo} and $P_\varepsilon\tau_i=\tau_i$ we deduce that
  \begin{align*}
    \{(n_\varepsilon\cdot\nabla)\nabla u\}^Tn_\varepsilon = \Delta u-\sum_{i=1,2}[(\tau_i\cdot\nabla_{\Gamma_\varepsilon})\{(\tau_i\cdot\nabla)u\}-\{(\tau_i\cdot\nabla_{\Gamma_\varepsilon})\tau_i\cdot\nabla\}u].
  \end{align*}
  Moreover, since $(\Delta u\otimes n_\varepsilon)n_\varepsilon=(n_\varepsilon\cdot n_\varepsilon)\Delta u=\Delta u$, it follows that
  \begin{multline} \label{Pf_LA:eta3_1}
    \{(n_\varepsilon\cdot\nabla)\nabla u\}^Tn_\varepsilon-(\Delta u\otimes n_\varepsilon)n_\varepsilon \\
    = -\sum_{i=1,2}[(\tau_i\cdot\nabla_{\Gamma_\varepsilon})\{(\tau_i\cdot\nabla)u\}-\{(\tau_i\cdot\nabla_{\Gamma_\varepsilon})\tau_i\cdot\nabla\}u].
  \end{multline}
  By \eqref{Pf_LA:ND_Bo}, \eqref{Pf_LA:Direc_TauN}, and \eqref{E:Gauss} we also observe that
  \begin{align*}
    (\tau_i\cdot\nabla_{\Gamma_\varepsilon})\{(\tau_i\cdot\nabla)u\} &= (\tau_i\cdot\nabla_{\Gamma_\varepsilon})\left\{\overline{\nabla}_i^\varepsilon u+(W_\varepsilon u\cdot\tau_i)n_\varepsilon\right\} \\
    &= \overline{\nabla}_i^\varepsilon\overline{\nabla}_i^\varepsilon u-(W_\varepsilon u\cdot\tau_i)W_\varepsilon\tau_i \\
    &\qquad\qquad +\left\{W_\varepsilon \overline{\nabla}_i^\varepsilon u\cdot\tau_i+\tau_i\cdot\nabla_{\Gamma_\varepsilon}(W_\varepsilon u\cdot\tau_i)\right\}n_\varepsilon
  \end{align*}
  and
  \begin{align*}
    \{(\tau_i\cdot\nabla_{\Gamma_\varepsilon})\tau_i\cdot\nabla\}u &= \left[\left\{\overline{\nabla}_i^\varepsilon\tau_i+(W_\varepsilon\tau_i\cdot\tau_i)n_\varepsilon\right\}\cdot\nabla\right]u \\
    &= \overline{\nabla}_{\overline{\nabla}_i^\varepsilon\tau_i}^\varepsilon u-(W_\varepsilon\tau_i\cdot\tau_i)(W_\varepsilon u+\tilde{\gamma}_\varepsilon u) \\
    &\qquad\qquad +\left(W_\varepsilon u\cdot\overline{\nabla}_i^\varepsilon \tau_i+\xi_\varepsilon W_\varepsilon\tau_i\cdot\tau_i\right)n_\varepsilon.
  \end{align*}
  We substitute these expressions for \eqref{Pf_LA:eta3_1} and use \eqref{E:MC_Local}, \eqref{E:Wtr_Cov},
  \begin{align*}
    \sum_{i=1,2}\left\{\tau_i\cdot\nabla_{\Gamma_\varepsilon}(W_\varepsilon u\cdot\tau_i)- W_\varepsilon u\cdot\overline{\nabla}_i^\varepsilon\tau_i\right\} = \sum_{i=1,2}\overline{\nabla}_i^\varepsilon(W_\varepsilon u)\cdot\tau_i = \mathrm{div}_{\Gamma_\varepsilon}(W_\varepsilon u)
  \end{align*}
  by \eqref{E:RiCo_Met} and \eqref{E:Sdiv_Cov}, and
  \begin{align*}
    \sum_{i=1,2}(W_\varepsilon u\cdot\tau_i)W_\varepsilon\tau_i = W_\varepsilon\sum_{i=1,2}(W_\varepsilon u\cdot\tau_i)\tau_i = W_\varepsilon^2 u
  \end{align*}
  by the facts that $W_\varepsilon u$ is tangential on $\Gamma_\varepsilon$ and that $\{\tau_1,\tau_2\}$ is an orthonormal basis of the tangent plane of $\Gamma_\varepsilon$.
  Then we have
  \begin{multline} \label{Pf_LA:eta3_2}
    \{(n_\varepsilon\cdot\nabla)\nabla u\}^Tn_\varepsilon-(\Delta u\otimes n_\varepsilon)n_\varepsilon \\
    = -\sum_{i=1,2}\left(\overline{\nabla}_i^\varepsilon\overline{\nabla}_i^\varepsilon u-\overline{\nabla}_{\overline{\nabla}_i^\varepsilon\tau_i}^\varepsilon u\right)+W_\varepsilon^2u-H_\varepsilon W_\varepsilon u-\tilde{\gamma}_\varepsilon H_\varepsilon u \\
    -\{\nabla_{\Gamma_\varepsilon}u:W_\varepsilon+\mathrm{div}_{\Gamma_\varepsilon}(W_\varepsilon u)-\xi_\varepsilon H_\varepsilon\}n_\varepsilon.
  \end{multline}
  Hence by \eqref{Pf_LA:ND_Bo}, \eqref{Pf_LA:Def_eta3}, and \eqref{Pf_LA:eta3_2} we get
  \begin{multline} \label{Pf_LA:eta3_Inn}
    \eta_3 = \sum_{i=1,2}\left(\overline{\nabla}_i^\varepsilon\overline{\nabla}_i^\varepsilon u-\overline{\nabla}_{\overline{\nabla}_i^\varepsilon\tau_i}^\varepsilon u\right)\cdot(W_\varepsilon u+\tilde{\gamma}_\varepsilon u) \\
    -(W_\varepsilon^2u-H_\varepsilon W_\varepsilon u-\tilde{\gamma}_\varepsilon H_\varepsilon u)\cdot(W_\varepsilon u+\tilde{\gamma}_\varepsilon u) \\
    -\xi_\varepsilon\{\nabla_{\Gamma_\varepsilon}u:W_\varepsilon+\mathrm{div}_{\Gamma_\varepsilon}(W_\varepsilon u)-\xi_\varepsilon H_\varepsilon\}.
  \end{multline}
  Now we observe by \eqref{E:Tgrad_Bo} and direct calculations that
  \begin{align*}
    \nabla_{\Gamma_\varepsilon}u:(\nabla_{\Gamma_\varepsilon}u)P_\varepsilon = P_\varepsilon(\nabla u):P_\varepsilon(\nabla u)P_\varepsilon = \nabla u:P_\varepsilon^TP_\varepsilon(\nabla u)P_\varepsilon.
  \end{align*}
  Since $P_\varepsilon^T=P_\varepsilon^2=P_\varepsilon$, the above equality implies that
  \begin{align} \label{Pf_LA:Mat_Inn_1}
    \nabla_{\Gamma_\varepsilon}u:(\nabla_{\Gamma_\varepsilon}u)P_\varepsilon = \nabla u:P_\varepsilon(\nabla u)P_\varepsilon.
  \end{align}
  By the same calculations with \eqref{E:Tgrad_Bo} and \eqref{E:PW_Bo} we have
  \begin{align} \label{Pf_LA:Mat_Inn_2}
    \begin{aligned}
      \nabla_{\Gamma_\varepsilon}(W_\varepsilon u):(\nabla_{\Gamma_\varepsilon}u)P_\varepsilon &= \{\nabla_{\Gamma_\varepsilon}W_\varepsilon\cdot u+(\nabla_{\Gamma_\varepsilon}u)W_\varepsilon\}:(\nabla_{\Gamma_\varepsilon}u)P_\varepsilon \\
      &= \{\nabla_{\Gamma_\varepsilon}W_\varepsilon\cdot u+(\nabla u)W_\varepsilon\}:P_\varepsilon(\nabla u)P_\varepsilon,
    \end{aligned}
  \end{align}
  where the matrix $\nabla_{\Gamma_\varepsilon}W_\varepsilon\cdot u$ is given by \eqref{Pf_LA:DW_u_Bo}, and
  \begin{align} \label{Pf_LA:Mat_Inn_3}
    W_\varepsilon(\nabla_{\Gamma_\varepsilon}u):(\nabla_{\Gamma_\varepsilon}u)P_\varepsilon = W_\varepsilon(\nabla u):(\nabla u)P_\varepsilon, \quad \nabla_{\Gamma_\varepsilon}u:W_\varepsilon = \nabla u:W_\varepsilon.
  \end{align}
  Also, it is easy to see that
  \begin{align} \label{Pf_LA:Div_Xi}
    \begin{aligned}
      W_\varepsilon u\cdot\nabla_{\Gamma_\varepsilon}\xi_\varepsilon &= \mathrm{div}_{\Gamma_\varepsilon}(\xi_\varepsilon W_\varepsilon u)-\xi_\varepsilon\mathrm{div}_{\Gamma_\varepsilon}(W_\varepsilon u), \\
      \mathrm{div}_{\Gamma_\varepsilon}(W_\varepsilon u) &= u\cdot\mathrm{div}_{\Gamma_\varepsilon}W_\varepsilon+\nabla_{\Gamma_\varepsilon}u: W_\varepsilon =  u\cdot\mathrm{div}_{\Gamma_\varepsilon}W_\varepsilon+\nabla u: W_\varepsilon,
    \end{aligned}
  \end{align}
  where the vector field $\mathrm{div}_{\Gamma_\varepsilon}W_\varepsilon$ is given by \eqref{Pf_LA:Div_W}.
  From \eqref{Pf_LA:Decom_Intg}, \eqref{Pf_LA:Sum_etai}, \eqref{Pf_LA:eta3_Inn}--\eqref{Pf_LA:Div_Xi}, $W_\varepsilon^T=W_\varepsilon$, and $\xi_\varepsilon=\nabla u:Q_\varepsilon$ we deduce that
  \begin{multline} \label{Pf_LA:Int_01}
    \int_{\Gamma_\varepsilon}\nabla u:\{(n_\varepsilon\cdot\nabla)\nabla u-n_\varepsilon\otimes\Delta u\}\,d\mathcal{H}^2 \\
    = \sum_{i=1,2}\int_{\Gamma_\varepsilon}\left(\overline{\nabla}_i^\varepsilon\overline{\nabla}_i^\varepsilon u-\overline{\nabla}_{\overline{\nabla}_i^\varepsilon\tau_i}^\varepsilon u\right)\cdot(W_\varepsilon u+\tilde{\gamma}_\varepsilon u)\,d\mathcal{H}^2\\
    +\int_{\Gamma_\varepsilon}\left(\frac{1}{2}\varphi_1+\sum_{k=2}^4\varphi_k\right)d\mathcal{H}^2+\int_{\Gamma_\varepsilon}\mathrm{div}_{\Gamma_\varepsilon}(\xi_\varepsilon W_\varepsilon u)\,d\mathcal{H}^2,
  \end{multline}
  where $\varphi_1$, $\dots$, $\varphi_4$ are given by \eqref{Pf_LA:Def_Phi}.
  Moreover, we apply \eqref{E:IbP_Cov} to the first term on the right-hand side and then use \eqref{E:InnP_Cov}, \eqref{Pf_LA:Mat_Inn_1}, and \eqref{Pf_LA:Mat_Inn_2} to get
  \begin{align} \label{Pf_LA:Int_02}
    \sum_{i=1,2}\int_{\Gamma_\varepsilon}\left(\overline{\nabla}_i^\varepsilon\overline{\nabla}_i^\varepsilon u-\overline{\nabla}_{\overline{\nabla}_i^\varepsilon\tau_i}^\varepsilon u\right)\cdot(W_\varepsilon u+\tilde{\gamma}_\varepsilon u)\,d\mathcal{H}^2 = \frac{1}{2}\int_{\Gamma_\varepsilon}\varphi_1\,d\mathcal{H}^2.
  \end{align}
  Also, since $u\in H^3(\Omega_\varepsilon)^3$ and $Q_\varepsilon,W_\varepsilon\in C^2(\Gamma_\varepsilon)^{3\times 3}$ by the $C^4$-regularity of $\Gamma_\varepsilon$,
  \begin{align*}
    \xi_\varepsilon = \nabla u:Q_\varepsilon \in H^1(\Gamma_\varepsilon), \quad W_\varepsilon u \in H^2(\Gamma_\varepsilon,T\Gamma_\varepsilon).
  \end{align*}
  Thus $\xi_\varepsilon W_\varepsilon u\in W^{1,1}(\Gamma_\varepsilon,T\Gamma_\varepsilon)$ and we can apply \eqref{E:IbP_WDivG_T} to $\xi_\varepsilon W_\varepsilon u$ to deduce that
  \begin{align} \label{Pf_LA:Int_03}
    \int_{\Gamma_\varepsilon}\mathrm{div}_{\Gamma_\varepsilon}(\xi_\varepsilon W_\varepsilon u)\,d\mathcal{H}^2 = 0.
  \end{align}
  Hence we obtain \eqref{Pf_LA:Int_Sum} by \eqref{Pf_LA:Int_01}--\eqref{Pf_LA:Int_03}.

  The second step for \eqref{Pf_LA:BD} is to show that
  \begin{align}
    \left|\int_{\Gamma_\varepsilon}\varphi_k\,d\mathcal{H}^2\right| &\leq c\left(\|u\|_{H^1(\Omega_\varepsilon)}^2+\|u\|_{H^1(\Omega_\varepsilon)}\|\nabla^2u\|_{L^2(\Omega_\varepsilon)}\right), \quad k=1,2, \label{Pf_LA:Est_Phi12}\\
    \left|\int_{\Gamma_\varepsilon}\varphi_k\,d\mathcal{H}^2\right| &\leq c\|u\|_{H^1(\Omega_\varepsilon)}^2, \quad k=3,4 \label{Pf_LA:Est_Phi34}
  \end{align}
  with a constant $c>0$ independent of $\varepsilon$.
  The estimate \eqref{Pf_LA:Est_Phi34} for $k=4$ is an easy consequence of \eqref{E:Fric_Upper}, \eqref{E:Poin_Bo}, and the uniform boundedness of $W_\varepsilon$ and $H_\varepsilon$ on $\Gamma_\varepsilon$:
  \begin{align*}
    \left|\int_{\Gamma_\varepsilon}\varphi_4\,d\mathcal{H}^2\right| \leq c\varepsilon\|u\|_{L^2(\Gamma_\varepsilon)}^2 \leq c\varepsilon(\varepsilon^{-1}\|u\|_{L^2(\Omega_\varepsilon)}^2+\varepsilon\|\partial_nu\|_{L^2(\Omega_\varepsilon)}^2) \leq c\|u\|_{H^1(\Omega_\varepsilon)}^2.
  \end{align*}
  Let us prove \eqref{Pf_LA:Est_Phi12} for $k=1$.
  As in the proof of Lemma \ref{L:Int_UgUn}, we interpolate the integrals over $\Gamma_\varepsilon^0$ and $\Gamma_\varepsilon^1$ to produce an integral over $\Omega_\varepsilon$ with a good estimate.
  In what follows, we use the notations \eqref{E:Pull_Dom} and \eqref{E:Pull_Bo} and sometimes suppress the arguments $y$ and $r$.
  For $y\in\Gamma$, $r\in[\varepsilon g_0(y),\varepsilon g_1(y)]$, and $j,k,l=1,2,3$ we set
  \begin{align*}
    F(y,r) &:= \frac{1}{\varepsilon g(y)}\bigl\{\bigl(r-\varepsilon g_0(y)\bigr)W_{\varepsilon,1}^\sharp(y)-\bigl(\varepsilon g_1(y)-r\bigr)W_{\varepsilon,0}^\sharp(y)\bigr\}, \\
    G_{jk}^l(y,r) &:= \frac{1}{\varepsilon g(y)}\bigl\{\bigl(r-\varepsilon g_0(y)\bigr)\bigl(\underline{D}_j^\varepsilon[W_\varepsilon]_{kl}\bigr)_1^\sharp(y)-\bigl(\varepsilon g_1(y)-r\bigr)\bigl(\underline{D}_j^\varepsilon[W_\varepsilon]_{kl}\bigr)_0^\sharp(y)\bigr\}, \\
    \tilde{\gamma}(y,r) &:= \frac{1}{\varepsilon g(y)}\bigl\{\bigl(r-\varepsilon g_0(y)\bigr)\tilde{\gamma}_\varepsilon^1-\bigl(\varepsilon g_1(y)-r\bigr)\tilde{\gamma}_\varepsilon^0\bigr\},
  \end{align*}
  where $\tilde{\gamma}_\varepsilon^i:=\gamma_\varepsilon^i/\nu$, $i=0,1$.
  Then we have
  \begin{multline} \label{Pf_LA:Aux_Eq_1}
    [\nabla_{\Gamma_\varepsilon}W_\varepsilon\cdot u+(\nabla u)W_\varepsilon+\tilde{\gamma}_\varepsilon\nabla u]_i^\sharp(y) \\
    = (-1)^{i+1}[G\cdot u^\sharp+(\nabla u)^\sharp F+\tilde{\gamma}(\nabla u)^\sharp](y,\varepsilon g_i(y)), \quad y\in\Gamma,\,i=0,1,
  \end{multline}
  where $G\cdot u^\sharp$ denotes a $3\times 3$ matrix whose $(j,k)$-entry is given by
  \begin{align*}
    [G\cdot u^\sharp]_{jk} := \sum_{l=1}^3G_{jk}^lu_l^\sharp, \quad j,k=1,2,3.
  \end{align*}
  Moreover, by \eqref{E:G_Inf}, \eqref{E:Fric_Upper}, \eqref{E:Diff_WH_IO} for $W_\varepsilon$ and $\underline{D}_j^\varepsilon W_\varepsilon$ with $j=1,2,3$,
  \begin{align} \label{Pf_LA:Est_Dist}
    |r-\varepsilon g_i(y)| \leq \varepsilon g(y) \leq c\varepsilon, \quad y\in\Gamma,\,r\in[\varepsilon g_0(y),\varepsilon g_1(y)],\,i=0,1,
  \end{align}
  and the uniform boundedness in $\varepsilon$ of $W_\varepsilon$ and $\underline{D}_j^\varepsilon W_\varepsilon$ on $\Gamma_\varepsilon$ we have
  \begin{align} \label{Pf_LA:Aux_Est_1}
    |\eta(y,r)|+\left|\frac{\partial \eta}{\partial r}(y,r)\right| \leq c, \quad \eta = F,G_{jk}^l,\tilde{\gamma}, \quad y\in\Gamma,\,r\in[\varepsilon g_0(y),\varepsilon g_1(y)]
  \end{align}
  with a constant $c>0$ independent of $\varepsilon$.
  We also define
  \begin{align*}
    R(y,r) := \frac{1}{\varepsilon g(y)}\bigl\{\bigl(r-\varepsilon g_0(y)\bigr)P_{\varepsilon,1}^\sharp(y)+\bigl(\varepsilon g_1(y)-r\bigr)P_{\varepsilon,0}^\sharp(y)\bigr\}
  \end{align*}
  for $y\in\Gamma$ and $r\in[\varepsilon g_0(y),\varepsilon g_1(y)]$, and
  \begin{align*}
    S_i(y) &:= \sqrt{1+\varepsilon^2|\tau_\varepsilon^i(y)|^2}\,P_{\varepsilon,i}^\sharp(y), \quad i=0,1, \\
    S(y,r) &:= \frac{1}{\varepsilon g(y)}\bigl\{\bigl(r-\varepsilon g_0(y)\bigr)S_1(y)+\bigl(\varepsilon g_1(y)-r\bigr)S_0(y)\bigr\},
  \end{align*}
  where $\tau_\varepsilon^0$ and $\tau_\varepsilon^1$ are given by \eqref{E:Def_NB_Aux}.
  Then
  \begin{align} \label{Pf_LA:Aux_Eq_2}
    \sqrt{1+\varepsilon^2|\tau_\varepsilon^i(y)|^2}\,[P_\varepsilon(\nabla u)P_\varepsilon]_i^\sharp(y) = [R(\nabla u)^\sharp S](y,\varepsilon g_i(y)), \quad y\in\Gamma,\,i=0,1.
  \end{align}
  Moreover, from \eqref{E:Diff_PQ_IO} for $P_\varepsilon$, \eqref{Pf_IU:Sqrt}, and \eqref{Pf_LA:Est_Dist} we deduce that
  \begin{align} \label{Pf_LA:Aux_Est_2}
    |\eta(y,r)|+\left|\frac{\partial\eta}{\partial r}(y,r)\right| \leq c, \quad \eta = R,S, \quad y\in\Gamma,\, r\in[\varepsilon g_0(y),\varepsilon g_1(y)].
  \end{align}
  Now we define a function $\Phi_1=\Phi_1(y,r)$ for $y\in\Gamma$ and $r\in[\varepsilon g_0(y),\varepsilon g_1(y)]$ by
  \begin{align*}
    \Phi_1(y,r) := -2[\{G\cdot u^\sharp+(\nabla u)^\sharp F+\tilde{\gamma}(\nabla u)^\sharp\}:R(\nabla u)^\sharp S](y,r)J(y,r),
  \end{align*}
  where $J$ is given by \eqref{E:Def_Jac}.
  Then by \eqref{E:CoV_Surf}, \eqref{Pf_LA:Aux_Eq_1}, and \eqref{Pf_LA:Aux_Eq_2} we have
  \begin{align*}
    \int_{\Gamma_\varepsilon}\varphi_1(x)\,d\mathcal{H}^2(x) &= \sum_{i=0,1}\int_{\Gamma_\varepsilon^i}\varphi_1(x)\,d\mathcal{H}^2(x) \\
    &= \int_\Gamma\{\Phi_1(y,\varepsilon g_1(y))-\Phi_1(y,\varepsilon g_0(y))\}\,d\mathcal{H}^2(y) \\
    &= \int_\Gamma\int_{\varepsilon g_0(y)}^{\varepsilon g_1(y)}\frac{\partial\Phi_1}{\partial r}(y,r)\,dr\,d\mathcal{H}^2(y).
  \end{align*}
  Furthermore, the inequalities \eqref{E:Jac_Bound_01}, \eqref{Pf_LA:Aux_Est_1}, and \eqref{Pf_LA:Aux_Est_2} imply that
  \begin{align*}
    \left|\frac{\partial\Phi_1}{\partial r}\right| \leq c\{|u^\sharp|^2+|(\nabla u)^\sharp|^2+(|u^\sharp|+|(\nabla u)^\sharp|)|(\nabla^2u)^\sharp|\}
  \end{align*}
  with some constant $c>0$ independent of $\varepsilon$ (here we also used Young's inequality).
  From the above relations, \eqref{E:CoV_Equiv}, and H\"{o}lder's inequality it follows that
  \begin{align*}
    \left|\int_{\Gamma_\varepsilon}\varphi_1\,d\mathcal{H}^2\right| &\leq c\int_\Gamma\int_{\varepsilon g_0}^{\varepsilon g_1}\{|u^\sharp|^2+|(\nabla u)^\sharp|^2+(|u^\sharp|+|(\nabla u)^\sharp|)|(\nabla^2u)^\sharp|\}\,dr\,d\mathcal{H}^2 \\
    &\leq c\left(\|u\|_{H^1(\Omega_\varepsilon)}^2+\|u\|_{H^1(\Omega_\varepsilon)}\|\nabla^2u\|_{L^2(\Omega_\varepsilon)}\right).
  \end{align*}
  Thus the inequality \eqref{Pf_LA:Est_Phi12} for $k=1$ is valid.
  By the same arguments we can prove \eqref{Pf_LA:Est_Phi12} for $k=2$ and \eqref{Pf_LA:Est_Phi34} for $k=3$.

  Finally, we obtain \eqref{Pf_LA:BD} by \eqref{Pf_LA:Int_Sum}, \eqref{Pf_LA:Est_Phi12}, and \eqref{Pf_LA:Est_Phi34}, and we apply \eqref{Pf_LA:BD} to \eqref{Pf_LA:IbP} and then use Young's inequality to get
  \begin{align*}
    \|\nabla^2u\|_{L^2(\Omega_\varepsilon)}^2 &\leq \|\Delta u\|_{L^2(\Omega_\varepsilon)}^2+c\left(\|u\|_{H^1(\Omega_\varepsilon)}^2+\|u\|_{H^1(\Omega_\varepsilon)}\|\nabla^2u\|_{L^2(\Omega_\varepsilon)}\right) \\
    &\leq \|\Delta u\|_{L^2(\Omega_\varepsilon)}^2+\frac{1}{2}\|\nabla^2u\|_{L^2(\Omega_\varepsilon)}^2+c\|u\|_{H^1(\Omega_\varepsilon)}^2,
  \end{align*}
  which yields \eqref{Pf_LA:Goal}.
  Hence the inequality \eqref{E:Lap_Apri} is valid.
\end{proof}

\section{Proofs of the main results} \label{S:Pf_Main}
In this section we establish Theorems \ref{T:Uni_aeps}, \ref{T:Comp_Sto_Lap}, and \ref{T:Stokes_H2}.
First we give an integration by parts formula related to the slip boundary conditions \eqref{E:Bo_Slip}.

\begin{lemma} \label{L:IbP_St}
  For $u_1\in H^2(\Omega_\varepsilon)^3$ and $u_2\in H^1(\Omega_\varepsilon)^3$ we have
  \begin{multline} \label{E:IbP_St}
    \int_{\Omega_\varepsilon}\{\Delta u_1+\nabla(\mathrm{div}\,u_1)\}\cdot u_2\,dx \\
    = -2\int_{\Omega_\varepsilon}D(u_1):D(u_2)\,dx+2\int_{\Gamma_\varepsilon}[D(u_1)n_\varepsilon]\cdot u_2\,d\mathcal{H}^2.
  \end{multline}
  In particular, if $u_1$ satisfies
  \begin{align*}
    \mathrm{div}\,u_1 = 0 \quad\text{in}\quad \Omega_\varepsilon, \quad u_1\cdot n_\varepsilon = 0, \quad 2\nu P_\varepsilon D(u_1)n_\varepsilon+\gamma_\varepsilon u_1 = 0 \quad\text{on}\quad \Gamma_\varepsilon
  \end{align*}
  and $u_2$ satisfies $u_2\cdot n_\varepsilon=0$ on $\Gamma_\varepsilon$, then
  \begin{align*}
    \nu\int_{\Omega_\varepsilon}\Delta u_1\cdot u_2\,dx = -2\nu\int_{\Omega_\varepsilon}D(u_1):D(u_2)\,dx-\sum_{i=0,1}\gamma_\varepsilon^i\int_{\Gamma_\varepsilon^i}u_1\cdot u_2\,d\mathcal{H}^2.
  \end{align*}
\end{lemma}

\begin{proof}
  Since $\Delta u_1+\nabla(\mathrm{div}\,u_1)=2\mathrm{div}[D(u_1)]$ in $\Omega_\varepsilon$ for $u_1\in H^2(\Omega_\varepsilon)^3$,
  \begin{align} \label{Pf_ISt:dU1U2}
    \int_{\Omega_\varepsilon}\{\Delta u_1+\nabla(\mathrm{div}\,u_1)\}\cdot u_2\,dx &= 2\int_{\Omega_\varepsilon}\mathrm{div}[D(u_1)]\cdot u_2\,dx.
  \end{align}
  Moreover, for $A=(A_{ij})_{i,j}\in H^1(\Omega_\varepsilon)^{3\times 3}$ and $u=(u^1,u^2,u^3)^T\in H^1(\Omega_\varepsilon)^3$, we have
  \begin{align*}
    \int_{\Omega_\varepsilon}\mathrm{div}\,A\cdot u\,dx &= \sum_{i,j=1}^3\int_{\Omega_\varepsilon}(\partial_iA_{ij})u^j\,dx = \sum_{i,j=1}^3\left(\int_{\Gamma_\varepsilon}n_\varepsilon^iA_{ij}u^j\,d\mathcal{H}^2-\int_{\Omega_\varepsilon}A_{ij}\partial_iu^j\,dx\right) \\
    &= \int_{\Gamma_\varepsilon}(A^Tn_\varepsilon)\cdot u\,d\mathcal{H}^2-\int_{\Omega_\varepsilon}A:\nabla u\,dx
  \end{align*}
  by integration by parts, where $n_\varepsilon^i$, $i=1,2,3$ is the $i$-th component of $n_\varepsilon$.
  Applying this formula with $A=D(u_1)$ and $u=u_2$ to \eqref{Pf_ISt:dU1U2} and using
  \begin{align*}
    D(u_1)^T = D(u_1), \quad D(u_1):\nabla u_2 = D(u_1):(\nabla u_2)^T = D(u_1):D(u_2)
  \end{align*}
  we obtain \eqref{E:IbP_St}.
\end{proof}

By Lemma \ref{L:IbP_St} we see that the bilinear form for the Stokes problem \eqref{E:Stokes_CTD} is given by \eqref{E:Def_aeps}, i.e.
\begin{align*}
  a_\varepsilon(u_1,u_2) = 2\nu\bigl(D(u_1),D(u_2)\bigr)_{L^2(\Omega_\varepsilon)}+\sum_{i=0,1}\gamma_\varepsilon^i(u_1,u_2)_{L^2(\Gamma_\varepsilon^i)}, \quad u_1,u_2\in H^1(\Omega_\varepsilon)^3.
\end{align*}
Now we impose Assumptions \ref{Assump_1} and \ref{Assump_2} and define the function spaces $\mathcal{H}_\varepsilon$ and $\mathcal{V}_\varepsilon$ by \eqref{E:Def_Heps}.
Let us show that $a_\varepsilon$ is uniformly bounded and coercive on $\mathcal{V}_\varepsilon$.

\begin{proof}[Proof of Theorem \ref{T:Uni_aeps}]
  Let $u\in \mathcal{V}_\varepsilon$.
  By \eqref{E:Fric_Upper} in Assumption \ref{Assump_1} and \eqref{E:Poin_Bo},
  \begin{align*}
    \gamma_\varepsilon^i\|u\|_{L^2(\Gamma_\varepsilon^i)}^2 \leq c\varepsilon\left(\varepsilon^{-1}\|u\|_{L^2(\Omega_\varepsilon)}^2+\varepsilon\|\partial_nu\|_{L^2(\Omega_\varepsilon)}^2\right) \leq c\|u\|_{H^1(\Omega_\varepsilon)}^2
  \end{align*}
  for $i=0,1$.
  Combining this inequality with $\|D(u)\|_{L^2(\Omega_\varepsilon)}\leq \|\nabla u\|_{L^2(\Omega_\varepsilon)}$ we obtain the right-hand inequality of \eqref{E:Uni_aeps}.

  Let us prove the left-hand inequality of \eqref{E:Uni_aeps}.
  First we suppose that the condition (A1) of Assumption \ref{Assump_2} is satisfied.
  Without loss of generality, we may assume $\gamma_\varepsilon^0\geq c\varepsilon$ for all $\varepsilon\in(0,1]$.
  For $u\in\mathcal{V}_\varepsilon$ we use \eqref{E:Poin_Dom} with $i=0$ to get
  \begin{align*}
    \|u\|_{L^2(\Omega_\varepsilon)}^2 \leq c\left(\varepsilon\|u\|_{L^2(\Gamma_\varepsilon^0)}^2+\varepsilon^2\|\nabla u\|_{L^2(\Omega_\varepsilon)}^2\right).
  \end{align*}
  Moreover, by $\gamma_\varepsilon^0\geq c\varepsilon$ and \eqref{E:Korn_Grad} (note that $u\in\mathcal{V}_\varepsilon$ satisfies \eqref{E:Bo_Imp}),
  \begin{align*}
    \|u\|_{L^2(\Omega_\varepsilon)}^2 &\leq c\left(\gamma_\varepsilon^0\|u\|_{L^2(\Gamma_\varepsilon^0)}^2+\varepsilon^2\|D(u)\|_{L^2(\Omega_\varepsilon)}^2+\varepsilon^2\|u\|_{L^2(\Omega_\varepsilon)}^2\right) \\
    &\leq c_1a_\varepsilon(u,u)+c_2\varepsilon^2\|u\|_{L^2(\Omega_\varepsilon)}^2
  \end{align*}
  with positive constants $c_1$ and $c_2$ independent of $\varepsilon$.
  We set $\varepsilon_1:=1/\sqrt{2c_2}$ and take $\varepsilon\in(0,\varepsilon_1]$ in the above inequality to get
  \begin{align*}
    \|u\|_{L^2(\Omega_\varepsilon)}^2 \leq 2c_1a_\varepsilon(u,u).
  \end{align*}
  From this inequality and \eqref{E:Korn_Grad} we also deduce that
  \begin{align*}
    \|\nabla u\|_{L^2(\Omega_\varepsilon)}^2 \leq ca_\varepsilon(u,u).
  \end{align*}
  These two inequalities imply the left-hand inequality of \eqref{E:Uni_aeps}.

  Next we suppose that the condition (A2) or (A3) of Assumption \ref{Assump_2} is satisfied.
  Then $u\in\mathcal{V}_\varepsilon$ satisfies \eqref{E:Bo_Imp} and \eqref{E:Korn_Orth} (resp. \eqref{E:KoRg_Orth}) with $\beta=0$ under the condition (A2) (resp. (A3)).
  Hence Lemmas \ref{L:Korn_H1} and \ref{L:Korn_Rg} imply that there exist $\varepsilon_{K,0}\in(0,1]$ and $c_{K,0}>0$ such that
  \begin{align*}
    \|u\|_{H^1(\Omega_\varepsilon)}^2 \leq c_{K,0}\|D(u)\|_{L^2(\Omega_\varepsilon)}^2 \leq c_{K,0}a_\varepsilon(u,u)
  \end{align*}
  for all $\varepsilon\in(0,\varepsilon_{K,0}]$ and $u\in\mathcal{V}_\varepsilon$, i.e. the left-hand inequality of \eqref{E:Uni_aeps} holds.

  Therefore, we conclude that the theorem is valid with $\varepsilon_0:=\min\{\varepsilon_1,\varepsilon_{K,0}\}$.
\end{proof}

As in Section \ref{S:Main} we fix the constant $\varepsilon_0$ given in Theorem \ref{T:Uni_aeps} and denote by $A_\varepsilon$ the Stokes operator for $\Omega_\varepsilon$ under the slip boundary conditions for $\varepsilon\in(0,\varepsilon_0]$.

Next we derive the uniform difference estimate \eqref{E:Comp_Sto_Lap} for $A_\varepsilon$ and $-\nu\Delta$.
For this purpose, we give an integration by parts formula for the curl of a vector field on $\Omega_\varepsilon$.
Let $n_\varepsilon^0$ and $n_\varepsilon^1$ be the vector fields on $\Gamma$ given by \eqref{E:Def_NB} and
\begin{align} \label{E:Def_Wieps}
  W_\varepsilon^i(x) := -\{I_3-\bar{n}_\varepsilon^i(x)\otimes\bar{n}_\varepsilon^i(x)\}\nabla\bar{n}_\varepsilon^i(x), \quad x\in N,\,i=0,1.
\end{align}
Here $\bar{n}_\varepsilon^i=n_\varepsilon^i\circ\pi$, $i=0,1$ is the constant extension of $n_\varepsilon^i$.
For $x\in N$ we set
\begin{align} \label{E:Def_Tilde}
  \begin{aligned}
    \tilde{n}_1(x) &:= \frac{1}{\varepsilon\bar{g}(x)}\bigl\{\bigl(d(x)-\varepsilon\bar{g}_0(x)\bigr)\bar{n}_\varepsilon^1(x)-\bigl(\varepsilon\bar{g}_1(x)-d(x)\bigr)\bar{n}_\varepsilon^0(x)\bigr\}, \\
    \tilde{n}_2(x) &:= \frac{1}{\varepsilon\bar{g}(x)}\left\{\bigl(d(x)-\varepsilon\bar{g}_0(x)\bigr)\frac{\gamma_\varepsilon^1}{\nu}\bar{n}_\varepsilon^1(x)+\bigl(\varepsilon\bar{g}_1(x)-d(x)\bigr)\frac{\gamma_\varepsilon^0}{\nu}\bar{n}_\varepsilon^0(x)\right\}, \\
    \widetilde{W}(x) &:= \frac{1}{\varepsilon\bar{g}(x)}\bigl\{\bigl(d(x)-\varepsilon\bar{g}_0(x)\bigr)W_\varepsilon^1(x)-\bigl(\varepsilon\bar{g}_1(x)-d(x)\bigr)W_\varepsilon^0(x)\bigr\}.
  \end{aligned}
\end{align}
From these definitions and Lemma \ref{L:Nor_Bo} it follows that
\begin{align} \label{E:Tilde_Bo}
  \tilde{n}_1 = (-1)^{i+1}n_\varepsilon, \quad \tilde{n}_2 = \frac{\gamma_\varepsilon}{\nu}n_\varepsilon, \quad \widetilde{W} = (-1)^{i+1}W_\varepsilon \quad\text{on}\quad \Gamma_\varepsilon^i,\, i=0,1.
\end{align}
For a vector field $u\colon\Omega_\varepsilon\to\mathbb{R}^3$ we define $G(u)\colon\Omega_\varepsilon\to\mathbb{R}^3$ by
\begin{align} \label{E:Def_Gu}
  G(u) := G_1(u)+G_2(u), \quad G_1(u) := 2\tilde{n}_1\times \widetilde{W}u, \quad G_2(u) := \tilde{n}_2\times u.
\end{align}

\begin{lemma} \label{L:G_Bound}
  Suppose that the inequalities \eqref{E:Fric_Upper} are valid.
  Then there exists a constant $c>0$ independent of $\varepsilon$ such that
  \begin{align} \label{E:G_Bound}
    |G(u)| \leq c|u|, \quad |\nabla G(u)| \leq c(|u|+|\nabla u|) \quad\text{in}\quad \Omega_\varepsilon
  \end{align}
  for all $u\in C^1(\Omega_\varepsilon)^3$, where $G(u)$ is the vector field on $\Omega_\varepsilon$ given by \eqref{E:Def_Gu}.
\end{lemma}

Lemma \ref{L:G_Bound} is proved just by direct calculations and the application of the results given in Section \ref{S:Pre}.
We give its proof in Appendix \ref{S:Ap_Proof}.

\begin{lemma} \label{L:IbP_Curl}
  The integration by parts formula
  \begin{multline} \label{E:IbP_Curl}
    \int_{\Omega_\varepsilon}\mathrm{curl}\,\mathrm{curl}\,u\cdot\Phi\,dx \\
    = -\int_{\Omega_\varepsilon}\mathrm{curl}\,G(u)\cdot\Phi\,dx+\int_{\Omega_\varepsilon}\{\mathrm{curl}\,u+G(u)\}\cdot\mathrm{curl}\,\Phi\,dx
  \end{multline}
  holds for all $u\in H^2(\Omega_\varepsilon)^3$ satisfying \eqref{E:Bo_Slip} and $\Phi\in L^2(\Omega_\varepsilon)^3$ with $\mathrm{curl}\,\Phi\in L^2(\Omega_\varepsilon)^3$, where $G(u)$ is the vector field on $\Omega_\varepsilon$ given by \eqref{E:Def_Gu}.
\end{lemma}

The proof of \eqref{E:IbP_Curl} is the same as in the case of a flat thin domain (see the proofs of \cite{Ho08}*{Lemma 2.3} and \cite{Ho10}*{Lemma 5.2}).
Here we give it for the completeness.

\begin{proof}
  By standard cut-off, dilatation, and mollification arguments, we can show as in the proof of \cite{Te79}*{Chapter 1, Theorem 1.1} that for $\Phi\in L^2(\Omega_\varepsilon)^3$ with $\mathrm{curl}\,\Phi\in L^2(\Omega_\varepsilon)^3$ there exists a sequence $\{\Phi_k\}_{k=1}^\infty$ in $C^\infty(\overline{\Omega}_\varepsilon)^3$ such that
  \begin{align*}
    \lim_{k\to\infty}\|\Phi-\Phi_k\|_{L^2(\Omega_\varepsilon)} = \lim_{k\to\infty}\|\mathrm{curl}\,\Phi-\mathrm{curl}\,\Phi_k\|_{L^2(\Omega_\varepsilon)} = 0.
  \end{align*}
  Thus, by a density argument, it is sufficient to prove \eqref{E:IbP_Curl} for all $\Phi\in C^\infty(\overline{\Omega}_\varepsilon)^3$.

  Let $u\in H^2(\Omega_\varepsilon)^3$ satisfy \eqref{E:Bo_Slip} and $\Phi\in C^\infty(\overline{\Omega}_\varepsilon)^3$.
  Then
  \begin{align} \label{Pf_IC:IBP}
    \int_{\Omega_\varepsilon}\mathrm{curl}\,\mathrm{curl}\,u\cdot\Phi\,dx = \int_{\Gamma_\varepsilon}(n_\varepsilon\times\mathrm{curl}\,u)\cdot\Phi\,d\mathcal{H}^2+\int_{\Omega_\varepsilon}\mathrm{curl}\,u\cdot\mathrm{curl}\,\Phi\,dx
  \end{align}
  by integration by parts.
  Since $u$ satisfies \eqref{E:Bo_Slip},
  \begin{align*}
    n_\varepsilon\times\mathrm{curl}\,u &= -n_\varepsilon\times\left\{n_\varepsilon\times\left(2W_\varepsilon u+\frac{\gamma_\varepsilon}{\nu}u\right)\right\} \\
    &= -n_\varepsilon\times\left(2\tilde{n}_1\times\widetilde{W}u+\tilde{n}_2\times u\right) = -n_\varepsilon\times G(u)
  \end{align*}
  on $\Gamma_\varepsilon$ by \eqref{E:NSl_Curl}, \eqref{E:Tilde_Bo}, and \eqref{E:Def_Gu}.
  Hence integration by parts yields
  \begin{align*}
    \int_{\Gamma_\varepsilon}(n_\varepsilon\times\mathrm{curl}\,u)\cdot\Phi\,d\mathcal{H}^2 &= -\int_{\Gamma_\varepsilon}\{n_\varepsilon\times G(u)\}\cdot\Phi\,d\mathcal{H}^2 \\
    &= \int_{\Omega_\varepsilon}\{G(u)\cdot\mathrm{curl}\,\Phi-\mathrm{curl}\,G(u)\cdot\Phi\}\,dx.
  \end{align*}
  Substituting this for \eqref{Pf_IC:IBP} we obtain \eqref{E:IbP_Curl}.
\end{proof}

Now let us prove \eqref{E:Comp_Sto_Lap}.
We follow the idea of the proof of a similar estimate for a flat thin domain given in \cite{Ho08}*{Theorem 2.1} and \cite{Ho10}*{Corollary 5.3}.
Main tools are the integration by parts formula \eqref{E:IbP_Curl} and the standard Helmholtz--Leray projection from $L^2(\Omega_\varepsilon)^3$ onto $L_\sigma^2(\Omega_\varepsilon)$ which we denote by $\mathbb{L}_\varepsilon$.
It is well known (see \cites{BoFa13,CoFo88,So01,Te79}) that the Helmholtz--Leray decomposition
\begin{align*}
  u = \mathbb{L}_\varepsilon u+\nabla q \quad\text{in}\quad L^2(\Omega_\varepsilon)^3, \quad \mathbb{L}_\varepsilon u \in L_\sigma^2(\Omega_\varepsilon), \quad \nabla q\in L_\sigma^2(\Omega_\varepsilon)^\perp
\end{align*}
holds for each $u\in L^2(\Omega_\varepsilon)^3$, where $q\in H^1(\Omega_\varepsilon)$ is a weak solution to the Neumann problem of Poisson's equation
\begin{align*}
  \Delta q = \mathrm{div}\,u \quad\text{in}\quad \Omega_\varepsilon, \quad \frac{\partial q}{\partial n_\varepsilon} = u\cdot n_\varepsilon \quad\text{on}\quad \Gamma_\varepsilon.
\end{align*}
Note that $\mathbb{L}_\varepsilon$ may differ from the orthogonal projection $\mathbb{P}_\varepsilon$ from $L^2(\Omega_\varepsilon)^3$ onto the closed subspace $\mathcal{H}_\varepsilon$ given by \eqref{E:Def_Heps} under the condition (A3) of Assumption \ref{Assump_2}.
In this case we require a little more discussions to establish \eqref{E:Comp_Sto_Lap}.

\begin{proof}[Proof of Theorem \ref{T:Comp_Sto_Lap}]
  We first show that there exists a constant $c>0$ such that
  \begin{align} \label{Pf_CSL:Diff_HL}
    \|\nu\Delta u-\nu\mathbb{L}_\varepsilon\Delta u\|_{L^2(\Omega_\varepsilon)} \leq c\|u\|_{H^1(\Omega_\varepsilon)}
  \end{align}
  for all $\varepsilon\in(0,\varepsilon_0]$ and $u\in D(A_\varepsilon)$.
  By the Helmholtz--Leray decomposition
  \begin{align*}
    \nu\Delta u = \nu\mathbb{L}_\varepsilon\Delta u+\nabla q \quad\text{in}\quad L^2(\Omega_\varepsilon)^3, \quad q\in H^1(\Omega_\varepsilon), \quad (\nu\mathbb{L}_\varepsilon\Delta u,\nabla q)_{L^2(\Omega_\varepsilon)} = 0
  \end{align*}
  for $u\in D(A_\varepsilon)$ and $\Delta u=-\mathrm{curl}\,\mathrm{curl}\,u$ in $\Omega_\varepsilon$ (note that $\mathrm{div}\,u=0$ in $\Omega_\varepsilon$) we have
  \begin{align*}
    \|\nu\Delta u-\nu\mathbb{L}_\varepsilon\Delta u\|_{L^2(\Omega_\varepsilon)}^2 = (\nu\Delta u-\nu\mathbb{L}_\varepsilon\Delta u,\nabla q)_{L^2(\Omega_\varepsilon)} = -\nu(\mathrm{curl}\,\mathrm{curl}\,u,\nabla q)_{L^2(\Omega_\varepsilon)}.
  \end{align*}
  Noting that $\mathrm{curl}\,\nabla q=0$ in $\Omega_\varepsilon$, we apply \eqref{E:IbP_Curl} with $\Phi=\nabla q$ to the last term to get
  \begin{align*}
    -\nu(\mathrm{curl}\,\mathrm{curl}\,u,\nabla q)_{L^2(\Omega_\varepsilon)} = \nu(\mathrm{curl}\,G(u),\nabla q)_{L^2(\Omega_\varepsilon)} \leq c\|\nabla G(u)\|_{L^2(\Omega_\varepsilon)}\|\nabla q\|_{L^2(\Omega_\varepsilon)}
  \end{align*}
  with $G(u)$ given by \eqref{E:Def_Gu}.
  Since the inequalities \eqref{E:Fric_Upper} hold by Assumption \ref{Assump_1}, we can use \eqref{E:G_Bound} to the right-hand side of this inequality.
  Hence
  \begin{align*}
    \|\nu\Delta u-\nu\mathbb{L}_\varepsilon\Delta u\|_{L^2(\Omega_\varepsilon)}^2 &\leq c\|\nabla G(u)\|_{L^2(\Omega_\varepsilon)}\|\nabla q\|_{L^2(\Omega_\varepsilon)} \\
    &\leq c\|u\|_{H^1(\Omega_\varepsilon)}\|\nu\Delta u-\nu\mathbb{L}_\varepsilon\Delta u\|_{L^2(\Omega_\varepsilon)}
  \end{align*}
  and \eqref{Pf_CSL:Diff_HL} follows (note that $\nabla q=\nu\Delta u-\nu\mathbb{L}_\varepsilon\Delta u$ in $L^2(\Omega_\varepsilon)^3$).
  When the condition (A1) or (A2) of Assumption \ref{Assump_2} is satisfied, we have
  \begin{align} \label{Pf_CSL:Au}
    A_\varepsilon u = -\nu\mathbb{P}_\varepsilon\Delta u = -\nu\mathbb{L}_\varepsilon\Delta u \quad\text{in}\quad L^2(\Omega_\varepsilon)^3, \quad u \in D(A_\varepsilon)
  \end{align}
  since $\mathbb{L}_\varepsilon$ agrees with the orthogonal projection $\mathbb{P}_\varepsilon$ onto $\mathcal{H}_\varepsilon=L_\sigma^2(\Omega_\varepsilon)$.
  Hence \eqref{E:Comp_Sto_Lap} is an immediate consequence of \eqref{Pf_CSL:Diff_HL}.

  Next we suppose that the condition (A3) of Assumption \ref{Assump_2} is satisfied.
  Then
  \begin{align*}
    A_\varepsilon u = -\nu\mathbb{P}_\varepsilon\Delta u \in \mathcal{H}_\varepsilon = L_\sigma^2(\Omega_\varepsilon)\cap\mathcal{R}_g^\perp, \quad u\in D(A_\varepsilon),
  \end{align*}
  where $\mathcal{R}_g$ is the space of infinitesimal rigid displacements of $\mathbb{R}^3$ given by \eqref{E:Def_Rg}.
  In this case, however, we still have \eqref{Pf_CSL:Au}.
  To see this, let $w\in\mathcal{R}_g$.
  Then $w$ belongs to $L_\sigma^2(\Omega_\varepsilon)$ by the assumption $\mathcal{R}_g=\mathcal{R}_0\cap\mathcal{R}_1$ and Lemma \ref{L:IR_Sole} and thus
  \begin{align*}
    (\mathbb{L}_\varepsilon\Delta u,w)_{L^2(\Omega_\varepsilon)} = (\Delta u,w)_{L^2(\Omega_\varepsilon)},
  \end{align*}
  since $\mathbb{L}_\varepsilon$ is the orthogonal projection from $L^2(\Omega_\varepsilon)^3$ onto $L_\sigma^2(\Omega_\varepsilon)$.
  Moreover, under the assumptions $\mathcal{R}_g=\mathcal{R}_0\cap\mathcal{R}_1$ and $\gamma_\varepsilon^0=\gamma_\varepsilon^1=0$, the vector fields $u\in D(A_\varepsilon)$ and $w\in\mathcal{R}_g$ satisfy (note that $w$ is of the form $w(x)=a\times x+b$)
  \begin{gather*}
    \mathrm{div}\,u = 0, \quad D(w) = 0 \quad\text{in}\quad \Omega_\varepsilon, \\
    u\cdot n_\varepsilon = 0, \quad P_\varepsilon D(u)n_\varepsilon = 0, \quad w\cdot n_\varepsilon = 0 \quad\text{on}\quad \Gamma_\varepsilon
  \end{gather*}
  by \eqref{E:Dom_St} and Lemma \ref{L:IR_CTD}.
  These equalities and the formula \eqref{E:IbP_St} yield
  \begin{align*}
    (\Delta u,w)_{L^2(\Omega_\varepsilon)} = -2\bigl(D(u),D(w)\bigr)_{L^2(\Omega_\varepsilon)}+2(D(u)n_\varepsilon,w)_{L^2(\Gamma_\varepsilon)} = 0.
  \end{align*}
  Hence $(\mathbb{L}_\varepsilon\Delta u,w)_{L^2(\Omega_\varepsilon)}=0$ for all $w\in\mathcal{R}_g$, i.e.
  \begin{align*}
    \mathbb{L}_\varepsilon\Delta u \in L_\sigma^2(\Omega_\varepsilon)\cap \mathcal{R}_g^\perp = \mathcal{H}_\varepsilon.
  \end{align*}
  Now since the Helmholtz--Leray decomposition
  \begin{align*}
    \Delta u = \mathbb{L}_\varepsilon\Delta u+\nabla\tilde{q} \quad\text{in}\quad L^2(\Omega_\varepsilon)^3, \quad \mathbb{L}_\varepsilon\Delta u \in \mathcal{H}_\varepsilon, \quad \nabla\tilde{q}\in L_\sigma^2(\Omega_\varepsilon)^\perp \subset \mathcal{H}_\varepsilon^\perp
  \end{align*}
  holds and $\mathbb{P}_\varepsilon$ is the orthogonal projection from $L^2(\Omega_\varepsilon)^3$ onto $\mathcal{H}_\varepsilon$, we have
  \begin{align*}
    \mathbb{P}_\varepsilon\Delta u = \mathbb{P}_\varepsilon\mathbb{L}_\varepsilon\Delta u = \mathbb{L}_\varepsilon\Delta u \quad\text{in}\quad L^2(\Omega_\varepsilon)^3, \quad u \in D(A_\varepsilon).
  \end{align*}
  Thus the relation \eqref{Pf_CSL:Au} holds and \eqref{E:Comp_Sto_Lap} follows from \eqref{Pf_CSL:Diff_HL}.
\end{proof}

Finally, we prove the uniform norm equivalence \eqref{E:Stokes_H2} for $A_\varepsilon$ by using the uniform estimates \eqref{E:Comp_Sto_Lap} and \eqref{E:Lap_Apri}.

\begin{proof}[Proof of Theorem \ref{T:Stokes_H2}]
  Let $u\in D(A_\varepsilon)$.
  Since $u$ satisfies the slip boundary conditions \eqref{E:Bo_Slip} by \eqref{E:Dom_St}, we can apply \eqref{E:Comp_Sto_Lap} and \eqref{E:Lap_Apri} to $u$ to get
  \begin{align*}
    \|u\|_{H^2(\Omega_\varepsilon)} &\leq c\left(\|\Delta u\|_{L^2(\Omega_\varepsilon)}+\|u\|_{H^1(\Omega_\varepsilon)}\right) \\
    &\leq c\left(\|A_\varepsilon u\|_{L^2(\Omega_\varepsilon)}+\|A_\varepsilon u+\nu\Delta u\|_{L^2(\Omega_\varepsilon)}+\|u\|_{H^1(\Omega_\varepsilon)}\right) \\
    &\leq c\left(\|A_\varepsilon u\|_{L^2(\Omega_\varepsilon)}+\|u\|_{H^1(\Omega_\varepsilon)}\right).
  \end{align*}
  Applying \eqref{E:Stokes_H1} and \eqref{E:Stokes_Po} to the second term on the last line we obtain the left-hand inequality of \eqref{E:Stokes_H2}.
  Also, by \eqref{E:Comp_Sto_Lap} and $\|u\|_{H^1(\Omega_\varepsilon)}\leq\|u\|_{H^2(\Omega_\varepsilon)}$,
  \begin{align*}
    \|A_\varepsilon u\|_{L^2(\Omega_\varepsilon)} \leq \|A_\varepsilon u+\nu\Delta u\|_{L^2(\Omega_\varepsilon)}+\|\nu\Delta u\|_{L^2(\Omega_\varepsilon)} \leq c\|u\|_{H^2(\Omega_\varepsilon)}.
  \end{align*}
  Hence the right-hand inequality of \eqref{E:Stokes_H2} holds.
\end{proof}

\appendix
\section{Notations on vectors and matrices} \label{S:Ap_Vec}
In this appendix we fix notations on vectors and matrices.
For $m\in\mathbb{N}$ we consider a vector $a\in\mathbb{R}^m$ as a column vector
\begin{align*}
  a =
  \begin{pmatrix}
    a_1 \\ \vdots \\ a_m
  \end{pmatrix}
  = (a_1, \cdots, a_m)^T
\end{align*}
and denote the $i$-th component of $a$ by $a_i$ or sometimes by $a^i$ or $[a]_i$ for $i=1,\dots,m$.
A matrix $A\in\mathbb{R}^{l\times m}$ with $l,m\in\mathbb{N}$ is expressed as
\begin{align*}
  A = (A_{ij})_{i,j} =
  \begin{pmatrix}
    A_{11} & \cdots & A_{1m} \\
    \vdots & & \vdots \\
    A_{l1} & \cdots & A_{lm}
  \end{pmatrix}.
\end{align*}
For $i=1,\dots,l$ and $j=1,\dots,m$ we write $A_{ij}$ or sometimes $[A]_{ij}$ for the $(i,j)$-entry of $A$.
Also, we denote the transpose of $A$ by $A^T$ and, when $l=m$, the symmetric part of $A$ by $A_S:=(A+A^T)/2$ and the $m\times m$ identity matrix by $I_m$.
We define the tensor product of vectors $a\in\mathbb{R}^l$ and $b\in\mathbb{R}^m$ with $l,m\in\mathbb{N}$ by
  \begin{align*}
    a\otimes b :=
    \begin{pmatrix}
      a_1b_1 & \cdots & a_1b_m \\
      \vdots & & \vdots \\
      a_lb_1 & \cdots & a_lb_m
    \end{pmatrix}, \quad
    a =
    \begin{pmatrix}
      a_1 \\ \vdots \\ a_l
    \end{pmatrix}, \quad
    b =
    \begin{pmatrix}
      b_1 \\ \vdots \\ b_m
    \end{pmatrix}.
  \end{align*}
For three-dimensional vector fields $u=(u_1,u_2,u_3)^T$ and $\varphi$ on an open set in $\mathbb{R}^3$ let
\begin{gather*}
  \nabla u :=
  \begin{pmatrix}
    \partial_1u_1 & \partial_1u_2 & \partial_1u_3 \\
    \partial_2u_1 & \partial_2u_2 & \partial_2u_3 \\
    \partial_3u_1 & \partial_3u_2 & \partial_3u_3
  \end{pmatrix}, \quad
  |\nabla^2u|^2 := \sum_{i,j,k=1}^3|\partial_i\partial_ju_k|^2 \quad\left(\partial_i := \frac{\partial}{\partial x_i}\right), \\
  (\varphi\cdot\nabla)u :=
  \begin{pmatrix}
    \varphi\cdot\nabla u_1 \\
    \varphi\cdot\nabla u_2 \\
    \varphi\cdot\nabla u_3
  \end{pmatrix}
  = (\nabla u)^T\varphi.
\end{gather*}
Also, for a $3\times3$ matrix-valued function $A=(A_{ij})_{i,j}$ on an open set in $\mathbb{R}^3$ we set
\begin{align*}
  \mathrm{div}\,A :=
  \begin{pmatrix}
    [\mathrm{div}\,A]_1 \\
    [\mathrm{div}\,A]_2 \\
    [\mathrm{div}\,A]_3
  \end{pmatrix}, \quad
  [\mathrm{div}\,A]_j := \sum_{i=1}^3\partial_iA_{ij}, \quad j=1,2,3.
\end{align*}
We define the inner product of matrices $A,B\in\mathbb{R}^{3\times3}$ and the norm of $A$ as
\begin{align*}
  A: B := \mathrm{tr}[A^TB] = \sum_{i=1}^3AE_i\cdot BE_i, \quad |A| := \sqrt{A:A},
\end{align*}
where $\{E_1,E_2,E_3\}$ is an orthonormal basis of $\mathbb{R}^3$.
Note that $A:B$ does not depend on a choice of $\{E_1,E_2,E_3\}$.
In particular, taking the standard basis of $\mathbb{R}^3$ we get
\begin{align*}
  A:B = \sum_{i,j=1}^3A_{ij}B_{ij} = B:A = A^T:B^T, \quad AB:C = A:CB^T = B:A^TC
\end{align*}
for $A,B,C\in\mathbb{R}^{3\times3}$.
Also, for $a,b\in\mathbb{R}^3$ we have $|a\otimes b|=|a||b|$.

\section{Auxiliary results related to a closes surface} \label{S:Ap_RCS}
This appendix presents some auxiliary results related to a closed surface.

Let $\Gamma$ be a closed, connected, and oriented surface in $\mathbb{R}^3$ of class $C^\ell$ with $\ell\geq2$.
We use the notations given in Section \ref{SS:Pre_Surf}.
First we provide some properties of the Riemannian metric of $\Gamma$ used in this and the next appendices.

\begin{lemma} \label{L:Metric}
  Let $U$ be an open set in $\mathbb{R}^2$, $\mu\colon U\to\Gamma$ a $C^\ell$ local parametrization of $\Gamma$, and $\mathcal{K}$ a compact subset of $U$.
  Then there exists a constant $c>0$ such that
  \begin{align} \label{E:Mu_Bound}
    |\partial_{s_i}\mu(s)| \leq c, \quad |\partial_{s_i}\partial_{s_j}\mu(s)| \leq c \quad\text{for all}\quad s\in\mathcal{K},\,i,j=1,2.
  \end{align}
  We define the Riemannian metric $\theta=(\theta_{ij})_{i,j}$ of $\Gamma$ by
  \begin{align} \label{E:Def_Met}
    \theta(s) := \nabla_s\mu(s)\{\nabla_s\mu(s)\}^T, \quad s\in U, \quad \nabla_s\mu :=
    \begin{pmatrix}
      \partial_{s_1}\mu_1 & \partial_{s_1}\mu_2 & \partial_{s_1}\mu_3 \\
      \partial_{s_2}\mu_1 & \partial_{s_2}\mu_2 & \partial_{s_2}\mu_3
    \end{pmatrix}
  \end{align}
  and denote by $\theta^{-1}=(\theta^{ij})_{i,j}$ the inverse matrix of $\theta$.
  Then
  \begin{align} \label{E:Metric}
      |\theta^k(s)| \leq c, \quad |\partial_{s_i}\theta^k(s)| \leq c, \quad c^{-1} \leq \det\theta(s) \leq c
  \end{align}
  for all $s\in\mathcal{K}$, $i=1,2$, and $k=\pm1$.
  Moreover,
  \begin{align} \label{E:Met_Inv}
    c^{-1}|a|^2 \leq \theta^{-1}(s)a\cdot a \leq c|a|^2
  \end{align}
  for all $s\in\mathcal{K}$ and $a\in\mathbb{R}^2$.
\end{lemma}

\begin{proof}
  The inequalities \eqref{E:Mu_Bound} follow from the $C^\ell$-regularity of $\mu$ on $U$ and the compactness of $\mathcal{K}$ in $U$.
  Using them and the relation
  \begin{align*}
    \partial_{s_i}\theta^{-1} = -\theta^{-1}(\partial_{s_i}\theta)\theta^{-1} \quad\text{in}\quad U
  \end{align*}
  we get the first and second inequalities of \eqref{E:Metric}.
  Also, the third inequality is valid since $\det\theta$ is continuous and strictly positive on $U$ and $\mathcal{K}$ is compact in $U$.

  Let us show \eqref{E:Met_Inv}.
  For $s\in U$ and $a=(a_1,a_2)^T\in\mathbb{R}^2$ we set
  \begin{align*}
    X(s,a) := \sum_{i,j=1}^2\theta^{ij}(s)a_i\partial_{s_j}\mu(s).
  \end{align*}
  Since $\partial_{s_1}\mu(s)$ and $\partial_{s_2}\mu(s)$ are linearly independent, $X(s,a)$ vanishes if and only if
  \begin{align*}
    \sum_{i=1,2}\theta^{ij}(s)a_i = \sum_{i=1,2}\theta^{ji}(s)a_i = 0 \quad\text{for}\quad j=1,2, \quad\text{i.e.}\quad \theta^{-1}(s)a = 0,
  \end{align*}
  which is equivalent to $a=0$ (note that $\theta^{-1}$ is symmetric since $\theta$ is so).
  From this fact it follows that the continuous and nonnegative function
  \begin{align*}
    |X(s,a)|^2 = \sum_{i,j=1}^2\theta^{ij}(s)a_ia_j = \theta^{-1}(s)a\cdot a, \quad (s,a)\in U\times\mathbb{R}^2
  \end{align*}
  does not vanish for $a\neq0$, and thus it is bounded from above and below by positive constants on the compact set $\mathcal{K}\times S^1$, where $S^1$ is the unit circle in $\mathbb{R}^2$.
  Hence \eqref{E:Met_Inv} follows.
\end{proof}

Hereafter we always write $\theta=(\theta_{ij})_{i,j}$ and $\theta^{-1}=(\theta^{ij})_{i,j}$ for the Riemannian metric of $\Gamma$ given by \eqref{E:Def_Met} and its inverse.
We show that the differential operators on $\Gamma$ given in Section \ref{SS:Pre_Surf} agree with those defined in differential geometry.

\begin{lemma} \label{L:TGr_DG}
  Let $U$ be an open set in $\mathbb{R}^2$ and $\mu\colon U\to\Gamma$ a $C^\ell$ local parametrization of $\Gamma$.
  For $\eta\in C^1(\Gamma)$ let $\eta^\flat:=\eta\circ\mu$ on $U$.
  Then
  \begin{align} \label{E:TGr_DG}
    \nabla_\Gamma\eta(\mu(s)) = \sum_{i,j=1}^2\theta^{ij}(s)\partial_{s_i}\eta^\flat(s)\partial_{s_j}\mu(s), \quad s\in U,
  \end{align}
  where $\nabla_\Gamma\eta$ is the tangential gradient of $\eta$ defined by \eqref{E:Def_TGr}.
\end{lemma}

\begin{proof}
  Let $s\in U$.
  Since $\nabla_\Gamma\eta$ is tangential on $\Gamma$ and $\{\partial_{s_1}\mu(s),\partial_{s_2}\mu(s)\}$ is a basis of the tangent plane of $\Gamma$ at $\mu(s)$, we have
  \begin{align*}
    \nabla_\Gamma\eta(\mu(s)) = \sum_{j=1}^2a_j\partial_{s_j}\mu(s)
  \end{align*}
  with some $a_1,a_2\in\mathbb{R}$ and thus
  \begin{align} \label{Pf_TDG:Inner}
    \partial_{s_i}\mu(s)\cdot\nabla_\Gamma\eta(\mu(s)) = \sum_{j=1}^2a_j\theta_{ij}(s), \quad i=1,2.
  \end{align}
  Let $\tilde{\eta}$ be an arbitrary $C^1$-extension of $\eta$ to $N$ with $\tilde{\eta}|_\Gamma=\eta$.
  Then
  \begin{align} \label{Pf_TDG:Flat}
    \partial_{s_i}\eta^\flat(s) = \partial_{s_i}\Bigl(\tilde{\eta}(\mu(s))\Bigr) = \partial_{s_i}\mu(s)\cdot\nabla\tilde{\eta}(\mu(s)) = \partial_{s_i}\mu(s)\cdot\nabla_\Gamma\eta(\mu(s))
  \end{align}
  for $i=1,2$ since $\partial_{s_i}\mu(s)$ is tangent to $\Gamma$ at $\mu(s)$.
  By \eqref{Pf_TDG:Inner} and \eqref{Pf_TDG:Flat} we have
  \begin{align*}
    \partial_{s_i}\eta^\flat(s) = \sum_{j=1}^2\theta_{ij}(s)a_j, \quad i=1,2, \quad\text{i.e.}\quad
    \begin{pmatrix}
      \partial_{s_1}\eta^\flat(s) \\
      \partial_{s_2}\eta^\flat(s)
    \end{pmatrix}
    = \theta(s)
    \begin{pmatrix}
      a_1 \\
      a_2
    \end{pmatrix}.
  \end{align*}
  Therefore,
  \begin{align*}
    \begin{pmatrix}
      a_1 \\
      a_2
    \end{pmatrix}
    = \theta^{-1}(s)
    \begin{pmatrix}
      \partial_{s_1}\eta^\flat(s) \\
      \partial_{s_2}\eta^\flat(s)
    \end{pmatrix}, \quad\text{i.e.}\quad
    a_j = \sum_{i=1}^2\theta^{ji}(s)\partial_{s_i}\eta^\flat(s), \quad j=1,2
  \end{align*}
  and the equality \eqref{E:TGr_DG} holds (note that $\theta^{ji}=\theta^{ij}$).
\end{proof}

\begin{lemma} \label{L:DivG_DG}
  Let $U$ be an open set in $\mathbb{R}^2$ and $\mu\colon U\to\Gamma$ a $C^\ell$ local parametrization of $\Gamma$.
  For $X\in C^1(\Gamma,T\Gamma)$ let
  \begin{align*}
    X^i(s) := \sum_{j=1}^2\theta^{ij}(s)\partial_{s_j}\mu(s)\cdot X(\mu(s)), \quad s\in U, \, i=1,2.
  \end{align*}
  Then the surface divergence of $X$ defined by \eqref{E:Def_DivG} is locally of the form
  \begin{align} \label{E:DivG_DG}
    \mathrm{div}_\Gamma X(\mu(s)) = \frac{1}{\sqrt{\det\theta(s)}}\sum_{i=1}^2\partial_{s_i}\Bigl(X^i(s)\sqrt{\det\theta(s)}\Bigr), \quad s\in U.
  \end{align}
\end{lemma}

\begin{proof}
  Since $X$ is tangential on $\Gamma$, we can show
  \begin{align} \label{Pf_DDG:X_Loc}
    X(\mu(s)) = \sum_{i=1}^2X^i(s)\partial_{s_i}\mu(s), \quad s \in U
  \end{align}
  as in the proof of Lemma \ref{L:TGr_DG}.
  In what follows, we write $\eta^\flat(s):=\eta(\mu(s))$, $s\in U$ for a function $\eta$ on $\Gamma$ and suppress the argument $s\in U$.
  Since
  \begin{align*}
    (\underline{D}_mX_m)^\flat = \sum_{i,j=1}^2\theta^{ij}(\partial_{s_i}X_m^\flat)\partial_{s_j}\mu_m, \quad m=1,2,3
  \end{align*}
  by \eqref{E:TGr_DG}, where $X=(X_1,X_2,X_3)^T$, the surface divergence of $X$ is of the form
  \begin{align*}
    (\mathrm{div}_\Gamma X)^\flat = \sum_{m=1}^3(\underline{D}_mX_m)^\flat = \sum_{m=1}^3\sum_{i,j=1}^2\theta^{ij}(\partial_{s_i}X_m^\flat)\partial_{s_j}\mu_m = \sum_{i,j=1}^2\theta^{ij}\partial_{s_i}X^\flat\cdot\partial_{s_j}\mu.
  \end{align*}
  To the last term we apply \eqref{Pf_DDG:X_Loc}.
  Then since
  \begin{align*}
    \partial_{s_i}X^\flat\cdot\partial_{s_j}\mu &= \sum_{k=1}^2\{(\partial_{s_i}X^k)\partial_{s_k}\mu\cdot\partial_{s_j}\mu+X^k\partial_{s_i}\partial_{s_k}\mu\cdot\partial_{s_j}\mu\} \\
    &= \sum_{k=1}^2\{\theta^{kj}\partial_{s_i}X^k+X^k\partial_{s_i}\partial_{s_k}\mu\cdot\partial_{s_j}\mu\}
  \end{align*}
  and $\sum_{j=1}^2\theta^{ij}\theta^{kj}=\sum_{j=1}^2\theta^{ij}\theta^{jk}=\delta_{ik}$, where $\delta_{ik}$ is the Kronecker delta, we have
  \begin{align} \label{Pf_DDG:Left}
    \begin{aligned}
      (\mathrm{div}_\Gamma X)^\flat &= \sum_{i=1}^2\partial_{s_i}X^i+\sum_{i,j,k=1}^2\theta^{ij}X^k\partial_{s_i}\partial_{s_k}\mu\cdot\partial_{s_j}\mu \\
      &= \sum_{i=1}^2\partial_{s_i}X^i+\sum_{i,j,k=1}^2X^i\theta^{jk}\partial_{s_i}\partial_{s_k}\mu\cdot\partial_{s_j}\mu.
    \end{aligned}
  \end{align}
  Here we rewrote the second term on the right-hand side by exchanging the indices $i$ and $k$ and using $\theta^{kj}=\theta^{jk}$ and $\partial_{s_i}\partial_{s_k}\mu=\partial_{s_k}\partial_{s_i}\mu$.
  On the other hand,
  \begin{align} \label{Pf_DDG:Right}
    \frac{1}{\sqrt{\det\theta}}\sum_{i=1}^2\partial_{s_i}\Bigl(X^i\sqrt{\det\theta}\Bigr) = \sum_{i=1}^2\partial_{s_i}X^i+\frac{1}{\sqrt{\det\theta}}\sum_{i=1}^2X^i\partial_{s_i}\Bigl(\sqrt{\det\theta}\Bigr).
  \end{align}
  Moreover, by Jacobi's formula
  \begin{align*}
    \partial_{s_i}(\det\theta) = \mathrm{tr}(\theta^{-1}\partial_{s_i}\theta)\det\theta, \quad i=1,2
  \end{align*}
  and $\theta^{kj}=\theta^{jk}$ for $j,k=1,2$ we observe that
  \begin{align*}
    \frac{1}{\sqrt{\det\theta}}\partial_{s_i}\Bigl(\sqrt{\det\theta}\Bigr) &= \frac{1}{2}\mathrm{tr}(\theta^{-1}\partial_{s_i}\theta) = \frac{1}{2}\sum_{j,k=1}^2\theta^{jk}\partial_{s_i}\theta_{kj} \\
    &= \frac{1}{2}\sum_{j,k=1}^2\theta^{jk}(\partial_{s_i}\partial_{s_k}\mu\cdot\partial_{s_j}\mu+\partial_{s_k}\mu\cdot\partial_{s_i}\partial_{s_j}\mu) \\
    &= \sum_{j,k=1}^2\theta^{jk}\partial_{s_i}\partial_{s_k}\mu\cdot\partial_{s_j}\mu.
  \end{align*}
  Thus the right-hand side of \eqref{Pf_DDG:Right} is equal to that of \eqref{Pf_DDG:Left} and we get \eqref{E:DivG_DG}.
\end{proof}

Next we consider the local expression of a function in $L^p(\Gamma)$ or $W^{1,p}(\Gamma)$.

\begin{lemma} \label{L:Lp_Loc}
  Let $U$ be an open set in $\mathbb{R}^2$, $\mu\colon U\to\Gamma$ a $C^\ell$ local parametrization of $\Gamma$, and $\mathcal{K}$ a compact subset of $U$.
  For $p\in[1,\infty]$ if $\eta\in L^p(\Gamma)$ is supported in $\mu(\mathcal{K})$, then $\eta^\flat:=\eta\circ\mu\in L^p(U)$ and
  \begin{align} \label{E:Lp_Loc}
    c^{-1}\|\eta^\flat\|_{L^p(U)} \leq \|\eta\|_{L^p(\Gamma)} \leq c\|\eta^\flat\|_{L^p(U)}.
  \end{align}
  If in addition $\eta\in W^{1,p}(\Gamma)$, then $\eta^\flat\in W^{1,p}(U)$, \eqref{E:TGr_DG} holds in $L^p(U)^3$, and
  \begin{align} \label{E:W1p_Loc}
    c^{-1}\|\nabla_s\eta^\flat\|_{L^p(U)} \leq \|\nabla_\Gamma\eta\|_{L^p(\Gamma)} \leq c\|\nabla_s\eta^\flat\|_{L^p(U)}.
  \end{align}
  Here $\nabla_s\eta^\flat=(\partial_{s_1}\eta^\flat,\partial_{s_2}\eta^\flat)^T$ is the gradient of $\eta^\flat$ in $s\in\mathbb{R}^2$.
\end{lemma}

\begin{proof}
  Let $\eta\in L^p(\Gamma)$, $p\in[1,\infty]$ be supported in $\mu(\mathcal{K})$ and $\eta^\flat:=\eta\circ\mu$ on $U$.
  When $p\neq\infty$, by the definition of an integral over a surface we have
  \begin{align*}
    \|\eta\|_{L^p(\Gamma)}^p = \int_\Gamma|\eta(y)|^p\,d\mathcal{H}^2(y) = \int_U|\eta^\flat(s)|^p\sqrt{\det\theta(s)}\,ds.
  \end{align*}
  Since $\eta^\flat$ is supported in $\mathcal{K}$, we can apply \eqref{E:Metric} to the right-hand side of this equality to get $\eta^\flat\in L^p(U)$ and \eqref{E:Lp_Loc}.
  Also, if $p=\infty$, then
  \begin{align*}
    \|\eta\|_{L^\infty(\Gamma)} = \|\eta\|_{L^\infty(\mu(U))} = \|\eta\circ\mu\|_{L^\infty(U)} = \|\eta^\flat\|_{L^\infty(U)}
  \end{align*}
  since $\eta$ is supported in $\mu(\mathcal{K})\subset\mu(U)$.
  Thus $\eta^\flat\in L^\infty(U)$ and \eqref{E:Lp_Loc} holds.

  Next let $\eta\in W^{1,p}(\Gamma)$, $p\in[1,\infty]$ be supported in $\mu(\mathcal{K})$.
  Then $\eta^\flat\in L^p(U)$ by the first part of the proof.
  Let us show $\partial_{s_i}\eta^\flat\in L^p(U)$ for $i=1,2$.
  Hereafter we write $\xi^\flat:=\xi\circ\mu$ on $U$ for a function $\xi$ on $\Gamma$.
  We prove that
  \begin{align} \label{Pf_LpL:D_Flat}
    \partial_{s_i}\eta^\flat = \partial_{s_i}\mu\cdot(\nabla_\Gamma\eta)^\flat \in L^p(U), \quad i=1,2.
  \end{align}
  Note that the right-hand side is in $L^p(U)$.
  Indeed, since $\nabla_\Gamma\eta\in L^p(\Gamma)^3$ is supported in $\mu(\mathcal{K})$, we have $(\nabla_\Gamma\eta)^\flat\in L^p(U)^3$ by the first part of the proof.
  Moreover,
  \begin{align*}
    |\partial_{s_i}\mu\cdot(\nabla_\Gamma\eta)^\flat| \leq c|(\nabla_\Gamma\eta)^\flat| \quad\text{on}\quad \mathcal{K}
  \end{align*}
  by \eqref{E:Mu_Bound} and $(\nabla_\Gamma\eta)^\flat$ is supported in $\mathcal{K}$.
  Hence $\partial_{s_i}\mu\cdot(\nabla_\Gamma\eta)^\flat\in L^p(U)$.
  Let us show the equality \eqref{Pf_LpL:D_Flat}.
  Fix $i=1,2$.
  For $\varphi\in C_c^1(U)$ we define
  \begin{align*}
    X(\mu(s)) := \frac{\varphi(s)}{\sqrt{\det\theta(s)}}\partial_{s_i}\mu(s), \quad s\in U
  \end{align*}
  and extend $X$ to $\Gamma$ by setting zero outside $\mu(U)$.
  Then $X\in C^1(\Gamma,T\Gamma)$ and
  \begin{align*}
    \mathrm{div}_\Gamma X(\mu(s)) = \frac{\partial_{s_i}\varphi(s)}{\sqrt{\det\theta(s)}}, \quad s\in U
  \end{align*}
  by Lemma \ref{L:DivG_DG}.
  Since $\eta^\flat$ is supported in $U$, we deduce from this equality that
  \begin{align*}
    \int_U\eta^\flat\partial_{s_i}\varphi\,ds &= \int_U\eta^\flat(\mathrm{div}_\Gamma X)^\flat\sqrt{\det\theta}\,ds = \int_\Gamma\eta\,\mathrm{div}_\Gamma X\,d\mathcal{H}^2.
  \end{align*}
  Moreover, we apply \eqref{E:IbP_WTGr} with $v=X$ and $X\cdot n=0$ on $\Gamma$ to the last term to get
  \begin{align*}
    \int_U\eta^\flat\partial_{s_i}\varphi\,ds &= -\int_\Gamma\nabla_\Gamma\eta\cdot X\,d\mathcal{H}^2 = -\int_U(\nabla_\Gamma\eta)^\flat\cdot X^\flat\sqrt{\det\theta}\,ds \\
    &= -\int_U\{\partial_{s_i}\mu\cdot(\nabla_\Gamma\eta)^\flat\}\varphi\,ds
  \end{align*}
  for all $\varphi\in C_c^1(U)$.
  Hence \eqref{Pf_LpL:D_Flat} is valid and $\eta^\flat\in W^{1,p}(U)$.
  Now we observe that
  \begin{align*}
    \int_\Gamma(\nabla_\Gamma\eta\cdot n)\xi\,d\mathcal{H}^2 = -\int_\Gamma\eta\{\mathrm{div}_\Gamma(\xi n)+H\xi|n|^2\}\,d\mathcal{H}^2 = 0
  \end{align*}
  for all $\xi\in C^1(\Gamma)$ by \eqref{E:IbP_WTGr} with $v=\xi n\in C^1(\Gamma)^3$ and
  \begin{align*}
    |n|^2 = 1, \quad \mathrm{div}_\Gamma(\xi n) = \nabla_\Gamma\xi\cdot n+\xi\,\mathrm{div}_\Gamma n = -\xi H \quad\text{on}\quad \Gamma,
  \end{align*}
  where the last equality follows from \eqref{E:P_TGr} for $\xi\in C^1(\Gamma)$.
  Hence
  \begin{align} \label{Pf_LpL:N_WTGr}
    \nabla_\Gamma\eta\cdot n = 0 \quad\text{on}\quad \Gamma, \quad\text{i.e.}\quad (\nabla_\Gamma\eta)^\flat\cdot n^\flat = 0 \quad\text{on}\quad U.
  \end{align}
  Since $(\nabla_\Gamma\eta)^\flat\in L^p(U)^3$, $\nabla_s\eta^\flat\in L^p(U)^2$, and $\{\partial_{s_1}\mu,\partial_{s_2}\mu,n^\flat\}$ is a basis of $\mathbb{R}^3$ on $U$, we see by \eqref{Pf_LpL:D_Flat} and \eqref{Pf_LpL:N_WTGr} that \eqref{E:TGr_DG} holds in $L^p(U)^3$ as in the proof of Lemma \ref{L:TGr_DG} (note that here we do not use a density argument).
  Thus
  \begin{align*}
    |(\nabla_\Gamma\eta)^\flat|^2 = \sum_{i,j=1}^2\theta^{ij}(\partial_{s_i}\eta^\flat)(\partial_{s_j}\eta^\flat) = (\theta^{-1}\nabla_s\eta^\flat)\cdot\nabla_s\eta^\flat \quad\text{on}\quad U
  \end{align*}
  and it follows from \eqref{E:Met_Inv} that
  \begin{align*}
    c^{-1}|\nabla_s\eta^\flat| \leq |(\nabla_\Gamma\eta)^\flat| \leq c|\nabla_s\eta^\flat| \quad\text{on}\quad \mathcal{K}.
  \end{align*}
  Noting that $(\nabla_\Gamma\eta)^\flat$ is supported in $\mathcal{K}$, we apply this inequality and \eqref{E:Metric} to
  \begin{align*}
    \|\nabla_\Gamma\eta\|_{L^p(\Gamma)}^p &= \int_U|(\nabla_\Gamma\eta)^\flat|^p\sqrt{\det\theta}\,ds, \quad p\neq\infty, \\
    \|\nabla_\Gamma\eta\|_{L^\infty(\Gamma)} &= \|(\nabla_\Gamma\eta)^\flat\|_{L^\infty(U)}
  \end{align*}
  to obtain \eqref{E:W1p_Loc}.
\end{proof}

Let us prove two lemmas related to a parametrized surface used in the proofs of Lemmas \ref{L:Nor_Bo} and \ref{L:CoV_Surf}.
For $h\in C^1(\Gamma)$ satisfying $|h|<\delta$ on $\Gamma$ we set
\begin{align} \label{E:Def_Para_Surf}
  \Gamma_h := \{y+h(y)n(y)\mid y\in\Gamma\} \subset \mathbb{R}^3.
\end{align}
Note that $\Gamma_h\subset N$ by $|h|<\delta$ on $\Gamma$ (see Section \ref{SS:Pre_Surf}).
We also define
\begin{align} \label{E:N_Para}
  \tau_h(y) := \{I_3-h(y)W(y)\}^{-1}\nabla_\Gamma h(y), \quad n_h(y) := \frac{n(y)-\tau_h(y)}{\sqrt{1+|\tau_h(y)|^2}}
\end{align}
for $y\in\Gamma$.
Note that $\tau_h$ is tangential on $\Gamma$.
We assume that the orientation of $\Gamma_h$ is the same as that of $\Gamma$.

\begin{lemma} \label{L:Para_Nor}
  The constant extension of $n_h$ in the normal direction of $\Gamma$ gives the unit outward normal vector field of $\Gamma_h$.
\end{lemma}

\begin{proof}
  Let $\bar{n}_h=n_h\circ\pi$ be the constant extension of $n_h$ in the normal direction of $\Gamma$.
  Since $|n_h|=1$ on $\Gamma$ and the direction of $n_h$ is the same as that of $n$, it is sufficient to show that $\bar{n}_h$ is perpendicular to the tangent plane of $\Gamma_h$.

  Let $\mu\colon U\to\Gamma$ be a local parametrization of $\Gamma$ with an open set $U$ of $\mathbb{R}^2$ and
  \begin{align} \label{Pf_PaNo:LoPa}
    \mu_h(s) := \mu(s)+h(\mu(s))n(\mu(s)), \quad s\in U.
  \end{align}
  Then $\mu_h$ is a local parametrization of $\Gamma_h$ and $\{\partial_{s_1}\mu_h(s),\partial_{s_2}\mu_h(s)\}$ is a basis of the tangent plane of $\Gamma_h$ at $\mu_h(s)$.
  Hence to show that $\bar{n}_h$ is perpendicular to the tangent plane of $\Gamma_h$ it suffices to prove
  \begin{align} \label{Pf_PaNo:Perp}
    \bar{n}_h(\mu_h(s))\cdot\partial_{s_k}\mu_h(s) = 0, \quad s\in U,\,k=1,2.
  \end{align}
  Moreover, $\bar{n}_h(\mu_h(s))=n_h(\mu(s))$ for $s\in U$ since $\pi(\mu_h(s))=\mu(s)\in\Gamma$.
  By this fact and \eqref{E:N_Para} the condition \eqref{Pf_PaNo:Perp} reduces to
  \begin{align} \label{Pf_PaNo:Goal}
    n(\mu(s))\cdot\partial_{s_k}\mu_h(s) = \tau_h(\mu(s))\cdot\partial_{s_k}\mu_h(s), \quad s\in U,\,k=1,2.
  \end{align}
  Let us prove \eqref{Pf_PaNo:Goal}.
  Hereafter we write $\eta^\flat(s):=\eta(\mu(s))$, $s\in U$ for a function $\eta$ on $\Gamma$ and suppress the argument $s\in U$.
  For $k=1,2$ we differentiate $\mu_h=\mu+h^\flat n^\flat$ with respect to $s_k$ and apply \eqref{Pf_TDG:Flat} and $-\nabla_\Gamma n=W=W^T$ on $\Gamma$ to get
  \begin{align*}
    \partial_{s_k}\mu_h = (I_3-h^\flat W^\flat)\partial_{s_k}\mu+\{\partial_{s_k}\mu\cdot(\nabla_\Gamma h)^\flat\}n^\flat.
  \end{align*}
  Since $\partial_{s_k}\mu(s)$ is tangent to $\Gamma$ at $\mu(s)$ and $W^Tn=Wn=0$ on $\Gamma$, we deduce from the above equality that
  \begin{align*}
    n^\flat\cdot\partial_{s_k}\mu_h = \partial_{s_k}\mu\cdot(\nabla_\Gamma h)^\flat.
  \end{align*}
  Also, since $\tau_h=(I_3-hW)^{-1}\nabla_\Gamma h$ is tangential and $W$ is symmetric on $\Gamma$,
  \begin{align*}
    \tau_h^\flat\cdot\partial_{s_k}\mu_h = (I_3-h^\flat W^\flat)^{-1}(\nabla_\Gamma h)^\flat\cdot(I_3-h^\flat W^\flat)\partial_{s_k}\mu = (\nabla_\Gamma h)^\flat\cdot\partial_{s_k}\mu.
  \end{align*}
  The above two equalities imply \eqref{Pf_PaNo:Goal} and thus the claim is valid.
\end{proof}

\begin{lemma} \label{L:CoV_Para}
  For $\varphi\in L^1(\Gamma_h)$ we have the change of variables formula
  \begin{align} \label{E:CoV_Para}
    \int_{\Gamma_h}\varphi(x)\,d\mathcal{H}^2(x) = \int_\Gamma\varphi_h^\sharp(y)J(y,h(y))\sqrt{1+|\tau_h(y)|^2}\,d\mathcal{H}^2(y),
  \end{align}
  where $J$ and $\tau_h$ are the functions given by \eqref{E:Def_Jac} and \eqref{E:N_Para} and
  \begin{align} \label{E:Pull_Para}
    \varphi_h^\sharp(y) := \varphi(y+h(y)n(y)), \quad y\in\Gamma.
  \end{align}
\end{lemma}

To prove Lemma \ref{L:CoV_Para} we use the tangential gradient of $\varphi\in C^1(\Gamma_h)$ given by
\begin{align} \label{E:Para_TGr}
  \nabla_{\Gamma_h}\varphi(x) := \{I_3-\bar{n}_h(x)\otimes\bar{n}_h(x)\}\nabla\tilde{\varphi}(x), \quad x\in\Gamma_h,
\end{align}
where $\tilde{\varphi}$ is an arbitrary extension of $\varphi$ to $N$ satisfying $\tilde{\varphi}|_{\Gamma_h}=\varphi$.

\begin{proof}
  Since $\Gamma$ is compact and without boundary, we can take a finite number of open sets in $\mathbb{R}^2$ and local parametrizations of $\Gamma$
  \begin{align*}
    U_k \subset \mathbb{R}^2, \quad \mu^k\colon U_k\to\Gamma, \quad k=1,\dots,k_0
  \end{align*}
  such that $\{\mu^k(U_k)\}_{k=1}^{k_0}$ is an open covering of $\Gamma$.
  Let $\{\eta^k\}_{k=1}^{k_0}$ be a partition of unity on $\Gamma$ subordinate to $\{\mu^k(U_k)\}_{k=1}^{k_0}$.
  For $k=1,\dots,k_0$ and $s\in U_k$ we define
  \begin{align*}
  \mu_h^k(s) := \mu^k(s)+h(\mu^k(s))n(\mu^k(s)), \quad \eta_h^k(\mu_h^k(s)) := \eta^k(\mu^k(s)).
  \end{align*}
  Then $\mu_h^k\colon U_k\to\Gamma_h$, $k=1,\dots,k_0$ are local parametrizations of $\Gamma_h$, $\{\mu_h^k(U_k)\}_{k=1}^{k_0}$ is an open covering of $\Gamma_h$, and $\{\eta_h^k\}_{k=1}^{k_0}$ is a partition of unity on $\Gamma_h$ subordinate to $\{\mu_h^k(U_k)\}_{k=1}^{k_0}$.
  Moreover, for $k=1,\dots,k_0$ we observe by \eqref{E:Pull_Para} that
  \begin{gather*}
    \eta_h^k(\mu_h^k(s))\varphi(\mu_h^k(s)) = \eta^k(\mu^k(s))\varphi_h^\sharp(\mu^k(s)), \quad s\in U_k, \\
    \text{i.e.} \quad (\eta_h^k\varphi)_h^\sharp = \eta^k\varphi_h^\sharp \quad\text{on}\quad \mu^k(U_k) \subset \Gamma.
  \end{gather*}
  Hence it suffices to prove \eqref{E:CoV_Para} for $\varphi^k:=\eta_h^k\varphi$ instead of $\varphi$.

  From now on, we fix and suppress $k$.
  Let $\mu\colon U\to\Gamma$ be a local parametrization of $\Gamma$ with an open set $U$ in $\mathbb{R}^2$, $\nabla_s\mu$ and $\theta$ the gradient matrix of $\mu$ and the Riemannian metric of $\Gamma$ given by \eqref{E:Def_Met}, $\mu_h$ the local parametrization of $\Gamma_h$ given by \eqref{Pf_PaNo:LoPa}, and $\nabla_s\mu_h$ and $\theta_h$ the gradient matrix of $\mu_h$ and the Riemannian metric of $\Gamma_h$ defined similarly as in \eqref{E:Def_Met}.
  By the first part of the proof we may assume that $\varphi\circ\mu_h$ is compactly supported in $U$.
  Then since $\varphi\circ\mu_h=\varphi_h^\sharp\circ\mu$ on $U$ and
  \begin{align*}
    \int_{\Gamma_h}\varphi\,d\mathcal{H}^2 &= \int_U\varphi\circ\mu_h\sqrt{\det\theta_h}\,ds, \\
    \int_\Gamma\varphi_h^\sharp J(\cdot,h)\sqrt{1+|\tau_h|^2}\,d\mathcal{H}^2 &= \int_U(\varphi_h^\sharp\circ\mu)J(\mu,h\circ\mu)\sqrt{(1+|\tau_h\circ\mu|^2)\det\theta}\,ds,
  \end{align*}
  it is sufficient for \eqref{E:CoV_Para} to show that
  \begin{align} \label{Pf_CoVP:Goal}
    \sqrt{\det\theta_h} = J(\mu,h\circ\mu)\sqrt{(1+|\tau_h\circ\mu|^2)\det\theta} \quad\text{on}\quad U.
  \end{align}
  Hereafter we write $\eta^\flat(s):=\eta(\mu(s))$, $s\in U$ for a function $\eta$ on $\Gamma$ and suppress the argument $s\in U$.
  Let $\bar{h}=h\circ\pi$ be the constant extension of $h$.
  First we prove
  \begin{align} \label{Pf_CoVP:First}
    (1-|(\nabla_{\Gamma_h}\bar{h})\circ\mu_h|^2)\det\theta_h = J(\mu,h^\flat)^2\det\theta.
  \end{align}
  We differentiate $\mu_h=\mu+h^\flat n^\flat$ and use \eqref{Pf_TDG:Flat} and $-\nabla_\Gamma n=W=W^T$ on $\Gamma$ to get
  \begin{align*}
    \nabla_s\mu_h=\nabla_s\mu(I_3-h^\flat W^\flat)+\nabla_sh^\flat\otimes n^\flat.
  \end{align*}
  Since $\partial_{s_1}\mu$ and $\partial_{s_2}\mu$ are tangent to $\Gamma$ at $\mu(s)$, it follows that
  \begin{align*}
    (\nabla_s\mu)n^\flat =
    \begin{pmatrix}
      \partial_{s_1}\mu\cdot n^\flat \\
      \partial_{s_2}\mu\cdot n^\flat
    \end{pmatrix}
    = 0.
  \end{align*}
  We also have $W^\flat n^\flat=0$ and
  \begin{align*}
    (\nabla_sh^\flat\otimes n^\flat)(n^\flat\otimes\nabla_s h^\flat) = |n^\flat|^2\nabla_sh^\flat\otimes\nabla_sh^\flat = \nabla_sh^\flat\otimes\nabla_sh^\flat.
  \end{align*}
  Noting that $W^\flat$ is symmetric, we deduce from these equalities that
  \begin{align*}
    \theta_h=\nabla_s\mu_h(\nabla_s\mu_h)^T=\nabla_s\mu(I_3-h^\flat W^\flat)^2(\nabla_s\mu)^T+\nabla_sh^\flat\otimes\nabla_sh^\flat
  \end{align*}
  and thus
  \begin{align*}
    \det(\theta_h-\nabla_sh^\flat\otimes\nabla_sh^\flat) = \det[\nabla_s\mu(I_3-h^\flat W^\flat)^2(\nabla_s\mu)^T].
  \end{align*}
  Let $\theta_h^{-1}=(\theta_h^{ij})_{i,j}$ be the inverse matrix of $\theta_h$.
  To the above equality we apply
  \begin{align*}
    \det(\theta_h-\nabla_sh^\flat\otimes\nabla_sh^\flat) &= \det[I_2-(\theta_h^{-1}\nabla_sh^\flat)\otimes\nabla_sh^\flat]\det\theta_h \\
    &= \{1-(\theta_h^{-1}\nabla_sh^\flat)\cdot\nabla_sh^\flat\}\det\theta_h
  \end{align*}
  by $\det(I_2+a\otimes b)=1+a\cdot b$ for $a,b\in\mathbb{R}^2$ and
  \begin{align} \label{Pf_CoVP:Later}
    \det[\nabla_s\mu(I_3-h^\flat W^\flat)^2(\nabla_s\mu)^T] = J(\mu,h^\flat)^2\det\theta
  \end{align}
  which we prove at the end of the proof.
  Then we obtain
  \begin{align} \label{Pf_CoVP:Det}
    \{1-(\theta_h^{-1}\nabla_sh^\flat)\cdot\nabla_sh^\flat\}\det\theta_h = J(\mu,h^\flat)^2\det\theta.
  \end{align}
  Now we recall that the tangential gradient of $\bar{h}=h\circ\pi$ on $\Gamma_h$ is expressed as
  \begin{align*}
    \nabla_{\Gamma_h}\bar{h}(\mu_h(s)) = \sum_{i,j=1}^2\theta_h^{ij}(s)\frac{\partial(\bar{h}\circ\mu_h)}{\partial s_i}(s)\partial_{s_j}\mu_h(s), \quad s\in U.
  \end{align*}
  Since $\pi(\mu_h(s))=\mu(s)$ for $s\in U$, we have $\bar{h}(\mu_h(s))=h(\mu(s))=h^\flat(s)$ and thus
  \begin{align*}
    (\nabla_{\Gamma_h}\bar{h})\circ\mu_h &= \sum_{i,j=1}^2\theta_h^{ij}(\partial_{s_i}h^\flat)\partial_{s_j}\mu_h, \\
    \left|(\nabla_{\Gamma_h}\bar{h})\circ\mu_h\right|^2 &= \sum_{i,j=1}^2\theta_h^{ij}(\partial_{s_i}h^\flat)(\partial_{s_j}h^\flat) = (\theta_h^{-1}\nabla_sh^\flat)\cdot\nabla_sh^\flat.
  \end{align*}
  Applying this equality to the left-hand side of \eqref{Pf_CoVP:Det} we obtain \eqref{Pf_CoVP:First}.

  Let $\tau_h$ and $n_h$ be given by \eqref{E:N_Para} and $\alpha:=(1+|\tau_h|^2)^{1/2}$ on $\Gamma$.
  We next show
  \begin{align} \label{Pf_CoVP:Second}
    1-\left|\nabla_{\Gamma_h}\bar{h}(y+h(y)n(y))\right|^2 = \frac{1}{\alpha(y)^2} = \frac{1}{1+|\tau_h(y)|^2}, \quad y\in\Gamma.
  \end{align}
  For $y\in\Gamma$ we see by \eqref{E:ConDer_Dom}, $d(y+h(y)n(y))=h(y)$, and $\pi(y+h(y)n(y))=y$ that
  \begin{align*}
    \nabla\bar{h}(y+h(y)n(y)) = \{I_3-h(y)W(y)\}^{-1}\nabla_\Gamma h(y) = \tau_h(y).
  \end{align*}
  From this equality, \eqref{E:Para_TGr}, $\bar{n}_h(y+h(y)n(y))=n_h(y)$, and
  \begin{align} \label{PF_CoVP:Tau}
    n_h = \alpha^{-1}(n-\tau_h), \quad \tau_h\cdot n = 0, \quad \alpha^2 = 1+|\tau_h|^2 \quad\text{on}\quad \Gamma
  \end{align}
  we deduce that (here we suppress the argument $y\in\Gamma$ of functions on $\Gamma$)
  \begin{align*}
    \nabla_{\Gamma_h}\bar{h}(y+hn) = (I_3-n_h\otimes n_h)\nabla\bar{h}(y+hn)= \alpha^{-2}(|\tau_h|^2n+\tau_h).
  \end{align*}
  Hence we again use the second and third equalities of \eqref{PF_CoVP:Tau} to obtain
  \begin{align*}
    1-\left|\nabla_{\Gamma_h}\bar{h}(y+hn)\right|^2 = 1-\alpha^{-4}(|\tau_h|^4+|\tau_h|^2) = \alpha^{-2}.
  \end{align*}
  This shows \eqref{Pf_CoVP:Second} and we obtain \eqref{Pf_CoVP:Goal} by \eqref{Pf_CoVP:First} and \eqref{Pf_CoVP:Second}.

  It remains to prove \eqref{Pf_CoVP:Later}.
  We define $3\times 3$ matrices
  \begin{align*}
    A :=
    \begin{pmatrix}
      \nabla_s\mu \\
      (n^\flat)^T
    \end{pmatrix}, \quad
    A_h :=
    \begin{pmatrix}
      \nabla_s\mu(I_3-h^\flat W^\flat) \\
      (n^\flat)^T
    \end{pmatrix}.
  \end{align*}
  Recall that we consider $n^\flat\in\mathbb{R}^3$ as a column vector.
  Then by the symmetry of $W^\flat$ and the equalities $(\nabla_s\mu)n^\flat=0$ and $W^\flat n^\flat=0$ we have
  \begin{gather*}
    A_h = A(I_3-h^\flat W^\flat), \\
    AA^T =
    \begin{pmatrix}
      \theta & 0 \\
      0 & 1
    \end{pmatrix}, \quad
    A_hA_h^T =
    \begin{pmatrix}
      \nabla_s\mu(I_3-h^\flat W^\flat)^2(\nabla_s\mu)^T & 0 \\
      0 & 1
    \end{pmatrix}.
  \end{gather*}
  From these equalities and $\det(I_3-h^\flat W^\flat)=J(\mu,h^\flat)$ we deduce that
  \begin{align*}
    \begin{aligned}
      \det[\nabla_s\mu(I_3-h^\flat W^\flat)^2(\nabla_s\mu)^T] &= \det[A_hA_h^T] = \det[A(I_3-h^\flat W^\flat)^2A^T] \\
      &= \det[(I_3-h^\flat W^\flat)^2]\det[AA^T] \\
      &= J(\mu,h^\flat)^2\det\theta.
    \end{aligned}
  \end{align*}
  Note that $A$ and $I_3-h^\flat W^\flat$ are $3\times3$ matrices.
  Hence \eqref{Pf_CoVP:Later} is valid.
\end{proof}

\begin{remark} \label{R:Cov_Para}
  In the proof of Lemma \ref{L:CoV_Para} we used \eqref{E:ConDer_Dom} which we will prove in Appendix \ref{S:Ap_Proof}.
  We note that we do not apply Lemma \ref{L:CoV_Para} to show \eqref{E:ConDer_Dom}.
\end{remark}

Let us give a regularity result for a Killing vector field on $\Gamma$, i.e. a tangential vector field $v$ on $\Gamma$ satisfying
\begin{align*}
  D_\Gamma(v) = P(\nabla_\Gamma v)_SP = 0 \quad\text{on}\quad \Gamma, \quad (\nabla_\Gamma v)_S = \frac{\nabla_\Gamma v+(\nabla_\Gamma v)^T}{2}.
\end{align*}

\begin{lemma} \label{L:Kil_Reg}
  If $\Gamma$ is of class $C^\ell$ with $\ell\geq3$ and $v\in H^1(\Gamma,T\Gamma)$ satisfies $D_\Gamma(v)=0$ on $\Gamma$, then $v$ is of class $C^{\ell-3}$ on $\Gamma$.
  In particular, $v$ is smooth if $\Gamma$ is smooth.
\end{lemma}

\begin{proof}
  Let $v\in H^1(\Gamma,T\Gamma)$ satisfy $D_\Gamma(v)=0$ on $\Gamma$.
  Then since
  \begin{align} \label{Pf_KiRe:Sym_TGr}
    \begin{aligned}
      \nabla_\Gamma v+(\nabla_\Gamma v)^T &= 2D_\Gamma(v)+(Wv)\otimes n+n\otimes(Wv) \\
      &= (Wv)\otimes n+n\otimes(Wv)
    \end{aligned}
  \end{align}
  on $\Gamma$ by \eqref{E:Grad_W} and $P^T=P$ on $\Gamma$, we have
  \begin{align*}
    \mathrm{div}_\Gamma v = \frac{1}{2}\mathrm{tr}\Bigl[\nabla_\Gamma v+(\nabla_\Gamma v)^T\Bigr] = (Wv)\cdot n = 0 \quad\text{on}\quad \Gamma.
  \end{align*}
  For $i=1,2,3$ we show that the $i$-th component of $v$ satisfies Poisson's equation
  \begin{align} \label{Pf_KiRe:Poi}
    \Delta_\Gamma v_i = \nabla_\Gamma(n_iH)\cdot v \quad\text{on}\quad \Gamma
  \end{align}
  in the weak sense.
  By \eqref{Pf_KiRe:Sym_TGr} we have
  \begin{align*}
    \nabla_\Gamma v_i = -\underline{D}_iv+n_iWv+[Wv]_in \quad\text{on}\quad \Gamma, \quad \underline{D}_iv=(\underline{D}_iv_1,\underline{D}_iv_2,\underline{D}_iv_3)^T.
  \end{align*}
  Let $\xi\in C^2(\Gamma)$.
  By the above equality, \eqref{E:P_TGr}, and $W^T=W$ on $\Gamma$ we have
  \begin{align} \label{Pf_KiRe:TGr_Vi}
    (\nabla_\Gamma v_i,\nabla_\Gamma\xi)_{L^2(\Gamma)} = -(\underline{D}_iv,\nabla_\Gamma\xi)_{L^2(\Gamma)}+(n_iv,W\nabla_\Gamma\xi)_{L^2(\Gamma)}.
  \end{align}
  Let us compute $J_1:=-(\underline{D}_iv,\nabla_\Gamma\xi)_{L^2(\Gamma)}$.
  By \eqref{E:Def_WTD} with $\eta=v_k$ and $\eta_i=\underline{D}_iv_k$,
  \begin{align*}
    J_1 = -\sum_{k=1}^3(\underline{D}_iv_k,\underline{D}_k\xi)_{L^2(\Gamma)} = \sum_{k=1}^3\{(v_k,\underline{D}_i\underline{D}_k\xi)_{L^2(\Gamma)}+(v_kHn_i,\underline{D}_k\xi)_{L^2(\Gamma)}\}.
  \end{align*}
  To the first term on the right-hand side we apply \eqref{E:TD_Exc}.
  Then we get
  \begin{align*}
    \sum_{k=1}^3(v_k,\underline{D}_i\underline{D}_k\xi)_{L^2(\Gamma)} &= \sum_{k=1}^3(v_k,\underline{D}_k\underline{D}_i\xi+[W\nabla_\Gamma\xi]_in_k-[W\nabla_\Gamma\xi]_kn_i)_{L^2(\Gamma)}.
  \end{align*}
  Moreover, by \eqref{E:Def_WTD}, $\mathrm{div}_\Gamma v=0$, and $v\cdot n=0$ on $\Gamma$,
  \begin{align*}
    \sum_{k=1}^3(v_k,\underline{D}_k\underline{D}_i\xi)_{L^2(\Gamma)} &= -\sum_{k=1}^3(\underline{D}_kv_k+v_kHn_k,\underline{D}_i\xi)_{L^2(\Gamma)} \\
    &= -(\mathrm{div}_\Gamma v+(v\cdot n)H,\underline{D}_i\xi)_{L^2(\Gamma)} = 0, \\
    \sum_{k=1}^3(v_k,[W\nabla_\Gamma\xi]_in_k)_{L^2(\Gamma)} &= (v\cdot n,[W\nabla_\Gamma\xi]_i)_{L^2(\Gamma)} = 0.
  \end{align*}
  By these equalities and $\sum_{k=1}^3(v_k,[W_\Gamma\xi]_kn_i)_{L^2(\Gamma)}=(n_iv,W\nabla_\Gamma\xi)_{L^2(\Gamma)}$ we get
  \begin{align} \label{Pf_KiRe:J1}
    J_1 = -(n_iv,W\nabla_\Gamma\xi)_{L^2(\Gamma)}+\sum_{k=1}^3(v_kHn_i,\underline{D}_k\xi)_{L^2(\Gamma)}.
  \end{align}
  To the last term we again use \eqref{E:Def_WTD}, $\mathrm{div}_\Gamma v=0$, and $v\cdot n=0$ on $\Gamma$.
  Then
  \begin{align*}
    \sum_{k=1}^3(v_kHn_i,\underline{D}_k\xi)_{L^2(\Gamma)} &= -\sum_{k=1}^3(\underline{D}_k(v_kHn_i)+(v_kHn_i)Hn_k,\xi)_{L^2(\Gamma)} \\
    &= -(\mathrm{div}_\Gamma(n_iHv)+n_iH^2(v\cdot n),\xi)_{L^2(\Gamma)} \\
    &= -(\nabla_\Gamma(n_iH)\cdot v,\xi)_{L^2(\Gamma)}.
  \end{align*}
  Combining this equality with \eqref{Pf_KiRe:TGr_Vi} and \eqref{Pf_KiRe:J1} we obtain
  \begin{align*}
    (\nabla_\Gamma v_i,\nabla_\Gamma\xi)_{L^2(\Gamma)} = -(\nabla_\Gamma(n_iH)\cdot v,\xi)_{L^2(\Gamma)} \quad\text{for all}\quad \xi \in C^2(\Gamma).
  \end{align*}
  Since $C^2(\Gamma)$ is dense in $H^1(\Gamma)$ by Lemma \ref{L:Wmp_Appr}, this equality is valid for all $\xi\in H^1(\Gamma)$.
  Hence $v_i$ satisfies \eqref{Pf_KiRe:Poi} in the weak sense for $i=1,2,3$.

  Now we recall that $\Gamma$ is assumed to be of class $C^\ell$ with $\ell\geq3$.
  Let us prove the $C^{\ell-3}$-regularity of $v$ on $\Gamma$.
  By a localization argument with a partition of unity on $\Gamma$ we may assume that $v$ is compactly supported in a relatively open subset $\mu(U)$ of $\Gamma$, where $U$ is an open set in $\mathbb{R}^2$ and $\mu\colon U\to\Gamma$ is a $C^\ell$ local parametrization of $\Gamma$.
  Let $\theta$ be the Riemannian metric of $\Gamma$ on $U$, $\theta^{-1}=(\theta^{kl})_{k,l}$ its inverse, and
  \begin{align*}
    v_i^\flat(s) := v_i(\mu(s)), \quad b_i^j(s) := [\underline{D}_j(n_iH)](\mu(s)), \quad s\in U,\,i,j=1,2,3.
  \end{align*}
  For $i=1,2,3$ the function $v_i^\flat$ is compactly supported in $U$ and belongs to $H^1(U)$ by Lemma \ref{L:Lp_Loc}.
  Moreover, by \eqref{Pf_KiRe:Poi} it is a weak solution to the elliptic equation
  \begin{align*}
    \frac{1}{\sqrt{\det\theta}}\sum_{k,l=1}^2\partial_{s_k}\Bigl(\theta^{kl}(\partial_{s_l}v_i^\flat)\sqrt{\det\theta}\Bigr) = \sum_{j=1}^3b_i^jv_j^\flat \quad\text{on}\quad U.
  \end{align*}
  Noting that $\theta,\theta^{-1}\in C^{\ell-1}(U)^{2\times2}$ and $b_i^j\in C^{\ell-3}(U)$ by the $C^\ell$-regularity of $\Gamma$ (see Section \ref{SS:Pre_Surf}) and $v_i^\flat$ is compactly supported in $U$, we apply the elliptic regularity theorem (see \cites{Ev10,GiTr01}) and a bootstrap argument to the above equation to get $v_i^\flat\in H^{\ell-1}(U)$.
  Hence the Sobolev embedding theorem (see \cite{AdFo03}) yields $v_i^\flat\in C^{\ell-3}(U)$ for $i=1,2,3$, which implies the $C^{\ell-3}$-regularity of $v=(v_1,v_2,v_3)^T$ on $\Gamma$ (note that $v$ is compactly supported in $\mu(U)$ and $\mu$ is of class $C^\ell$).
\end{proof}

\begin{remark} \label{R:Kil_Reg}
  We used \eqref{E:TD_Exc} and Lemma \ref{L:Wmp_Appr} to prove Lemma \ref{L:Kil_Reg}.
  In Appendix \ref{S:Ap_Proof} below we show \eqref{E:TD_Exc} and Lemma \ref{L:Wmp_Appr} without applying Lemma \ref{L:Kil_Reg}.
\end{remark}

Finally, we show that the perfect slip boundary conditions
\begin{align*}
  u\cdot n = 0, \quad 2PD(u)n = 0 \quad\text{on}\quad \Gamma
\end{align*}
are different by the curvatures of $\Gamma$ from the Hodge boundary conditions
\begin{align*}
  u\cdot n = 0, \quad \mathrm{curl}\,u\times n = 0 \quad\text{on}\quad \Gamma
\end{align*}
for a vector field $u\colon\Omega\to\mathbb{R}^3$, where $\Omega$ is a bounded domain in $\mathbb{R}^3$ with $\partial\Omega=\Gamma$.

\begin{lemma} \label{L:Diff_SH}
  Suppose that $u\in C^1(\overline{\Omega})^3\cup H^2(\Omega)^3$ satisfies $u\cdot n=0$ on $\Gamma$.
  Then
  \begin{align} \label{E:Diff_SH}
    2PD(u)n-\mathrm{curl}\,u\times n = 2Wu \quad\text{on}\quad \Gamma.
  \end{align}
  Here $W=-\nabla_\Gamma n$ is the Weingarten map of $\Gamma$.
\end{lemma}

\begin{proof}
  First note that $\mathrm{curl}\,u\times n$ is tangential on $\Gamma$.
  By this fact,
  \begin{align*}
    2D(u) = \nabla u+(\nabla u)^T, \quad \mathrm{curl}\,u\times n = (\nabla u)^Tn-(\nabla u)n
  \end{align*}
  under our notation for $\nabla u$ given in Appendix \ref{S:Ap_Vec}, and \eqref{E:Tgrad_Surf} we have
  \begin{align*}
    2PD(u)n-\mathrm{curl}\,u\times n = P\{2D(u)n-\mathrm{curl}\,u\times n\} = 2P(\nabla u)n = 2(\nabla_\Gamma u)n
  \end{align*}
  on $\Gamma$.
  Noting that $u\cdot n=0$ on $\Gamma$, we apply \eqref{E:Grad_W} to the last term to get \eqref{E:Diff_SH}.
\end{proof}

\section{Proofs of auxiliary lemmas} \label{S:Ap_Proof}
The purpose of this appendix is to give the proofs of the lemmas in Section \ref{S:Pre} and Lemmas \ref{L:CoV_Fixed}, \ref{L:KAux_Du}, and \ref{L:G_Bound} involving elementary but long calculations of differential geometry of the surfaces $\Gamma$, $\Gamma_\varepsilon^0$, and $\Gamma_\varepsilon^1$.

As in Appendix \ref{S:Ap_RCS}, let $\Gamma$ be a closed, connected, and oriented surface in $\mathbb{R}^3$ of class $C^\ell$ with $\ell\geq2$.
First we prove the lemmas in Section \ref{SS:Pre_Surf}.

\begin{proof}[Proof of Lemma \ref{L:TD_Exc}]
  For $\eta\in C^2(\Gamma)$ let $\tilde{\eta}$ be an arbitrary $C^2$-extension of $\eta$ to $N$ with $\tilde{\eta}|_\Gamma=\eta$ and
  \begin{align*}
    \varphi_j(x) := \sum_{l=1}^3\{\delta_{jl}-\partial_jd(x)\partial_ld(x)\}\partial_l\tilde{\eta}(x), \quad x\in N,\, j=1,2,3,
  \end{align*}
  where $\delta_{jl}$ is the Kronecker delta.
  Then since
  \begin{align} \label{Pf_TDE:P}
    \delta_{jl}-\partial_jd(y)\partial_ld(y) = \delta_{jl}-n_j(y)n_l(y) = P_{jl}(y), \quad y\in\Gamma
  \end{align}
  by \eqref{E:Nor_Coord}, we have $\varphi_j|_\Gamma=\underline{D}_j\eta$, i.e. $\varphi_j$ is a $C^1$-extension of $\underline{D}_j\eta$ to $N$.
  Therefore,
  \begin{align*}
    \underline{D}_i\underline{D}_j\eta(y) = \sum_{k=1}^3\{\delta_{ik}-\partial_id(y)\partial_kd(y)\}\partial_k\varphi_j(y) = \sum_{m=1}^3\Phi_m(y)
  \end{align*}
  for $y\in\Gamma$ and $i,j=1,2,3$, where
  \begin{align*}
    \Phi_1(y) &:= \sum_{k,l=1}^3\{\delta_{ik}-\partial_id(y)\partial_kd(y)\}\{\delta_{jl}-\partial_jd(y)\partial_ld(y)\}\partial_k\partial_l\tilde{\eta}(y), \\
    \Phi_2(y) &:= -\sum_{k,l=1}^3\{\delta_{ik}-\partial_id(y)\partial_kd(y)\}\partial_k\partial_jd(y)\partial_ld(y)\partial_l\tilde{\eta}(y), \\
    \Phi_3(y) &:= -\sum_{k,l=1}^3\{\delta_{ik}-\partial_id(y)\partial_kd(y)\}\partial_jd(y)\partial_k\partial_ld(y)\partial_l\tilde{\eta}(y).
  \end{align*}
  Similarly, $\underline{D}_j\underline{D}_i\eta(y)=\sum_{m=1}^3\Psi_m(y)$ for $y\in\Gamma$, where
  \begin{align*}
    \Psi_1(y) &:= \sum_{k,l=1}^3\{\delta_{jl}-\partial_jd(y)\partial_ld(y)\}\{\delta_{ik}-\partial_id(y)\partial_kd(y)\}\partial_l\partial_k\tilde{\eta}(y), \\
    \Psi_2(y) &:= -\sum_{k,l=1}^3\{\delta_{jl}-\partial_jd(y)\partial_ld(y)\}\partial_l\partial_id(y)\partial_kd(y)\partial_k\tilde{\eta}(y), \\
    \Psi_3(y) &:= -\sum_{k,l=1}^3\{\delta_{jl}-\partial_jd(y)\partial_ld(y)\}\partial_id(y)\partial_l\partial_kd(y)\partial_k\tilde{\eta}(y).
  \end{align*}
  In what follows, we fix and suppress the argument $y\in\Gamma$.
  Since $\partial_k\partial_l\tilde{\eta}=\partial_l\partial_k\tilde{\eta}$, we have $\Phi_1=\Psi_1$.
  By \eqref{E:Nor_Coord} and \eqref{E:ConDer_Surf} we observe that
  \begin{align} \label{Pf_TDE:W}
    \nabla^2d = \nabla\bar{n} = -W, \quad\text{i.e.}\quad \partial_k\partial_jd = -W_{kj}.
  \end{align}
  Hence it follows from \eqref{E:Nor_Coord}, \eqref{E:Form_W}, \eqref{Pf_TDE:P}, and \eqref{Pf_TDE:W} that
  \begin{align*}
    \Phi_2 = \sum_{k,l=1}^3P_{ik}W_{kj}n_l\partial_l\tilde{\eta} = (n\cdot\nabla\tilde{\eta})[PW]_{ij} = (n\cdot\nabla\tilde{\eta})W_{ij}, \quad \Psi_2 = (n\cdot\nabla\tilde{\eta})W_{ji}
  \end{align*}
  and thus $\Phi_2=\Psi_2$ since $W$ is symmetric by Lemma \ref{L:Form_W}.
  We also have
  \begin{align*}
    \Phi_3 = \sum_{k,l=1}^3P_{ik}n_jW_{kl}\partial_l\tilde{\eta} = [PW\nabla\tilde{\eta}]_in_j = [W\nabla_\Gamma\eta]_in_j, \quad \Psi_3 = [W\nabla_\Gamma\eta]_jn_i
  \end{align*}
  by \eqref{E:Nor_Coord}, \eqref{Pf_TDE:P}, \eqref{Pf_TDE:W}, and $PW\nabla\tilde{\eta}=WP\nabla\tilde{\eta} = W\nabla_\Gamma\eta$.
  Therefore,
  \begin{align*}
    \underline{D}_i\underline{D}_j\eta-\underline{D}_j\underline{D}_i\eta = \sum_{m=1}^3\Phi_m-\sum_{m=1}^3\Psi_m = [W\nabla_\Gamma\eta]_in_j-[W\nabla_\Gamma\eta]_jn_i
  \end{align*}
  and the equality \eqref{E:TD_Exc} is valid.
\end{proof}

\begin{proof}[Proof of Lemma \ref{L:Wein}]
  Since $W$ has the eigenvalues zero, $\kappa_1$, and $\kappa_2$,
  \begin{align*}
    \det[I_3-rW(y)] = \{1-r\kappa_1(y)\}\{1-r\kappa_2(y)\} > 0, \quad y\in\Gamma,\,r\in(-\delta,\delta)
  \end{align*}
  by \eqref{E:Curv_Bound}.
  Hence $I_3-rW(y)$ is invertible.
  Also, \eqref{E:WReso_P} follows from \eqref{E:Form_W}.

  Let us prove \eqref{E:Wein_Bound} and \eqref{E:Wein_Diff}.
  We fix and suppress $y\in\Gamma$.
  Since $W$ is real and symmetric by Lemma \ref{L:Form_W} and has the eigenvalues $\kappa_1$, $\kappa_2$, and zero with $Wn=0$, we can take an orthonormal basis $\{\tau_1,\tau_2,n\}$ of $\mathbb{R}^3$ such that $W\tau_i=\kappa_i\tau_i$, $i=1,2$.
  Then for $r\in(-\delta,\delta)$, $i=1,2$, and $k=\pm1$ we have
  \begin{align} \label{Pf_We:Inv}
    (I_3-rW)^k\tau_i = (1-r\kappa_i)^k\tau_i, \quad (I_3-rW)^kn = n.
  \end{align}
  Since $\{\tau_1,\tau_2,n\}$ is an orthonormal basis of $\mathbb{R}^3$, these formulas imply that
  \begin{align*}
    (I_3-rW)^ka &= \sum_{i=1,2}(a\cdot\tau_i)(I_3-rW)^k\tau_i+(a\cdot n)(I_3-rW)^kn \\
    &= \sum_{i=1,2}(a\cdot\tau_i)(1-r\kappa_i)^k\tau_i+(a\cdot n)n
  \end{align*}
  for all $a\in\mathbb{R}^3$ and $k=\pm1$.
  Hence
  \begin{align*}
    \bigl|(I_3-rW)^ka\bigr|^2 = \sum_{i=1,2}(a\cdot\tau_i)^2(1-r\kappa_i)^{2k}+(a\cdot n)^2
  \end{align*}
  and \eqref{E:Wein_Bound} follows from \eqref{E:Curv_Bound} and $|a|^2=(a\cdot\tau_1)^2+(a\cdot\tau_2)^2+(a\cdot n)^2$.
  Also,
  \begin{align*}
    \bigl|I_3-(I_3-rW)^{-1}\bigr|^2 = \sum_{i=1,2}|1-(1-r\kappa_i)^{-1}|^2 = \sum_{i=1,2}|r\kappa_i(1-r\kappa_i)^{-1}|^2 \leq c|r|^2
  \end{align*}
  by \eqref{Pf_We:Inv}, $|\tau_1|=|\tau_2|=1$, and \eqref{E:Curv_Bound}.
  Hence \eqref{E:Wein_Diff} is valid.
\end{proof}

\begin{proof}[Proof of Lemma \ref{L:Pi_Der}]
  Let $x\in N$.
  Since $\pi(x)\in\Gamma$ we have
  \begin{gather*}
    \pi(x) = x-d(x)n(\pi(x)) = x-d(x)\bar{n}(\pi(x)), \\
    -\nabla\bar{n}(\pi(x)) = -\overline{\nabla_\Gamma n}(\pi(x)) = \overline{W}(\pi(x)) = \overline{W}(x)
  \end{gather*}
  by \eqref{E:Nor_Coord} and \eqref{E:ConDer_Surf}.
  We differentiate both sides of the first equality with respect to $x$ and use the second equality and $\nabla d(x)=\bar{n}(x)=\bar{n}(\pi(x))$ to get
  \begin{align*}
    \nabla\pi(x) = I_3-\bar{n}(x)\otimes\bar{n}(x)-d(x)\nabla\pi(x)\nabla\bar{n}(\pi(x)) = \overline{P}(x)+d(x)\nabla\pi(x)\overline{W}(x)
  \end{align*}
  for $x\in N$ and thus
  \begin{align*}
    \nabla\pi(x)\left\{I_3-d(x)\overline{W}(x)\right\} = \overline{P}(x), \quad x\in N.
  \end{align*}
  Since $I_3-d\overline{W}$ is invertible in $N$ by Lemma \ref{L:Wein}, we obtain \eqref{E:Pi_Der} by the above equality and \eqref{E:WReso_P}.

  Let $\eta\in C^1(\Gamma)$ and $\bar{\eta}=\eta\circ\pi$ be its constant extension.
  Since $\bar{\eta}(x)=\bar{\eta}(\pi(x))$ and $\pi(x)\in\Gamma$ for $x\in N$, we observe by \eqref{E:ConDer_Surf} and \eqref{E:Pi_Der} that
  \begin{align*}
    \nabla\bar{\eta}(x) &= \nabla\pi(x)\nabla\bar{\eta}(\pi(x)) = \left\{I_3-d(x)\overline{W}(x)\right\}^{-1}\overline{P}(x)\overline{\nabla_\Gamma\eta}(\pi(x)) \\
    &= \left\{I_3-d(x)\overline{W}(x)\right\}^{-1}\overline{P}(x)\overline{\nabla_\Gamma\eta}(x).
  \end{align*}
  Hence we obtain \eqref{E:ConDer_Dom} by applying \eqref{E:P_TGr} to the last line of the above equality.
  We also have \eqref{E:ConDer_Bound} and \eqref{E:ConDer_Diff} by \eqref{E:Wein_Bound}, \eqref{E:Wein_Diff}, and \eqref{E:ConDer_Dom}.

  Now let $\Gamma$ be of class $C^3$ and $\eta\in C^2(\Gamma)$.
  For $i=1,2,3$ we differentiate both sides of \eqref{E:ConDer_Dom} with respect to $x_i$ to get
  \begin{align} \label{Pf_PD:Second}
    \partial_i\nabla\bar{\eta} = \left\{\partial_i\Bigl(I_3-d\overline{W}\Bigr)^{-1}\right\}\overline{\nabla_\Gamma\eta}+\Bigl(I_3-d\overline{W}\Bigr)^{-1}\partial_i\Bigl(\overline{\nabla_\Gamma\eta}\Bigr) \quad\text{in}\quad N.
  \end{align}
  To estimate the right-hand side we differentiate both sides of
  \begin{align*}
    \left\{I_3-d(x)\overline{W}(x)\right\}^{-1}\left\{I_3-d(x)\overline{W}(x)\right\} = I_3, \quad x\in N
  \end{align*}
  with respect to $x_i$ and use $\nabla d(x)=\bar{n}(x)$ to get
  \begin{align} \label{Pf_PD:Der_Res}
    \partial_i\Bigl(I_3-d\overline{W}\Bigr)^{-1} = \Bigl(I_3-d\overline{W}\Bigr)^{-1}\Bigl(\bar{n}_i\overline{W}+d\partial_i\overline{W}\Bigr)\Bigl(I_3-d\overline{W}\Bigr)^{-1} \quad\text{in}\quad N.
  \end{align}
  The right-hand side of \eqref{Pf_PD:Der_Res} is bounded on $N$ by \eqref{E:Wein_Bound} and \eqref{E:ConDer_Bound} since $W$ is of class $C^1$ on $\Gamma$ by the $C^3$-regularity of $\Gamma$.
  By this fact, \eqref{E:Wein_Bound}, \eqref{E:ConDer_Bound}, and \eqref{Pf_PD:Second},
  \begin{align*}
    |\partial_i\nabla\bar{\eta}| \leq c\left(\left|\overline{\nabla_\Gamma\eta}\right|+\left|\overline{\nabla_\Gamma^2\eta}\right|\right) \quad\text{in}\quad N,\,i=1,2,3,
  \end{align*}
  which shows \eqref{E:Con_Hess}.
\end{proof}

\begin{proof}[Proof of Lemma \ref{L:Wmp_Appr}]
  Here we only show the density of $C^\ell(\Gamma)$ in $W^{m,p}(\Gamma)$ for $\ell=m=2$ and $p\in[1,\infty)$.
  The assertion in other cases are proved similarly.

  Let $\eta\in W^{2,p}(\Gamma)$.
  Since $\Gamma$ is closed and of class $C^2$, by a localization argument with a partition of unity on $\Gamma$ we may assume that there exist an open set $U$ in $\mathbb{R}^2$, a compact subset $\mathcal{K}$ of $U$, and a $C^2$ local parametrization $\mu\colon U\to\Gamma$ of $\Gamma$ such that $\eta$ is supported in $\mu(\mathcal{K})$.
  Then $\eta^\flat:=\eta\circ\mu$ is supported in $\mathcal{K}$ and belongs to $W^{1,p}(U)$ by Lemma \ref{L:Lp_Loc}.
  Let us prove $\partial_{s_i}\partial_{s_j}\eta^\flat\in L^p(U)$ for $i,j=1,2$.
  In what follows, we write $\xi^\flat:=\xi\circ\mu$ on $U$ for a function $\xi$ on $\Gamma$.
  We show that
  \begin{align} \label{Pf_WA:D2_EF}
    \partial_{s_i}\partial_{s_j}\eta^\flat = \partial_{s_i}\mu\cdot\{(\nabla_\Gamma^2\eta)^\flat\partial_{s_j}\mu\}+\partial_{s_i}\partial_{s_j}\mu\cdot(\nabla_\Gamma\eta)^\flat \in L^p(U), \, i,j=1,2.
  \end{align}
  First note that the right-hand side is in $L^p(U)$.
  Indeed, since
  \begin{align*}
    |\partial_{s_i}\mu\cdot\{(\nabla_\Gamma^2\eta)^\flat\partial_{s_j}\mu\}+\partial_{s_i}\partial_{s_j}\mu\cdot(\nabla_\Gamma\eta)^\flat| \leq c(|(\nabla_\Gamma\eta)^\flat|+|(\nabla_\Gamma^2\eta)^\flat|) \quad\text{on}\quad \mathcal{K}
  \end{align*}
  by \eqref{E:Mu_Bound} and $(\nabla_\Gamma\eta)^\flat$ and $(\nabla_\Gamma^2\eta)^\flat$ are supported in $\mathcal{K}$, we see by \eqref{E:Metric} that
  \begin{align*}
    &\int_U|\partial_{s_i}\mu\cdot\{(\nabla_\Gamma^2\eta)^\flat\partial_{s_j}\mu\}+\partial_{s_i}\partial_{s_j}\mu\cdot(\nabla_\Gamma\eta)^\flat|^p\,ds \\
    &\qquad = \int_{\mathcal{K}}|\partial_{s_i}\mu\cdot\{(\nabla_\Gamma^2\eta)^\flat\partial_{s_j}\mu\}+\partial_{s_i}\partial_{s_j}\mu\cdot(\nabla_\Gamma\eta)^\flat|^p\,ds \\
    &\qquad \leq c\int_{\mathcal{K}}(|(\nabla_\Gamma\eta)^\flat|^p+|(\nabla_\Gamma^2\eta)^\flat|^p)\sqrt{\det\theta}\,ds \\
    &\qquad = c\left(\|\nabla_\Gamma\eta\|_{L^p(\Gamma)}^p+\|\nabla_\Gamma^2\eta\|_{L^p(\Gamma)}^p\right).
  \end{align*}
  Fix $i,j=1,2$.
  Let $\varphi\in C_c^1(U)$ and $X_m\colon\mu(U)\to\mathbb{R}^3$ be given by
  \begin{align*}
    X_m(\mu(s)) := \frac{\varphi(s)\partial_{s_j}\mu_m(s)}{\sqrt{\det\theta(s)}}\partial_{s_i}\mu(s), \quad s\in U, \, m=1,2,3.
  \end{align*}
  Note that here $X_m$ does not stand for the $m$-th component of a vector field $X$.
  We extend $X_m$ to $\Gamma$ by setting zero outside $\mu(U)$.
  Then $X_m\in C^1(\Gamma,T\Gamma)$ and
  \begin{align*}
    \mathrm{div}_\Gamma X_m(\mu(s)) = \frac{\partial_{s_i}\bigl(\varphi(s)\partial_{s_j}\mu_m(s)\bigr)}{\sqrt{\det\theta(s)}} = \frac{\partial_{s_i}\varphi(s)\partial_{s_j}\mu_m(s)}{\sqrt{\det\theta(s)}}+\frac{\varphi(s)\partial_{s_i}\partial_{s_j}\mu_m(s)}{\sqrt{\det\theta(s)}}
  \end{align*}
  for $s\in U$ and $m=1,2,3$ by Lemma \ref{L:DivG_DG}.
  By this equality and \eqref{Pf_LpL:D_Flat} we have
  \begin{align*}
    \int_U(\partial_{s_j}\eta^\flat)\partial_{s_i}\varphi\,ds &= \int_U\{\partial_{s_j}\mu\cdot(\nabla_\Gamma\eta)^\flat\}\partial_{s_i}\varphi\,ds \\
    &= \sum_{m=1}^3\int_U(\underline{D}_m\eta)^\flat(\partial_{s_i}\varphi)\partial_{s_j}\mu_m\,ds = J_1+J_2,
  \end{align*}
  where
  \begin{align*}
    J_1 &:= \sum_{m=1}^3\int_U(\underline{D}_m\eta)^\flat(\mathrm{div}_\Gamma X_m)^\flat\sqrt{\det\theta}\,ds = \sum_{m=1}^3\int_\Gamma(\underline{D}_m\eta)\,\mathrm{div}_\Gamma X_m\,d\mathcal{H}^2, \\
    J_2 &:= -\sum_{m=1}^3\int_U(\underline{D}_m\eta)^\flat(\partial_{s_i}\partial_{s_j}\mu_m)\varphi\,ds = -\int_U\{\partial_{s_i}\partial_{s_j}\mu\cdot(\nabla_\Gamma\eta)^\flat\}\varphi\,ds.
  \end{align*}
  We apply \eqref{E:IbP_WTGr} with $v=X_m$ and $X_m\cdot n=0$ on $\Gamma$ to $J_1$.
  Then
  \begin{align*}
    J_1 &= -\sum_{m=1}^3\int_\Gamma\nabla_\Gamma(\underline{D}_m\eta)\cdot X_m\,d\mathcal{H}^2 = -\sum_{m=1}^3\int_U[\nabla_\Gamma(\underline{D}_m\eta)]^\flat\cdot X_m^\flat\sqrt{\det\theta}\,ds \\
    &= -\sum_{l,m=1}^3\int_U(\underline{D}_l\underline{D}_m\eta)^\flat(\partial_{s_i}\mu_l)(\partial_{s_j}\mu_m)\varphi\,ds = -\int_U[\partial_{s_i}\mu\cdot\{(\nabla_\Gamma^2\eta)^\flat\partial_{s_j}\mu\}]\varphi\,ds.
  \end{align*}
  Therefore,
  \begin{align*}
    \int_U(\partial_{s_j}\eta^\flat)\partial_{s_i}\varphi\,ds = J_1+J_2 = -\int_U[\partial_{s_i}\mu\cdot\{(\nabla_\Gamma^2\eta)^\flat\partial_{s_j}\mu\}+\partial_{s_i}\partial_{s_j}\mu\cdot(\nabla_\Gamma\eta)^\flat]\varphi\,ds
  \end{align*}
  for all $\varphi\in C_c^1(U)$ and we obtain \eqref{Pf_WA:D2_EF} and $\eta^\flat\in W^{2,p}(U)$.
  Now we prove
  \begin{align} \label{Pf_WA:W2p}
    \|\eta\|_{W^{2,p}(\Gamma)} \leq c\|\eta^\flat\|_{W^{2,p}(U)}.
  \end{align}
  Since $\eta\in W^{2,p}(\Gamma)$ and \eqref{E:TGr_DG} holds for a function in $W^{1,p}(\Gamma)$ by Lemma \ref{L:Lp_Loc},
  \begin{align*}
    (\underline{D}_m\eta)^\flat = \sum_{i,j=1}^2\theta^{ij}(\partial_{s_i}\eta^\flat)\partial_{s_j}\mu_m \quad\text{on}\quad U, \, m=1,2,3.
  \end{align*}
  Here both sides belong to $W^{1,p}(U)$ by \eqref{E:Mu_Bound}, \eqref{E:Metric}, and the fact that $\eta\in W^{2,p}(\Gamma)$ and $\eta^\flat\in W^{2,p}(U)$ are supported in $\mu(\mathcal{K})$ and $\mathcal{K}$.
  Hence we can differentiate both sides with respect to $s\in U$ in the weak sense and apply \eqref{E:Mu_Bound} and \eqref{E:Metric} to get
  \begin{align*}
    |\nabla_s(\underline{D}_m\eta)^\flat| \leq c(|\nabla_s\eta^\flat|+|\nabla_s^2\eta^\flat|) \quad\text{on}\quad \mathcal{K},
  \end{align*}
  where $\nabla_s^2\eta^\flat:=(\partial_{s_i}\partial_{s_j}\eta^\flat)_{i,j}$ is the Hessian matrix of $\eta^\flat$.
  Noting that $(\underline{D}_m\eta)^\flat$ and $\eta^\flat$ are supported in $\mathcal{K}$, we deduce from the above inequality that
  \begin{align*}
    \|\nabla_s(\underline{D}_m\eta)^\flat\|_{L^p(U)} \leq c\left(\|\nabla_s\eta^\flat\|_{L^p(U)}+\|\nabla_s^2\eta^\flat\|_{L^p(U)}\right) \leq c\|\eta^\flat\|_{W^{2,p}(U)}.
  \end{align*}
  We use \eqref{E:W1p_Loc} with $\eta$ replaced by $\underline{D}_m\eta\in W^{1,p}(\Gamma)$ and the above inequality to get
  \begin{align*}
    \|\nabla_\Gamma(\underline{D}_m\eta)\|_{L^p(\Gamma)} \leq c\|\nabla_s(\underline{D}_m\eta)^\flat\|_{L^p(U)} \leq c\|\eta^\flat\|_{W^{2,p}(U)}, \quad m=1,2,3.
  \end{align*}
  Hence $\|\nabla_\Gamma^2\eta\|_{L^p(\Gamma)}\leq c\|\eta^\flat\|_{W^{2,p}(U)}$ and \eqref{Pf_WA:W2p} follows from this inequality, \eqref{E:Lp_Loc}, and \eqref{E:W1p_Loc}.
  Now since $\eta^\flat$ belongs to $W^{2,p}(U)$ and is compactly supported in $U$, we can take a sequence $\{\eta_k^\flat\}_{k=1}^\infty$ in $C_c^\infty(U)$ that converges to $\eta^\flat$ strongly in $W^{2,p}(U)$ by a standard mollification argument.
  For each $k\in\mathbb{N}$ we set $\eta_k(\mu(s)):=\eta_k^\flat(s)$, $s\in U$ and extend $\eta_k$ to $\Gamma$ by setting zero outside $\mu(U)$.
  Then $\eta_k\in C^2(\Gamma)$ (note that $\mu$ is of class $C^2$ on $U$) and
  \begin{align*}
    \|\eta-\eta_k\|_{W^{2,p}(\Gamma)} \leq c\|\eta^\flat-\eta_k^\flat\|_{W^{2,p}(U)} \to 0 \quad\text{as}\quad k\to\infty
  \end{align*}
  by \eqref{Pf_WA:W2p}.
  Hence $C^2(\Gamma)$ is dense in $W^{2,p}(\Gamma)$.
\end{proof}

\begin{remark} \label{R:Proof_WA}
  In the proof of Lemma \ref{L:Wmp_Appr} we applied Lemma \ref{L:Lp_Loc}.
  We note that we did not use the density of $C^m(\Gamma)$ in $W^{m,p}(\Gamma)$ with $m=1,2$ to prove Lemma \ref{L:Lp_Loc}, and of course to prove Lemma \ref{L:Wmp_Appr}.
\end{remark}

Next we assume that $\Gamma$ is of class $C^5$ and prove the formulas and inequalities in Section \ref{SS:Pre_Dom} for the surface quantities of the boundaries $\Gamma_\varepsilon^0$ and $\Gamma_\varepsilon^1$ of $\Omega_\varepsilon$.

\begin{proof}[Proof of Lemma \ref{L:NB_Aux}]
  First note that, since $W\in C^3(\Gamma)^{3\times3}$ by the $C^5$-regularity of $\Gamma$ and $g_0,g_1\in C^4(\Gamma)$, they are bounded on $\Gamma$ along with their first and second order tangential derivatives.

  Let $\tau_\varepsilon^i$ and $n_\varepsilon^i$, $i=0,1$ be the vector fields on $\Gamma$ given by \eqref{E:Def_NB_Aux} and \eqref{E:Def_NB}.
  Then the first inequalities of \eqref{E:Tau_Bound} and \eqref{E:Tau_Diff} immediately follow from \eqref{E:Wein_Bound} and \eqref{E:Wein_Diff}.
  To show the second inequalities of \eqref{E:Tau_Bound} and \eqref{E:Tau_Diff} we set
  \begin{align} \label{Pf_NBA:Def_R}
    R_\varepsilon^i(y) := \{I_3-\varepsilon g_i(y)W(y)\}^{-1}, \quad y\in\Gamma
  \end{align}
  and apply $\underline{D}_k$, $k=1,2,3$ to both sides of $R_\varepsilon^i(I_3-\varepsilon g_iW)=I_3$ on $\Gamma$ to get
  \begin{align} \label{Pf_NBA:D_Inv}
    \underline{D}_kR_\varepsilon^i = \varepsilon R_\varepsilon^i\{(\underline{D}_kg_i)W+g_i\underline{D}_kW\}R_\varepsilon^i \quad\text{on}\quad \Gamma.
  \end{align}
  Hence by \eqref{E:Wein_Bound} there exists a constant $c>0$ independent of $\varepsilon$ such that
  \begin{align} \label{Pf_NBA:Est_D_Inv}
    |\underline{D}_kR_\varepsilon^i| \leq c\varepsilon \quad\text{on}\quad \Gamma.
  \end{align}
  Applying \eqref{E:Wein_Bound}, \eqref{E:Wein_Diff}, and \eqref{Pf_NBA:Est_D_Inv} to $\underline{D}_k\tau_\varepsilon^i=(\underline{D}_kR_\varepsilon^i)\nabla_\Gamma g_i+R_\varepsilon^i(\underline{D}_k\nabla_\Gamma g)$ we obtain
  \begin{align*}
    |\underline{D}_k\tau_\varepsilon^i| \leq c, \quad |\underline{D}_k\tau_\varepsilon^i-\underline{D}_k\nabla_\Gamma g| &\leq |(\underline{D}_kR_\varepsilon^i)\nabla_\Gamma g_i|+|(R_\varepsilon^i-I_3)(\underline{D}_k\nabla_\Gamma g)| \leq c\varepsilon
  \end{align*}
  on $\Gamma$ for $k=1,2,3$.
  Hence the second inequalities of \eqref{E:Tau_Bound} and \eqref{E:Tau_Diff} are valid.
  We further apply $\underline{D}_l$, $l=1,2,3$ to both sides of \eqref{Pf_NBA:D_Inv} and use \eqref{E:Wein_Bound} and \eqref{Pf_NBA:Est_D_Inv} to obtain $|\underline{D}_l\underline{D}_kR_\varepsilon^i|\leq c\varepsilon$ on $\Gamma$.
  Using this inequality, \eqref{E:Wein_Bound}, and \eqref{Pf_NBA:Est_D_Inv} to
  \begin{multline*}
    \underline{D}_l\underline{D}_k\tau_\varepsilon^i = (\underline{D}_l\underline{D}_kR_\varepsilon^i)\nabla_\Gamma g_i+(\underline{D}_kR_\varepsilon^i)(\underline{D}_l\nabla_\Gamma g_i) \\+(\underline{D}_lR_\varepsilon^i)(\underline{D}_k\nabla_\Gamma g_i)+R_\varepsilon^i(\underline{D}_l\underline{D}_k\nabla_\Gamma g_i)
  \end{multline*}
  we get $|\underline{D}_l\underline{D}_k\tau_\varepsilon^i|\leq c$ on $\Gamma$ for $k,l=1,2,3$, i.e. the third inequality of \eqref{E:Tau_Bound} holds.

  Next we show \eqref{E:N_Bound} and \eqref{E:N_Diff}.
  The first equality of \eqref{E:N_Bound} is due to \eqref{E:Def_NB}.
  We also have the other inequalities of \eqref{E:N_Bound} by \eqref{E:Tau_Bound}.
  To prove \eqref{E:N_Diff} let
  \begin{align*}
    \varphi_\varepsilon := \frac{1}{\sqrt{1+\varepsilon^2|\tau_\varepsilon^1|^2}}-\frac{1}{\sqrt{1+\varepsilon^2|\tau_\varepsilon^0|^2}}, \quad \tau_\varepsilon := -\frac{\tau_\varepsilon^1}{\sqrt{1+\varepsilon^2|\tau_\varepsilon^1|^2}}+\frac{\tau_\varepsilon^0}{\sqrt{1+\varepsilon^2|\tau_\varepsilon^0|^2}}
  \end{align*}
  so that $n_\varepsilon^0+n_\varepsilon^1=\varphi_\varepsilon n+\varepsilon\tau_\varepsilon$ on $\Gamma$.
  From \eqref{E:Tau_Bound} we deduce that
  \begin{align} \label{Pf_NBA:Est_Tau}
    |\tau_\varepsilon| \leq c, \quad |\nabla_\Gamma\tau_\varepsilon| \leq c \quad\text{on}\quad \Gamma
  \end{align}
  with a constant $c>0$ independent of $\varepsilon$.
  Also,
  \begin{align} \label{Pf_NBA:Est_Phi}
    |\varphi_\varepsilon| \leq \frac{\varepsilon^2}{2}\bigl||\tau_\varepsilon^1|^2-|\tau_\varepsilon^0|^2\bigr| \leq c\varepsilon^2 \quad\text{on}\quad \Gamma
  \end{align}
  by the mean value theorem for $(1+s)^{-1/2}$, $s\geq 0$ and \eqref{E:Tau_Bound}.
  Since
  \begin{align*}
    \nabla_\Gamma\left(\frac{1}{\sqrt{1+\varepsilon^2|\tau_\varepsilon^i|^2}}\right) = -\frac{\varepsilon^2(\nabla_\Gamma\tau_\varepsilon^i)\tau_\varepsilon^i}{(1+\varepsilon^2|\tau_\varepsilon^i|^2)^{3/2}} \quad\text{on}\quad \Gamma,\, i=0,1,
  \end{align*}
  we have $|\nabla_\Gamma\varphi_\varepsilon|\leq c\varepsilon^2$ on $\Gamma$ by \eqref{E:Tau_Bound}.
  Applying this inequality, \eqref{Pf_NBA:Est_Tau}, and \eqref{Pf_NBA:Est_Phi} to $n_\varepsilon^0+n_\varepsilon^1=\varphi_\varepsilon n+\varepsilon\tau_\varepsilon$ and its tangential gradient matrix we obtain \eqref{E:N_Diff}.
\end{proof}

\begin{proof}[Proof of Lemma \ref{L:Comp_Nor}]
  Throughout the proof we write $c$ for a general positive constant independent of $\varepsilon$ and denote by $\bar{\eta}=\eta\circ\pi$ the constant extension of a function $\eta$ on $\Gamma$.
  First note that $\bar{n}$, $\overline{P}$, and $\overline{W}$ are bounded on $N$ independently of $\varepsilon$ along their first and second order derivatives by \eqref{E:ConDer_Bound}, \eqref{E:Con_Hess}, and the $C^5$-regularity of $\Gamma$.
  In the sequel we use this fact without mention.

  For $i=0,1$ let $\tau_\varepsilon^i$ and $n_\varepsilon^i$ be given by \eqref{E:Def_NB_Aux} and \eqref{E:Def_NB}, and
  \begin{align*}
    \varphi_\varepsilon^i(x) := \frac{1}{\sqrt{1+\varepsilon^2|\bar{\tau}_\varepsilon^i(x)|^2}}-1, \quad x\in N.
  \end{align*}
  By the mean value theorem for $(1+s)^{-1/2}$, $s\geq0$ and \eqref{E:Tau_Bound} we have
  \begin{align} \label{Pf_CN:Aux_1}
    |\varphi_\varepsilon^i(x)| \leq \frac{\varepsilon^2}{2}|\bar{\tau}_\varepsilon^i(x)|^2 \leq c\varepsilon^2, \quad x\in N.
  \end{align}
  We also differentiate $\varphi_\varepsilon^i$ once or twice and use \eqref{E:ConDer_Bound}, \eqref{E:Con_Hess}, and \eqref{E:Tau_Bound} to get
  \begin{align} \label{Pf_CN:Aux_1_Der}
    |\partial_x^\alpha\varphi_\varepsilon^i(x)| \leq c\varepsilon^2, \quad x\in N,\,|\alpha|=1,2,
  \end{align}
  where $\partial_x^\alpha=\partial_1^{\alpha_1}\partial_2^{\alpha_2}\partial_3^{\alpha_3}$ for $\alpha=(\alpha_1,\alpha_2,\alpha_3)\in\mathbb{Z}^3$ with $\alpha_j\geq0$, $j=1,2,3$.
  Since
  \begin{align*}
    n_\varepsilon-(-1)^{i+1}\Bigl(\bar{n}-\varepsilon\overline{\nabla_\Gamma g_i}\Bigr) = (-1)^{i+1}\varphi_\varepsilon^i(\bar{n}-\varepsilon\bar{\tau}_\varepsilon^i)-(-1)^{i+1}\varepsilon\Bigl(\bar{\tau}_\varepsilon^i-\overline{\nabla_\Gamma g_i}\Bigr)
  \end{align*}
  on $\Gamma_\varepsilon^i$ by \eqref{E:Def_NB} and \eqref{E:Nor_Bo}, the inequality \eqref{E:Comp_N} follows from \eqref{E:Tau_Bound}, \eqref{E:Tau_Diff}, and \eqref{Pf_CN:Aux_1}.
  We also have \eqref{E:Comp_P} by \eqref{E:Comp_N} and the definitions of $P$, $Q$, $P_\varepsilon$, and $Q_\varepsilon$.

  Next we prove \eqref{E:Comp_W}.
  For $x\in N$ we set
  \begin{align*}
    \Phi_\varepsilon^i(x) := (-1)^{i+1}\left\{\varphi_\varepsilon^i(x)\bar{n}(x)-\frac{\varepsilon\bar{\tau}_\varepsilon^i(x)}{\sqrt{1+\varepsilon^2|\bar{\tau}_\varepsilon^i(x)|^2}}\right\}.
  \end{align*}
  Then it follows from \eqref{E:ConDer_Bound}, \eqref{E:Con_Hess}, \eqref{E:Tau_Bound}, \eqref{Pf_CN:Aux_1}, and \eqref{Pf_CN:Aux_1_Der} that
  \begin{align} \label{Pf_CN:Aux_2}
    |\partial_x^\alpha\Phi_\varepsilon^i(x)| \leq c\varepsilon, \quad x\in N,\,|\alpha|=0,1,2.
  \end{align}
  Since $\bar{n}_\varepsilon^i(x)=(-1)^{i+1}\bar{n}(x)+\Phi_\varepsilon^i(x)$ for $x\in N$, we observe by \eqref{E:Nor_Grad} that
  \begin{align*}
    \nabla\bar{n}_\varepsilon^i(x) = (-1)^i\left\{I_3-d(x)\overline{W}(x)\right\}^{-1}\overline{W}(x)+\nabla\Phi_\varepsilon^i(x), \quad x\in N.
  \end{align*}
  Moreover, since $\bar{n}_\varepsilon^i$ is an extension of $n_\varepsilon|_{\Gamma_\varepsilon^i}$ to $N$, the Weingarten map of $\Gamma_\varepsilon^i$ is given by $W_\varepsilon=-P_\varepsilon\nabla\bar{n}_\varepsilon^i$ on $\Gamma_\varepsilon^i$.
  Thus the above equality yields
  \begin{align} \label{Pf_CN:Wein}
    W_\varepsilon(x) = P_\varepsilon(x)\left\{(-1)^{i+1}\overline{R}_\varepsilon^i(x)\overline{W}(x)-\nabla\Phi_\varepsilon^i(x)\right\}, \quad x\in\Gamma_\varepsilon^i,
  \end{align}
  where $R_\varepsilon^i$ is given by \eqref{Pf_NBA:Def_R}, and it follows from \eqref{E:Form_W} that
  \begin{align*}
    \left|W_\varepsilon-(-1)^{i+1}\overline{W}\right| &\leq \left|\Bigl(P_\varepsilon-\overline{P}\Bigr)\overline{R}_\varepsilon^i\overline{W}\right|+\left|\overline{P}\Bigl(\overline{R}_\varepsilon^i-I_3\Bigr)\overline{W}\right|+|P_\varepsilon\nabla\Phi_\varepsilon^i| \quad\text{on}\quad \Gamma_\varepsilon^i.
  \end{align*}
  Hence we obtain the first inequality of \eqref{E:Comp_W} by applying \eqref{E:Wein_Bound}, \eqref{E:Wein_Diff}, \eqref{E:Comp_P}, and \eqref{Pf_CN:Aux_2} to the above inequality.
  Also, the second inequality of \eqref{E:Comp_W} follows from the first one since $H=\mathrm{tr}[W]$ and $H_\varepsilon=\mathrm{tr}[W_\varepsilon]$.

  Let us show \eqref{E:Comp_DW}.
  Based on \eqref{Pf_CN:Wein} we define an extension of $W_\varepsilon|_{\Gamma_\varepsilon^i}$ to $N$ by
  \begin{align*}
    \widetilde{W}_\varepsilon^i(x) := \overline{P}_\varepsilon^i(x)\left\{(-1)^{i+1}\overline{R}_\varepsilon^i(x)\overline{W}(x)-\nabla\Phi_\varepsilon^i(x)\right\}, \quad x\in N,
  \end{align*}
  where $P_\varepsilon^i:=I_3-n_\varepsilon^i\otimes n_\varepsilon^i$ on $\Gamma$.
  For $x\in N$ let
  \begin{align*}
    E_\varepsilon^i(x) &:= (-1)^{i+1}\left\{\overline{P}_\varepsilon^i(x)-\overline{P}(x)\right\}\overline{R}_\varepsilon^i(x)\overline{W}(x), \\
    F_\varepsilon^i(x) &:= (-1)^{i+1}\overline{P}(x)\left\{\overline{R}_\varepsilon^i(x)-I_3\right\}\overline{W}(x), \\
    G_\varepsilon^i(x) &:= -\overline{P}_\varepsilon^i(x)\nabla\Phi_\varepsilon^i(x)
  \end{align*}
  so that $\widetilde{W}_\varepsilon^i=(-1)^{i+1}\overline{W}+E_\varepsilon^i+F_\varepsilon^i+G_\varepsilon^i$ in $N$ by \eqref{E:Form_W}.
  Then by \eqref{E:ConDer_Dom} we get
  \begin{align} \label{Pf_CN:Der_Wein}
    \partial_j\widetilde{W}_\varepsilon^i = \sum_{k=1}^3(-1)^{i+1}\left[\Bigl(I_3-d\overline{W}\Bigr)^{-1}\right]_{jk}\overline{\underline{D}_kW}+\partial_jE_\varepsilon^i+\partial_jF_\varepsilon^i+\partial_jG_\varepsilon^i
  \end{align}
  in $N$ for $j=1,2,3$.
  To estimate the last three terms we see that
  \begin{align*}
    \overline{P}_\varepsilon^i-\overline{P} = (-1)^i(\bar{n}\otimes\Phi_\varepsilon^i+\Phi_\varepsilon^i\otimes\bar{n})-\Phi_\varepsilon^i\otimes\Phi_\varepsilon^i \quad\text{in}\quad N
  \end{align*}
  by $\bar{n}_\varepsilon^i=(-1)^{i+1}\bar{n}+\Phi_\varepsilon^i$ in $N$ and the definitions of $P$ and $P_\varepsilon^i$.
  Hence
  \begin{align*}
    \left|\overline{P}_\varepsilon^i-\overline{P}\right| \leq c\varepsilon, \quad \left|\partial_j\overline{P}_\varepsilon^i-\partial_j\overline{P}\right| \leq c\varepsilon \quad\text{in}\quad N,\, j=1,2,3
  \end{align*}
  by \eqref{Pf_CN:Aux_2}.
  These inequalities, \eqref{E:Wein_Bound}, \eqref{E:Wein_Diff}, \eqref{E:ConDer_Bound}, \eqref{Pf_NBA:Est_D_Inv}, and \eqref{Pf_CN:Aux_2} show that
  \begin{align*}
    |\partial_jE_\varepsilon^i| \leq c\varepsilon, \quad |\partial_jF_\varepsilon^i| \leq c\varepsilon, \quad |\partial_jG_\varepsilon^i| \leq c\varepsilon \quad\text{in}\quad N.
  \end{align*}
  Applying \eqref{E:Wein_Diff} and the above inequalities to \eqref{Pf_CN:Der_Wein} we get
  \begin{align} \label{Pf_CN:Diff_DW}
    \left|\partial_j\widetilde{W}_\varepsilon^i(x)-(-1)^{i+1}\overline{\underline{D}_jW}(x)\right| \leq c(|d(x)|+\varepsilon), \quad x\in N,\, j=1,2,3.
  \end{align}
  Now we observe that $\underline{D}_j^\varepsilon W_\varepsilon=\sum_{k=1}^3[P_\varepsilon]_{jk}\partial_k\widetilde{W}_\varepsilon^i$ on $\Gamma_\varepsilon^i$ since $\widetilde{W}_\varepsilon^i$ is an extension of $W_\varepsilon|_{\Gamma_\varepsilon^i}$ to $N$.
  From this fact and $\underline{D}_jW=\sum_{k=1}^3P_{jk}\underline{D}_kW$ on $\Gamma$ by \eqref{E:P_TGr} we have
  \begin{multline*}
    \left|\underline{D}_j^\varepsilon W_\varepsilon-(-1)^{i+1}\overline{\underline{D}_jW}\right| \\
    \leq \sum_{k=1}^3\left(\left|\Bigl[P_\varepsilon-\overline{P}\Bigr]_{jk}\partial_k\widetilde{W}_\varepsilon^i\right|+\left|\overline{P}_{jk}\left\{\partial_k\widetilde{W}_\varepsilon^i-(-1)^{i+1}\overline{\underline{D}_kW}\right\}\right|\right)
  \end{multline*}
  on $\Gamma_\varepsilon^i$ for $j=1,2,3$.
  Applying \eqref{E:Comp_P} and \eqref{Pf_CN:Diff_DW} with $|d|=\varepsilon|\bar{g}_i|\leq c\varepsilon$ on $\Gamma_\varepsilon^i$ to the right-hand side we conclude that \eqref{E:Comp_DW} is valid.
\end{proof}

\begin{proof}[Proof of Lemma \ref{L:Diff_SQ_IO}]
  Let $\overline{P}=P\circ\pi$ be the constant extension of $P$.
  Since
  \begin{align*}
    \overline{P}(y+\varepsilon g_0(y)n(y)) = \overline{P}(y+\varepsilon g_1(y)n(y)) = P(y), \quad y\in\Gamma,
  \end{align*}
  we observe that
  \begin{align*}
    |P_\varepsilon(y+\varepsilon g_1(y)n(y))-P_\varepsilon(y+\varepsilon g_0(y)n(y))| \leq \sum_{i=0,1}\left|\Bigl[P_\varepsilon-\overline{P}\Bigr](y+\varepsilon g_i(y)n(y))\right|.
  \end{align*}
  To the right-hand side we apply \eqref{E:Comp_P} to get \eqref{E:Diff_PQ_IO} for $F_\varepsilon=P_\varepsilon$.
  Using \eqref{E:Comp_P}--\eqref{E:Comp_DW} we can show the other inequalities in the same way.
\end{proof}

Now let us give the proofs of Lemmas \ref{L:CoV_Fixed} and \ref{L:KAux_Du}.

\begin{proof}[Proof of Lemma \ref{L:CoV_Fixed}]
  Let $\Phi_\varepsilon$ be the mapping given by \eqref{E:Def_Bij}, i.e.
  \begin{align*}
    \Phi_\varepsilon(X) := \pi(X)+\varepsilon d(X)\bar{n}(X), \quad X \in \Omega_1.
  \end{align*}
  By \eqref{E:Nor_Coord} we easily see that $\Phi_\varepsilon$ is a bijection from $\Omega_1$ onto $\Omega_\varepsilon$ and its inverse is
  \begin{align*}
    \Phi_\varepsilon^{-1}(x) := \pi(x)+\varepsilon^{-1}d(x)\bar{n}(x), \quad x\in\Omega_\varepsilon.
  \end{align*}
  Moreover, by $\nabla d=\bar{n}$ in $N$, \eqref{E:Form_W}, \eqref{E:Pi_Der}, and \eqref{E:Nor_Grad} we have
  \begin{align*}
    \nabla\Phi_\varepsilon(X) = \left\{I_3-d(X)\overline{W}(X)\right\}^{-1}\left\{I_3-\varepsilon d(X)\overline{W}(X)\right\}\overline{P}(X)+\varepsilon\overline{Q}(X), \quad X\in\Omega_1.
  \end{align*}
  Since $W$ has the eigenvalues zero, $\kappa_1$, and $\kappa_2$ with $Wn=0$ on $\Gamma$, for each $X\in\Omega_1$ we can take an orthonormal basis $\{\tau_1,\tau_2,\bar{n}(X)\}$ such that
  \begin{align*}
    \overline{W}(X)\tau_i = \bar{\kappa}_i(X)\tau_i, \quad \overline{P}(X)\tau_i = \tau_i, \quad \overline{Q}(X)\tau_i = 0, \quad i=1,2.
  \end{align*}
  Then for $i=1,2$ we have
  \begin{align*}
    [\nabla\Phi_\varepsilon(X)]\tau_i &= \{1-d(X)\bar{\kappa}_i(X)\}^{-1}\{1-\varepsilon d(X)\bar{\kappa}_i(X)\}\tau_i \\
    &= \{1-d(X)\kappa_i(\pi(X))\}^{-1}\{1-\varepsilon d(X)\kappa_i(\pi(X))\}\tau_i.
  \end{align*}
  Also, $[\nabla\Phi_\varepsilon(X)]\bar{n}(X)=\varepsilon\bar{n}(X)$ by $Pn=0$ and $Qn=n$ on $\Gamma$.
  Thus
  \begin{align*}
    \det\nabla\Phi_\varepsilon(X) = \varepsilon J(\pi(X),d(X))^{-1}J(\pi(X),\varepsilon d(X)), \quad X\in\Omega_1
  \end{align*}
  and the change of variables formula \eqref{E:CoV_Fixed} holds.
  Moreover, when $\varphi\in L^2(\Omega_\varepsilon)$,
  \begin{align*}
    \|\varphi\|_{L^2(\Omega_\varepsilon)}^2 = \varepsilon\int_{\Omega_1}|\xi(X)|^2J(\pi(X),d(X))^{-1}J(\pi(X),\varepsilon d(X))\,dX
  \end{align*}
  with $\xi:=\varphi\circ\Phi_\varepsilon$ on $\Omega_1$ by \eqref{E:CoV_Fixed} and thus \eqref{E:L2_Fixed} follows from \eqref{E:Jac_Bound_01}.

  Let $\varphi\in H^1(\Omega_\varepsilon)$.
  Then the right-hand inequality of \eqref{E:L2_Fixed} yields
  \begin{align} \label{Pf_Fix:L2_Gr}
    \varepsilon^{-1}\|\nabla\varphi\|_{L^2(\Omega_\varepsilon)}^2 \geq c\|(\nabla\varphi)\circ\Phi_\varepsilon\|_{L^2(\Omega_1)}^2.
  \end{align}
  To estimate the right-hand side we observe that
  \begin{align*}
    \nabla\Phi_\varepsilon^{-1}(x) = \left\{I_3-d(x)\overline{W}(x)\right\}^{-1}\left\{I_3-\varepsilon^{-1}d(x)\overline{W}(x)\right\}\overline{P}(x)+\varepsilon^{-1}\overline{Q}(x), \quad x\in\Omega_\varepsilon
  \end{align*}
  by $\nabla d=\bar{n}$ in $N$, \eqref{E:Form_W}, \eqref{E:Pi_Der}, and \eqref{E:Nor_Grad}.
  Setting $x=\Phi_\varepsilon(X)$ with $X\in\Omega_1$ in this equality and using $d(\Phi_\varepsilon(X))=\varepsilon d(X)$ and $\pi(\Phi_\varepsilon(X))=\pi(X)$ we get
  \begin{align} \label{Pf_Fix:Gr_Pinv}
    \begin{gathered}
      \nabla\Phi_\varepsilon^{-1}(\Phi_\varepsilon(X))=\Lambda_\varepsilon(X)\overline{P}(X)+\varepsilon^{-1}\overline{Q}(X), \quad X\in\Omega_1, \\
      \Lambda_\varepsilon(X) := \left\{I_3-\varepsilon d(X)\overline{W}(X)\right\}^{-1}\left\{I_3-d(X)\overline{W}(X)\right\}.
    \end{gathered}
  \end{align}
  Let $\xi=\varphi\circ\Phi_\varepsilon$ on $\Omega_1$.
  Then since $\varphi=\xi\circ\Phi_\varepsilon^{-1}$ on $\Omega_\varepsilon$,
  \begin{align*}
    \nabla\varphi(x) = \nabla\Phi_\varepsilon^{-1}(x)\nabla\xi(\Phi_\varepsilon^{-1}(x)), \quad x\in\Omega_\varepsilon
  \end{align*}
  and by setting $x=\Phi_\varepsilon(X)$ we get
  \begin{align} \label{Pf_Fix:GrP_Phi}
    \nabla\varphi(\Phi_\varepsilon(X)) = [\nabla\Phi_\varepsilon^{-1}(\Phi_\varepsilon(X))]\nabla\xi(X), \quad X\in\Omega_1.
  \end{align}
  We apply \eqref{Pf_Fix:Gr_Pinv} to this equality and use
  \begin{align} \label{Pf_Fix:LPPL}
    \Lambda_\varepsilon(X)\overline{P}(X) = \overline{P}(X)\Lambda_\varepsilon(X),
  \end{align}
  which follows from \eqref{E:Form_W} and \eqref{E:WReso_P}, and
  \begin{align*}
    \overline{Q}(X)\nabla\xi(X) = \{\bar{n}(X)\cdot\nabla\xi(X)\}\bar{n}(X) = \partial_n\xi(X)\bar{n}(X)
  \end{align*}
  to obtain
  \begin{align*}
    [(\nabla\varphi)\circ\Phi_\varepsilon](X) = \overline{P}(X)\Lambda_\varepsilon(X)\nabla\xi(X)+\varepsilon^{-1}\partial_n\xi(X)\bar{n}(X), \quad X\in\Omega_1.
  \end{align*}
  Here the two terms on the right-hand side are orthogonal to each other.
  Moreover,
  \begin{align*}
    \left|\overline{P}(X)\Lambda_\varepsilon(X)\nabla\xi(X)\right| = \left|\Lambda_\varepsilon(X)\overline{P}(X)\nabla\xi(X)\right| \geq c\left|\overline{P}(X)\nabla\xi(X)\right|
  \end{align*}
  by \eqref{E:Wein_Bound} and \eqref{Pf_Fix:LPPL}.
  Hence
  \begin{align*}
    |[(\nabla\varphi)\circ\Phi_\varepsilon](X)|^2 &= \left|\overline{P}(X)\Lambda_\varepsilon(X)\nabla\xi(X)\right|^2+\varepsilon^{-2}|\partial_n\xi(X)\bar{n}(X)|^2 \\
    &\geq c\left|\overline{P}(X)\nabla\xi(X)\right|^2+\varepsilon^{-2}|\partial_n\xi(X)|^2
  \end{align*}
  for $X\in\Omega_1$.
  Combining this inequality and \eqref{Pf_Fix:L2_Gr} we obtain
  \begin{align*}
    \varepsilon^{-1}\|\nabla\varphi\|_{L^2(\Omega_\varepsilon)}^2 \geq c\|(\nabla\varphi)\circ\Phi_\varepsilon\|_{L^2(\Omega_1)}^2 \geq c\left(\left\|\overline{P}\nabla\xi\right\|_{L^2(\Omega_1)}^2+\varepsilon^{-2}\|\partial_n\xi\|_{L^2(\Omega_1)}^2\right).
  \end{align*}
  Thus \eqref{E:H1_Fixed} is valid.
  Also, noting that
  \begin{align*}
    \|\nabla\xi\|_{L^2(\Omega_1)}^2 = \left\|\overline{P}\nabla\xi\right\|_{L^2(\Omega_1)}^2+\left\|\overline{Q}\nabla\xi\right\|_{L^2(\Omega_1)}^2 = \left\|\overline{P}\nabla\xi\right\|_{L^2(\Omega_1)}^2+\|\partial_n\xi\|_{L^2(\Omega_1)}^2,
  \end{align*}
  we have $\xi\in H^1(\Omega_1)$ by \eqref{E:L2_Fixed} and \eqref{E:H1_Fixed}.
\end{proof}

To prove Lemma \ref{L:KAux_Du} we present an auxiliary result.

\begin{lemma} \label{L:Norm_Mat}
  Let $\tau_1,\tau_2,n_0\in\mathbb{R}^3$ and $A\in\mathbb{R}^{3\times3}$ satisfy
  \begin{align} \label{E:Mat_Cd}
    |n_0| = 1, \quad n_0\cdot\tau_1 = n_0\cdot\tau_2 = 0, \quad An_0 = A^Tn_0 = 0.
  \end{align}
  Then for $B:=A+\tau_1\otimes n_0+n_0\otimes\tau_2+cn_0\otimes n_0$ with $c\in\mathbb{R}^3$ we have
  \begin{align} \label{E:Norm_Mat}
    |B|^2 = |A|^2+|\tau_1|^2+|\tau_2|^2+|c|^2.
  \end{align}
\end{lemma}

\begin{proof}
  By direct calculations and $|n_0|=1$, $n_0\cdot\tau_1=0$, and $A^Tn_0=0$ we have
  \begin{multline*}
    B^TB = A^TA+\tau_2\otimes\tau_2+(|\tau_1|^2+|c|^2)n_0\otimes n_0 \\
    +(A^T\tau_1)\otimes n_0+n_0\otimes(A^T\tau_1)+c(\tau_2\otimes n_0+n_0\otimes\tau_2).
  \end{multline*}
  Hence $|B|^2=\mathrm{tr}[B^TB]$ is of the form
  \begin{align*}
    |B|^2 = |A|^2+|\tau_2|^2+(|\tau_1|^2+|c|^2)|n_0|^2+2(A^T\tau_1)\cdot n_0+2c(\tau_2\cdot n_0).
  \end{align*}
  Since $|n_0|=1$, $\tau_2\cdot n_0=0$, and $(A^T\tau_1)\cdot n_0=\tau_1\cdot(An_0)=0$ by $An_0=0$, we conclude by the above equality that \eqref{E:Norm_Mat} is valid.
\end{proof}

\begin{proof}[Proof of Lemma \ref{L:KAux_Du}]
  Let $\Phi_\varepsilon$ be the bijection from $\Omega_1$ onto $\Omega_\varepsilon$ given by \eqref{E:Def_Bij}.
  For $u\in H^1(\Omega_\varepsilon)^3$ we have $U:=u\circ\Phi_\varepsilon\in H^1(\Omega_1)^3$ by Lemma \ref{L:CoV_Fixed}.
  Also, denoting by $u_i$ and $U_i$ the $i$-th components of $u$ and $U$ for $i=1,2,3$ we see by \eqref{E:H1_Fixed} that
  \begin{align*}
    \varepsilon^{-1}\|\nabla u_i\|_{L^2(\Omega_\varepsilon)}^2 \geq c\left(\left\|\overline{P}\nabla U_i\right\|_{L^2(\Omega_1)}^2+\varepsilon^{-2}\|\partial_nU_i\|_{L^2(\Omega_1)}^2\right), \quad i=1,2,3.
  \end{align*}
  Since $\|\nabla u\|_{L^2(\Omega_\varepsilon)}^2 = \sum_{i=1}^3\|\nabla u_i\|_{L^2(\Omega_\varepsilon)}^2$, $\|\partial_nU\|_{L^2(\Omega_1)}^2 = \sum_{i=1}^3\|\partial_nU_i\|_{L^2(\Omega_1)}^2$, and
  \begin{align*}
    \left\|\overline{P}\nabla U\right\|_{L^2(\Omega_1)}^2 = \sum_{i,j,k=1}^3\left\|\overline{P}_{ij}\partial_jU_k\right\|_{L^2(\Omega_1)^2}^2 = \sum_{k=1}^3\left\|\overline{P}\nabla U_k\right\|_{L^2(\Omega_1)}^2,
  \end{align*}
  by the above inequality we get \eqref{E:H1_Fixed} with $\varphi$ and $\xi$ replaced by $u$ and $U$.

  Let us prove \eqref{E:KAux_Du}.
  Hereafter we carry out all calculations in $\Omega_1$ and suppress the argument $X\in\Omega_1$ unless otherwise stated.
  From \eqref{Pf_Fix:Gr_Pinv}--\eqref{Pf_Fix:LPPL} and
  \begin{align*}
    \overline{Q}\nabla U &= \bar{n}\otimes[(\nabla U)^T\bar{n}], \quad (\nabla U)^T\bar{n} = (\bar{n}\cdot\nabla)U = \partial_nU
  \end{align*}
  we deduce that
  \begin{align*}
    (\nabla u)\circ\Phi_\varepsilon = \overline{P}F_\varepsilon(U)+\varepsilon^{-1}\bar{n}\otimes\partial_nU, \quad F_\varepsilon(U):=\Lambda_\varepsilon\nabla U.
  \end{align*}
  Moreover, by $I_3=P+Q$ on $\Gamma$ and \eqref{E:NorDer_Con} with $\eta=n,P$ we see that
  \begin{align*}
    \overline{P}F_\varepsilon(U) &= \overline{P}F_\varepsilon(U)\overline{P}+\overline{P}F_\varepsilon(U)\overline{Q} = \overline{P}F_\varepsilon(U)\overline{P}+\Bigl[\overline{P}F_\varepsilon(U)\bar{n}\Bigr]\otimes\bar{n}, \\
    \partial_nU &= \partial_n\Bigl[\overline{P}U+(U\cdot\bar{n})\bar{n}\Bigr] = \overline{P}\partial_nU+\{\partial_n(U\cdot\bar{n})\}\bar{n}.
  \end{align*}
  By the above equalities we get (note that $B_S=(B+B^T)/2$ for $B\in\mathbb{R}^{3\times3}$)
  \begin{gather*}
    D(u)\circ\Phi_\varepsilon = (\nabla u)_S\circ\Phi_\varepsilon = A+\tau_1\otimes n_0+n_0\otimes\tau_2+\varepsilon^{-1}\{\partial_n(U\cdot\bar{n})\}n_0\otimes n_0, \\
    A := \overline{P}F_\varepsilon(U)_S\overline{P}, \quad \tau_1 = \tau_2 := \frac{1}{2}\overline{P}\{F_\varepsilon(U)\bar{n}+\varepsilon^{-1}\partial_nU\}, \quad n_0 := \bar{n}.
  \end{gather*}
  Moreover, since $Pn=P^Tn=0$ on $\Gamma$, we see that $A$, $\tau_1$, $\tau_2$, and $n_0$ satisfy \eqref{E:Mat_Cd}.
  Hence we can apply \eqref{E:Norm_Mat} to $B=D(u)\circ\Phi_\varepsilon$ to get
  \begin{align*}
    |D(u)\circ\Phi_\varepsilon|^2 = |A|^2+|\tau_1|^2+|\tau_2|^2+\varepsilon^{-2}|\partial_n(U\cdot\bar{n})|^2 \geq |A|^2+\varepsilon^{-2}|\partial_n(U\cdot\bar{n})|^2
  \end{align*}
  in $\Omega_1$.
  From this inequality and \eqref{E:L2_Fixed} we deduce that
  \begin{align*}
    \varepsilon^{-1}\|D(u)\|_{L^2(\Omega_\varepsilon)}^2 \geq c\|D(u)\circ\Phi_\varepsilon\|_{L^2(\Omega_1)}^2 \geq c\left(\|A\|_{L^2(\Omega_1)}^2+\varepsilon^{-2}\|\partial_n(U\cdot\bar{n})\|_{L^2(\Omega_1)}^2\right).
  \end{align*}
  Since $A=\overline{P}F_\varepsilon(U)_S\overline{P}$ and $F_\varepsilon(U)=\Lambda_\varepsilon\nabla U$ is of the form \eqref{E:KAux_Matrix}, we obtain \eqref{E:KAux_Du} by the above inequality.
\end{proof}

Finally, we prove Lemma \ref{L:G_Bound} for the vector field $G(u)$ given by \eqref{E:Def_Gu}.

\begin{proof}[Proof of Lemma \ref{L:G_Bound}]
  For a function $\eta$ on $\Gamma$ we denote by $\bar{\eta}=\eta\circ\pi$ its constant extension in the normal direction of $\Gamma$.
  Let $n_\varepsilon^0$ and $n_\varepsilon^1$ be the vector fields on $\Gamma$ given by \eqref{E:Def_NB} and $W_\varepsilon^0$, $W_\varepsilon^1$, $\tilde{n}_1$, $\tilde{n}_2$, and $\widetilde{W}$ the functions on $N$ given by \eqref{E:Def_Wieps}--\eqref{E:Def_Tilde}.
  By \eqref{E:ConDer_Bound}, \eqref{E:Con_Hess}, \eqref{E:N_Bound}, and $g_0,g_1\in C^4(\Gamma)$ we see that
  \begin{align} \label{Pf_GB:Est_N_G}
    |\partial_x^\alpha\bar{n}_\varepsilon^i(x)| \leq c, \quad |\partial_x^\alpha\bar{g}_i(x)| \leq c, \quad x \in N,\,i=0,1,\,|\alpha|=0,1,2,
  \end{align}
  where $\partial_x^\alpha=\partial_1^{\alpha_1}\partial_2^{\alpha_2}\partial_3^{\alpha_3}$ for $\alpha=(\alpha_1,\alpha_2,\alpha_3)^T\in\mathbb{Z}^3$ with $\alpha_j\geq0$, $j=1,2,3$ and $c>0$ is a constant independent of $\varepsilon$.
  From \eqref{Pf_GB:Est_N_G} it also follows that
  \begin{align} \label{Pf_GB:Est_W}
    |W_\varepsilon^i(x)| \leq c, \quad |\partial_kW_\varepsilon^i(x)| \leq c, \quad x \in N, \,i=0,1,\,k=1,2,3.
  \end{align}
  By \eqref{E:Fric_Upper}, \eqref{Pf_GB:Est_N_G}, \eqref{Pf_GB:Est_W} and
  \begin{align} \label{Pf_GB:Width}
    0 \leq d(x)-\varepsilon\bar{g}_0(x) \leq \varepsilon\bar{g}(x), \quad 0 \leq \varepsilon\bar{g}_1(x)-d(x) \leq \varepsilon\bar{g}(x), \quad x\in \Omega_\varepsilon
  \end{align}
  we have
  \begin{align} \label{Pf_GB:nW_Bound}
    |\tilde{n}_1| \leq c, \quad |\tilde{n}_2| \leq c\varepsilon, \quad \left|\widetilde{W}\right| \leq c \quad\text{in}\quad \Omega_\varepsilon.
  \end{align}
  Applying \eqref{Pf_GB:nW_Bound} to \eqref{E:Def_Gu} we obtain the first inequality of \eqref{E:G_Bound}.

  To prove the second inequality of \eqref{E:G_Bound} we estimate the first order derivatives of $\tilde{n}_1$, $\tilde{n}_2$, and $\widetilde{W}$.
  We differentiate $\tilde{n}_1$ and use $\nabla d=\bar{n}$ in $N$ to get
  \begin{align*}
    \nabla\tilde{n}_1 = \frac{1}{\varepsilon\bar{g}}\bar{n}\otimes(\bar{n}_\varepsilon^0+\bar{n}_\varepsilon^1)+A_1 \quad\text{in}\quad N,
  \end{align*}
  where $A_1$ is a $3\times 3$ matrix-valued function on $N$ defined by
  \begin{multline*}
    A_1 := -\frac{\nabla\bar{g}}{\bar{g}}\otimes\tilde{n}_1-\frac{1}{\bar{g}}(\nabla\bar{g}_0\otimes\bar{n}_\varepsilon^1+\nabla\bar{g}_1\otimes\bar{n}_\varepsilon^0) \\
    +\frac{1}{\varepsilon\bar{g}}\{(d-\varepsilon\bar{g}_0)\nabla\bar{n}_\varepsilon^1-(\varepsilon\bar{g}_1-d)\nabla\bar{n}_\varepsilon^0\}.
  \end{multline*}
  By \eqref{E:G_Inf}, \eqref{Pf_GB:Est_N_G}, \eqref{Pf_GB:Width}, and \eqref{Pf_GB:nW_Bound} we observe that $A_1$ is bounded on $\Omega_\varepsilon$ uniformly in $\varepsilon$.
  From this fact, \eqref{E:G_Inf}, and \eqref{E:N_Diff} it follows that
  \begin{align} \label{Pf_GB:n1_Grad}
    |\nabla\tilde{n}_1| \leq \frac{1}{\varepsilon\bar{g}}|\bar{n}_\varepsilon^0+\bar{n}_\varepsilon^1|+|A_1| \leq c \quad\text{in}\quad \Omega_\varepsilon.
  \end{align}
  Similarly, we differentiate $\tilde{n}_2$ and $\widetilde{W}$ and use $\nabla d=\bar{n}$ in $N$, \eqref{E:G_Inf}, and \eqref{Pf_GB:Est_N_G}--\eqref{Pf_GB:nW_Bound} to deduce that
  \begin{align*}
    \nabla\tilde{n}_2 &= \frac{1}{\varepsilon\bar{g}}\bar{n}\otimes\left(\frac{\gamma_\varepsilon^1}{\nu}\bar{n}_\varepsilon^1-\frac{\gamma_\varepsilon^0}{\nu}\bar{n}_\varepsilon^0\right)+A_2, \\
    \partial_k\widetilde{W} &= \frac{1}{\varepsilon\bar{g}}\bar{n}_k(W_\varepsilon^0+W_\varepsilon^1)+B_k, \quad k=1,2,3
  \end{align*}
  in $N$, where the matrix-valued functions $A_2$ and $B_k$, $k=1,2,3$ are bounded on $\Omega_\varepsilon$ uniformly in $\varepsilon$.
  We apply \eqref{E:G_Inf}, \eqref{E:Fric_Upper}, and \eqref{E:N_Bound} to $\nabla\tilde{n}_2$.
  Then we have
  \begin{align} \label{Pf_GB:n2_Grad}
    |\nabla\tilde{n}_2| \leq \frac{c(\gamma_\varepsilon^0+\gamma_\varepsilon^1)}{\varepsilon\bar{g}}+|A_2| \leq c \quad\text{in}\quad \Omega_\varepsilon.
  \end{align}
  To estimate the first order derivatives of $\widetilde{W}$ we see by \eqref{E:ConDer_Dom} that
  \begin{multline*}
    W_\varepsilon^0+W_\varepsilon^1 = \{(\bar{n}_\varepsilon^0+\bar{n}_\varepsilon^1)\otimes\bar{n}_\varepsilon^0-\bar{n}_\varepsilon^1\otimes(\bar{n}_\varepsilon^0+\bar{n}_\varepsilon^1)\}\Bigl(I_3-d\overline{W}\Bigr)^{-1}\overline{\nabla_\Gamma n_\varepsilon^0} \\
    -(I_3-\bar{n}_\varepsilon^1\otimes\bar{n}_\varepsilon^1)\Bigl(I_3-d\overline{W}\Bigr)^{-1}\left(\overline{\nabla_\Gamma n_\varepsilon^0}+\overline{\nabla_\Gamma n_\varepsilon^1}\right)
  \end{multline*}
  in $N$.
  Hence $|W_\varepsilon^0+W_\varepsilon^1|\leq c\varepsilon$ in $N$ by \eqref{E:Wein_Bound}, \eqref{E:N_Bound}, and \eqref{E:N_Diff} and we get
  \begin{align} \label{Pf_GB:W1_Bound}
    \left|\partial_k\widetilde{W}\right| \leq \frac{1}{\varepsilon\bar{g}}|W_\varepsilon^0+W_\varepsilon^1|+|B_k| \leq c \quad\text{in}\quad \Omega_\varepsilon,\,k=1,2,3.
  \end{align}
  Here we also used \eqref{E:G_Inf} and the uniform in $\varepsilon$ boundedness of $B_k$ on $\Omega_\varepsilon$ in the second inequality.
  Noting that $G(u)$ is given by \eqref{E:Def_Gu}, we apply \eqref{Pf_GB:nW_Bound}--\eqref{Pf_GB:W1_Bound} to $\nabla G(u)$ to obtain the second inequality of \eqref{E:G_Bound}.
\end{proof}

\section{Formulas for the covariant derivatives} \label{S:Ap_Cov}
In this appendix we present formulas for the covariant derivatives of tangential vector fields on an embedded surface in $\mathbb{R}^3$ used in the proof of Lemma \ref{L:Lap_Apri}.

Let $\Gamma$ be a closed, connected, and oriented surface in $\mathbb{R}^3$ of class $C^3$.
We use the notations given in Section \ref{SS:Pre_Surf}.
For $X\in C^1(\Gamma,T\Gamma)$ and $Y\in C(\Gamma,T\Gamma)$ we define the covariant derivative of $X$ along $Y$ by
\begin{align} \label{E:Def_Covari}
  \overline{\nabla}_YX := P(Y\cdot\nabla)\widetilde{X} \quad\text{on}\quad \Gamma,
\end{align}
where $\widetilde{X}$ is a $C^1$-extension of $X$ to an open neighborhood of $\Gamma$ with $\widetilde{X}|_\Gamma=X$.
Since $Y$ is tangential on $\Gamma$, we observe by \eqref{E:Tgrad_Surf} that
\begin{align*}
  (Y\cdot\nabla)\widetilde{X} = (Y\cdot\nabla_\Gamma)X \quad\text{on}\quad \Gamma.
\end{align*}
Thus the value of $\overline{\nabla}_YX$ does not depend on the choice of an extension of $X$.
The directional derivative $(Y\cdot\nabla_\Gamma)X$ is expressed by $\overline{\nabla}_YX$ and the Weingarten map $W$.

\begin{lemma} \label{L:Gauss}
  For $X\in C^1(\Gamma,T\Gamma)$ and $Y\in C(\Gamma,T\Gamma)$ we have
  \begin{align} \label{E:Gauss}
    (Y\cdot\nabla)\widetilde{X} = (Y\cdot\nabla_\Gamma)X = \overline{\nabla}_YX+(WX\cdot Y)n \quad\text{on}\quad \Gamma,
  \end{align}
  where $\widetilde{X}$ is any $C^1$-extension of $X$ to an open neighborhood of $\Gamma$ with $\widetilde{X}|_\Gamma=X$.
\end{lemma}

\begin{proof}
  Since $X\cdot n=0$ and $-\nabla_\Gamma n=W$ on $\Gamma$,
  \begin{align*}
    (Y\cdot\nabla_\Gamma)X\cdot n &= Y\cdot\nabla_\Gamma(X\cdot n)-X\cdot(Y\cdot\nabla_\Gamma)n \\
    &= X\cdot(-\nabla_\Gamma n)^TY = X\cdot W^TY = WX\cdot Y
  \end{align*}
  on $\Gamma$.
  Combining this with \eqref{E:Def_Covari} we obtain \eqref{E:Gauss}.
\end{proof}

The formula \eqref{E:Gauss} is called the Gauss formula (see e.g. \cites{Ch15,Lee18}).
Let us show fundamental properties of the covariant derivative.

\begin{lemma} \label{L:RiCo}
  The following equalities hold on $\Gamma$:
  \begin{itemize}
    \item For $X\in C^1(\Gamma,T\Gamma)$, $Y,Z\in C(\Gamma,T\Gamma)$, and $\eta,\xi\in C(\Gamma)$,
    \begin{align} \label{E:RiCo_AfY}
      \overline{\nabla}_{\eta Y+\xi Z}X = \eta\overline{\nabla}_YX+\xi\overline{\nabla}_ZX.
    \end{align}
    \item For $X\in C^1(\Gamma,T\Gamma)$, $Y\in C(\Gamma,T\Gamma)$, and $\eta\in C^1(\Gamma)$,
    \begin{align} \label{E:RiCo_AfX}
      \overline{\nabla}_Y(\eta X) = (Y\cdot\nabla_\Gamma\eta)X+\eta\overline{\nabla}_YX.
    \end{align}
    \item For $X,Y\in C^1(\Gamma,T\Gamma)$ and $Z\in C(\Gamma,T\Gamma)$,
    \begin{align} \label{E:RiCo_Met}
      Z\cdot\nabla_\Gamma(X\cdot Y) = \overline{\nabla}_ZX\cdot Y+X\cdot\overline{\nabla}_ZY.
    \end{align}
    \item For $X,Y\in C^1(\Gamma,T\Gamma)$ and $\eta\in C^2(\Gamma)$,
    \begin{align} \label{E:RiCo_Tor}
      X\cdot\nabla_\Gamma(Y\cdot\nabla_\Gamma\eta)-Y\cdot\nabla_\Gamma(X\cdot\nabla_\Gamma\eta) = \Bigl(\overline{\nabla}_XY-\overline{\nabla}_YX\Bigr)\cdot\nabla_\Gamma\eta.
    \end{align}
  \end{itemize}
\end{lemma}

\begin{proof}
  The equalities \eqref{E:RiCo_AfY} and \eqref{E:RiCo_AfX} immediately follow from \eqref{E:Def_Covari}.
  We also apply \eqref{E:Gauss} and $X\cdot n=Y\cdot n=0$ on $\Gamma$ to the right-hand side of
  \begin{align*}
    Z\cdot\nabla_\Gamma(X\cdot Y) = (Z\cdot\nabla_\Gamma)X\cdot Y+X\cdot(Z\cdot\nabla_\Gamma)Y \quad\text{on}\quad \Gamma
  \end{align*}
  to get \eqref{E:RiCo_Met}.
  Let us prove \eqref{E:RiCo_Tor}.
  The left-hand side of \eqref{E:RiCo_Tor} is of the form
  \begin{gather*}
    \sum_{i,j=1}^3\{X_i\underline{D}_i(Y_j\underline{D}_j\eta)-Y_i\underline{D}_i(X_j\underline{D}_j\eta)\} = J_1+J_2, \\
    J_1 := \sum_{i,j=1}^3(X_i\underline{D}_iY_j-Y_i\underline{D}_iX_j)\underline{D}_j\eta, \quad J_2 := \sum_{i,j=1}^3(X_iY_j-X_jY_i)\underline{D}_i\underline{D}_j\eta.
  \end{gather*}
  By \eqref{E:Gauss} and $\nabla_\Gamma\eta\cdot n=0$ on $\Gamma$ we have
  \begin{align*}
    J_1 = \{(X\cdot\nabla_\Gamma)Y-(Y\cdot\nabla_\Gamma)X\}\cdot\nabla_\Gamma\eta = \Bigl(\overline{\nabla}_XY-\overline{\nabla}_YX\Bigr)\cdot\nabla_\Gamma\eta.
  \end{align*}
  Also, using \eqref{E:TD_Exc} and $X\cdot n=Y\cdot n=0$ on $\Gamma$ we observe that
  \begin{align*}
    J_2 &= \sum_{i,j=1}^3X_iY_j(\underline{D}_i\underline{D}_j\eta-\underline{D}_j\underline{D}_i\eta) = \sum_{i,j=1}^3X_iY_j([W\nabla\eta]_in_j-[W\nabla_\Gamma\eta]_jn_i) \\
    &= (X\cdot W\nabla_\Gamma\eta)(Y\cdot n)-(X\cdot n)(Y\cdot W\nabla_\Gamma\eta) = 0.
  \end{align*}
  Combining the above three equalities we obtain \eqref{E:RiCo_Tor}.
\end{proof}

Lemma \ref{L:RiCo} shows that the mapping
\begin{align*}
  \overline{\nabla}\colon C^1(\Gamma,T\Gamma)\times C(\Gamma,T\Gamma)\to C(\Gamma), \quad (X,Y) \mapsto \overline{\nabla}_YX
\end{align*}
is the Riemannian (or Levi-Civita) connection on $\Gamma$ (see e.g. \cites{Ch15,Lee18}).
Note that the formula \eqref{E:RiCo_Tor} stands for the torsion-free condition
\begin{align*}
  [X,Y] := XY-YX = \overline{\nabla}_XY-\overline{\nabla}_YX,
\end{align*}
where $[X,Y]$ is the Lie bracket of $X$ and $Y$.

Let $O$ be a relatively open subset of $\Gamma$.
If $O$ is sufficiently small, then by the $C^3$-regularity of $\Gamma$ we can take $C^2$ vector fields $\tau_1$ and $\tau_2$ on $O$ such that $\{\tau_1(y),\tau_2(y)\}$ is an orthonormal basis of the tangent plane of $\Gamma$ at each $y\in\Gamma$.
We call the pair $\{\tau_1,\tau_2\}$ of such vector fields a local orthonormal frame for the tangent bundle of $\Gamma$ on $O$, or simply a local orthonormal frame on $O$.
Note that
\begin{align} \label{E:MC_Local}
  H = \mathrm{tr}[W] = W\tau_1\cdot\tau_1+W\tau_2\cdot\tau_2 \quad\text{on}\quad O
\end{align}
since $\{\tau_1,\tau_2,n\}$ is an orthonormal basis of $\mathbb{R}^3$ and $Wn=0$ on $\Gamma$.
We express several quantities related to the tangential gradient matrix of tangential vector fields on $\Gamma$ in terms of the covariant derivatives and the local orthonormal frame.

\begin{lemma} \label{L:Form_Cov}
  Let $\{\tau_1,\tau_2\}$ be a local orthonormal frame for the tangent bundle of $\Gamma$ on a relatively open subset $O$ of $\Gamma$.
  For $X,Y\in C^1(\Gamma,T\Gamma)$ we have
  \begin{align}
    \mathrm{div}_\Gamma X &= \sum_{i=1,2}\overline{\nabla}_iX\cdot\tau_i, \label{E:Sdiv_Cov} \\
    \nabla_\Gamma X:W &= \sum_{i=1,2}\overline{\nabla}_iX\cdot W\tau_i = \sum_{i=1,2}W\overline{\nabla}_iX\cdot\tau_i, \label{E:Wtr_Cov} \\
    \nabla_\Gamma X:(\nabla_\Gamma Y)P &= \sum_{i=1,2}\overline{\nabla}_iX\cdot\overline{\nabla}_iY, \label{E:InnP_Cov} \\
    W\nabla_\Gamma X:(\nabla_\Gamma Y)P &= \sum_{i=1,2}\overline{\nabla}_{W\tau_i}X\cdot\overline{\nabla}_iY \label{E:Winn_Cov}
  \end{align}
  on $O$, where $\overline{\nabla}_i:=\overline{\nabla}_{\tau_i}$ for $i=1,2$.
\end{lemma}

\begin{proof}
  We carry out calculations on $O$.
  By \eqref{E:P_TGr} and \eqref{E:Gauss} we have
  \begin{align} \label{Pf_FC:Form}
    \begin{aligned}
      (\nabla_\Gamma X)^T\tau_i &= (\tau_i\cdot\nabla_\Gamma)X = \overline{\nabla}_iX+(WX\cdot\tau_i)n, \quad i=1,2, \\
      (\nabla_\Gamma X)^Tn &= (n\cdot\nabla_\Gamma)X = 0.
    \end{aligned}
  \end{align}
  Since $\{\tau_1,\tau_2,n\}$ forms an orthonormal basis of $\mathbb{R}^3$,
  \begin{align*}
    \mathrm{div}_\Gamma X = \mathrm{tr}[\nabla_\Gamma X] = \sum_{i=1,2}(\nabla_\Gamma X)^T\tau_i\cdot\tau_i+(\nabla_\Gamma X)^Tn\cdot n.
  \end{align*}
  The equality \eqref{E:Sdiv_Cov} follows from this equality and \eqref{Pf_FC:Form}.
  We also obtain \eqref{E:Wtr_Cov} by applying \eqref{Pf_FC:Form}, $W^T=W$, and $Wn=0$ to the right-hand side of
  \begin{align*}
    \nabla_\Gamma X:W = (\nabla_\Gamma X)^T:W^T = \sum_{i=1,2}(\nabla_\Gamma X)^T\tau_i\cdot W^T\tau_i+(\nabla_\Gamma X)^Tn\cdot W^Tn.
  \end{align*}
  Let us prove \eqref{E:InnP_Cov}.
  By \eqref{Pf_FC:Form}, $P^T=P$, $Pn=0$, and $P\overline{\nabla}_iY=\overline{\nabla}_iY$,
  \begin{align*}
    [(\nabla_\Gamma Y)P]^T\tau_i = P[(\nabla_\Gamma Y)^T\tau_i] = P\Bigl\{\overline{\nabla}_iY+(WY\cdot\tau_i)n\Bigr\} = \overline{\nabla}_iY
  \end{align*}
  for $i=1,2$.
  From this equality, \eqref{Pf_FC:Form}, and $\overline{\nabla}_iY\cdot n=0$ it follows that
  \begin{align*}
    \nabla_\Gamma X:(\nabla_\Gamma Y)P &= (\nabla_\Gamma X)^T:[(\nabla_\Gamma Y)P]^T \\
    &= \sum_{i=1,2}(\nabla_\Gamma X)^T\tau_i\cdot[(\nabla_\Gamma Y)P]^T\tau_i+(\nabla_\Gamma X)^Tn\cdot[(\nabla_\Gamma Y)P]^Tn \\
    &= \sum_{i=1,2}\Bigl\{\overline{\nabla}_iX+(WX\cdot\tau_i)n\Bigr\}\cdot\overline{\nabla}_iY = \sum_{i=1,2}\overline{\nabla}_iX\cdot\overline{\nabla}_iY.
  \end{align*}
  Thus \eqref{E:InnP_Cov} is valid.
  Similarly, we can prove \eqref{E:Winn_Cov} by using the formulas
  \begin{align*}
  [W(\nabla_\Gamma X)]^T\tau_i &= (\nabla_\Gamma X)^TW\tau_i = (W\tau_i\cdot\nabla_\Gamma)X = \overline{\nabla}_{W\tau_i}X+(WX\cdot W\tau_i)n, \\
  [W(\nabla_\Gamma X)]^Tn &= (\nabla_\Gamma X)^TWn = 0
  \end{align*}
  by $W^T=W$, $Wn=0$, and \eqref{E:Gauss} (note that $W\tau_i$ is tangential on $\Gamma$).
\end{proof}

Next we give an integration by parts formula for integrals over $\Gamma$ of the covariant derivatives along vector fields of a local orthonormal frame.

\begin{lemma} \label{L:IbP_Cov}
  Let $\{\tau_1,\tau_2\}$ be a local orthonormal frame for the tangent bundle of $\Gamma$ on a relatively open subset $O$ of $\Gamma$ and $\overline{\nabla}_i:=\overline{\nabla}_{\tau_i}$ for $i=1,2$.
  Suppose that $X\in C^2(\Gamma,T\Gamma)$ and $Y\in C^1(\Gamma,T\Gamma)$ are compactly supported in $O$.
  Then we have
  \begin{align} \label{E:IbP_Cov}
    \sum_{i=1,2}\int_\Gamma\Bigl(\overline{\nabla}_i\overline{\nabla}_iX-\overline{\nabla}_{\overline{\nabla}_i\tau_i}X\Bigr)\cdot Y\,d\mathcal{H}^2 = -\sum_{i=1,2}\int_\Gamma\overline{\nabla}_iX\cdot\overline{\nabla}_iY\,d\mathcal{H}^2.
  \end{align}
\end{lemma}

\begin{proof}
  The proof is basically the same as that of \cite{Pe06}*{Proposition 34}.
  We set
  \begin{align*}
    Z := \sum_{i=1,2}\Bigl(\overline{\nabla}_iX\cdot Y\Bigr)\tau_i \quad\text{on}\quad O
  \end{align*}
  and extend $Z$ to $\Gamma$ by setting zero outside of $O$.
  Then $Z\in C^1(\Gamma,T\Gamma)$ since $\tau_1$ and $\tau_2$ are of class $C^2$ on $O$ and $X\in C^2(\Gamma,T\Gamma)$ and $Y\in C^1(\Gamma,T\Gamma)$ are compactly supported in $O$.
  Moreover, since $\{\tau_1,\tau_2\}$ is a local orthonormal frame, we have
  \begin{align*}
    Z\cdot V = \sum_{i=1,2}\overline{\nabla}_{(V\cdot\tau_i)\tau_i}X\cdot Y = \overline{\nabla}_VX\cdot Y \quad\text{on}\quad O
  \end{align*}
  for all $V\in C(\Gamma,T\Gamma)$ by \eqref{E:RiCo_AfY}.
  From this fact and \eqref{E:RiCo_Met} we deduce that
  \begin{align*}
    \overline{\nabla}_iZ\cdot\tau_i = \tau_i\cdot\nabla_\Gamma(Z\cdot\tau_i)-Z\cdot\overline{\nabla}_i\tau_i = \tau_i\cdot\nabla_\Gamma\Bigl(\overline{\nabla}_iX\cdot Y\Bigr)-\overline{\nabla}_{\overline{\nabla}_i\tau_i}X\cdot Y \quad\text{on}\quad O
  \end{align*}
  for $i=1,2$.
  Applying \eqref{E:RiCo_Met} again to the first term on the right-hand side we get
  \begin{align*}
    \overline{\nabla}_iZ\cdot\tau_i = \Bigl(\overline{\nabla}_i\overline{\nabla}_iX-\overline{\nabla}_{\overline{\nabla}_i\tau_i}X\Bigr)\cdot Y+\overline{\nabla}_iX\cdot\overline{\nabla}_iY \quad\text{on}\quad O.
  \end{align*}
  By this equality and \eqref{E:Sdiv_Cov} we see that
  \begin{align} \label{Pf_ICov:Div}
    \mathrm{div}_\Gamma Z = \sum_{i=1,2}\left\{\Bigl(\overline{\nabla}_i\overline{\nabla}_iX-\overline{\nabla}_{\overline{\nabla}_i\tau_i}X\Bigr)\cdot Y+\overline{\nabla}_iX\cdot\overline{\nabla}_iY\right\} \quad\text{on}\quad O.
  \end{align}
  Since $X$, $Y$, and $Z$ are compactly supported in $O$, we may assume that \eqref{Pf_ICov:Div} holds on the whole surface $\Gamma$.
  Hence we obtain \eqref{E:IbP_Cov} by integrating both sides of \eqref{Pf_ICov:Div} over $\Gamma$ and using \eqref{Pf_ITD:SD_Thm} with $X$ replaced by $Z\in C^1(\Gamma,T\Gamma)$.
\end{proof}

\begin{remark} \label{R:RC}
  Since $\Gamma$ is of class $C^3$, the space $C^2(\Gamma,T\Gamma)$ is dense in $H^m(\Gamma,T\Gamma)$ for $m=0,1,2$ by Lemma \ref{L:Wmp_Tan_Appr}.
  Hence the formulas given in this appendix are also valid (a.e. on $\Gamma$) if we replace $C^m(\Gamma,T\Gamma)$, $m=0,1,2$ with $H^m(\Gamma,T\Gamma)$.
\end{remark}

\section{Infinitesimal rigid displacements on a closed surface} \label{S:Ap_Rig}
In this appendix we show several results on infinitesimal rigid displacements of $\mathbb{R}^3$ related to the axial symmetry of a closed surface and a curved thin domain.

Let $\Gamma$ be a $C^2$ closed, connected, and oriented surface in $\mathbb{R}^3$ and $\mathcal{R}$ the set of the form \eqref{E:Def_R} which consists of infinitesimal rigid displacements of $\mathbb{R}^3$ with tangential restrictions on $\Gamma$.

\begin{lemma} \label{L:IR_Surf}
  Let $w(x)=a\times x+b\in\mathcal{R}$.
  If $w\not\equiv0$, then $a\neq0$, $a\cdot b=0$, and $\Gamma$ is axially symmetric around the line parallel to the vector $a$ and passing through the point $b_a:=|a|^{-2}(a\times b)$.
  Conversely, if $\Gamma$ is axially symmetric around the line parallel to $a\neq0$ and passing through $\tilde{b}\in\mathbb{R}^3$, then $\tilde{w}(x)=a\times(x-\tilde{b})\in\mathcal{R}\setminus\{0\}$.
\end{lemma}

\begin{proof}
  Suppose that $w(x)=a\times x+b\in\mathcal{R}$ does not identically vanish.
  If $a=0$, then $w\cdot n=b\cdot n=0$ on $\Gamma$, which yields $b=0$ since $\Gamma$ is closed.
  Hence $a\neq0$ when $w\not\equiv0$.
  Now we consider the flow map $x(\cdot,t)\colon\mathbb{R}^3\to\mathbb{R}^3$ of $w$ (here $\dot{x}=\partial x/\partial t$)
  \begin{align} \label{Pf_IRS:Flow}
    x(X,0) = X \in \mathbb{R}^3, \quad \dot{x}(X,t) = w(x(X,t),t) = a\times x(X,t)+b, \quad t>0.
  \end{align}
  Since $a\neq0$, we can take an orthonormal basis $\{E_1,E_2,E_3\}$ of $\mathbb{R}^3$ such that
  \begin{align} \label{Pf_IRS:ONB}
    E_3 = |a|^{-1}a, \quad E_1\times E_2 = E_3, \quad E_2\times E_3 = E_1, \quad E_3\times E_1 = E_2.
  \end{align}
  In what follows, we write $x_E^i:=x\cdot E_i$ for $x\in\mathbb{R}^3$ and $i=1,2,3$.
  Then since
  \begin{align*}
    a\times x = |a|E_3\times(x_E^1E_1+x_E^2E_2+x_E^3E_3) = -|a|x_E^2E_1+|a|x_E^1E_2,
  \end{align*}
  the system \eqref{Pf_IRS:Flow} for $x(t)=\sum_{i=1}^3x_E^i(t)E_i$ is equivalent to
  \begin{align*}
    \left\{
    \begin{aligned}
      \dot{x}_E^1(t) &= -|a|x_E^2(t)+b_E^1, & & x_E^1(0) = X_E^1, \\
      \dot{x}_E^2(t) &= |a|x_E^1(t)+b_E^2, & & x_E^2(0) = X_E^2, \\
      \dot{x}_E^3(t) &= b_E^3, & & x_E^3(0) = X_E^3.
    \end{aligned}
    \right.
  \end{align*}
  Here we suppressed the argument $X$.
  We solve this system to get
  \begin{align} \label{Pf_IRS:Sol}
    \left\{
    \begin{aligned}
      x_E^1(t)+|a|^{-1}b_E^2 &= (X_E^1+|a|^{-1}b_E^2)\cos(|a|t)-(X_E^2-|a|^{-1}b_E^1)\sin(|a|t), \\
      x_E^2(t)-|a|^{-1}b_E^1 &= (X_E^1+|a|^{-1}b_E^2)\sin(|a|t)+(X_E^2-|a|^{-1}b_E^1)\cos(|a|t), \\
      x_E^3(t) &= X_E^3+b_E^3t.
    \end{aligned}
    \right.
  \end{align}
  Now we observe by $w\cdot n=0$ on $\Gamma$ that $x(\cdot,t)$ maps $\Gamma$ into itself for all $t>0$.
  This fact and the compactness of $\Gamma$ imply that $x_E^3(t)=X_E^3+b_E^3t$ remains bounded for $X\in\Gamma$, which yields $b_E^3=b\cdot E_3=0$, i.e. $a\cdot b=0$.
  Also, setting
  \begin{align} \label{Pf_IRS:Ba}
    b_a := -|a|^{-1}b_E^2E_1+|a|^{-1}b_E^1E_2 = |a|^{-2}a\times b,
  \end{align}
  where the second equality follows from $b=b_E^1E_1+b_E^2E_2$ and \eqref{Pf_IRS:ONB}, we get
  \begin{align*}
    x(X,t)-b_a &= \{x_E^1(X,t)+|a|^{-1}b_E^2\}E_1+\{x_E^2(X,t)-|a|^{-1}b_E^1\}E_2+x_E^3(X,t)E_3.
  \end{align*}
  Hence by \eqref{Pf_IRS:Sol} with $b_E^3=0$ we obtain
  \begin{align} \label{Pf_IRS:Rot}
    x(X,t)-b_a = P_aR_a(t)P_a^T(X-b_a), \quad X\in\mathbb{R}^3,\,t\geq0,
  \end{align}
  where $P_a:=(E_1 \, E_2 \, E_3)$ is a $3\times 3$ matrix with $i$-th column $E_i$ for $i=1,2,3$ and
  \begin{align} \label{Pf_IRS:Def_Ra}
    R_a(t) :=
    \begin{pmatrix}
      \cos(|a|t) & -\sin(|a|t) & 0 \\
      \sin(|a|t) & \cos(|a|t) & 0 \\
      0 & 0 & 1
    \end{pmatrix}.
  \end{align}
  By \eqref{Pf_IRS:Rot} we observe that the flow map $x(\cdot,t)$ of $w$ is given by the rotation through the angle $|a|t$ around the axis parallel to $E_3=|a|^{-1}a$ and passing through $b_a$.
  Since $x(\cdot,t)$ maps $\Gamma$ into itself for all $t>0$ by $w\cdot n=0$ on $\Gamma$, we conclude that $\Gamma$ is axially symmetric around the line parallel to $a$ and passing through $b_a$.

  Conversely, suppose that $\Gamma$ is axially symmetric around the line parallel to $a\neq0$ and passing through $\tilde{b}\in\mathbb{R}^3$.
  Let $\{E_1,E_2,E_3\}$ be an orthonormal basis of $\mathbb{R}^3$ satisfying \eqref{Pf_IRS:ONB}.
  Then by the first part of the proof we see that the mapping
  \begin{align*}
    \Phi_t(X) := P_aR_a(t)P_a^T(X-\tilde{b})+\tilde{b}, \quad X\in\mathbb{R}^3
  \end{align*}
  preserves $\Gamma$ for all $t\in\mathbb{R}$, where $P_a=(E_1 \, E_2 \, E_3)$ and $R_a(t)$ is given by \eqref{Pf_IRS:Def_Ra}.
  Hence for each $Y\in\Gamma$ the time derivative of $\Phi_t(Y)$ at $t=0$ gives a tangent vector on $\Gamma$ at $Y$.
  Moreover, setting $Z:=Y-\tilde{b}$ and $Z_E^i:=Z\cdot E_i$, $i=1,2,3$ we have
  \begin{align*}
    \frac{d}{dt}\Phi_t(Y)\Bigl|_{t=0} &= \frac{d}{dt}P_aR_a(t)P_a^TZ\Bigl|_{t=0} = P_a
    \begin{pmatrix}
      0 & -|a| & 0 \\
      |a| & 0 & 0 \\
      0 & 0 & 0
    \end{pmatrix}
    P_a^TZ \\
    &= |a|(-Z_E^2E_1+Z_E^1E_2) = |a|E_3\times(Z_E^1E_1+Z_E^2E_2+Z_E^3E_3) \\
    &= |a|E_3\times Z = a\times(Y-\tilde{b})
  \end{align*}
  by \eqref{Pf_IRS:ONB}, $E_3\times E_3=0$, and $Z=\sum_{i=1}^3Z_E^iE_i$.
  Hence $\tilde{w}(x)=a\times(x-\tilde{b})\in\mathcal{R}\setminus\{0\}$.
\end{proof}

\begin{lemma} \label{L:SR_Eigen}
  Let $w(x)=a\times x+b\in\mathcal{R}$ satisfy $w\not\equiv0$.
  Then
  \begin{align} \label{E:SR_Eigen}
    W(y)w(y) = \lambda(y)w(y), \quad a\times n(y) = -\lambda(y)w(y) \quad\text{for all}\quad y\in\Gamma
  \end{align}
  with some $\lambda(y)\in\mathbb{R}$.
  Here $W=-\nabla_\Gamma n$ is the Weingarten map of $\Gamma$.
\end{lemma}

\begin{proof}
  Since $w(x)=a\times x+b\in\mathcal{R}$ and $w\not\equiv0$, Lemma \ref{L:IR_Surf} implies that $a\neq0$, $a\cdot b=0$, and $\Gamma$ is axially symmetric around the line parallel to $a$ and passing through $b_a=|a|^{-2}(a\times b)$.
  Also, since $a\times b_a=-b$ by \eqref{Pf_IRS:ONB}, \eqref{Pf_IRS:Ba}, and $a\cdot b=0$, we have $w(x)=a\times(x-b_a)$.
  Hence by a translation along $b_a$ and a rotation of coordinates we may assume that $\Gamma$ is axially symmetric around the $x_3$-axis and $w(x)=\alpha(e_3\times x)$, where $\alpha=|a|>0$ and $e_3=(0,0,1)^T$.
  Replacing $w$ by $\alpha^{-1}w$ we may further assume $\alpha=1$, i.e. $a=e_3$ and $w(x)=e_3\times x$.
  Under these assumptions, $\Gamma$ is represented as a surface of revolution
  \begin{align} \label{Pf_SRE:SoR}
    \Gamma = \{\mu(s,\vartheta) = (\varphi(s)\cos\vartheta,\varphi(s)\sin\vartheta,\psi(s)) \mid s\in[0,L],\,\vartheta\in[0,2\pi]\}.
  \end{align}
  Here $\gamma(s)=(\varphi(s),0,\psi(s))$ is a $C^2$ curve parametrized by the arc length $s\in[0,L]$, $L>0$ such that $\varphi(s)>0$ for $s\neq0,L$.
  We may further assume that for $s=0,L$ if $\varphi(s)=0$ then $\psi'(s)=0$, otherwise $\Gamma$ is not of class $C^2$ at the point
  \begin{align*}
    \mu(s,\vartheta) = (\varphi(s)\cos\vartheta,\varphi(s)\sin\vartheta,\psi(s)) = (0,0,\psi(s)), \quad \vartheta\in[0,2\pi].
  \end{align*}
  By the arc length parametrization of $\gamma$ we have
  \begin{align} \label{Pf_SRE:Arc}
    \{\varphi'(s)\}^2+\{\psi'(s)\}^2 = 1, \quad s\in[0,L].
  \end{align}
  Let $y=\mu(s,\vartheta)\in\Gamma$.
  We suppress the arguments of $\mu$ and its derivatives.
  From
  \begin{align} \label{Pf_SRE:Tau}
    \partial_s\mu =
    \begin{pmatrix}
      \varphi'(s)\cos\vartheta \\ \varphi'(s)\sin\vartheta \\ \psi'(s)
    \end{pmatrix}, \quad
    \partial_\vartheta\mu =
    \begin{pmatrix}
      -\varphi(s)\sin\vartheta \\ \varphi(s)\cos\vartheta \\ 0
    \end{pmatrix}
    = e_3\times \mu = w(y)
  \end{align}
  and \eqref{Pf_SRE:Arc} we deduce that
  \begin{align*}
    \partial_s\mu\times\partial_\vartheta\mu = \varphi(s)
    \begin{pmatrix}
      -\psi'(s)\cos\vartheta \\ -\psi'(s)\sin\vartheta \\ \varphi'(s)
    \end{pmatrix}, \quad
    |\partial_s\mu\times\partial_\vartheta\mu| = \varphi(s).
  \end{align*}
  Suppose that $\varphi(s)>0$.
  Without loss of generality, we may assume that the direction of $\partial_s\mu\times\partial_\vartheta\mu$ is the same as that of the unit outward normal $n(y)$.
  Then
  \begin{align} \label{Pf_SRE:Nor}
    n(y) = n(\mu(s,\vartheta)) = \frac{\partial_s\mu\times\partial_\vartheta\mu}{|\partial_s\mu\times\partial_\vartheta\mu|} =
    \begin{pmatrix}
      -\psi'(s)\cos\vartheta \\ -\psi'(s)\sin\vartheta \\ \varphi'(s)
    \end{pmatrix}.
  \end{align}
  We differentiate $n(\mu(s,\vartheta))$ with respect to $\vartheta$ and use \eqref{Pf_SRE:Tau} to get
  \begin{align*}
    \frac{\partial}{\partial \vartheta}\bigl(n(\mu(s,\vartheta))\bigr) =
    \begin{pmatrix}
      \psi'(s)\sin\vartheta \\ -\psi'(s)\cos\vartheta \\ 0
    \end{pmatrix}
    = -\lambda(y)w(y), \quad \lambda(y) := \frac{\psi'(s)}{\varphi(s)}.
  \end{align*}
  Moreover, by \eqref{E:ConDer_Surf} with $y=\mu(s,\vartheta)\in\Gamma$, $-\nabla_\Gamma n=W=W^T$ on $\Gamma$, and \eqref{Pf_SRE:Tau},
  \begin{align*}
    \frac{\partial}{\partial \vartheta}\bigl(n(\mu(s,\vartheta))\bigr) = \frac{\partial}{\partial \vartheta}\bigl(\bar{n}(\mu(s,\vartheta))\bigr) = [\nabla\bar{n}(\mu(s,\vartheta))]^T\partial_\vartheta\mu = -W(y)w(y).
  \end{align*}
  Hence $W(y)w(y)=\lambda(y)w(y)$.
  We also have $e_3\times n(y)=-\lambda(y)w(y)$ by \eqref{Pf_SRE:Tau} and \eqref{Pf_SRE:Nor}.
  Therefore, \eqref{E:SR_Eigen} is valid when $\varphi(s)>0$ (note that we assume $a=e_3$).

  Now for $s=0,L$ suppose that $\varphi(s)=0$.
  Then $\psi'(s)=0$ by our assumption and thus the tangent plane of $\Gamma$ at the point
  \begin{align*}
    y = \mu(s,\vartheta) = (\varphi(s)\cos\vartheta,\varphi(s)\sin\vartheta,\psi(s)) = (0,0,\psi(s)), \quad \vartheta\in[0,2\pi]
  \end{align*}
  is orthogonal to the $x_3$-axis.
  Hence $n(y)=\pm e_3$ and we obtain $a\times n(y)=0$ by the assumption $a=e_3$.
  Moreover, $w(y)=0$ by \eqref{Pf_SRE:Tau} and $\varphi(s)=0$.
  By these facts we conclude that the equalities \eqref{E:SR_Eigen} are valid for any $\lambda(y)\in\mathbb{R}$.
\end{proof}

Let $\mathcal{K}(\Gamma)$ be the space of Killing vector fields on $\Gamma$ given by \eqref{E:Def_Kil}.
We show that $\mathcal{K}(\Gamma)$ agrees with $\mathcal{R}|_\Gamma=\{w|_\Gamma\mid w\in\mathcal{R}\}$ if $\Gamma$ is axially symmetric.

\begin{lemma} \label{L:IR_Kil}
  Suppose that $\Gamma$ is of class $C^5$ and $\mathcal{R}\neq\{0\}$.
  Then $\mathcal{K}(\Gamma)=\mathcal{R}|_\Gamma$.
\end{lemma}

\begin{proof}
  For $w(x)=a\times x+b$ with $a=(a_1,a_2,a_3)^T\in\mathbb{R}^3$ and $b\in\mathbb{R}^3$ we have
  \begin{align*}
    \nabla_\Gamma w = P\nabla w = PA \quad\text{on}\quad \Gamma, \quad A :=
    \begin{pmatrix}
      0 & a_3 & -a_2 \\
      -a_3 & 0 & a_1 \\
      a_2 & -a_1 & 0
    \end{pmatrix}
    = -A^T.
  \end{align*}
  From these equalities and $P^T=P^2=P$ on $\Gamma$ it follows that
  \begin{align*}
    D_\Gamma(w) = P\left\{\frac{\nabla_\Gamma w+(\nabla_\Gamma w)^T}{2}\right\}P = P\left(\frac{A+A^T}{2}\right)P = 0 \quad\text{on}\quad \Gamma.
  \end{align*}
  In particular, $w\in\mathcal{K}(\Gamma)$ if it is tangential on $\Gamma$, which shows $\mathcal{R}|_\Gamma\subset\mathcal{K}(\Gamma)$.

  Suppose that $\Gamma$ is a sphere in $\mathbb{R}^3$.
  By a translation we may assume that the sphere $\Gamma$ is centered at the origin.
  Then since
  \begin{align*}
    \mathcal{R}|_\Gamma = \{w(y)=a\times y,\,y\in\Gamma\mid a\in\mathbb{R}^3\}
  \end{align*}
  is a three-dimensional subspace of $\mathcal{K}(\Gamma)$ and the dimension of $\mathcal{K}(\Gamma)$ is at most three (see e.g. \cite{Pe06}*{Theorem 35}), we have $\mathcal{K}(\Gamma)=\mathcal{R}|_\Gamma$.

  Next suppose that $\Gamma$ is not a sphere.
  Since $\Gamma$ is axially symmetric by $\mathcal{R}\neq\{0\}$ and Lemma \ref{L:IR_Surf}, as in the proof of Lemma \ref{L:SR_Eigen} we may assume that $\Gamma$ is axially symmetric around the $x_3$-axis, i.e.
  \begin{align} \label{Pf_IRK:Incl}
    \{w(y) = c(e_3\times y), \, y\in\Gamma \mid c\in\mathbb{R}\} \subset \mathcal{R}|_\Gamma, \quad e_3 = (0,0,1)^T.
  \end{align}
  We may further assume that $\Gamma$ is a surface of revolution of the form \eqref{Pf_SRE:SoR} with $C^5$ functions $\varphi$ and $\psi$ satisfying \eqref{Pf_SRE:Arc} and $\varphi(s)>0$ for $s\neq0,L$.
  Then the Gaussian curvature of $\Gamma$ is given by (see e.g. \cite{ON06}*{Section 5.7})
  \begin{align} \label{Pf_IRK:Gau}
    K(\mu(s,\vartheta)) = -\frac{\varphi''(s)}{\varphi(s)}, \quad s\in(0,L),\,\vartheta\in[0,2\pi].
  \end{align}
  We use this formula later.
  Differentiating both sides of \eqref{Pf_SRE:Arc} we also have
  \begin{align} \label{Pf_IRK:D_Arc}
    \varphi'(s)\varphi''(s)+\psi'(s)\psi''(s) = 0, \quad s\in(0,L).
  \end{align}
  Let $X\in\mathcal{K}(\Gamma)$ be of the form
  \begin{align*}
    X(\mu(s,\vartheta)) = X^s(s,\vartheta)\partial_s\mu(s,\vartheta)+X^\vartheta(s,\vartheta)\partial_\vartheta\mu(s,\vartheta), \quad s\in[0,L],\,\vartheta\in[0,2\pi].
  \end{align*}
  Note that $X\in C^2(\Gamma,T\Gamma)$ by the $C^5$-regularity of $\Gamma$ and Lemma \ref{L:Kil_Reg}.
  Also,
  \begin{align*}
    (Y\cdot\nabla_\Gamma)X\cdot Z+Y\cdot(Z\cdot\nabla_\Gamma)X = 2D_\Gamma(X)Y\cdot Z = 0 \quad\text{on}\quad \Gamma
  \end{align*}
  for all $Y,Z\in C(\Gamma,T\Gamma)$ by $PY=Y$, $PZ=Z$, and $D_\Gamma(X)=0$ on $\Gamma$.
  We substitute $\partial_s\mu$ and $\partial_\vartheta\mu$ for $Y$ and $Z$ and then use \eqref{Pf_SRE:Arc}, \eqref{Pf_SRE:Tau}, \eqref{Pf_IRK:D_Arc},
  \begin{gather*}
    (\partial_s\mu\cdot\nabla_\Gamma)X = \frac{\partial(X\circ\mu)}{\partial s}= (\partial_sX^s)\partial_s\mu+X^s\partial_s^2\mu+(\partial_sX^\vartheta)\partial_\vartheta\mu+X^\vartheta\partial_s\partial_\vartheta\mu, \\
    \partial_s^2\mu =
    \begin{pmatrix}
      \varphi''(s)\cos\vartheta \\ \varphi''(s)\sin\vartheta \\ \psi''(s)
    \end{pmatrix}, \quad
    \partial_s\partial_\vartheta\mu =
    \begin{pmatrix}
      -\varphi'(s)\sin\vartheta \\ \varphi'(s)\cos\vartheta \\ 0
    \end{pmatrix}, \quad
    \partial_\vartheta^2\mu =
    \begin{pmatrix}
      -\varphi(s)\cos\vartheta \\ -\varphi(s)\sin\vartheta \\ 0
    \end{pmatrix},
  \end{gather*}
  and a similar equality for $(\partial_\vartheta\mu\cdot\nabla_\Gamma)X$ to get
  \begin{align} \label{Pf_IRK:Kil_Eq}
    \partial_sX^s = 0, \quad \partial_\vartheta X^s+\varphi^2\partial_sX^\vartheta = 0, \quad \varphi^2\partial_\vartheta X^\vartheta+\varphi\varphi'X^s = 0.
  \end{align}
  If $X^s\equiv0$ then $X^\vartheta\equiv c$ is constant by the second and third equations of \eqref{Pf_IRK:Kil_Eq} (note that $\varphi>0$ on $(0,L)$ and $X$ is of class $C^2$).
  In this case,
  \begin{align*}
    X(y) = c\partial_\vartheta\mu(s,\vartheta) = c(e_3\times y)\in\mathcal{R}|_\Gamma, \quad y=\mu(s,\vartheta)\in\Gamma
  \end{align*}
  by \eqref{Pf_SRE:Tau} and \eqref{Pf_IRK:Incl}.
  Let us show that each $X\in\mathcal{K}(\Gamma)$ is of this form (here the arguments are essentially the same as in \cite{Ei49}*{Section 74}).
  Assume to the contrary that $X^s\not\equiv0$.
  By the first equation of \eqref{Pf_IRK:Kil_Eq}, $X^s=X^s(\vartheta)$ is independent of $s$.
  Since $X^s$ continuous and $X^s\not\equiv0$, it does not vanish on some open interval $I\subset[0,2\pi]$.
  By the second and third equations of \eqref{Pf_IRK:Kil_Eq},
  \begin{align*}
    \partial_sX^\vartheta(s,\vartheta) = -\frac{\partial_\vartheta X^s(\vartheta)}{\{\varphi(s)\}^2}, \quad \partial_\vartheta X^\vartheta(s,\vartheta) = -\frac{\varphi'(s)X^s(\vartheta)}{\varphi(s)}
  \end{align*}
  for $s\in(0,L)$ and $\vartheta\in[0,2\pi]$ (note that $\varphi(s)>0$ for $s\neq0,L$).
  Since $X$ is of class $C^2$, we have $\partial_\vartheta\partial_s X^\vartheta=\partial_s\partial_\vartheta X^\vartheta$.
  Thus it follows from the above equations that
  \begin{align*}
    \frac{\partial_\vartheta^2 X^s(\vartheta)}{X^s(\vartheta)} = \varphi(s)\varphi''(s)-\{\varphi'(s)\}^2, \quad s\in(0,L),\,\vartheta\in I.
  \end{align*}
  Noting that the left-hand side is independent of $s$ and the function $\varphi$ is of class $C^5$, we differentiate both sides of this equality with respect to $s$ to get
  \begin{align*}
    \varphi(s)\varphi'''(s)-\varphi'(s)\varphi''(s) = 0, \quad s\in(0,L).
  \end{align*}
  Now we observe by this equality and \eqref{Pf_IRK:Gau} that
  \begin{align*}
    \frac{\partial}{\partial s}\bigl(K(\mu(s,\vartheta))\bigr) = -\frac{\varphi(s)\varphi'''(s)-\varphi'(s)\varphi''(s)}{\{\varphi(s)\}^2} = 0, \quad s\in(0,L),\,\vartheta\in[0,2\pi].
  \end{align*}
  By this fact and the continuity of $K$ and $\mu$ on $\Gamma$ and $[0,L]\times[0,2\pi]$, the Gaussian curvature $K$ is constant on the whole closed surface $\Gamma$.
  Thus Liebmann's theorem (see e.g. \cite{ON06}*{Section 6.3, Theorem 3.7}) implies that $\Gamma$ is a sphere, which contradicts with our assumption that $\Gamma$ is not a sphere.
  Hence $\mathcal{K}(\Gamma)$ contains only vector fields of the form $w(y)=c(e_3\times y)$, $y\in\Gamma$ with $c\in\mathbb{R}$.
  By this fact and \eqref{Pf_IRK:Incl} we get $\mathcal{K}(\Gamma)\subset\mathcal{R}|_\Gamma$ and, since $\mathcal{R}|_\Gamma$ is a subspace of $\mathcal{K}(\Gamma)$, we conclude that $\mathcal{K}(\Gamma)=\mathcal{R}|_\Gamma$.
\end{proof}

\begin{remark} \label{R:IR_Kil}
  By the proof of Lemma \ref{L:IR_Kil} we see that
  \begin{itemize}
    \item $\mathcal{R}=\{0\}$ if $\Gamma$ is not axially symmetric,
    \item the dimension of $\mathcal{R}$ is one if $\Gamma$ is axially symmetric but not a sphere, and
    \item the dimension of $\mathcal{R}$ is three if $\Gamma$ is a sphere.
  \end{itemize}
  In particular, if $\Gamma$ is axially symmetric around some line and it is not a sphere, then it is not axially symmetric around other lines (see also Lemma \ref{L:IR_Surf}).
\end{remark}

Now we assume again that $\Gamma$ is of class $C^2$ and take $g_0,g_1\in C^1(\Gamma)$ satisfying \eqref{E:G_Inf}.
Let $\mathcal{R}_0$, $\mathcal{R}_1$, and $\mathcal{R}_g$ be the subspaces of $\mathcal{R}$ given by \eqref{E:Def_Rg} and $\Omega_\varepsilon$ the curved thin domain of the form \eqref{E:CTD_Intro} with boundary $\Gamma_\varepsilon$.
As in Section \ref{SS:Pre_Dom} we scale $g_i$ to assume $|g_i|<\delta$ on $\Gamma$ for $i=0,1$, where $\delta$ is the radius of the tubular neighborhood $N$ of $\Gamma$ given in Section \ref{SS:Pre_Surf}, and thus $\overline{\Omega}_\varepsilon\subset N$ for all $\varepsilon\in(0,1]$.

\begin{lemma} \label{L:IR_CTD}
  For an infinitesimal rigid displacement $w(x)=a\times x+b$ of $\mathbb{R}^3$ with $a,b\in\mathbb{R}^3$ the following conditions are equivalent:
  \begin{itemize}
    \item[(a)] For all $\varepsilon\in(0,1]$ the restriction of $w$ on $\Gamma_\varepsilon$ satisfies $w|_{\Gamma_\varepsilon}\cdot n_\varepsilon=0$ on $\Gamma_\varepsilon$.
    \item[(b)] There exists a sequence $\{\varepsilon_k\}_{k=1}^\infty$ of positive numbers such that
    \begin{align*}
      \lim_{k\to\infty}\varepsilon_k = 0, \quad w|_{\Gamma_{\varepsilon_k}}\cdot n_{\varepsilon_k} = 0 \quad\text{on}\quad \Gamma_{\varepsilon_k} \quad\text{for all}\quad k\in\mathbb{N}.
    \end{align*}
    \item[(c)] The vector field $w$ belongs to $\mathcal{R}_0\cap\mathcal{R}_1$.
  \end{itemize}
\end{lemma}

\begin{proof}
  For $\varepsilon\in(0,1]$ and $i=0,1$ let $\tau_\varepsilon^i$ be given by \eqref{E:Def_NB_Aux}.
  Then
  \begin{align*}
    n_\varepsilon(y+\varepsilon g_i(y)n(y)) = (-1)^{i+1}\frac{n(y)-\varepsilon\tau_\varepsilon^i(y)}{\sqrt{1+\varepsilon^2|\tau_\varepsilon^i(y)|^2}}, \quad y+\varepsilon g_i(y)n(y) \in \Gamma_\varepsilon^i
  \end{align*}
  with $y\in\Gamma$ by Lemma \ref{L:Nor_Bo}.
  Moreover, for $w(x)=a\times x+b$ we have
  \begin{align*}
    w(y+\varepsilon g_i(y)n(y)) = w(y)+\varepsilon g_i(y)\{a\times n(y)\}, \quad \{a\times n(y)\}\cdot n(y) = 0, \quad y\in\Gamma.
  \end{align*}
  Hence the condition $w|_{\Gamma_\varepsilon^i}\cdot n_\varepsilon=0$ on $\Gamma_\varepsilon^i$ is equivalent to
  \begin{align} \label{Pf_IRC:Reduce}
    w|_\Gamma\cdot n-\varepsilon w|_\Gamma\cdot\tau_\varepsilon^i-\varepsilon^2g_i(a\times n)\cdot \tau_\varepsilon^i = 0 \quad\text{on}\quad\Gamma.
  \end{align}
  Now let us prove the lemma.
  The condition (a) clearly implies (b).
  Let us show that (b) yields (c).
  Suppose that (b) is satisfied.
  Then by \eqref{Pf_IRC:Reduce} we have
  \begin{align*}
    w|_\Gamma\cdot n-\varepsilon_kw|_\Gamma\cdot\tau_{\varepsilon_k}^i-\varepsilon_k^2g_i(a\times n)\cdot \tau_{\varepsilon_k}^i = 0 \quad\text{on}\quad\Gamma
  \end{align*}
  for $k\in\mathbb{N}$ and $i=0,1$.
  Letting $k\to\infty$ in this equality we obtain $w|_\Gamma\cdot n=0$ on $\Gamma$ by \eqref{E:Tau_Bound}.
  Hence $w\in\mathcal{R}$.
  Moreover, from the above equalities we deduce that
  \begin{align*}
    w|_\Gamma\cdot\tau_{\varepsilon_k}^i+\varepsilon_kg_i(a\times n)\cdot\tau_{\varepsilon_k}^i = 0 \quad\text{on}\quad \Gamma.
  \end{align*}
  Since $\{\tau_{\varepsilon_k}^i\}_{k=1}^\infty$ converges to $\nabla_\Gamma g_i$ uniformly on $\Gamma$ by \eqref{E:Tau_Diff}, we send $k\to\infty$ in this equality to get $w|_\Gamma\cdot\nabla_\Gamma g_i=0$ on $\Gamma$ for $i=0,1$.
  Thus $w\in\mathcal{R}_0\cap\mathcal{R}_1$, i.e. (c) is valid.

  Let us show that (c) implies (a).
  If $w\equiv0$ then (a) is trivial.
  Suppose that $w\not\equiv0$ belongs to $\mathcal{R}_0\cap\mathcal{R}_1$.
  Let $\varepsilon\in(0,1]$ and $i=0,1$.
  Since the condition $w|_{\Gamma_\varepsilon^i}\cdot n_\varepsilon=0$ on $\Gamma_\varepsilon^i$ is equivalent to \eqref{Pf_IRC:Reduce} and $w\in\mathcal{R}_0\cap\mathcal{R}_1\subset\mathcal{R}$ satisfies $w|_\Gamma\cdot n=0$ on $\Gamma$, it is sufficient for (a) to show that
  \begin{align} \label{Pf_IRC:Goal}
    w(y)\cdot\tau_\varepsilon^i(y) = 0, \quad \{a\times n(y)\}\cdot\tau_\varepsilon^i(y) = 0 \quad\text{for all}\quad y\in\Gamma.
  \end{align}
  Hereafter we fix and suppress the argument $y$.
  If $w=0$, then $a\times n=0$ by \eqref{E:SR_Eigen} and the equalities \eqref{Pf_IRC:Goal} are valid (note that we can apply Lemma \ref{L:SR_Eigen} by $w\not\equiv0$).
  Suppose $w\neq 0$.
  Then $w$ is the eigenvector of $W$ corresponding to the eigenvalue $\lambda$ by \eqref{E:SR_Eigen}.
  Since $W$ has the eigenvalues $\kappa_1$, $\kappa_2$, and zero with $Wn=0$ and $w\neq n$ by $w\cdot n=0$, we have $\lambda=\kappa_1$ or $\lambda=\kappa_2$.
  Without loss of generality, we may assume $\lambda=\kappa_1$, i.e. $Ww=\kappa_1w$.
  Then since
  \begin{align*}
    (I_3-\varepsilon g_iW)w = (1-\varepsilon g_i\kappa_1)w, \quad 1-\varepsilon g_i\kappa_1 > 0,
  \end{align*}
  and $I_3-\varepsilon g_iW$ is invertible by $|g_i|<\delta$ on $\Gamma$, \eqref{E:Curv_Bound}, and Lemma \ref{L:Wein},
  \begin{align} \label{Pf_IRC:WRe_W}
    (I_3-\varepsilon g_iW)^{-1}w = (1-\varepsilon g_i\kappa_1)^{-1}w.
  \end{align}
  We use \eqref{E:Def_NB_Aux}, the symmetry of $W$, \eqref{Pf_IRC:WRe_W}, and $w\cdot\nabla_\Gamma g_i=0$ by $w\in\mathcal{R}_i$ to get
  \begin{align} \label{Pf_IRC:W_Tau}
    w\cdot\tau_\varepsilon^i = (I_3-\varepsilon g_iW)^{-1}w\cdot\nabla_\Gamma g_i = (1-\varepsilon g_i\kappa_1)^{-1}(w\cdot\nabla_\Gamma g_i) = 0.
  \end{align}
  Moreover, by \eqref{E:SR_Eigen} with $\lambda=\kappa_1$ and \eqref{Pf_IRC:WRe_W},
  \begin{align*}
    (I_3-\varepsilon g_iW)^{-1}(a\times n) = -\kappa_1(I_3-\varepsilon g_iW)^{-1}w = -\kappa_1(1-\varepsilon g_i\kappa_1)^{-1}w.
  \end{align*}
  Using this equality we get $(a\times n)\cdot\tau_\varepsilon^i=0$ as in \eqref{Pf_IRC:W_Tau}.
  Thus \eqref{Pf_IRC:Goal} holds and we conclude that $w|_{\Gamma_\varepsilon^i}\cdot n_\varepsilon=0$ on $\Gamma_\varepsilon^i$ for all $\varepsilon\in(0,1]$ and $i=0,1$, i.e. (a) is valid.
\end{proof}

By Lemmas \ref{L:IR_Surf} and \ref{L:IR_CTD} we observe that the nontriviality of $\mathcal{R}_0\cap\mathcal{R}_1$ implies the uniform axial symmetry of $\Omega_\varepsilon$.

\begin{lemma} \label{L:CTD_AS}
  If there exists a vector field $w(x)=a\times x+b\in\mathcal{R}_0\cap\mathcal{R}_1$ such that $w\not\equiv0$, then $a\neq 0$, $a\cdot b=0$, and $\Omega_\varepsilon$ is axially symmetric around the line parallel to $a$ and passing through $b_a=|a|^{-2}(a\times b)$ for all $\varepsilon\in(0,1]$.
\end{lemma}

\begin{proof}
  Let $w(x)=a\times x+b\in\mathcal{R}_0\cap\mathcal{R}_1$.
  Then $w|_{\Gamma_\varepsilon^i}\cdot n_\varepsilon=0$ on $\Gamma_\varepsilon^i$ for each $\varepsilon\in(0,1]$ and $i=0,1$ by Lemma \ref{L:IR_CTD}.
  Hence if $w\not\equiv0$ then Lemma \ref{L:IR_Surf} implies that $a\neq0$, $a\cdot b=0$, and both $\Gamma_\varepsilon^0$ and $\Gamma_\varepsilon^1$ are axially symmetric around the line parallel to $a$ and passing through $b_a$, which yields the same axial symmetry of $\Omega_\varepsilon$.
\end{proof}

Next we see that the triviality of $\mathcal{R}_g$ yields the axial asymmetry of $\Omega_\varepsilon$.

\begin{lemma} \label{L:CTD_Rg}
  If $\mathcal{R}_g=\{0\}$, then there exists a constant $\tilde{\varepsilon}\in(0,1]$ such that $\Omega_\varepsilon$ is not axially symmetric around any line for all $\varepsilon\in(0,\tilde{\varepsilon}]$.
\end{lemma}

\begin{proof}
  We prove the contrapositive statement: if there exists a sequence $\{\varepsilon_k\}_{k=1}^\infty$ convergent to zero such that $\Omega_{\varepsilon_k}$ is (and thus $\Gamma_{\varepsilon_k}^0$ and $\Gamma_{\varepsilon_k}^1$ are) axially symmetric around some line $l_k$ for each $k\in\mathbb{N}$, then $\mathcal{R}_g\neq\{0\}$.

  Suppose that such a sequence $\{\varepsilon_k\}_{k=1}^\infty$ exists and that for each $k\in\mathbb{N}$ the line $l_k$ is of the form $l_k=\{sa_k+b_k\mid s\in\mathbb{R}\}$ with $a_k,b_k\in\mathbb{R}^3$, $a_k\neq0$, i.e. $l_k$ is parallel to $a_k$ and passing through $b_k$.
  Replacing $a_k$ with $a_k/|a_k|$ we may assume $a_k\in S^2$ for all $k\in\mathbb{N}$ without changing $l_k$ (here $S^2$ is the unit sphere in $\mathbb{R}^3$).
  Since $\Omega_\varepsilon$ is contained in the bounded set $N$ for all $\varepsilon\in(0,1]$, there exists an open ball $B_R$ centered at the origin of radius $R>0$ such that $\Omega_{\varepsilon_k}\subset B_R$ for all $k\in\mathbb{N}$.
  Then, by the axial symmetry of $\Omega_{\varepsilon_k}$ around the line $l_k$, the intersection $l_k\cap B_R$ is not empty: otherwise the ball generated by the rotation of $B_R$ through the angle $\pi$ around $l_k$ does not intersect with $B_R$ and thus $\Omega_{\varepsilon_k}\subset B_R$ is not axially symmetric around $l_k$.
  Hence we may assume $b_k\in l_k\cap B_R$ for all $k\in\mathbb{N}$ by replacing $b_k$ with $b_k-sa_k$ for an appropriate $s\in\mathbb{R}$.
  Now $\{a_k\}_{k=1}^\infty$ and $\{b_k\}_{k=1}^\infty$ are bounded in $\mathbb{R}^3$ and thus converge (up to subsequences) to some $a\in S^2$ and $b\in\mathbb{R}^3$, respectively.

  Let us prove $w(x):=a\times(x-b)\in\mathcal{R}_g$.
  For $k\in\mathbb{N}$ and $i=0,1$ let $\tau_{\varepsilon_k}^i$ be the vector field on $\Gamma$ given by \eqref{E:Def_NB_Aux} and $w_k(x):=a_k\times(x-b_k)$, $x\in\mathbb{R}^3$.
  Then
  \begin{align} \label{Pf_CRg:Lim_Wk}
    \lim_{k\to\infty}\tau_{\varepsilon_k}^i(y) = \nabla_\Gamma g_i(y), \quad \lim_{k\to\infty}w_k(y) = w(y) \quad\text{for all}\quad y\in\Gamma
  \end{align}
  by \eqref{E:Tau_Diff}, $\lim_{k\to\infty}a_k=a$, and $\lim_{k\to\infty}b_k=b$.
  For each $k\in\mathbb{N}$ and $i=0,1$, since $\Gamma_{\varepsilon_k}^i$ is axially symmetric around the line $l_k$, Lemma \ref{L:IR_Surf} implies that $w_k|_{\Gamma_{\varepsilon_k}^i}\cdot n_{\varepsilon_k}=0$ on $\Gamma_{\varepsilon_k}^i$.
  By the proof of Lemma \ref{L:IR_CTD} (see \eqref{Pf_IRC:Reduce}) this condition is equivalent to
  \begin{align} \label{Pf_CRg:Reduce}
    w_k|_\Gamma\cdot n-\varepsilon_k w_k|_\Gamma\cdot\tau_{\varepsilon_k}^i-\varepsilon_k^2g_i(a_k\times n)\cdot\tau_{\varepsilon_k}^i = 0 \quad\text{on}\quad \Gamma.
  \end{align}
  Letting $k\to\infty$ in \eqref{Pf_CRg:Reduce} we get $w|_\Gamma\cdot n=0$ on $\Gamma$ by \eqref{Pf_CRg:Lim_Wk} and $\lim_{k\to\infty}a_k=a$.
  Thus $w\in\mathcal{R}$.
  Next we subtract \eqref{Pf_CRg:Reduce} for $i=1$ from that for $i=0$ and divide the resulting equality by $\varepsilon_k$.
  Then since $w_k|_\Gamma\cdot n$ does not depend on $i$ we have
  \begin{align*}
    w_k|_\Gamma\cdot(\tau_{\varepsilon_k}^1-\tau_{\varepsilon_k}^0)+\varepsilon_k(a_k\times n)\cdot(g_1\tau_{\varepsilon_k}^1-g_0\tau_{\varepsilon_k}^0) = 0 \quad\text{on}\quad \Gamma.
  \end{align*}
  We send $k\to\infty$ in this equality and use \eqref{Pf_CRg:Lim_Wk} and $\lim_{k\to\infty}a_k=a$ to obtain
  \begin{align*}
    w|_\Gamma\cdot(\nabla_\Gamma g_1-\nabla_\Gamma g_0) = w|_\Gamma\cdot\nabla_\Gamma g = 0 \quad\text{on}\quad \Gamma.
  \end{align*}
  This shows $w\in\mathcal{R}_g$.
  Since $w\not\equiv0$ by $a\in S^2$, we conclude that $\mathcal{R}_g\neq\{0\}$.
\end{proof}

Finally, we show that the restriction on $\Omega_\varepsilon$ of a vector field in $\mathcal{R}_0\cap\mathcal{R}_1$ belongs to the solenoidal space $L_\sigma^2(\Omega_\varepsilon)=\{u\in L^2(\Omega_\varepsilon)^3\mid\text{$\mathrm{div}\,u=0$ in $\Omega_\varepsilon$, $u\cdot n_\varepsilon=0$ on $\Gamma_\varepsilon$}\}$.

\begin{lemma} \label{L:IR_Sole}
  The inclusion $\mathcal{R}_0\cap\mathcal{R}_1\subset L_\sigma^2(\Omega_\varepsilon)$ holds for each $\varepsilon\in(0,1]$.
\end{lemma}

\begin{proof}
  Let $w(x)=a\times x+b\in\mathcal{R}_0\cap\mathcal{R}_1$.
  Then $\mathrm{div}\,w=0$ in $\mathbb{R}^3$ by direct calculations.
  Moreover, $w|_{\Gamma_\varepsilon}\cdot n_\varepsilon=0$ on $\Gamma_\varepsilon$ by Lemma \ref{L:IR_CTD}.
  Hence $w\in L_\sigma^2(\Omega_\varepsilon)$.
\end{proof}

\section*{Acknowledgments}
This work is an expanded version of a part of the doctoral thesis of the author \cite{Miu_DT} completed under the supervision of Professor Yoshikazu Giga at the University of Tokyo.
The author is grateful to him for his valuable comments on this work and also would like to thank Mr. Yuuki Shimizu for fruitful discussions on Killing vector fields on surfaces.

The work of the author was supported by Grant-in-Aid for JSPS Fellows No. 16J02664 and No. 19J00693, and by the Program for Leading Graduate Schools, MEXT, Japan.


\end{document}